\documentclass[aos,preprint]{imsart}

\usepackage{amsmath, amsthm, amsfonts, graphicx, extarrows, bbm, amssymb}
\usepackage{tikz}
\usepackage[hmargin=1.2in, bmargin=1.5in]{geometry}
\usepackage{enumerate}
\usepackage[colorlinks,citecolor=blue,urlcolor=blue]{hyperref}

\graphicspath{
{./fig/finite_large_noise/}
{./fig/finite_large_samp/}
{./fig/theo_ROC_cpr/}
{./fig/theo_ROC_2S_sis_cpr/}
{./fig/corr_design/}
{./fig/revision_fig/}
}
\numberwithin{figure}{section}
\numberwithin{table}{section}

\startlocaldefs

\newcommand{\argmin}[1]{\underset{#1}{\text{argmin}}\;}
\newcommand{\follow}{\overset{i.i.d.}{\sim}}
\newcommand{\esssup}{\text{esssup}}

\newcommand{\sgn}{\text{sgn}}
\newcommand{\1}[1]{\mathbbm{1}_{\{#1\}}}
\newcommand{\ave}{\text{Ave}} 

\newcommand{\fdp}{\text{FDP}}
\newcommand{\tpp}{\text{TPP}}
\newcommand{\afdp}{\text{AFDP}}
\newcommand{\atpp}{\text{ATPP}}

\newcommand{\amse}{\mathrm{AMSE}}

\newcommand{\lasso}{\text{LASSO}}

\newcommand{\aeps}{b_\epsilon}

\newcommand{\aGoChStar}{\frac{b_\epsilon \tilde{G}}{\alpha_*^{\frac{1}{2-q}}}}

\newcommand{\ZoChiStarTaus}{\frac{\tau_* Z}{\alpha_*^{\frac{1}{2-q}}}}

\newcommand{\rightarrowas}{\xrightarrow{a.s.}}

\numberwithin{equation}{section}

\newtheorem{theorem}{Theorem}[section]
\newtheorem{corollary}{Corollary}[section]
\newtheorem{lemma}{Lemma}[section]
\newtheorem{remark}{Remark}[section]

\theoremstyle{definition}
\newtheorem{definition}{Definition}[section]
\theoremstyle{plain}

\endlocaldefs

\begin{document}

\begin{frontmatter}

\title{Which bridge estimator is The Best For Variable Selection?}
\runtitle{Which Regularizer is Optimal For Variable Selection?}


\begin{aug}
\author{\fnms{Shuaiwen} \snm{Wang}\thanksref{m1} \ead[label=e1]{sw2853@columbia.edu}}
\author{\fnms{Haolei} \snm{Weng}\thanksref{m2} \ead[label=e2]{hw2375@columbia.edu}}
\and
\author{\fnms{Arian} \snm{Maleki}\thanksref{m3} \ead[label=e3]{arian@stat.columbia.edu}}

\runauthor{S. Wang, H. Weng, A. Maleki}

\affiliation{Columbia University \thanksmark{m1}\thanksmark{m2}\thanksmark{m3} }

\address{S. Wang\\
Department of Statistics \\
Columbia University \\
1255 Amsterdam Avenue \\
New York, NY, 10027 \\
USA \\
\printead{e1}\\
}

\address{H. Weng\\
Department of Statistics \\
Columbia University \\
1255 Amsterdam Avenue \\
New York, NY, 10027 \\
USA \\
\printead{e2}\\
}

\address{A. Maleki \\
Department of Statistics \\
Columbia University \\
1255 Amsterdam Avenue \\
New York, NY, 10027 \\
USA \\
\printead{e3}\\
}

\end{aug}

\begin{abstract}

We study the problem of variable selection for linear models under the
high-dimensional asymptotic setting, where the number of observations $n$ grows
at the same rate as the number of predictors $p$. We consider two-stage
variable selection techniques (TVS) in which the first stage uses bridge
estimators to obtain an estimate of the regression coefficients, and the second
stage simply thresholds this estimate to select the
``important" predictors. The asymptotic false discovery proportion ($\afdp$)
and true positive proportion (ATPP) of these TVS are evaluated. We prove that
for a fixed ATPP, in order to obtain a smaller $\afdp$, one should pick a
bridge estimator with smaller asymptotic mean square error in the first stage
of TVS. Based on such principled discovery, we present a sharp comparison of
different TVS, via an in-depth investigation of the estimation properties of
bridge estimators. Rather than ``order-wise" error bounds with loose constants,
our analysis focuses on precise error characterization. Various interesting signal-to-noise
ratio and sparsity settings are studied. Our results offer new and thorough 
insights into high-dimensional variable selection. For instance, we prove that
a TVS with Ridge in its first stage outperforms TVS with other bridge
estimators in large noise settings; two-stage LASSO becomes inferior when the
signal is rare and weak. As a by-product, we show that two-stage
methods outperform some standard variable selection techniques, such as
$\lasso$ and Sure Independence Screening, under certain conditions.
\end{abstract}



\end{frontmatter}

\section{Introduction}

\subsection{Motivation and problem statement}\label{ssec:motive}
Although linear models can be traced back to two hundred years ago, they keep shining in the modern statistical research. A problem of major interest in this literature is {\em variable selection}. Consider the linear regression model
\begin{equation*}\label{intro:problem:describe}
y=X\beta+w,
\end{equation*}
with $y\in\mathbb{R}^n$, $X\in\mathbb{R}^{n\times p}$, $\beta\in\mathbb{R}^p$
and $w \in \mathbb{R}^n$. Suppose only a few elements of $\beta$ are nonzero.
The problem of variable selection is to find these nonzero
locations of $\beta$.  Motivated by the concerns about the instability and high computational cost of classical variable selection techniques, such as best subset selection and stepwise selection, Tibshirani proposed $\lasso$  \cite{tibshirani1996regression} to perform parameter estimation and variable selection simultaneously. The $\lasso$ estimate is given by
\begin{equation} \label{lasso96}
\hat{\beta} (1, \lambda) 
:=
\argmin{\beta}\frac{1}{2}\|y-X\beta\|_2^2+\lambda\|\beta\|_1,
\end{equation}
where $\lambda\in(0,\infty)$ is the tuning parameter, and $\|\cdot\|_1$ is the
$\ell_1$ norm. The regularization term $\|\beta\|_1$ stablizes the variable
selection process while the convex formulation of \eqref{lasso96} reduces the computational cost.

Compared to $\lasso$, other convex regularizers such as $\|\beta\|_2^2$ imposes larger penalty to
large components of $\beta$. Hence, their estimates might be more stable than
$\lasso$. Even though the solutions of many of these regularizers are not
sparse (and thus not automatically perform variable selection), we may
threshold their estimates to select variables. This observation leads us to the
following questions: can such two-stage methods with other
regularizers outperform $\lasso$ in variable selection? If so, which
regularizer should be used in the first stage? The goal of this paper is to
address these questions. In particular, we study the performances of the
two-stage variable selection (TVS) techniques mentioned above, with the first stage
based on the class of bridge estimators \cite{frank1993statistical}:
\begin{equation}
\hat{\beta}(q, \lambda)
:=
\argmin{\beta}\frac{1}{2}\|y-X\beta\|_2^2+\lambda\|\beta\|_q^q, \label{eq:bridge regression problem}
\end{equation}
where $\|\beta\|_q^q=\sum_i |\beta_i|^q$ with $q\geq 1$. Our variable selection technique  takes $\hat{\beta}(q, \lambda)$ and returns the sparse estimate $\bar{\beta} (q, \lambda,s )$ defined as follows:
\begin{equation*}
\bar{\beta} (q, \lambda, s) = \eta_0 (\hat{\beta}(q, \lambda); s^2/2),
\end{equation*}
where $\eta_0(u; \chi) = u\1{|u|\geq \sqrt{2\chi}}$ denotes the hard threshold
function and it operates on a vector in a component-wise manner. The nonzero
elements of $\bar{\beta}(q, \lambda, s)$ are used as selected
variables. In this paper, we give a thorough
investigation of such TVS techniques under the asymptotic setting
$n/p\rightarrow \delta \in (0,\infty)$. Specifically the following fundamental
questions are addressed:
\begin{center}
\begin{tabular}{p{14cm}}
    \textit{
        Which value of $q$ offers the best variable selection
        performance? Does $\lasso$ outperform the two-stage methods based
        on other bridge estimators? What is the impact of the
        signal-to-noise ratio (SNR) and the sparisty level on the optimal choice of
        $q$?
    }
\end{tabular}
\end{center}

\subsection{Our Contribution}
Different from most of the previous works, our study adopts a high-dimensional
regime in which variable selection consistency is unattainable. Under our
asymptotic framework, we are able to obtain a sharp characterization of the
variable selection ``error'' (we will clarify our definition of this error in
Section \ref{sec:asymptoticanalysisDef}). The \emph{asymptotically exact expressions} we derive for the error
 open a new way for comparing the
aforementioned variable selection techniques accurately.

It turns out that the variable selection performance of TVS is closely
connected with the estimation quality of the bridge estimator in the first
stage; a bridge estimator with a smaller asymptotic mean square error (AMSE) 
in the first stage offers a better variable selection performance in the TVS. This novel
observation enables us to connect and translate the study of TVS to the comparison of the
estimation accuracy of different bridge estimators.

Due to the nature of different $\ell_q$ regularizers, each bridge estimator
has its own strength under different model settings. To clarify the
strength and weakness of different bridge estimators, we study and compare
their AMSE under the following important scenarios: (i) rare signal scenario; (ii) large noise scenario; 
(iii) large sample scenario. For the first two, new
phenomena are discovered: the Ridge estimator is optimal among all the bridge
estimators in large noise settings; in the setting of rare signals, LASSO achieves the 
best performance when the signal strength exceeds a certain level.
However for signals below that level, other bridge estimators may outperform 
LASSO. In the large sample scenario, we connect our analyese with the
fruits of the classical low-dimensional asymptotic studies. We will provide new comparison results not 
available in classical asymptotic analyses of bridge estimators.

In summary, our studies reveal the intricate impact of
the combination of SNR and sparsity level on the estimation
of the coefficients. New insights into high-dimensional variable selection are
discovered. We present our contributions more formally in Section \ref{sec:contribution}.

\subsection{Related Work}\label{ssec:relatedwork}
The literature on variable selection is very rich. Hence, the related works we choose to discuss can only be illustrative rather than exhaustive.  

Traditional methods of variable selection include  best subset selection and stepwise procedures. Best subset selection suffers from high computational complexity and high variance. The greedy nature of stepwise procedures reduces the computational complexity, but limits the number of models that are checked by such procedures. See \cite{miller2002subset} for a comprehensive treatment of classical subset selection. To overcome these limitations, \cite{tibshirani1996regression} proposed the $\lasso$ that aims to perform variable selection and parameter estimation simultaneously.  Both the variable selection and estimation performance of $\lasso$ have been studied extensively in the past decade. It has been justified in the works of
\cite{meinshausen2006high, zhao2006model, zhang2008sparsity} that
a type of ``irrepresentable condition'' is almost sufficient and necessary to
guarantee sign consistency for the $\lasso$. Later  \cite{wainwright2009sharp}
established sharp conditions under which $\lasso$ can perform a consistent
variable selection. One implication of  \cite{wainwright2009sharp} that is
relevant to our paper is that, consistent variable selection is impossible
under the linear asymptotic regime\footnote{Throughout the paper, the linear asymptotic is referred to  the asymptotic setting with (a) and (b) in Definition 2.1 satisfied. Typically in this case, we have $n$, $p$ and the number of nonzero coefficients $k$ go to
infinity proportionally.} that we consider in this paper. This result is consistent with that of \cite{su2015false} and our paper. Hence, we should expect that both the true positive proportion (TPP) and false discovery proportion (FDP) play a major role in our analyses and comparisons. It is worth mentioning that the rate of convergence for variable selection under Hamming loss has been studied in a sequel of works \cite{genovese2012comparison, ke2014, jin2014optimality, butucea2018variable}.

Since $\lasso$ requires strong conditions for variable selection consistency,
several authors have considered a few variants, such as adaptive $\lasso$
\cite{zou2006adaptive} and thresholded $\lasso$ \cite{meinshausen2009lasso}.
Thresholded $\lasso$ is an instance of two-stage variable selection schemes we
study in this paper. \cite{meinshausen2009lasso} proved that thresholded
$\lasso$ offers a variable selection consistency under weaker conditions than
the irrepresentable condition required by $\lasso$. As we will see later, even
the thresholded $\lasso$ does not obtain variable selection consistency under
the asymptotic framework of this paper. However, we will show that it
outperforms the $\lasso$ in variable selection. Other authors have also studied
two-step or even multi-step variable selection schemes in the
hope of weakening the required conditions \cite{zhou2009thresholding,
zhang2009some, luo2014sequential, weng2017regularization}. Note that none of these methods provide
consistent variable selection under the linear asymptotic setting we consider
in this paper.  Study and comparison of these other schemes under our
asymptotic setting is an interesting open problem for future research. 

A more delicate study of the $\lasso$ estimator and more generally the bridge estimators is necessary for an accurate analysis of two-stage methods under the linear asymptotic regime. Our analysis relies on the recent results in the study of bridge estimators \cite{donoho2009message, donoho2011noise, bayati2011dynamics, bayati2012lasso, weng2018overcoming, maleki2013asymptotic, su2015false}. These papers use the platform offered by approximate message passing (AMP) to characterize sharp asymptotic properties. In particular, the most relevant work to our paper is  \cite{su2015false} which studies the solution path of LASSO through the trade-off diagram of the asymptotic $\fdp$ and $\tpp$. The present paper makes further steps in the analysis of bridge estimator based two-stage methods under various interesting signal-to-noise ratio settings that have not been considered in \cite{su2015false}.

Another line of two-stage methods is the idea of screening \cite{fan2008sure,
wasserman2009high, ji2012ups, cho2012high}. For instance, in \cite{fan2008sure} a preliminary estimate
of the $j$th regression coefficient is obtained by regressing $y$ on only the
$j$th predictor. Then a hard threshold function is applied to all the estimates
to infer the location of the non-zero coefficients. As we will discuss in
Section \ref{sec:surescreen}, this approach is a special form of our TVS with a
debiasing performed in the first stage, and hence our variable selection
technique under appropriate tuning outperforms Sure Independence Screening of
\cite{fan2008sure}. Compared to Sure Independence Screening, the work of
\cite{wasserman2009high} uses more complicated estimators  in the first stage,
which is more aligned to our approach. However, \cite{wasserman2009high}
requires data splitting. While  this data splitting achieves certain
theoretical improvement, in practice (especially in high-dimensions) this may
degrade the performance of a variable selection technique. In this paper, we
avoid data splitting. We should also mention that two-stage or
multi-stage methods (that have a thresholding step) are also popular for
estimation purposes. See for instance \cite{yang2014elementary}. Due to limited space, the current paper will be focused on variable selection and not discuss the estimation
performance of TVS. However, an accurate analysis of multi-stage estimation techniques is
an interesting problem to study. 

Finally, there exists one stream of research with emphasis on the derivation of sufficient and necessary conditions for variable selection consistency under different types of restrictions on the model parameters \cite{fletcher2009necessary, wainwright2009information, aeron2010information, wang2010information, rad2011nearly, david2017high, ndaoud2018optimal}. These works typically assume that all the entries of the design matrix $X$ and error vector $w$ are independent zero-mean Gaussian, with which they are able to obtain accurate information theoretical thresholds and phase transition for exact support recovery of the coefficients $\beta$. We refer to \cite{ndaoud2018optimal} for a detailed discussion of such results. As will be shown shortly in Section \ref{sec:asymptoticanalysisDef}, we make the same assumption on the design $X$, but allow much weaker conditions on the error term $w$. More importantly, we push the analysis one step further by analyzing a class of TVS when exact recovery is impossible information theoretically. 

\section{Our Asymptotic Framework and Some Preliminaries}\label{sec:asymptoticanalysisDef}

\subsection{Asymptotic framework}
In this section, we review the asymptotic framework under which our studies are performed. We start with the definition of a converging sequence adapted from  \cite{bayati2012lasso}. 

\begin{definition}\label{definition:asymptotics}
The sequence of instances $\{\beta(p), w(p), X(p)\}_{p\in\mathbb{N}}$, indexed by $p$, is said to be a standard converging sequence if
\begin{itemize}
 \item[(a)] $n=n(p)$ such that $\frac{n}{p}\rightarrow\delta\in(0,\infty)$.

\item[(b)] The empirical distribution of the entries of $\beta(p)$ converges
    weakly to a probability measure $p_{B}$ on $\mathbb{R}$ with finite second moment. Further, $\frac{1}{p}\sum_{i=1}^p \beta_i(p)^2$ converges to the second moment of $p_B$; and $\frac{1}{p}\sum_{i=1}^p\mathbb{I}(\beta_i(p)=0)\rightarrow p_B(\{0\})$.

\item[(c)] The empirical distribution of the entries of $w(p)$ converges weakly to a zero-mean distribution with variance $\sigma^2$. Furthermore, $\frac{1}{n}\sum_{i=1}^n w_i(p)^2 \rightarrow \sigma^2$. 

\item[(d)] $X_{ij}(p) \overset{i.i.d.}{\sim} N \left(0, \frac{1}{n} \right)$. 
\end{itemize}
\end{definition}

The asymptotic scaling $n/p\rightarrow\delta$ specified in Condition (a) was proposed by Huber in 1973 \cite{huber1973robust}, and has become one of the most popular asymptotic settings especially for studying problems with moderately large dimensions \cite{el2010high, el2013robust, donoho2016high, sur2017likelihood, dobriban2018high, sur2018modern}. Regarding Condition (b), suppose the entries of $\beta(p)$ form a stationary ergodic sequence with marginal distribution determined by some probability measure $p_B$. According to Birkhoff's ergodic theorem, it is clear that Condition (b) will hold almost surely. Thus Condition (b) can be considered as a weaker notion of this Bayesian set-up. Similar interpretation works for Condition (c). Regarding Condition (d), as discussed in Section \ref{ssec:relatedwork}, many related works assume it as well. Moreover, we would like to point out that there are a lot of empirical and a few theoretical studies revealing the universal behavior of i.i.d. Gaussian design matrices over a wider class of distributions. See \cite{bayati2015universality} and references therein. Hence, the Gaussianity of the design does not play a critical role in our final results. The numerical studies  presented in Section \ref{ssec:general-design} confirm this claim.
The independence assumption of the design entries is critical for our analysis. Given that our analyses for i.i.d. matrices are already complicated, and the obtained results are highly non-trivial (as will be seen in Section \ref{sec:contribution}), we leave the study of general design matrices for a future research. However, the numerical studies performed in Section \ref{ssec:general-design} imply that the main conclusions of our paper are valid even when the design matrix is correlated.

 In the rest of the paper, we assume the vector of regression coefficients $\beta$ is sparse. More specifically, we assume $p_B=(1-\epsilon)\delta_0+\epsilon p_G$, where $\delta_0$ denotes a point mass at 0 and $p_G$ is a probability measure without any point mass at 0. Accordingly, the mixture proportion $\epsilon$ represents the sparsity level of $\beta(p)$ in the converging sequence. Throughout the paper, $B$ and $G$ will be used as random variables with distribution specified by $p_B$ and $p_G$, respectively. $Z$ represents a standard normal random variable. Subscripts like $i$ attached to a vector are used to denote its $i$th component.
 The \textit{asymptotic mean square error} (AMSE) of the bridge estimator $\hat{\beta}(q,\lambda)$ is defined as the almost sure limit
\begin{equation}
{\rm AMSE}(q,\lambda) \triangleq \lim_{p \rightarrow \infty} \frac{1}{p} \|\hat{\beta}(q, \lambda) - \beta\|_2^2. \label{eq:def:amse}
\end{equation}
According to \cite{bayati2011dynamics, weng2018overcoming}, AMSE$(q,\lambda)$ is well defined for $q\in [1,\infty)$ and $\lambda >0$. In this paper, one of our focuses will be on bridge estimators with optimal tuning $\lambda^*_q$ defined as
\[
\lambda_q^*\triangleq \arg\min_{\lambda > 0}\amse(q, \lambda).
\]
Further, we denote the thresholded estimators as 
\[
\bar{\beta}(q, \lambda, s) = \eta_0(\hat{\beta}(q, \lambda); s^2/2)=\hat{\beta}(q,\lambda)\1{|\hat{\beta}(q, \lambda)|\geq s}. 
\]
Since under our asymptotic setting the exact recovery of the non-zero locations
of $\beta$ is impossible \cite{wainwright2009sharp, reeves2013approximate}, we
expect to observe both false positives and false negatives. Hence, for a given sparse
estimator $\hat{\beta}$, we follow \cite{su2015false} and measure its variable
selection performance by the false discovery proportion ($\fdp$) and true
positive proportion ($\tpp$), defined as: 
\begin{equation*}
\fdp(\hat{\beta})=\frac{\#\{i: \hat{\beta}_i\neq0,\beta_i=0\}}{\#\{i: \hat{\beta}_i\neq0\}}, \quad \tpp(\hat{\beta})=\frac{\#\{i:\hat{\beta}_i\neq 0,\beta_i \neq 0\}}{\#\{i: \beta_i \neq 0\}}.
\end{equation*}
In particular, our study will focus on the
asymptotic version of FDP and TPP for the LASSO estimate $\hat{\beta}(1,\lambda)$ and
thresholded estimators $\bar{\beta}(q, \lambda, s)$. We define (the
limits are in almost surely senses)
\begin{eqnarray*}
\hspace{0.2cm} {\rm AFDP}(1,\lambda)=\lim_{p \rightarrow \infty} {\rm FDP}(\hat{\beta}(1,\lambda)),~~ {\rm AFDP}(q,\lambda, s)=\lim_{p \rightarrow \infty} {\rm FDP}(\bar{\beta}(q,\lambda,s)).
\end{eqnarray*}
Similar definitions are used for ${\rm ATPP}(1,\lambda)$ and $ {\rm
ATPP}(q,\lambda, s)$.  The following result adapted from
\cite{bogdan2013supplementary} characterizes the AFDP and ATPP for $\lasso$.

\begin{lemma} \label{CITE:WEIJIE}
For any given $\lambda>0$, almost surely
\begin{eqnarray}
{\rm AFDP}(1, \lambda)
&=&
\frac{(1-\epsilon)\mathbb{P}(|Z|>\alpha)}{(1-\epsilon)\mathbb{P}(|Z|>\alpha)+\epsilon\mathbb{P}(|G+\tau Z| > \alpha \tau)}, \nonumber \\
{\rm ATPP}(1, \lambda)
&=&\mathbb{P}(|G+\tau Z| > \alpha\tau), \label{lemma:afdp:atpp:l1}
\end{eqnarray}
where $(\alpha, \tau)$ is the unique solution to the following equations with $q=1$:
\begin{eqnarray} \label{eq:fixedpointfirstappearance}      
\tau^2
&=&
\sigma^2+\frac{1}{\delta}\mathbb{E}\left(\eta_q(B+\tau Z;\alpha\tau^{2-q})-B\right)^2, \label{state_evolution1} \\
\lambda
&=&
\alpha\tau^{2-q}\big(1-\frac{1}{\delta}\mathbb{E}\eta_q'(B+\tau Z;\alpha\tau^{2-q})\big), \label{state_evolution2}
\end{eqnarray}
with $\eta_q (\cdot; \cdot)$ being the proximal operator defined as
\begin{equation*}
\eta_q (u; \chi) = \arg\min_z \frac{1}{2} (u - z)^2 + \chi |z|^q,
\end{equation*}
and $\eta_q'(\cdot;\cdot)$ being the derivative of $\eta_q$ with respect to its first argument.
\end{lemma}
The formulas in this lemma have been derived in terms of convergence in probability in \cite{bogdan2013supplementary}. The extension to almost sure convergence is straightforward and is hence skipped. See Appendix C.1 of \cite{wang2017bridge} for more information. One of the main goals of this paper is to compare the performance of two-stage variable selection techniques with $\lasso$. In the next lemma we derive the AFDP and ATPP of the thresholded estimate $\bar{\beta}(q, \lambda, s)$. 

\begin{lemma}
\label{LEMMA:FDP:LQ}
For any given $q \in [1,\infty), \lambda>0, s > 0$, almost surely
\begin{eqnarray}
\hspace{-0.cm} {\rm AFDP}(q, \lambda, s)&=&
\frac{(1-\epsilon)\mathbb{P}(\eta_q(|Z|;\alpha) > \frac{s}{\tau})}{(1-\epsilon)\mathbb{P}(\eta_q(|Z|;\alpha) > \frac{s}{\tau})+\epsilon\mathbb{P}(|\eta_q(G+\tau Z; \alpha \tau^{2-q})|>s)}, \nonumber \\
\hspace{-0.cm} {\rm ATPP}(q, \lambda, s)&=&
\mathbb{P}(|\eta_q(G+\tau Z; \alpha \tau^{2-q})|>s), \label{afdp:atpp}
\end{eqnarray}
where $(\alpha, \tau)$ is the unique solution of \eqref{state_evolution1} and \eqref{state_evolution2}.
\end{lemma}
The proof of this lemma is presented in Appendix \ref{apxnewb}. 

\section{Our Main Contribution}\label{sec:contribution}

\subsection{How to compare two variable selection schemes?}

The main objective of this paper is to compare the performance of the TVS techniques under the asymptotic setting of Section \ref{sec:asymptoticanalysisDef}.  A natural way for performing this comparison is to set $\atpp$ to a fixed value $\zeta\in [0,1]$ for different variable selection schemes and then compare their $\afdp$s.

The first challenge we face in such a comparison is that the TVS
may have many different ways for setting $\atpp$ to $\zeta$. If
$q>1$, Lemma \ref{LEMMA:FDP:LQ} shows that for every given value of the
regularization parameter $\lambda$, we can set $s$ (the threshold parameter) in
a way that it returns the right level of $\atpp$. Which of these parameter
choices should be used when we compare a TVS with another variable selection
technique, such as LASSO? Despite the fact that different choices of
$(\lambda,s)$ achieve the same $\atpp$ level $\zeta$, they may result in
different values of $\afdp$. Thus for fair comparison we pick the one that minimizes $\afdp$. The next theorem explains how this optimal pair can be found.

\begin{theorem}\label{THM:MSEMINFDRMIN}
Consider $q \in (1, \infty)$. Given an {\rm ATPP} level $\zeta \in [0,1]$, for every value of $\lambda>0$ there exists $s=s(\lambda, \zeta)$ such that ${\rm ATPP}(q, \lambda, s) = \zeta$. Furthermore, the value of $\lambda$ that minimizes ${\rm AFDP}(q,\lambda,s(\lambda, \zeta))$ also minimizes {\rm AMSE}$(q,\lambda)$.
\end{theorem}
The proof of this theorem can be found in Appendix
\ref{proof:THM:MSEMINFDRMIN}. Before discussing the implications of this
theorem, we state a similar result for LASSO.

\begin{theorem}\label{THM:LASSOFDPCOMPARE}
For any $\zeta \in [0, {\rm ATPP}(1,\lambda_1^*)]$, there exists at least one $\lambda$ s.t. ${\rm ATPP}(1, \lambda)=\zeta$. Further there exists a unique $s=s_\zeta$ such that ${\rm ATPP}(1, \lambda_1^*, s) = \zeta$. There may also exist other $(\lambda, s)$ s.t. ${\rm ATPP}(1, \lambda, s) = \zeta$. Among all these estimators, the one that offers the minimal ${\rm AFDP}$ is $\bar{\beta}(1, \lambda^*_1, s_\zeta)$, i.e., the two-stage $\lasso$ with the optimal tuning value $\lambda=\lambda_1^*$.
\end{theorem}

The proof of this theorem can be found in Appendix \ref{proof:THM:LASSOFDPCOMPARE}. There are a couple of points we would like to emphasize here:
\begin{enumerate}[(i)]
	\item
	Consider a TVS technique. According to
        Theorems \ref{THM:MSEMINFDRMIN} and \ref{THM:LASSOFDPCOMPARE}, for
        $q\in (1,\infty)$, the optimal choice of $\lambda$ does not depend on
        the $\atpp$ level $\zeta$ we are interested in. Even for $q=1$, the
        optimal choice of $\lambda$ is independent of $\zeta$ in a large range
        of $\atpp$s. It is the optimal tuning $\lambda^*_q$ for AMSE.
	\item
	An implication of Theorem \ref{THM:LASSOFDPCOMPARE} is that, for a wide
        range of $\zeta$, a second thresholding step helps with the variable
        selection of $\lasso$. Figure \ref{fig:Lasso_roc_cpr} compares the
        AFDP-ATPP curve of $\lasso$ with that of the two-stage $\lasso$. As is
        clear in this figure, when SNR is higher,
        the gap between the performance of two-stage $\lasso$ and $\lasso$
        becomes larger. We should emphasize that the $\atpp$ level
        of the two-stage $\lasso$ (with optimal tuning) can not exceed that of $\hat{\beta}(1,\lambda^*_1)$. We discuss \textit{debiasing} to resolve this issue in Section \ref{ssec:discuss:debias}.

	\item Theorems \ref{THM:MSEMINFDRMIN} and \ref{THM:LASSOFDPCOMPARE} do
            not explain how $\lambda_q^*$ can be estimated in practice. This
            issue will be discussed in Section \ref{sec:simulation}. But in  a
            nutshell, any approach that optimizes $\lambda$ for minimizing the
            out-of-sample prediction error works well. 

\end{enumerate}
    
{
\setlength{\tabcolsep}{0pt}
\begin{figure}[t!]
\begin{center}
\begin{tabular}{rccc}
    & \scriptsize{$\sigma=0.5$} & \scriptsize{$\sigma=0.22$} & \scriptsize{$\sigma=0.15$} \\
	\rotatebox{90}{\qquad\qquad\quad\tiny{AFDP}} & \includegraphics[scale=0.4]{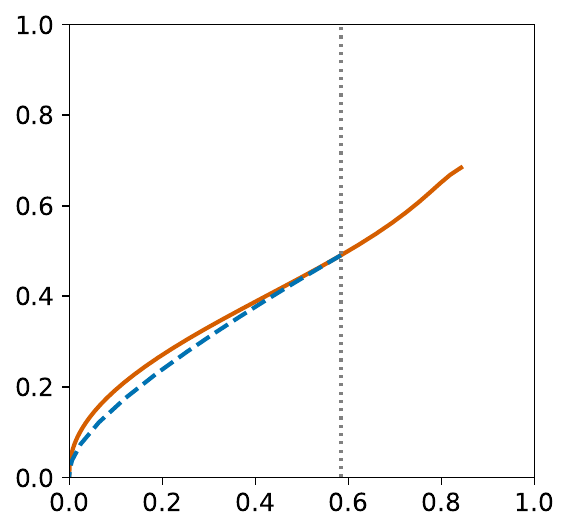}
	&
	\includegraphics[scale=0.4]{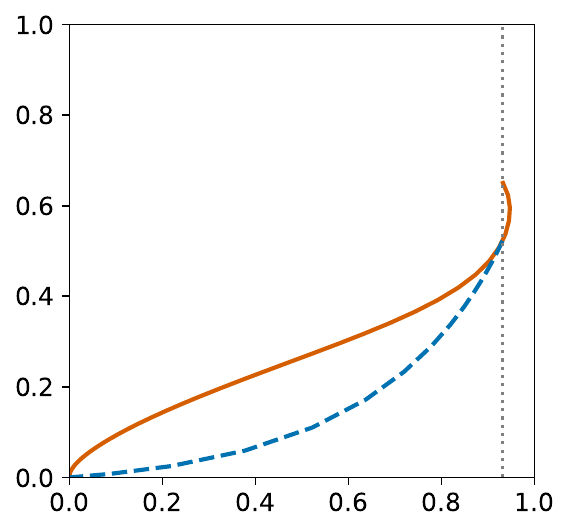}
	&
	\includegraphics[scale=0.4]{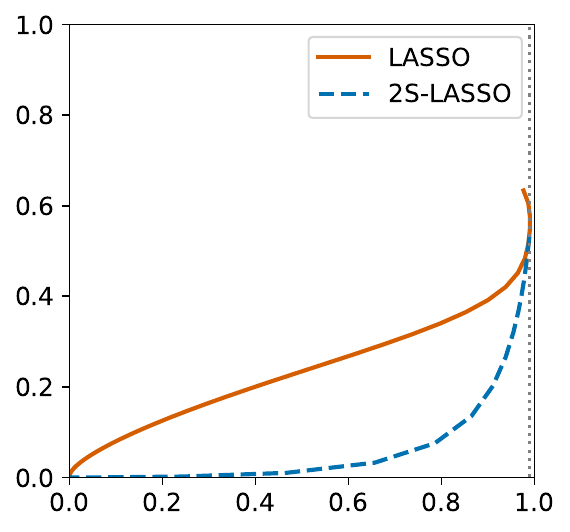} \\
	& \tiny{ATPP} & \tiny{ATPP} & \tiny{ATPP}
\end{tabular}
\caption{Comparison of {\rm AFDP-ATPP} curve between $\lasso$ and two-stage $\lasso$. Here we pick the setting $\delta=0.8$, $\epsilon=0.3$, $\sigma\in\{0.5, 0.22, 0.15\}, p_G=\delta_1$. For two-stage LASSO, we use optimal tuning $\lambda^*_1$ in the first stage. All the curves are calculated based on Equations \eqref{lemma:afdp:atpp:l1} and \eqref{afdp:atpp}. The gray dotted line is the upper bound of {\rm ATPP} that the two-stage $\lasso$ can reach. Notice that even for $\lasso$, there is an upper bound which it cannot exceed.}\label{fig:Lasso_roc_cpr}
\end{center}
\end{figure}
}

\begin{remark}
Theorems \ref{THM:MSEMINFDRMIN} and \ref{THM:LASSOFDPCOMPARE} prove that the optimal way to use two-stage variable selection is to set $\lambda =\lambda^*_q$ for the regularization parameter in the first stage. It is important to point out that $\lambda_q^*$ minimizes AMSE$(q, \lambda)$ and thus is the optimal tuning for parameter estimation. Therefore, the optimal tuning of the regularization parameter in bridge regression is the same for estimation and variable selection.  
\end{remark}

In the rest of the paper we will use the notation $s^*_q(\zeta)$ for the value of threshold that satisfies ${\rm ATPP}(q, \lambda^*_q, s^*_q(\zeta)) = \zeta$.

\subsection{The best bridge estimator for variable selection} \label{three:settings}
\subsubsection{Summary}
The two theorems we presented in the last section pave our way in addressing the question we raised in Section \ref{ssec:motive}, i.e., finding the best bridge estimator based TVS technique. Consider $q_1, q_2 \in [1,\infty)$. We would like to compare ${\rm AFDP}(q_1, \lambda^*_{q_1}, s^*_{q_1}(\zeta))$ and ${\rm AFDP}(q_2, \lambda^*_{q_2}, s^*_{q_2}(\zeta))$. The following corollary of Theorems \ref{THM:MSEMINFDRMIN} and \ref{THM:LASSOFDPCOMPARE}  shows the equivalence of the variable selection and estimation performance of bridge estimators. 

\begin{corollary}
\label{THM:FDPCOMPARE:LQ}
Let $q_1,q_2\geq 1$. If ${\rm AMSE}(q_1,\lambda^*_{q_1}) < {\rm AMSE}(q_2,\lambda^*_{q_2})$, then for every $\zeta \in [0,1]$
\[
{\rm AFDP}(q_1, \lambda^*_{q_1}, s^*_{q_1}(\zeta)) \leq {\rm AFDP}(q_2, \lambda^*_{q_2}, s^*_{q_2}(\zeta)).
\]
\end{corollary}

The proof of this result is presented in Appendix \ref{proof:THM:FDPCOMPARE:LQ}. According to Corollary \ref{THM:FDPCOMPARE:LQ}, in order to see which two-stage method is better, we can compare their $\amse$ under optimal tuning $\lambda_q^*$. Such $\amse$ is given by (see Theorem \ref{theorem:amp:bridge2} and Lemma \ref{lem:opttunesimpler} in the appendix)
\begin{equation*}
{\rm AMSE}(q, \lambda^*_q) = \mathbb{E}\left(\eta_q(B+\tau_* Z;\alpha_*\tau_*^{2-q})-B\right)^2,
\end{equation*}
where $\tau_*$ and $\alpha_*$ satisfy \eqref{state_evolution1} and \eqref{state_evolution2} with $\lambda=\lambda_q^*$. \\

The stage is finally set for comparing different two-stage variable selection
techniques. Note that in the calculation of ${\rm AMSE}(q, \lambda^*_q)$, the
values of $\alpha_*$ and $\tau_*$ are required and can only be calculated
through the fixed point equations \eqref{state_evolution1} and
\eqref{state_evolution2}. Therefore, we have no access to an explicit formula
for ${\rm AMSE}(q, \lambda^*_q)$. Furthermore, $\amse$ depends on many factors
including $\delta$, $\sigma$ and $p_B$. This poses an extra challenge to
completely evaluate and compare $\amse$ for different values of $q$. To address
these issues, we focus on a few regimes that
researchers have found useful in applications, and develop techniques to obtain explicit 
and accurate expressions for ${\rm AMSE}(q, \lambda^*_q)$. These
sharp results enable an accurate comparison among different TVS methods in each setting. The regimes we will consider are the following:   
\begin{itemize}

       \item[(i)] Nearly black objects or rare signals: In this regime,
           $\epsilon$ is assumed to be small. In other words, there are very
           few non-zero coefficients that need to be detected. This model is
           called nearly black objects \cite{donoho1992maximum} or rare signals
           \cite{donoho2015higher}. Intuitively speaking, it is also equivalent
           to the models considered in many other papers in which the sparsity
           level is assumed to be much smaller than the number of features. See
           for instance, \cite{meinshausen2006high, zhao2006model,
           zhang2008sparsity} and the references therein. We
           will allow the signal strength to vary with respect to $\epsilon$. It turns
           out that the rate of signal strength affects the choice of optimal bridge
           estimator.

	\item[(ii)] Low SNR: In this model, $\sigma$ is considered to be large. This assumption is accurate in many social and medical studies. For more information, the reader may refer to \cite{hastie2017extended}. To explain the effect of SNR on the best choice of $q$, we will also mention a result for high SNR. Such assumption is also standard in the engineering applications, where the quality of measurements is carefully controlled.  The analysis that is performed under the low noise setting is often called phase transition analysis, noise sensitivity analysis, or nearly exact recovery. See for instance \cite{oymak2016sharp, donoho2005sparse, donoho2011noise}.

	\item[(iii)] Large sample regime: In this regime the per-feature sample size
            $\delta$ is large. This regime, as will be seen later, is closely
            related to the classical asymptotic regime $n/p \rightarrow
            \infty$, and is appropriate for traditional applied statistical problems. See for instance \cite{knight2000asymptotics} for the asymptotic analysis of bridge estimators.

\end{itemize}

\subsubsection{Analysis of {\rm AMSE} for nearly black objects}
\label{sssec:nbo}
As discussed in the preceding section, the formulas of $\amse$ are implicit and
depend on $\delta$, $\sigma$ and $p_B$ in a complicated way. The goal of this
section is to obtain explicit and accurate expressions for ${\rm AMSE}(q,
\lambda_q^*)$ when $\epsilon$ is small (i.e. the signal is very sparse).
Towards this goal, a critical issue as made  in e.g. \cite{donoho1992maximum} for the case of orthogonal design, is that the strength of the signal affects the performance of each estimator. Hence, in our analysis we let the strength of the signal vary with $\epsilon$. This generalization requires an extra notation we introduce here. Recall $G$ is the random variable with probability measure $p_G$, which determines the values of the non-zero entries of $\beta$. Define
\[
b_{\epsilon}=\sqrt{\mathbb{E}G^2}, \quad \tilde{G}=G/b_{\epsilon}.
\]

Under this parameterization, $\mathbb{E}\tilde{G}^2=1$ and $b_\epsilon$ represents the (average) magnitude of each non-zero coefficient. We refer to $b_{\epsilon}$ as the signal strength and will allow it to change with the sparsity level $\epsilon$. Our first theorem characterizes the behavior of bridge estimators for $q>1$ and small values of $\epsilon$. 
\begin{theorem}\label{THM:STRONGRAREELL_Q}
Suppose that $b_\epsilon \rightarrow \infty$ and $b_\epsilon = O(1/
\sqrt{\epsilon})$.\footnote{$O$ notation used here is the standard big-$O$
notation. We will also use other standrd asymptotic notations. If the reader is
not familiar with these notation, he/she may refer to Appendix
\ref{appendix:notations}. } For $q > 1$, we have
\begin{itemize}
\item If $b_\epsilon = \omega(\epsilon^{\frac{1-q}{2}})$, then
\begin{equation*}
    \lim_{\epsilon \rightarrow 0}
    \epsilon^{-\frac{1}{q}}b_\epsilon^{-\frac{2(q-1)}{q}} \mathrm{AMSE}(q, \lambda_q^*)
    =
    q (q-1)^{\frac{1}{q} - 1} \sigma^\frac{2}{q}
    \big[\mathbb{E} |Z|^{\frac{2}{q-1}} \big]^{\frac{q-1}{q}}
    \big[ \mathbb{E} |\tilde{G}|^{2q - 2} \big]^{\frac{1}{q}}.
\end{equation*}

\item  If $b_\epsilon = o(\epsilon^{\frac{1-q}{2}})$, then $\lim_{\epsilon \rightarrow 0} \epsilon^{-1} b_\epsilon^{-2} {\rm AMSE}(q, \lambda_q^*)= 1$. 

\item If $\lim_{\epsilon \rightarrow 0} b_\epsilon \epsilon^{\frac{q-1}{2}} = c_r \in (0,\infty)$, then 
\[
\lim_{\epsilon \rightarrow 0} \epsilon^{-\frac{1}{q}} b_\epsilon^{-\frac{2(q-1)}{q}} {\rm AMSE} (q, \lambda^*_q) = \min_C h(C),
\]
where $h: \mathbb{R}^+ \rightarrow \mathbb{R}$ and $h(C)\triangleq
(Cq)^{-\frac{2}{q-1}}  \sigma^2  \mathbb{E} |Z|^{\frac{2}{q-1}} + \mathbb{E} \Big(\eta_q \big(c_r \tilde{G}; C
\sigma^{2-q} \big) - c_r \tilde{G} \Big)^2$. Furthermore, the minimizer of $h(C)$ is finite. 
\end{itemize}

We note that when $q > 2$, $b_\epsilon=o(\epsilon^{\frac{1 - q}{2}})$ always holds, hence
only the second item applies. When $q = 2$, only the second and the thrid items apply.

\end{theorem}
This theorem is proved in Appendix \ref{sec:proofThemstrongrareq}.  Before we
interpret this result, we characterize ${\rm AMSE}(1, \lambda_1^*)$ in Theorem
\ref{EQ:THM:STRONGRAREELL_1}.

\begin{theorem}\label{EQ:THM:STRONGRAREELL_1}
Suppose that $b_\epsilon \rightarrow \infty$ and $b_\epsilon = O(1/ \sqrt{\epsilon})$. We have
\begin{itemize}
\item If $b_\epsilon = \omega (\sqrt{\log \epsilon^{-1}})$, then $\lim_{\epsilon \rightarrow 0} \frac{{\rm AMSE}(1, \lambda_1^*)}{\epsilon \log \epsilon^{-1}} = 2 \sigma^2.$ 
\item If $b_\epsilon = o (\sqrt{\log \epsilon^{-1}})$, then $\lim_{\epsilon \rightarrow 0} \frac{{\rm AMSE}(1, \lambda_1^*)}{\epsilon b_\epsilon^2} = 1$.

\item If $\frac{b_\epsilon}{\sqrt{2\log \epsilon^{-1}}} \rightarrow c \in
    (0,\infty)$, then $\lim_{\epsilon \rightarrow 0} \frac{{\rm AMSE}(1,
    \lambda_1^*)}{\epsilon \log \epsilon^{-1}} =  \mathbb{E}
    (\eta_1(c\tilde{G}; \sigma) -c\tilde{G})^2$.
\end{itemize}
\end{theorem}
This theorem will be proved in Appendix \ref{thmsmallepsilonell_1}. There are a few points that we should emphasize about Theorems \ref{THM:STRONGRAREELL_Q} and  \ref{EQ:THM:STRONGRAREELL_1}.

\begin{remark}
First let us discuss the assumptions of these two theorems. It is
straightforward to show that  with $b_\epsilon = \omega(1/ \sqrt{\epsilon})$,
the SNR per measurement goes to infinity. Such scenarios seem uncommon in applications, and for the sake of brevity we have only considered  $b_\epsilon
= O(1/ \sqrt{\epsilon})$. Otherwise, the techniques we developed can be applied
to higher SNR as well. Furthermore, we postpone the discussion about the case $\aeps = O(1)$ to Theorem \ref{THM:SPARSE:MAIN}.\footnote{For the definitions of the asymptotic notations such as $\Omega$ refer to Appendix \ref{appendix:notations}. } 
\end{remark}

\begin{remark}
The work of \cite{donoho1992maximum} has studied the problem of estimating an extremely sparse signal under the orthogonal design. The main goal of \cite{donoho1992maximum} is to obtain the minimax risk for the class of $\epsilon$-sparse signals (similar to our model) without any constraint on the signals' power. They have shown that the approximately least favorable distribution has a point mass at $\Theta (\sqrt{\log(\epsilon^{-1})})$, and that LASSO achieves the minimax risk. Note that there are two major differences between Theorem \ref{EQ:THM:STRONGRAREELL_1} and the work of \cite{donoho1992maximum}: (i) our result is for non-orthogonal design, and (ii) we are not concerned with the minimax performance. In fact, we fix the power of the signal and obtain the asymptotic mean square error. This platform enables us to observe several delicate phenomena that are not observed in minimax settings. For instance, as is clear from Theorem \ref{EQ:THM:STRONGRAREELL_1}, the rate of ${\rm AMSE}(1, \lambda_1^*)$ undergoes a transition at the signal strength level $\Theta (\sqrt{\log(\epsilon^{-1})})$. As we will discuss later, below this threshold, LASSO is not necessarily optimal. However, since the risk of the Bayes estimator and LASSO is maximized for $\aeps = \Theta (\sqrt{\log(\epsilon^{-1})})$, this important information is missed in minimax analysis. 
\end{remark}

\begin{remark}
Compared to other bridge estimators, the performance of LASSO is much less sensitive to the strength of the signal: ${\rm AMSE}(1, \lambda_1^*) \sim \epsilon \log \epsilon^{-1}$ as long as $\aeps = \Omega(\sqrt{\log \epsilon^{-1}})$, while the order of ${\rm AMSE}(q,\lambda_q^*)$ continuously changes as $\aeps$ varies.\end{remark}

Theorems \ref{THM:STRONGRAREELL_Q} and \ref{EQ:THM:STRONGRAREELL_1} can be used for comparing different bridge estimators, as clarified in our next corollary.

\begin{corollary}
Suppose that $b_\epsilon = \epsilon^{-\gamma}$ for $\gamma \in (0,1/2]$. We have 
\begin{itemize}
\item If $q > 2\gamma + 1$, then ${\rm AMSE}(q, \lambda_q^*) \sim \epsilon^{1- 2 \gamma}$. 
\item If $1<q \leq 2\gamma + 1$, then ${\rm AMSE}(q, \lambda_q^*) \sim \epsilon^{\frac{1-2 \gamma (q-1)}{q}}$.
\item If $q =1$, then ${\rm AMSE}(q, \lambda_q^*) \sim \epsilon \log (\epsilon^{-1})$. 
\end{itemize}
\end{corollary}

The above result implies that in a wide range of signal strength, $q=1$ offers the smallest AMSE when the value of $\epsilon$ is very small. Consequently, according to Corollary
\ref{THM:FDPCOMPARE:LQ}, the two-stage LASSO provides the best variable selection performance. One can further confirm that the same conclusion continues to hold as long as $b_\epsilon = \omega (\sqrt{\log \epsilon^{-1}})$. \\

So far, we have seen that if the signal is reasonably strong, i.e. $b_\epsilon = \omega (\sqrt{\log \epsilon^{-1}})$, then two-stage LASSO outperforms all the other variable selection techniques. However, once $\aeps =  O (\sqrt{\log \epsilon^{-1}})$, we can see that ${\rm AMSE}(q, \lambda_q^*) \sim \epsilon \aeps^2$ for all $q \geq 1$. Hence, in order to provide a fair comparison, one should perform finer analyses and obtain a more accurate expression for $\amse$. Our next result shows how this can be done. 

 \begin{theorem}\label{THM:SPARSE:MAIN}
     Consider $\aeps=1$ and hence $\tilde{G} = G$. Assume $G$ is bounded from above. Then we have 
    \begin{align}
        \text{For } q = 1:&\quad
        \amse(1, \lambda_1^*)
        =
        \epsilon \mathbb{E} G^2 + o(\epsilon^k),
        \quad
        \forall k\in\mathbb{N}; \label{extreme:sparse:q1} \\
        \text{For } q > 1:&\quad
        \amse(q, \lambda^*_q)
        =
        \epsilon\mathbb{E} G^2 - \epsilon^2 \frac{ \mathbb{E}^2
        \Big( \big| \frac{G}{\sigma} + Z \big|^{\frac{1}{q-1}} \sgn \big(
        \frac{G}{\sigma} + Z \big) G \Big)} {\mathbb{E}|Z|^{\frac{2}{q-1}}}
        + o(\epsilon^2),\label{extreme:sparse:q12}
    \end{align}
    where $\sgn(\cdot)$ denotes the sign of a random variable.
\end{theorem}
The proof of this theorem is presented in Appendix
\ref{sec:proofFinitePowerExtremeSPARSE}. The first interesting observation
about this theorem is that, the first dominant term of $\amse$ is the same for
all bridge estimators. The second dominant term, on the other hand, is much
smaller for $q=1$ compared to the other values of $q$. Hence, LASSO is
\emph{suboptimal} in this setting. Accordingly, two-stage LASSO is outperformed
by other TVS methods. However, as is clear from Theorem \ref{THM:SPARSE:MAIN}, we should not expect the bridge estimator with $q>1$ to outperform LASSO by a large margin when $\epsilon$ is too small. In fact, the second dominant term is proportional to $\epsilon^2$ (for $q>1$), while the first dominant term is proportional to  $\epsilon$. Hence, the second dominant term is expected to become important for moderately small values of $\epsilon$. In such cases, we expect $q>1$ to offer more significant improvements. Regarding the optimal choice of $q$, it is determined by the constant of the second order term in \eqref{extreme:sparse:q12}. As is shown in Figure \ref{extreme:sparse:plot}, while the optimal value of $q$ is case-dependent, it gets closer to 1 as the signal strength increases. This observation is consistent with the message delivered by Theorems \ref{THM:STRONGRAREELL_Q} and \ref{EQ:THM:STRONGRAREELL_1}.

{
\setlength{\tabcolsep}{0pt}
\begin{figure}[t!]
\begin{center}
\begin{tabular}{rccc}
    & \scriptsize{$M=1$} & \scriptsize{$M=2$} & \scriptsize{$M=3$} \\
	\rotatebox{90}{\qquad\quad\scriptsize{2nd order coef}} & \includegraphics[scale=0.4]{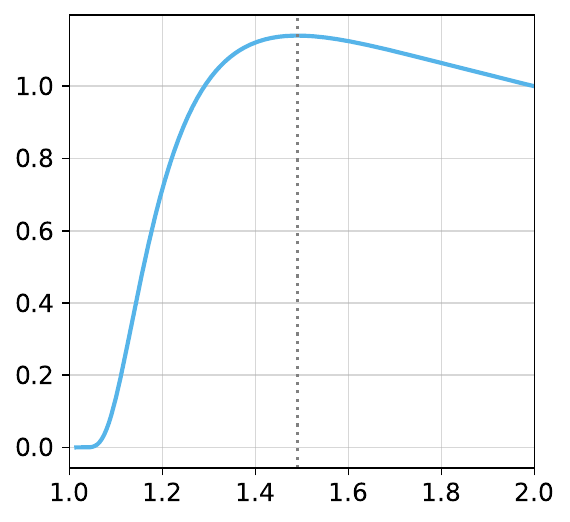}
	&
	\includegraphics[scale=0.4]{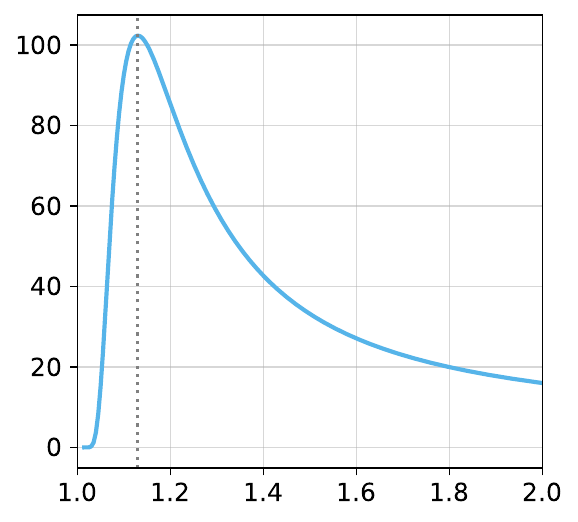}
	&
	\includegraphics[scale=0.4]{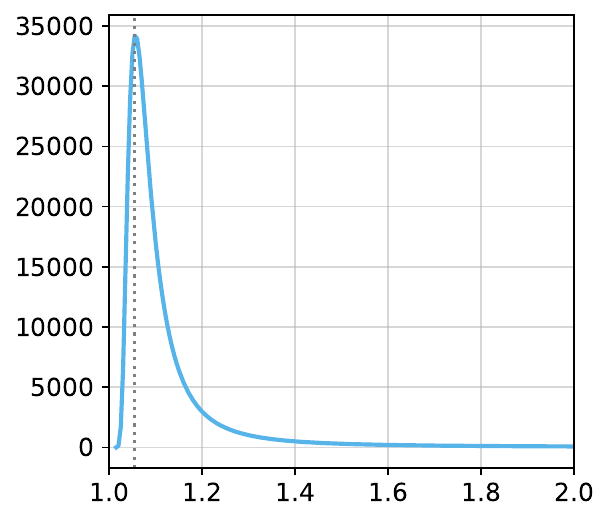} \\
	& \scriptsize{$q$} & \scriptsize{$q$} & \scriptsize{$q$}
\end{tabular}
\caption{The constant coefficient of the second order term in \eqref{extreme:sparse:q12}. We set $G=\delta_M$ with $M=1, 2, 3$ respectively and $\sigma = 1$. As the signal strength $M$ increases, the optimal
        choice of $q$ shifts towards 1.}\label{extreme:sparse:plot}
\end{center}
\end{figure}
}

\subsubsection{Analysis of {\rm AMSE} in large noise scenario} \label{large:noise:analysis}

This section aims to obtain explicit formulas for the optimal $\amse$ of bridge
estimators in low SNR. This regime is particularly important, since
in many social and medical studies, variable selection plays a key
role and the SNR is low. The following theorem summarizes the main result of this section. 

\begin{theorem}\label{THM:NOISY:MAIN}
As $\sigma\rightarrow\infty$, we have the following expansions of {\rm AMSE}$(q,\lambda_q^*)$:
\begin{enumerate}[(i)]
	\item
	For $q=1$, when $G$ has a sub-Gaussian tail, we have
\begin{equation} \label{largenoise:lasso}
{\rm AMSE}(1, \lambda_1^*)
=
\epsilon\mathbb{E}|G|^2+o(e^{-\frac{C^2\sigma^2}{2}}),
\end{equation}
where $C$ can be any positive number smaller than $C_0$, and $C_0>0$ is a constant only depending on $\epsilon$ and $G$. The explicit definition of $C_0$ can be found in the proof.

\vspace{0.2cm}
	\item
	For $1<q\leq 2$, if all the moments of $G$ are finite, then
\begin{equation} \label{largenoiseresult:q12}
{\rm AMSE}(q, \lambda_q^*)
=\epsilon \mathbb{E}|G|^2-\frac{\epsilon^2 (\mathbb{E}|G|^2)^2c_q}{\sigma^2}+o(\sigma^{-2}),
\end{equation}
with $c_q=\frac{\big(\mathbb{E}|Z|^{\frac{2-q}{q-1}}\big)^2}{(q-1)^2\mathbb{E}|Z|^{\frac{2}{q-1}}}$.

\item For $q > 2$, if $G$ has sub-Gaussian tail, then \eqref{largenoiseresult:q12} holds.
\end{enumerate}
\end{theorem}

We present our proofs in Appendix \ref{thmlargenoise:main}. Figure \ref{fig:inline_large_noise} compares the accuracy of the first-order approximation and second-order approximation for moderate values of $\sigma$. As is clear, for $q\in (1,\infty)$, the second-order approximation provides an accurate approximation of ${\rm AMSE}(q, \lambda_q^*)$  for a wide range of $\sigma$. Moreover, the first-order approximation for ${\rm AMSE}(1,\lambda_1^*)$ is already accurate as can be justified by its exponentially small second order term in \eqref{largenoise:lasso}. 

{
\setlength{\tabcolsep}{0pt}
\begin{figure}[t!]
\begin{center}
\begin{tabular}{rcc}
    & \scriptsize{$q = 1$} & \scriptsize{$q = 1.5$} \\
	\rotatebox{90}{\scriptsize{\qquad \qquad relative error}} & \includegraphics[scale=0.4]{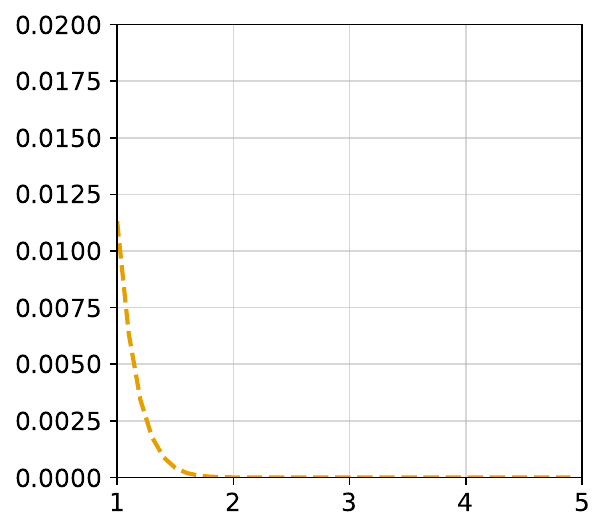}
	&
	\includegraphics[scale=0.4]{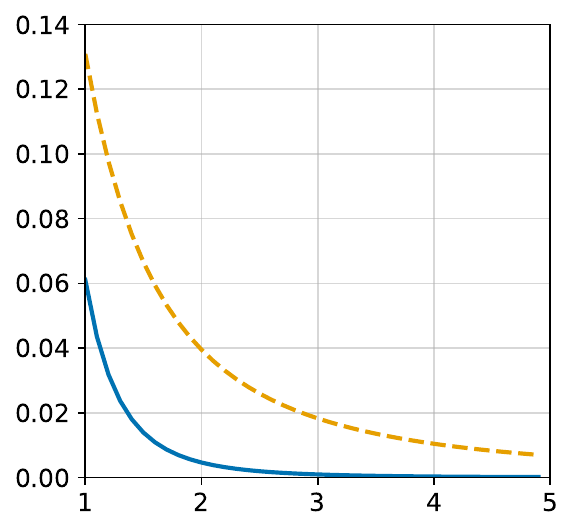} \\
	& \scriptsize{$\sigma$} & \scriptsize{$\sigma$} \\
	&& \\
	&& \\
    & \scriptsize{$q = 2$} & \scriptsize{$q = 2.5$} \\
	\rotatebox{90}{\scriptsize{\qquad \qquad relative error}} & \includegraphics[scale=0.4]{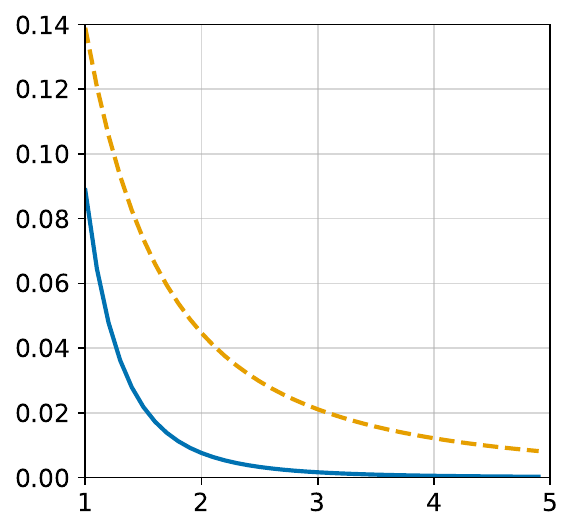}
	&
	\includegraphics[scale=0.4]{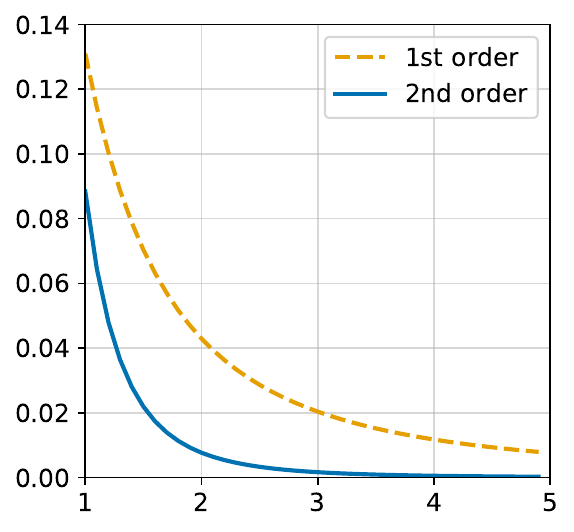} \\
	& \scriptsize{$\sigma$} & \scriptsize{$\sigma$}
\end{tabular}
\caption{Absolute relative error of first-order and second-order approximations of $\amse$ under large noise scenario. In these four figures, $p_B = (1 - \epsilon)\delta_0 + \epsilon\delta_1$, $\delta=0.4$, $\epsilon=0.2$.}\label{fig:inline_large_noise}
\end{center}
\end{figure}
}
  
According to this theorem, we can conclude that for sufficiently large $\sigma$, two-stage method with any $q > 1$ can outperform the two-stage $\lasso$. This is because while the first dominant term is the same for all the bridge estimators with $q\in [1,\infty)$, the second order term for $\lasso$ is exponentially smaller (in magnitude) than that of the other estimators. More interestingly, the following lemma shows that in fact $q=2$ leads to the smallest AMSE in the large noise regime.    

{
\setlength{\tabcolsep}{0pt}
\begin{figure}[t!]
\begin{center}
\begin{tabular}{rc}
	\rotatebox{90}{\qquad\qquad\quad\scriptsize{$C_q$}} & \includegraphics[scale=0.4]{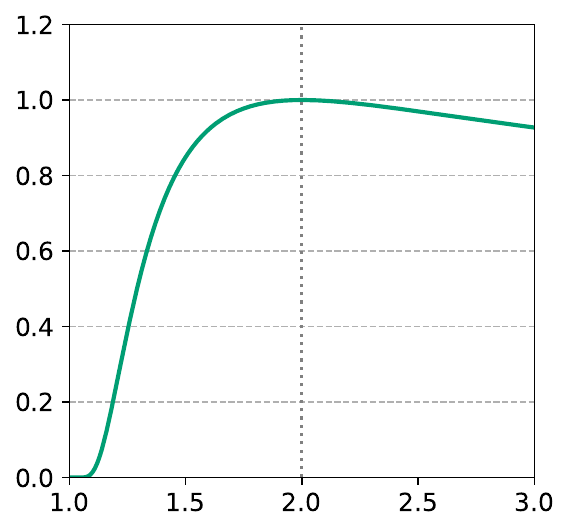} \\
	& \scriptsize{$q$}
\end{tabular}
\caption{The constant $c_q$ in Theorem \ref{THM:NOISY:MAIN} part (ii). The maximum is achieved at $q=2$. }\label{fig:c_q}
\end{center}
\end{figure}
}

\begin{lemma}\label{lemma:ridgebestlargenoise}
The maximum of $c_q$, defined in Theorem \ref{THM:NOISY:MAIN}, is achieved at $q=2$. 
\end{lemma}

See Figure \ref{fig:c_q} for the plot of $c_q$.

\begin{proof}
A simple integration by part yields:
\begin{eqnarray*}
\mathbb{E}|Z|^{\frac{2-q}{q-1}}
&=&
2(q-1)\int_0^\infty z^{\frac{q}{q-1}}\phi(z)dz
=
(q-1)\mathbb{E}|Z|^{\frac{q}{q-1}}
\end{eqnarray*}
We can then apply H$\ddot{o}$lders's inequality to obtain
\begin{equation*}
c_q
=
\frac{(\mathbb{E}|Z|^{\frac{q}{q-1}})^2}{\mathbb{E}|Z|^{\frac{2}{q-1}}}
\leq
\frac{\mathbb{E}|Z|^{\frac{2}{q-1}}\mathbb{E}Z^2}{\mathbb{E}|Z|^{\frac{2}{q-1}}}
=1
=c_2.
\end{equation*}
\end{proof}

Therefore, while the AMSE of all bridge estimators share the same first
dominant term, Ridge offers the largest second dominant term (in magnitude),
and hence the lowest AMSE. If we combine this result with Corollary
\ref{THM:FDPCOMPARE:LQ}, we conclude that in low SNR regime, two-stage Ridge obtains the best variable selection performance among TVS schemes with their first stage picked from the class of bridge estimators. 

A comparison of this result with that for the high SNR derived in
\cite{weng2018overcoming} clarifies the impact of SNR on the best choice of $q$.

\begin{theorem}
\label{THM:LOWNOISE:MAIN}
Assume $\epsilon \in (0, 1)$. As $\sigma\rightarrow 0$, we have the following expansions of {\rm AMSE}$(q, \lambda_q^*)$ in terms of $\sigma$.
	\begin{enumerate}[(i)]
	\item
	For $q = 1$, if $\mathbb{P}(|G| \geq \mu)=1$ for some $\mu>0$, $\delta> M_1(\epsilon)$, and $\mathbb{E}|G|^2<\infty$, then
	\begin{equation}
	{\rm AMSE}(1, \lambda_1^*)
	=
	\frac{\delta M_1(\epsilon)}{\delta - M_1(\epsilon)}\sigma^2
	+
	o\big(e^{\frac{(M_1(\epsilon)-\delta)\tilde{\mu}^2}{2\delta \sigma^2}}\big), \label{eq:lownoise:q=1}
	\end{equation}
	where
        $M_1(\epsilon)=\min_{\chi}(1-\epsilon)\mathbb{E}\eta_1^2(Z;\chi)+\epsilon
        (1+\chi^2)$, and $\tilde{\mu}$ can be any positive number smaller than $\mu$.
	\vspace{0.2cm}
	\item
	For $1 < q < 2$, if $\mathbb{P}(|G| \leq x)=O(x)$ (as $x\rightarrow 0$), $\delta>1$, and $\mathbb{E}|G|^2<\infty$ then
	\begin{equation}
	{\rm AMSE}(q, \lambda_q^*)
	=
	\frac{\sigma^2}{1-1/\delta}
	-
	\sigma^{2q}\frac{\delta^{q+1}(1-\epsilon)^2(\mathbb{E}|Z|^q)^2}{(\delta-1)^{q+1}\epsilon\mathbb{E}|G|^{2q-2}}+o(\sigma^{2q}). \label{eq:lownoise:2>q>1}
	\end{equation}
	
	\item
	For $q = 2$, if $\delta>1$ and $\mathbb{E}|G|^2<\infty$, we have
	\begin{equation}
	{\rm AMSE}(2, \lambda_2^*)
	=
	\frac{\sigma^2}{1-1/\delta}
	-
	\sigma^4 \frac{\delta^3}{(\delta-1)^3\epsilon \mathbb{E}|G|^2}+o(\sigma^4). \label{eq:lownoise:q=2}
	\end{equation}
	
	\item For $q>2$, if $\delta>1$ and $\mathbb{E}|G|^{2q-2}<\infty$, then
	\begin{equation} \label{expansion:amseq2}
	{\rm AMSE}(q, \lambda_q^*)
	=
	\frac{\sigma^2}{1-1/\delta}
	-
	\sigma^4\frac{\delta^3\epsilon (q-1)^2(\mathbb{E}|G|^{q-2})^2}{(\delta-1)^3\mathbb{E}|G|^{2q-2}}+o(\sigma^4).
	\end{equation}
\end{enumerate}
\end{theorem}

The results for $q \in [1,2]$ are taken from \cite{weng2018overcoming}. The proof for the case $q>2$ can be found in Appendix I of \cite{wang2017bridge}. It is straightforward to see that $M_1(\epsilon)$ is an increasing function of $\epsilon \in [0,1]$ and $M_1(1)=1$. This implies that ${\rm AMSE}(1,\lambda_1^*)$ is the smallest among all ${\rm AMSE}(q,\lambda_q^*)$ with $q\in [1,\infty)$. As is clear, the first order terms in the expansion of ${\rm AMSE}(q,\lambda_q^*)$ are the same for all $q\in (1,\infty)$. However, the second dominant term shows that the smaller values of $q$ are preferable (note the strict monotonicity only occurs in the range $(1,2]$). 

Combining the above results with Corollary \ref{THM:FDPCOMPARE:LQ} implies that in the high SNR setting, two-stage $\lasso$ offers the best variable selection performance. We should also emphasize that as depicted in Figure \ref{fig:Lasso_roc_cpr}, in this regime two-stage $\lasso$ offers a much better variable selection performance than $\lasso$.

\begin{remark}
Theorems \ref{THM:NOISY:MAIN} and \ref{THM:LOWNOISE:MAIN} together give a full and sharp evaluation of the noise-sensitivity of bridge estimators. Among all the bridge estimators with $q\in [1,\infty)$, LASSO and Ridge are optimal for parameter estimation and variable selection, in the low and large noise settings respectively. This result delivers an intriguing message: sparsity inducing regularization is not necessarily preferable even in sparse models. Such phenomenon might be well explained by the bias-variance tradeoff: variance is the major factor in very noisy settings, thus a regularization that produces more stable estimator is preferred, when the noise is large.
\end{remark}

\subsubsection{Analysis of {\rm AMSE} in large sample scenario} \label{large:sample:study}

Our analysis in this section is concerned with the large $\delta$ regime. Since $n/p\rightarrow \delta$ in our asymptotic setting, large $\delta$ means large sample size (relative to the dimension $p$). Intuitively speaking, this is similar to the classical asymptotic setting where $n \rightarrow \infty$ and $p$ is fixed (specially if we assume the fixed number $p$ is large). We will later connect the results we derive in the large $\delta$ regime to those obtained in classical asymptotic regime, and provide new insights. 

In our original set-up, the elements of the design matrix are $X_{ij} \overset{i.i.d.}{\sim} N(0, \frac{1}{n})$. This means the SNR ${\rm var}(\sum_j X_{ij} \beta_j)/{\rm var}(w_i)\rightarrow \frac{\mathbb{E}|B|^2}{\delta \sigma^2}$ as $n\rightarrow \infty$. Therefore, if we let $\delta \rightarrow \infty$, the SNR will decrease to zero, which is not consistent with the classical asymptotics in which the SNR is assumed to be fixed. To resolve this discrepancy we scale the noise term by $\sqrt{\delta}$ and use the model:
\begin{equation}\label{eq:largedeltabridge}
y= X\beta + \frac{1}{\sqrt{\delta}} w,
\end{equation}
where $\{\beta,w, X\}$ is the converging sequence in Definition \ref{definition:asymptotics}. Under this model we compare the AMSE of different bridge estimators. The next theorem summarizes the main result.

\begin{theorem}\label{THM:SAMPLE:MAIN}
Consider the model in \eqref{eq:largedeltabridge} and $\epsilon \in (0,1)$. As $\delta \rightarrow \infty$, we have 
\begin{enumerate}[(i)]
	\item
	For $q = 1$, if $\mathbb{P}(|G| \geq \mu)=1$ for some $\mu>0$ and  $\mathbb{E}|G|^2 < \infty$, then
	\begin{equation}
	{\rm AMSE}(1,\lambda^*_1)
	=
	\frac{M_1(\epsilon)\sigma^2}{\delta}+o(\delta^{-1}), \label{eq:large:sample:q=1}
	\end{equation}
        where $M_1(\epsilon)$ has the same definition as in Theorem \ref{THM:LOWNOISE:MAIN} (i).
	
	\vspace{0.2cm}
	\item
	For $1 < q < 2$, if $\mathbb{P}(|G| \leq x) = O(x)$ (as $x \rightarrow 0$) and  $\mathbb{E}|G|^2 < \infty$, then
	\begin{equation}
	{\rm AMSE}(q, \lambda^*_q)
	=
	\frac{\sigma^2}{\delta} - \frac{\sigma^{2q}}{\delta^q}\frac{(1-\epsilon)^2(\mathbb{E}|Z|^q)^2}{\epsilon\mathbb{E}|G|^{2q-2}}+o\big(\delta^{-q}\big)\label{eq:large:sample:1<q<2}
	\end{equation}
	
	\item
	For $q = 2$, if  $\mathbb{E}|G|^2 < \infty$, then we have
	\begin{equation}
	{\rm AMSE}(2,\lambda^*_2)
	=
	\frac{\sigma^2}{\delta} + \frac{\sigma^2}{\delta^2} \big[1 - \frac{\sigma^2}{\epsilon\mathbb{E}G^2}\big] + o(\delta^{-2})
	\end{equation}
	\item 
	For $q>2$, if  $\mathbb{E}|G|^{2q-2} < \infty$, then
	\begin{equation}
	{\rm AMSE}(q,\lambda^*_q)
	=
	\frac{\sigma^2}{\delta} + \frac{\sigma^2}{\delta^2}
	\bigg[1	- \frac{\epsilon (q-1)^2\sigma^2(\mathbb{E}|G|^{q-2})^2}{\mathbb{E}|G|^{2q-2}}\bigg] +o(\delta^{-2}).   \label{eq:large:sample:q=2}
	\end{equation}
\end{enumerate}
\end{theorem}

The proof of Theorem \ref{THM:SAMPLE:MAIN} can be found in Appendix \ref{apxc}. Figure \ref{fig:inline_large_sample} compares the accuracy of the first and second order expansions in large range of $\delta$. As is clear from this figure, the second-order term often offers an accurate approximation over a wide range of $\delta$. 

\begin{remark}
As mentioned in Section \ref{large:noise:analysis}, $M_1(\epsilon)$ is an increasing function of $\epsilon \in [0,1]$ and $M_1(1) =1$. This implies that ${\rm AMSE}(1,\lambda_1^*)$ is the smallest among all ${\rm AMSE}(q,\lambda_q^*)$ with $q\in [1,\infty)$. Therefore, in this regime $\lasso$ gives the smallest estimation error and thus two-stage $\lasso$ offers the best variable selection performance.  
\end{remark}

\begin{remark}
The ${\rm AMSE}(q,\lambda_q^*)$ with $q>1$ share the same first dominant term, but have different second order terms. Furthermore, for $q\in (1,2]$, the smaller $q$ is, the better its performance will be. Such monotonicity does not hold beyond $q=2$. 
\end{remark}

{
\setlength{\tabcolsep}{0pt}
\begin{figure}[t!]
\begin{center}
\begin{tabular}{rcc}
    & \scriptsize{$q = 1$} & \scriptsize{$q = 1.5$} \\
	\rotatebox{90}{\scriptsize{\qquad \qquad relative error}} & \includegraphics[scale=0.4]{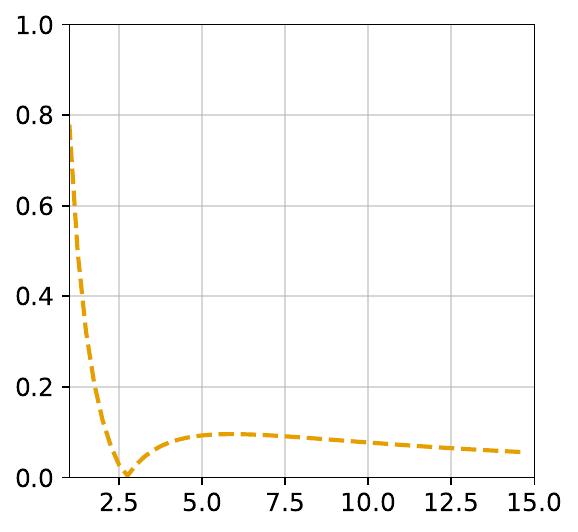}
	&
	\includegraphics[scale=0.4]{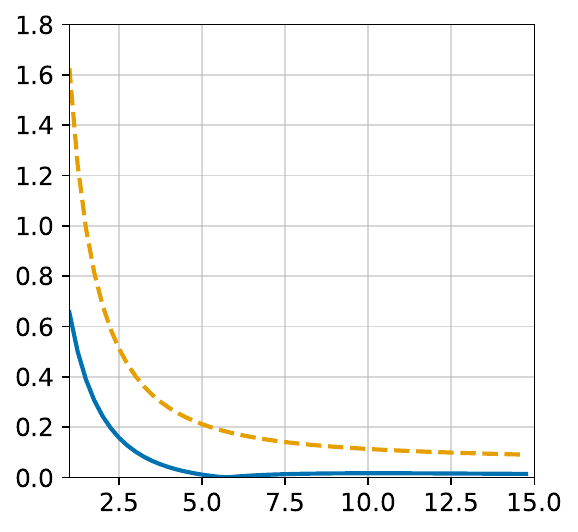} \\
	& \scriptsize{$\delta$} & \scriptsize{$\delta$} \\
	&& \\
	&& \\
    & \scriptsize{$q = 2$} & \scriptsize{$q = 2.5$} \\
	\rotatebox{90}{\scriptsize{\qquad \qquad relative error}} & \includegraphics[scale=0.4]{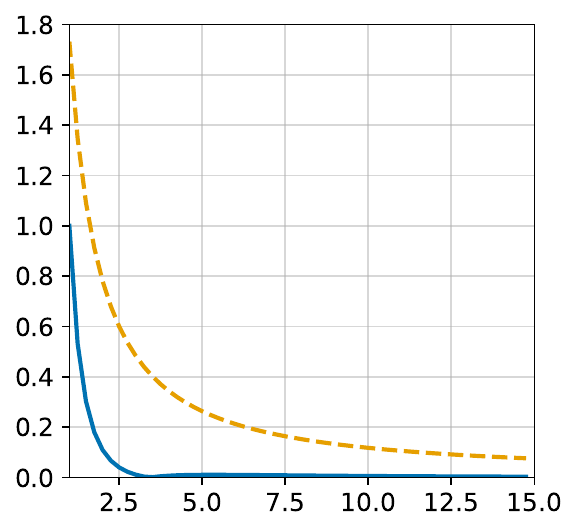}
	&
	\includegraphics[scale=0.4]{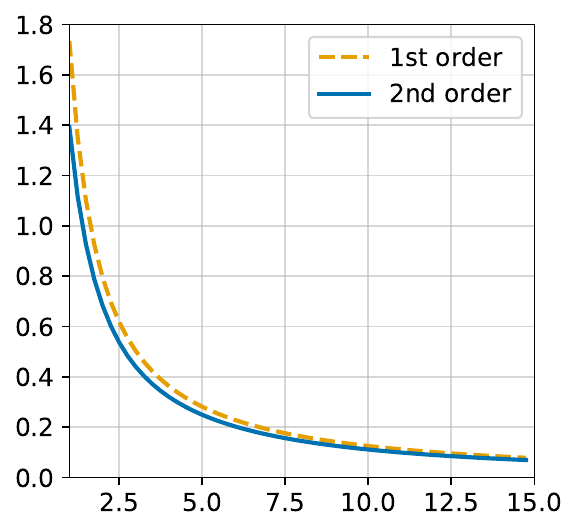} \\
	& \scriptsize{$\delta$} & \scriptsize{$\delta$}
\end{tabular}
\caption{Absolute relative error of first-order and second-order approximations of AMSE under large sample scenario. In these four figures, $p_B = (1 - \epsilon)\delta_0 + \epsilon\delta_1$, $\epsilon=0.5$, $\sigma=1$.}\label{fig:inline_large_sample}
\end{center}
\end{figure}
}

We now connect our results in this large $\delta$ regime to those obtained in classical asymptotic setting. The classical aymptotics ($p$ fixed) of bridge estimators for all the values of $q\in [0,\infty)$ is studied in \cite{knight2000asymptotics}. We explain $\lasso$ first. According to \cite{knight2000asymptotics}, if $\frac{\lambda}{\sqrt{n}}\rightarrow\lambda_0\geq 0$ and $\frac{1}{n}X^TX\rightarrow C$, then
\begin{equation} \label{classical:result}
\sqrt{n}(\hat{\beta} - \beta)
\xrightarrow{d}
\arg\min_u V(u),
\end{equation}
where $V(u) = -2u^TW + u^T C u + \lambda_0\sum_{j=1}^p[u_j\sgn(\beta_j)\1{\beta_j\neq 0} + |u_j|\1{\beta_j = 0}]$ with $W\sim\mathcal{N}(0, \sigma^2 C)$. We will do the following calculations to explore the connections. Since $X_{ij}\sim N(0,1/n)$ in our paper, we first make the following changes to $\lasso$ to make our set-up consistent with that of  \cite{knight2000asymptotics}:
\begin{equation*}
\frac{1}{2}\|y-X\beta\|_2^2 + \lambda\|\beta\|_1
=
\frac{1}{2}\Big(\|y-\sqrt{n}X\frac{\beta}{\sqrt{n}}\|_2^2 + 2\sqrt{n}\lambda\|\frac{\beta}{\sqrt{n}}\|_1\Big).
\end{equation*}

We thus have $C=\frac{1}{n}(\sqrt{n}X)^T(\sqrt{n}X)\rightarrow I$ and $\lambda_0 = 2\lambda$. Now suppose the result \eqref{classical:result} works for $\hat{\beta}(1,\lambda)$. Then we have
\begin{equation} \label{cal:formula}
\hat{\beta}(1, \lambda) - \beta
\xrightarrow{d}
\arg\min_u V(u),
\end{equation}
where $V(u) = -2u^TW + u^Tu + 2\lambda\sum_{j=1}^p[u_j\sgn(\beta_j)\1{\beta_j\neq 0} + |u_j|\1{\beta_j = 0}]$ with $W \sim\mathcal{N}(0, \frac{\sigma^2}{\delta}I)$. It is straightforward to see that the optimal choice of $u$ in \eqref{cal:formula} has the following form:
\begin{equation*}
\hat{u}_j
=
\left\{\begin{array}{ll}
	W_j - \lambda\sgn(\beta_j) & \text{ when }\beta_j\neq 0\\
	W_j - \lambda s(\hat{u}_j) & \text{ when }\beta_j = 0
\end{array}\right.
\end{equation*}
where $s(u_j)=\sgn(u_j)$ when $u_j\neq 0$ and $|s(u_j)|\leq 1$ when $u_j = 0$. Furthermore, for the case of $\beta_j=0$, $\hat{u}_j=0$ is equivalent to $|W_j|\leq\lambda$ and $\sgn(W_j)=\sgn(\hat{u}_j)$ when $\hat{u}_j\neq 0$. Based on this result, we do the following heuristic calculation to connect our results with those of \cite{knight2000asymptotics}:
\begin{align*}
\frac{1}{p}\|\hat{\beta}(1, \lambda) - \beta\|_2^2
\approx&
\frac{1}{p}\mathbb{E}\bigg[\sum_{j:\beta_j\neq 0}\big[W_j^2 - 2\lambda\sgn(\beta_j)W_j + \lambda^2\big]
+
\sum_{j:\beta_j=0, \hat{u}_j\neq 0}\big[W_j^2 - 2\lambda W_j\sgn(\hat{u}_j) + \lambda^2\big]
\bigg]\nonumber\\
\approx&
\frac{1}{p}\bigg[\sum_{j:\beta_j\neq 0}(\frac{\sigma^2}{\delta} + \lambda^2)
+
\sum_{j:\beta_j=0}\mathbb{E}\eta_1^2(W_j;\lambda)\bigg]
=
\frac{k}{p}(\frac{\sigma^2}{\delta} + \lambda^2) + \frac{p-k}{p}\mathbb{E}\eta_1^2(W_j;\lambda) \nonumber \\
=&
\frac{\sigma^2}{\delta}\Big[\frac{p-k}{p}\mathbb{E}\eta^2_1(Z; \sqrt{\delta}\lambda/\sigma)+\frac{k}{p}(1+(\sqrt{\delta}\lambda/\sigma)^2) \Big], \nonumber 
\end{align*}
where $k$ is the number of non-zero elements of $\beta$ and $Z \sim N(0,1)$. Note that in our asymptotic setting $k/p \rightarrow \epsilon$ and we consider the optimal tuning $\lambda_1^*$. Therefore following the above calculations we obtain
\[
\min_{\lambda}\frac{1}{p}\|\hat{\beta}(1, \lambda) - \beta\|_2^2\approx \frac{\sigma^2}{\delta}\min_{\chi} (1-\epsilon)\mathbb{E}\eta_1^2(Z;\chi)+\epsilon (1+\chi^2)=\frac{M_1(\epsilon)\sigma^2}{\delta}.
\]

This is consistent with \eqref{eq:large:sample:q=1} in our asymptotic analysis. We can do similar calculations to show that the asumptotic analysis of \cite{knight2000asymptotics} leads to the first order expansion of AMSE in Theorem \ref{THM:SAMPLE:MAIN} for the case $q>1$. 

Based on this heuristic argument, we may conclude that the information provided by the classical asymptotic analysis is reflected in the first order term of AMSE$(q,\lambda_q^*)$. Moreover, our large sample analysis is able to derive the second dominant term for $q>1$. This term enables us to compare the performance of different values of $q>1$ more accurately (note they all have the same first order term). Such comparisons cannot be performed  in \cite{knight2000asymptotics}. 

\section{Debiasing}\label{ssec:discuss:debias}

\subsection{Implications of debiasing for $\lasso$}

As is clear from Theorem \ref{THM:LASSOFDPCOMPARE}, since $\lasso$ produces a sparse solution, it is not possible for a $\lasso$ based two-stage method to achieve $\atpp$ values beyond what is already reached by the first stage. This problem can be resolved by \textit{debiasing}. In this approach, instead of thresholding the $\lasso$ estimate (or in general a bridge estimate), we threshold its debiased version. Below we will add a dagger $\dagger$ to aforementioned notations to denote their corresponding debiased version. Recall $\hat{\beta}(q, \lambda)$ denotes the solution of bridge regression for any $q\geq 1$. Define the debiased estimates as
\begin{enumerate}[(i)]
	\item
	For $q = 1$,
	\begin{equation*}
	\hat{\beta}^\dagger(1, \lambda) \triangleq \hat{\beta} (1, \lambda) + X^T \frac{y-X\hat{\beta}(1, \lambda) }{1- \|\hat{\beta}(1, \lambda)\|_0 / n},
	\end{equation*}
	where $\|\cdot \|_0$ counts the number of non-zero elements in a vector.  

	\item
	For $q > 1$,
	\begin{equation}
	\hat{\beta}^\dagger(q, \lambda) \triangleq \hat{\beta} (q, \lambda) + X^T \frac{y-X\hat{\beta}(q, \lambda) }{1- f(\hat{\beta}(q,\lambda),\hat{\gamma}_{\lambda}) / n}, \label{eq:debiasing:lq}
	\end{equation}
	where $f(v,w)=\sum_{i=1}^p\frac{1}{1+wq(q-1)|v_i|^{q-2}}$ and $\gamma=\hat{\gamma}_\lambda$ is the unique solution of the following equation:
	\begin{eqnarray}
	\frac{\lambda}{\gamma}=1-\frac{1}{n}f(\hat{\beta}(q,\lambda),\gamma). \label{eq:solvegamma}
	\end{eqnarray}
\end{enumerate}

We have the following theorem to confirm the validity of the debiasing estimator $\hat{\beta}^\dagger(q, \lambda)$.

\begin{theorem}\label{THEOREM:DEBIASING:VALID}
For any given $q\in [1,\infty)$, with probability one, the empirical distribution of the components of $\hat{\beta}^\dagger(q, \lambda)-\beta$ converges weakly to $N(0,\tau^2)$, where $\tau$ is the solution of \eqref{state_evolution1} and \eqref{state_evolution2}.
\end{theorem}
See Appendix \ref{apxd} for the proof. In order to perform variable selection, one may apply the hard thresholding function to these debiased estimates, i.e.,
\begin{equation*}
\bar{\beta}^\dagger(q, \lambda,s) =  \eta_0(\hat{\beta}^\dagger(q, \lambda); s^2/2)=\hat{\beta}^\dagger(q,\lambda)\1{|\hat{\beta}^\dagger(q, \lambda)|\geq s}. 
\end{equation*}
We use the notations ${\rm ATPP}^\dagger(q, \lambda, s)$ and ${\rm AFDP}^\dagger(q, \lambda, s)$ to denote the ATPP and AFDP of $\bar{\beta}^\dagger(q, \lambda, s)$ respectively. 
In the case of $\lasso$, note that unlike $\hat{\beta}(1,\lambda)$ the debiased estimator $\hat{\beta}^\dagger(1, \lambda)$ is dense. Hence we expect the two-stage variable selection estimate $\bar{\beta}^\dagger(1, \lambda,s) $ to be able to reach any value of $\atpp$ between $[0,1]$.  The following theorem confirms this claim.

\begin{theorem}\label{THEOREM:DEBIASING:ATTP}
Given the {\rm ATPP} level $\zeta \in [0,1]$, for every value of $\lambda>0$, there exists $s(\lambda, \zeta)$ such that ${\rm ATPP}^\dagger(1, \lambda, s(\lambda, \zeta)) = \zeta$. Furthermore, whenever $\bar{\beta}^{\dagger}(1,\lambda,s)$ and $\bar{\beta}(1,\lambda, \tilde{s})$ reach the same level of {\rm ATPP}, they have the same {\rm AFDP}. The value of $\lambda$ that minimizes ${\rm AFDP}^\dagger(1, \lambda, s(\lambda, \zeta))$ also minimizes ${\rm AMSE}(1, \lambda)$.   
\end{theorem}

As expected since the solution of bridge regression for $q>1$ is dense, the debiasing step does not help variable selection for $q>1$. Our next theorem confirms this claim.

\begin{theorem} \label{DEBIASING:Q>1}
Consider $q>1$. Given the {\rm ATPP} level $\zeta \in [0,1]$, for every value of $\lambda>0$, there exists $s(\lambda, \zeta)$ such that ${\rm ATPP}^\dagger(q, \lambda, s(\lambda, \zeta)) = \zeta$. Furthermore, whenever $\bar{\beta}^{\dagger}(q,\lambda,s)$ and $\bar{\beta}(q,\lambda, \tilde{s})$ reach the same level of {\rm ATPP}, they have the same {\rm AFDP}. Also, the value of $\lambda$ that minimizes ${\rm AFDP}^\dagger(q, \lambda, s(\lambda, \zeta))$ also minimizes ${\rm AMSE}(q, \lambda)$. As a result, the optimal value of ${\rm AFDP}^\dagger(q, \lambda, s(\lambda, \zeta))$ is the same as ${\rm AFDP}(q, \lambda^*_{q}, s^*_{q}(\zeta))$. 
\end{theorem}

For the proof of Theorems \ref{THEOREM:DEBIASING:ATTP} and \ref{DEBIASING:Q>1}, please refer to Appendix \ref{apxd}.
\begin{remark}\label{remark:bebiasing:lasso}
Comparing Theorem \ref{THEOREM:DEBIASING:ATTP} with Theorem \ref{THM:LASSOFDPCOMPARE}, we see that replacing $\lasso$ in the first stage with the debiased version enables to achieve wider range of ATPP level. On the other hand, given the value of $\lambda$, if $\bar{\beta}^{\dagger}(1,\lambda,s)$ and $\bar{\beta}(1,\lambda,\tilde{s})$ reach the same level of {\rm ATPP}, their {\rm AFDP} are equal as well. Therefore, the debiasing for $\lasso$ expands the range of {\rm AFDP}-{\rm ATPP} curve without changing the original one. Figure \ref{fig:Lasso_roc_cpr_dbs} compares the variable selection performance of $\lasso$ with that of the two-stage scheme having the debiased $\lasso$ estimate in the first stage. Compare this figure with Figure \ref{fig:Lasso_roc_cpr} to see the difference between the two-stage $\lasso$ and two-stage debiased $\lasso$.  
\end{remark}
{
\setlength{\tabcolsep}{0pt}
\begin{figure}[t!]
\begin{center}
\begin{tabular}{rccc}
    & \scriptsize{$\sigma=0.5$} & \scriptsize{$\sigma=0.22$} & \scriptsize{$\sigma=0.15$} \\
	\rotatebox{90}{\qquad\qquad\quad\tiny{AFDP}} & \includegraphics[scale=0.4]{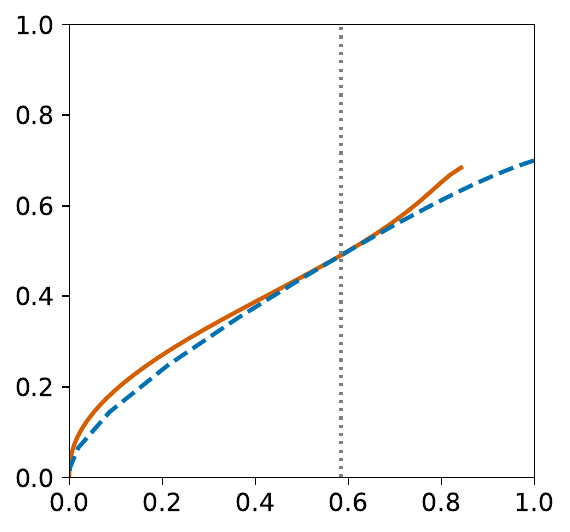}
	&
	\includegraphics[scale=0.4]{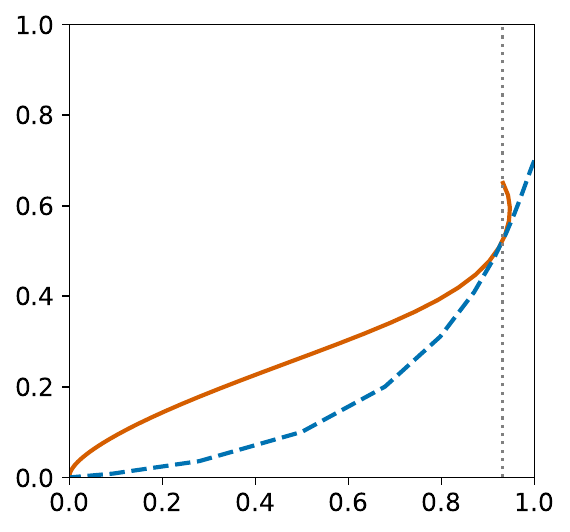}
	&
	\includegraphics[scale=0.4]{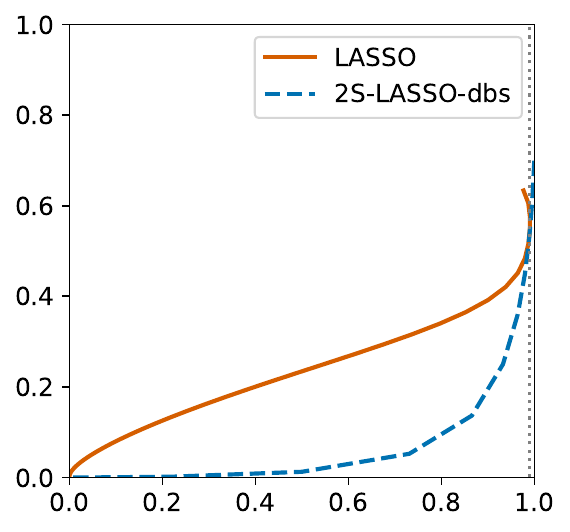} \\
	& \tiny{ATPP} & \tiny{ATPP} & \tiny{ATPP}
\end{tabular}
\caption{Comparison of {\rm AFDP-ATPP} curve between $\lasso$ and two-stage debiased $\lasso$. Here we pick the setting $\delta=0.8$, $\epsilon=0.3$, $\sigma\in\{0.5, 0.22, 0.15\}, p_G=\delta_1$. For the two-stage debiased $\lasso$, we use optimal tuning $\lambda_1^*$ in the first stage. The gray dotted line is the upper bound for the two-stage $\lasso$ without debiasing can reach.}\label{fig:Lasso_roc_cpr_dbs}
\end{center}
\end{figure}
}
\begin{remark}
The debiasing does not present any extra gain to the two-stage variable selection technique based on bridge estimators with $q>1$. In other words, debiasing does not change the {\rm AFDP-ATPP} curve for $q>1$. 
\end{remark}

\subsection{Debiasing and Sure Independence Screening}\label{sec:surescreen}

Sure Independence Screening (SIS) is  a variable selection scheme proposed for ultra-high dimensional settings \cite{fan2008sure}. Our asymptotic setting is not considered an ultra-high dimensional asymptotic. We are also aware that SIS is typically used for screening out irrelevant variables and other variable selection methods, such as LASSO, will be applied afterwards. Nevertheless, we present a connection and comparison between our two-stage methods and SIS in the linear asymptotic regime. Such comparisons shed more light on the performance of SIS. It is straightforward to confirm that Sure Independence Screening is equivalent to 
\begin{equation*}
\bar{\beta}^\dagger(q, \infty,s) =  \eta_0(\hat{\beta}^\dagger(q, \infty); s^2/2)=  \eta_0(X^Ty; s^2/2).
\end{equation*}
Therefore, the main difference between the approach we propose in this paper and SIS, is that SIS sets $\lambda$ to $\infty$, while we select the value of $\lambda$ that minimizes AMSE.\footnote{Our approach is more aligned with the approach proposed in \cite{wasserman2009high}. However, \cite{wasserman2009high} uses data splitting to select $\lambda$. } This simple difference may give a major boost to the variable selection performance. The following lemma confirms this claim. 

{
\setlength{\tabcolsep}{0pt}
\begin{figure}[t!]
\begin{center}
\begin{tabular}{rccc}
    & \scriptsize{$\sigma=0.5$} & \scriptsize{$\sigma=0.22$} & \scriptsize{$\sigma=0.15$} \\
	\rotatebox{90}{\qquad\qquad\quad\tiny{AFDP}} & \includegraphics[scale=0.4]{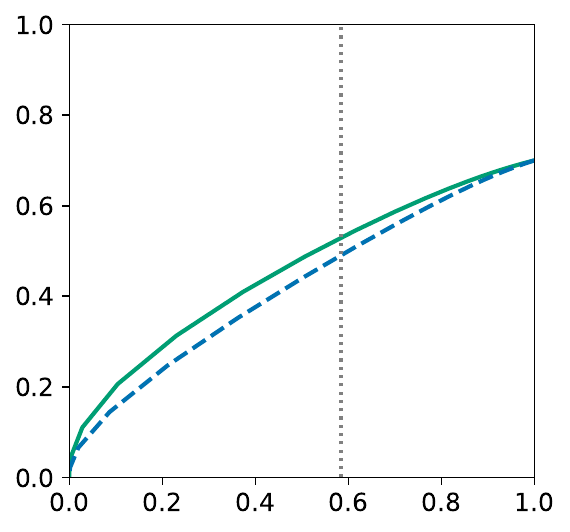}
	&
	\includegraphics[scale=0.4]{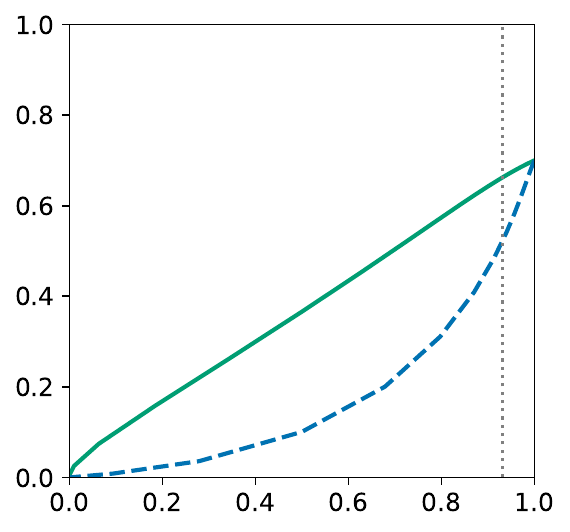}
	&
	\includegraphics[scale=0.4]{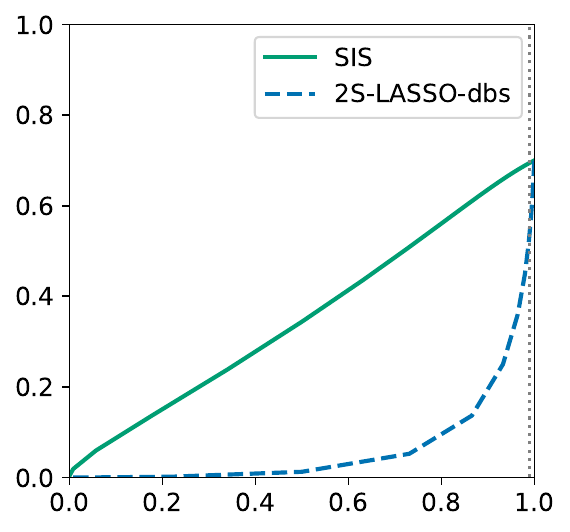} \\
	& \tiny{ATPP} & \tiny{ATPP} & \tiny{ATPP}
\end{tabular}
\caption{Comparison of {\rm AFDP-ATPP} curve between SIS and the two-stage debiased $\lasso$. Here we pick the setting $\delta=0.8$, $\epsilon=0.3$, $\sigma\in\{0.5, 0.22, 0.15\}, p_G=\delta_1$. For the two-stage debiased $\lasso$, we use optimal tuning $\lambda^*_{1}$ in the first stage. The gray dotted line is the upper bound that the two-stage $\lasso$ without debiasing can reach.}\label{fig:roc_cpr_sis_dbs}
\end{center}
\end{figure}
}

\begin{lemma}\label{COMP:SIS}
Consider $q\geq 1$. Given any {\rm ATPP} level $\zeta \in [0,1]$, let ${\rm AFDP}_{\rm sis} (\zeta)$ and ${\rm AFDP}^{\dagger}(q,\lambda_q^*,s(\lambda_q^*,\zeta))$ denote the asymptotic {\rm FDP} of SIS and two-stage debiased bridge estimator respectively, when their {\rm ATPP} is equal to $\zeta$. Then, 
${\rm AFDP}^{\dagger}(q, \lambda^*_{q}, s(\lambda^*_q,\zeta)) \leq {\rm AFDP}_{\rm sis} (\zeta)$.
\end{lemma}

Refer to Appendix \ref{apxd} for the proof. Note that when the noise $\sigma$
is large, we expect the optimally tuned $\lambda$ to be large, and hence the
performance of SIS gets closer to the TVS. However, as $\sigma$ decreases,
the gain obtained from using a better estimator in the first stage improves.
Figure \ref{fig:roc_cpr_sis_dbs} compares the performance of SIS and TVS under different noise settings.

\section{Numerical experiments}\label{sec:simulation}

\subsection{Objective and Simulation Set-up}
This section aims to investigate the finite sample performances of various
two-stage variable selection estimators under the three different regimes
analyzed in Section \ref{three:settings}. In particular, we will study to what
extent our theory works for more realistic situations, where model parameters
$\sigma$, $\epsilon$, $\delta$ are of moderate magnitudes or the iid-Gaussian
design assumption is violated. For brevity, we will use bridge estimator to refer to the corresponding 
two-stage method whenever it does not cause any confusion. More specifically, in all the figures, $\ell_q$ will be used to denote the TVS that uses the bridge estimator with $q$ in the first stage, and $\ell_1$-db denotes the two-stage debiased LASSO. The performances of different methods
will be compared via the $\afdp$-$\atpp$ curves.\footnote{Since the simulations
are in finite samples, the curve we calculate is actually FDP-TPP instead of
the asymptotic version. With a little abuse of notation, we will call it
AFDP-ATPP curve throughout the section.}

The organization of this section is as follows. In Sections
\ref{sec:optimalLambda} - \ref{ssec:lasso-vs-two-stage}, we focus on experiments
under iid-Gaussian design as assumed in our
theories. In Section \ref{ssec:general-design}, we present numerical results for non-i.i.d. or non-Gaussian designs to evaluate the accuracy of our results, when i.i.d. Gaussian assumption on $X$ is violated. 

We adopt the following settings for iid-Gaussian design. The settings for general design are described in Section \ref{ssec:general-design}.

\begin{enumerate}
    \item
    Number of variables is fixed at $p=5000$. Sample size $n=p\delta$ is then decided by $\delta$.
    \item
    Given the values of $\delta$, $\epsilon$, $\sigma$, we sample $X\in\mathbb{R}^{n\times p}$ with
    $X_{ij}\follow\mathcal{N}(0, \frac{1}{n})$. We pick the probability measure
    $p_G$ as a point mass at $M$ where $M$ will be specified in each scenario.
    We generate $\beta\in\mathbb{R}^p$ with $\beta_i\follow p_B=(1-\epsilon)\delta_0 + \epsilon p_G$, and 
    $w\in\mathbb{R}^n$ with $w_i\follow\mathcal{N}(0, \sigma^2)$ or
    $\mathcal{N}(0, \frac{\sigma^2}{\delta}).$\footnote{The setting $w_i\follow\mathcal{N}(0, \frac{\sigma^2}{\delta})$ will be  used in the large sample scenario, since we have scaled the error term by $\sqrt{\delta}$ in our asymptotic analysis in Section \ref{large:sample:study}.} Construct $y$ according to $y = X\beta + w$.
    \item
    For each data set $(y, X)$, $\afdp$-$\atpp$ curves will be generated for
    different variable selection methods. In each setting of parameters, 80 samples are drawn
    and the average $\afdp$-$\atpp$ curves are calculated. The associated one standard deviation confidence interval will be presented. 
\end{enumerate}

We compute bridge estimators via coordinate descent algorithm, with the proximal operator $\eta_q(x; \tau)$ calculated through a properly implemented Newton's method.

We discuss how to pick optimal tuning
under iid-Gaussian design in Section \ref{sec:optimalLambda}. Section \ref{ssec:smalllarge} presents the
large/small noise scenario. Section \ref{ssec:largenoisesim} is devoted to the
large sample regime. Section \ref{ssec:simu:near-black} covers the
nearly black object scenario. In Section \ref{ssec:lasso-vs-two-stage}, we compare the performance of LASSO and 
two-stage LASSO to shed more lights on our two-stage methods.

\subsection{Estimating the optimal tuning $\lambda_q^*$}\label{sec:optimalLambda}
For two-stage variable selection procedures, it is critical to have a good estimator in the first step. One challenge here is to search for the optimal tuning that minimizes $\amse$ of $\hat{\beta}(q, \lambda)$. According to the result of Theorem \ref{theorem:amp:bridge2} and the definition of ${\rm AMSE}$ in \eqref{eq:def:amse}, it is straightforward to see that $\tau^2 = \sigma ^ 2 + \frac{1}{\delta}\amse$. Hence, one can minimize $\tau^2$ to achieve the same optimal tuning. Motivated by \cite{mousavi2015consistent}, we can obtain a consistent estimator of $\tau^2$:
\begin{equation*}
	q = 1:\quad
	\hat{\tau}^2
	=
	  \frac{\|y - X\hat{\beta}(1, \lambda)\|_2^2}{n(1 -
        \|\hat{\beta}(1, \lambda)\|_0 / n)^2} ,
        \qquad
	q > 1:\quad
	\hat{\tau}^2
	=
	 \frac{\|y - X\hat{\beta}(q, \lambda)\|_2^2}{n(1 - f(\hat{\beta}(q, \lambda), \hat{\gamma}_\lambda) / n)^2} ,
\end{equation*}
where $f(\cdot,\cdot), \hat{\gamma}_{\lambda}$ are the same as the ones in
\eqref{eq:debiasing:lq} and \eqref{eq:solvegamma}. The consistency
$\hat{\tau}\overset{a.s.}{\rightarrow} \tau$ can be easily seen from the proof
of Theorem \ref{THEOREM:DEBIASING:VALID}. We thus do not repeat it. As a
result, we approximate $\lambda_q^*$ by searching for the $\lambda$ that
minimizes $\hat{\tau}^2$. Notice that this problem has been studied for
$\lasso$ in \cite{mousavi2015consistent} and a generalization is
straightforward for other bridge estimators. We use the following grid search strategy:
\begin{itemize}
	\item
	Initialization: An initial search region $[a, b]$, a window size $\Delta$ and a grid size $m$.
	\item
	Searching:
	A grid with size $m$ is built over $[a, b]$, upon which we search in descending order for $\lambda$ that minimizes $\hat{\tau}^2$ with warm initialization.
	\begin{itemize}
		\item
		If the minimal point $\hat{\lambda}\in(a, b)$, stop searching and return $\hat{\lambda}$.
		\item
		If $\hat{\lambda}=a$ or $b$, update the search region with $[\frac{a}{10}, a]$ or $[b, b + \Delta]$ and do the next round of searching.
	\end{itemize}
	\item
	Stability:
	If the optimal $\hat{\lambda}$ obtained from two consecutive search regions are smaller than a threshold $\epsilon_0$, we stop and return the previous optimal $\hat{\lambda}$; If the number of non-zero locations of a $\lasso$ estimator is larger than $n$ (which may happen numerically for very small tuning), we set its $\hat{\tau}^2$ to $\infty$.
\end{itemize}

For our experiments, we pick the initial $[a,b ] = [0.1, \frac{1}{2}\|X^Ty\|_\infty]$, $\Delta = \frac{1}{2}\|X^Ty\|_\infty$ and $m=15$. 

\subsection{From large noise to small noise}\label{ssec:smalllarge}
Theorems \ref{THM:NOISY:MAIN} and \ref{THM:LOWNOISE:MAIN} showed that in
low and high SNR situations, ridge and $\lasso$ offer the best
performances respectively. These results are obtained for limiting cases $\sigma \rightarrow
\infty$ and $\sigma \rightarrow 0$. In this section, we run a few simulations
to clarify the scope of applicability of our analysis. Toward this goal, we fix
the probability measure $p_G=\delta_M$ with $M=8$ and run TVS for $q\in\{1, 1.2, 2, 4\}$ and debiased LASSO\footnote{We include the results for two-stage debiased LASSO in Sections \ref{ssec:smalllarge} - \ref{ssec:simu:near-black} to validate the effect of debiasing stated in Theorem \ref{THEOREM:DEBIASING:ATTP} and Remark \ref{remark:bebiasing:lasso}.} under four settings:

\begin{enumerate}
	\item
	$\delta = 0.8$, $\epsilon = 0.2$: The results are shown in Figure
        \ref{fig:fLN_4}. Here we pick $\sigma \in \{1.5,3,5\}$. As
        expected from our theoretical results, for small values of noise
        $\lasso$ offers the best performance. As we increase the noise,
        eventually ridge outperforms $\lasso$ and the other bridge estimators. Note that under this setting, the
        outperformance occurs at a high noise level so that all estimators have large errors. In this example, we make $1>\delta> M_1(\epsilon)$. Refer to Theorem \ref{THM:LOWNOISE:MAIN} for the importance of this condition. 

{
\setlength{\tabcolsep}{0pt}
\begin{figure}[htbp!]
\begin{center}
\begin{tabular}{rccc}
    & \scriptsize{$\sigma=1.5$} & \scriptsize{$\sigma=3$} & \scriptsize{$\sigma=5$} \\
	\rotatebox{90}{\qquad\qquad\quad\tiny{AFDP}} & \includegraphics[scale=0.4]{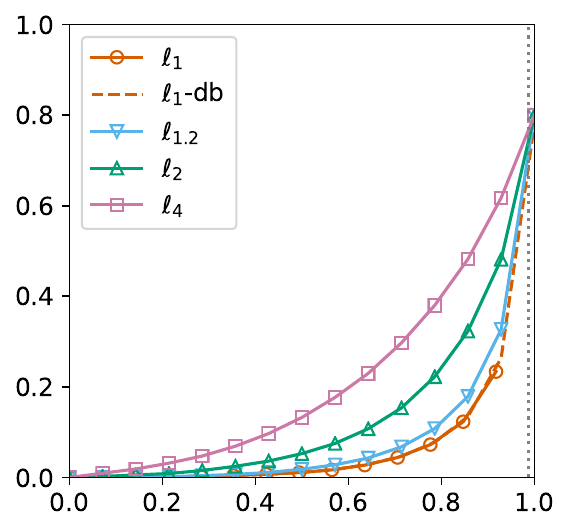}
	&
	\includegraphics[scale=0.4]{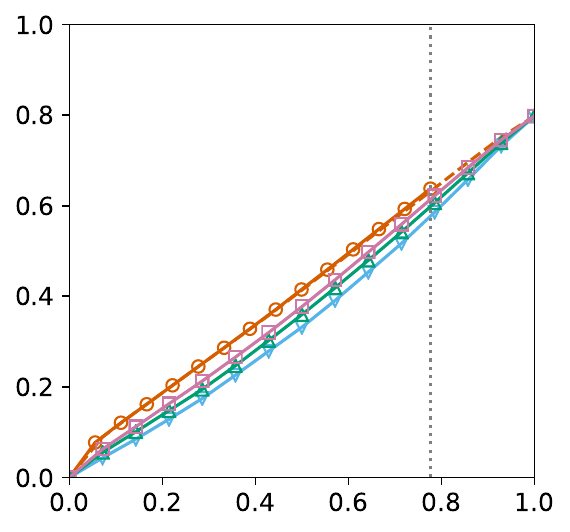}
	&
	\includegraphics[scale=0.4]{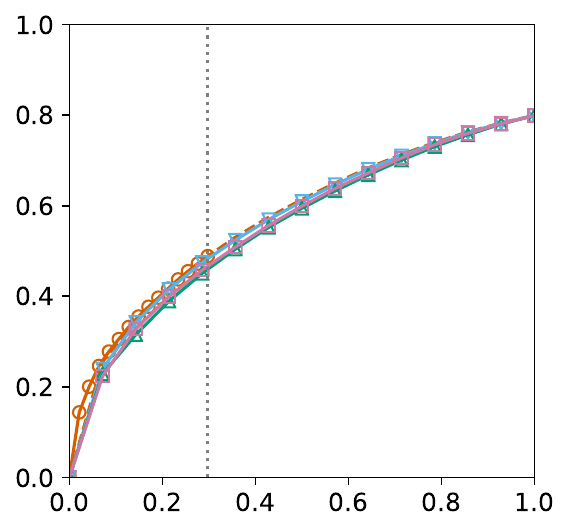} \\
	& \tiny{ATPP} & \tiny{ATPP} & \tiny{ATPP} \\
	&  &  &  \\
	\rotatebox{90}{\qquad\quad\tiny{$\mathrm{AFDP}_q - \mathrm{AFDP}_2$}} & \includegraphics[scale=0.4]{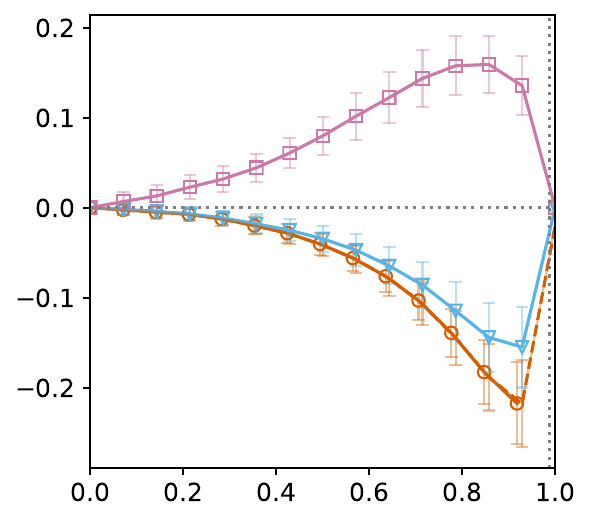}
	&
	\includegraphics[scale=0.4]{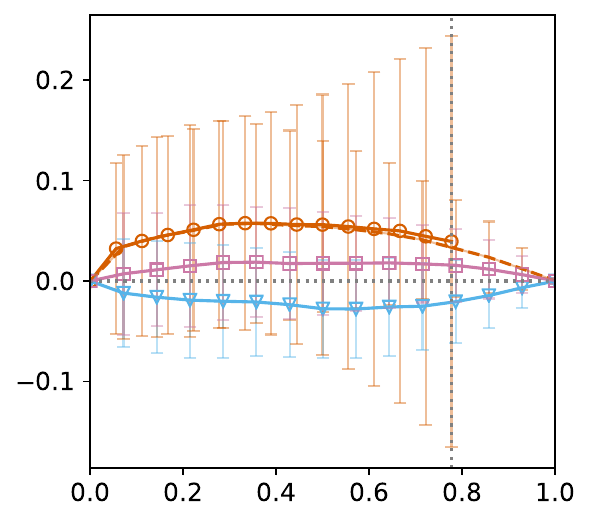}
	&
	\includegraphics[scale=0.4]{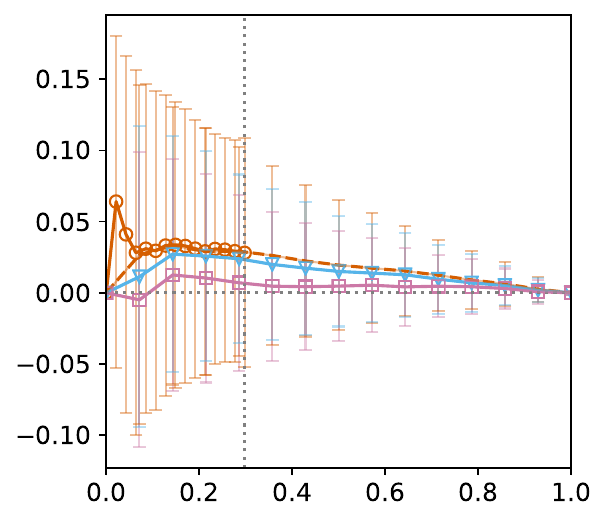} \\
	& \tiny{ATPP} & \tiny{ATPP} & \tiny{ATPP} \\
	&  &  &  \\
	\hline
	&  &  &  \\
    & \scriptsize{$\sigma=2$} & \scriptsize{$\sigma=4$} & \scriptsize{$\sigma=8$} \\
	\rotatebox{90}{\qquad\qquad\quad\tiny{AFDP}} & \includegraphics[scale=0.4]{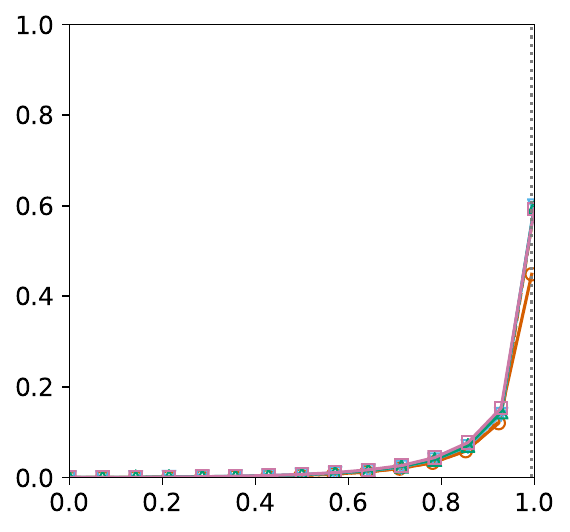}
	&
	\includegraphics[scale=0.4]{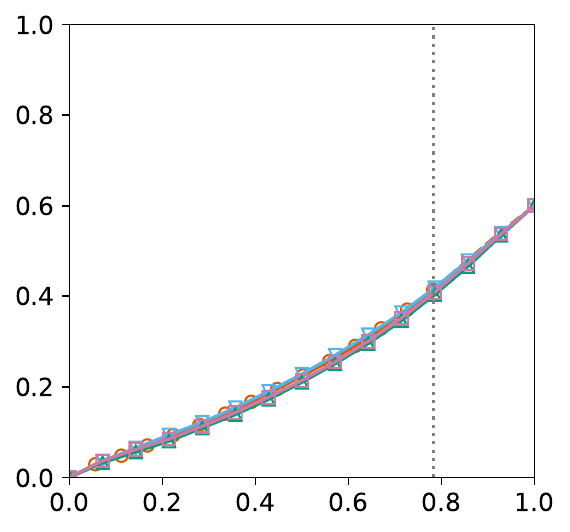}
	&
	\includegraphics[scale=0.4]{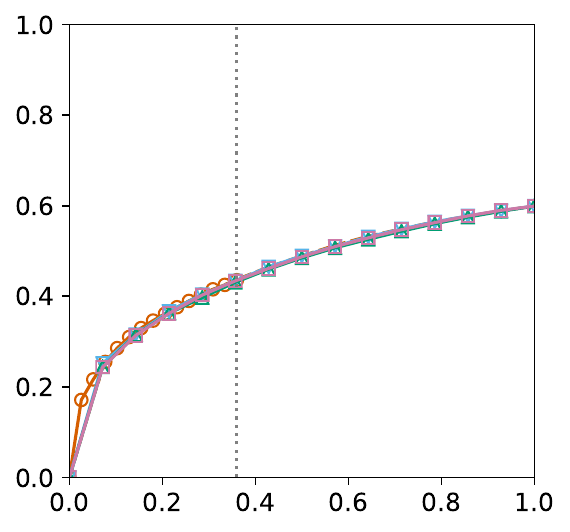} \\
	& \tiny{ATPP} & \tiny{ATPP} & \tiny{ATPP} \\
	&  &  &  \\
	\rotatebox{90}{\qquad\quad\tiny{$\mathrm{AFDP}_q - \mathrm{AFDP}_2$}} & \includegraphics[scale=0.4]{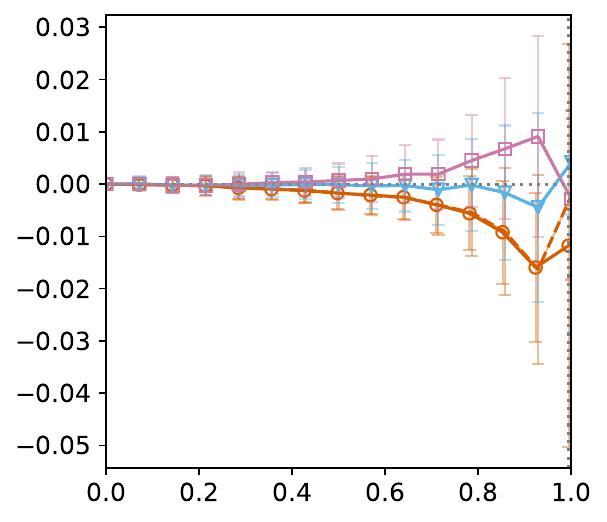}
	&
	\includegraphics[scale=0.4]{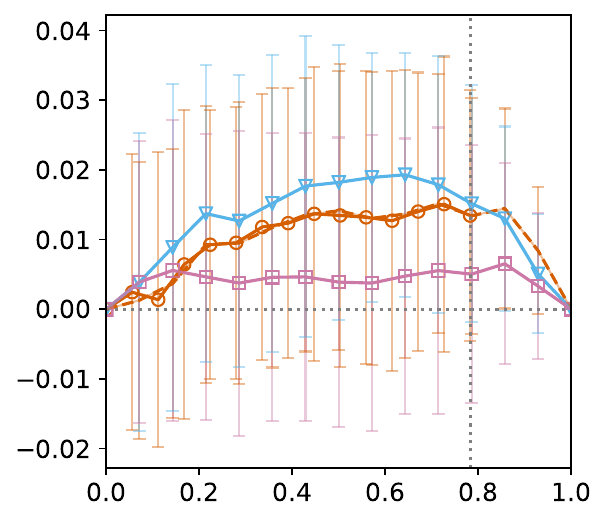}
	&
	\includegraphics[scale=0.4]{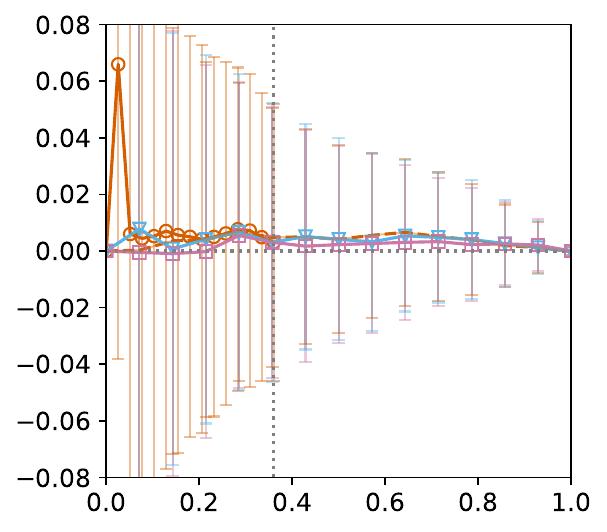} \\
	& \tiny{ATPP} & \tiny{ATPP} & \tiny{ATPP} \\
\end{tabular}
\caption{Top row: $\afdp$-$\atpp$ curve under the setting $\delta=0.8$, $\epsilon=0.2, \sigma\in\{1.5, 3, 5\}$. Second row: Y-axis is the difference of $\afdp$ between the other bridge estimators and ridge. One standard deviation of the difference is added. Third and fourth rows: the same type of plots as in the first two rows, under the setting $\delta=2$, $\epsilon=0.4, \sigma\in\{2, 4, 8\}$.}\label{fig:fLN_4}
\end{center}
\end{figure}
}

{
\setlength{\tabcolsep}{0pt}
\begin{figure}[thbp!]
\begin{center}
\begin{tabular}{rccc}
    & \scriptsize{$\sigma=0.25$} & \scriptsize{$\sigma=0.75$} & \scriptsize{$\sigma=2$} \\
	\rotatebox{90}{\qquad\qquad\quad\tiny{AFDP}} & \includegraphics[scale=0.4]{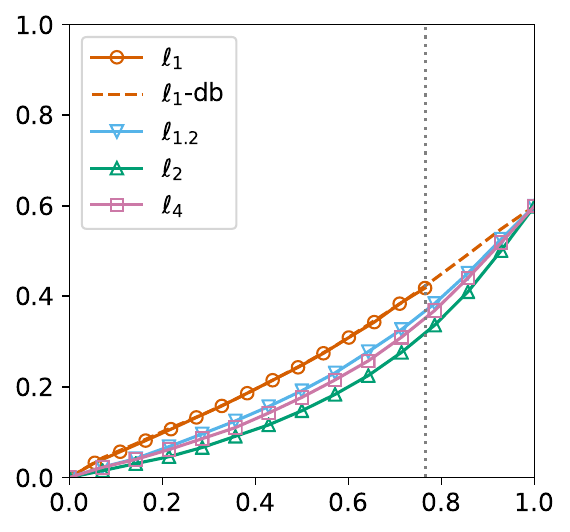}
	&
	\includegraphics[scale=0.4]{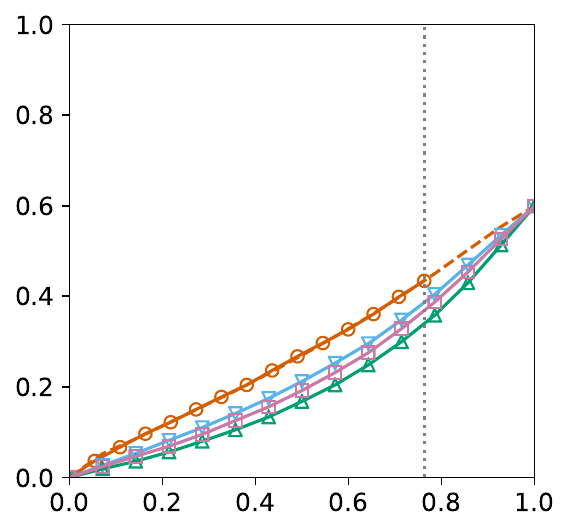}
	&
	\includegraphics[scale=0.4]{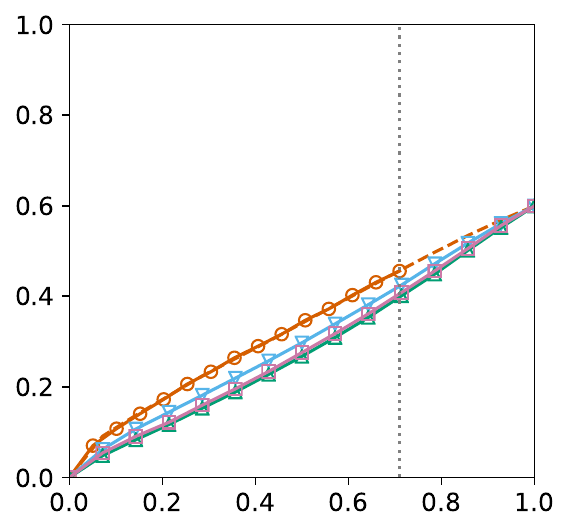} \\
	& \tiny{ATPP} & \tiny{ATPP} & \tiny{ATPP} \\
	&  &  &  \\
	\rotatebox{90}{\qquad\quad\tiny{$\mathrm{AFDP}_q - \mathrm{AFDP}_2$}} & \includegraphics[scale=0.4]{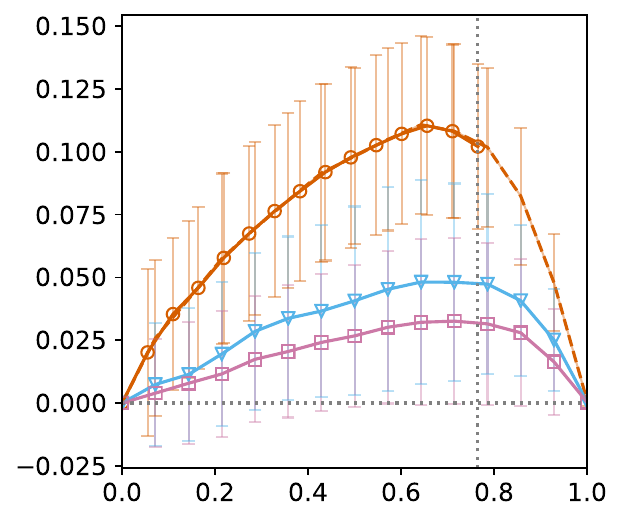}
	&
	\includegraphics[scale=0.4]{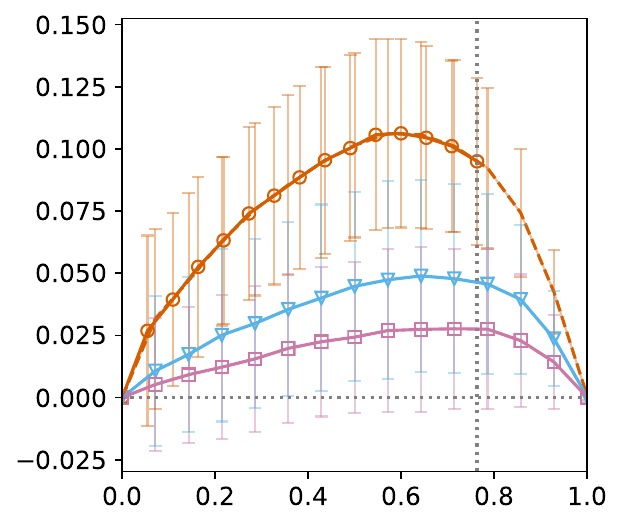}
	&
	\includegraphics[scale=0.4]{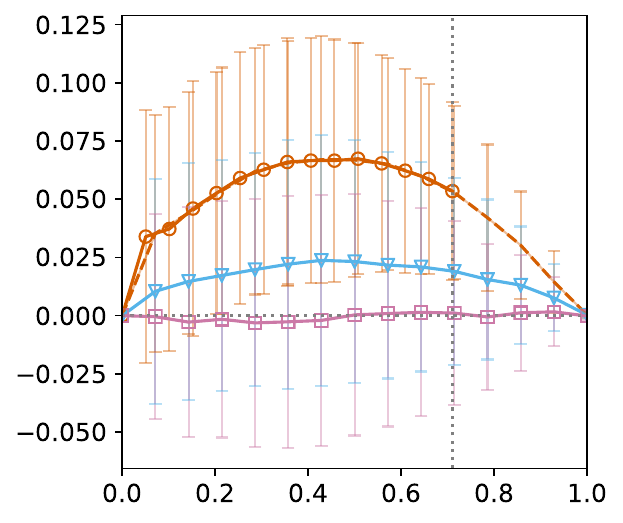} \\
	& \tiny{ATPP} & \tiny{ATPP} & \tiny{ATPP} \\
	&  &  &  \\
	\hline
	&  &  &  \\
    & \scriptsize{$\sigma=1.2$} & \scriptsize{$\sigma=1.5$} & \scriptsize{$\sigma=1.9$} \\
	\rotatebox{90}{\qquad\qquad\quad\tiny{AFDP}} & \includegraphics[scale=0.4]{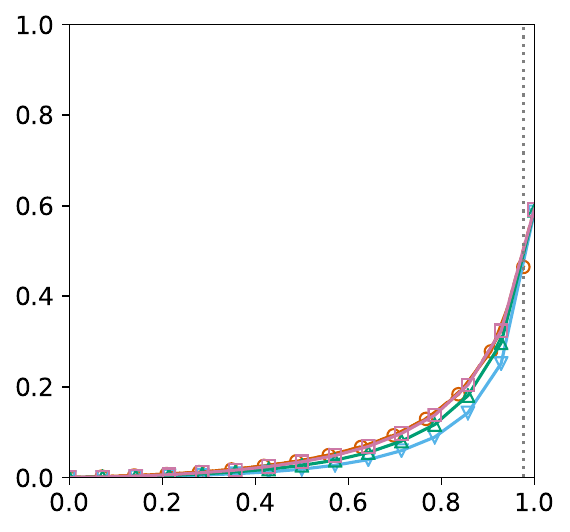}
	&
	\includegraphics[scale=0.4]{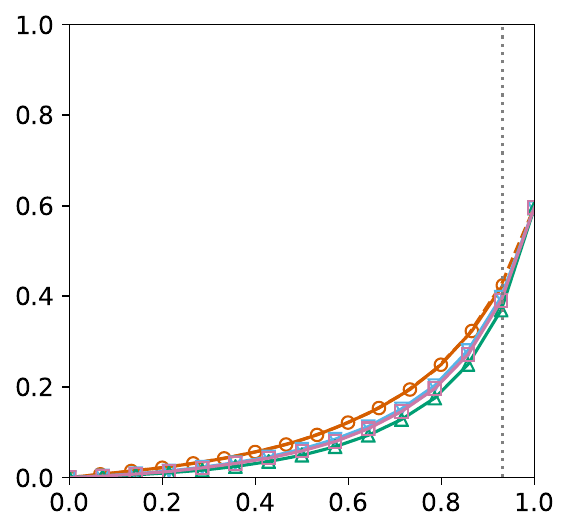}
	&
	\includegraphics[scale=0.4]{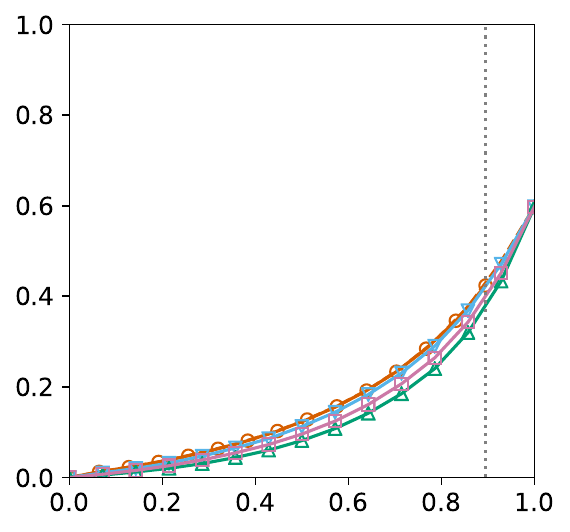} \\
	& \tiny{ATPP} & \tiny{ATPP} & \tiny{ATPP} \\
	&  &  &  \\
	\rotatebox{90}{\qquad\quad\tiny{$\mathrm{AFDP}_q - \mathrm{AFDP}_2$}} & \includegraphics[scale=0.4]{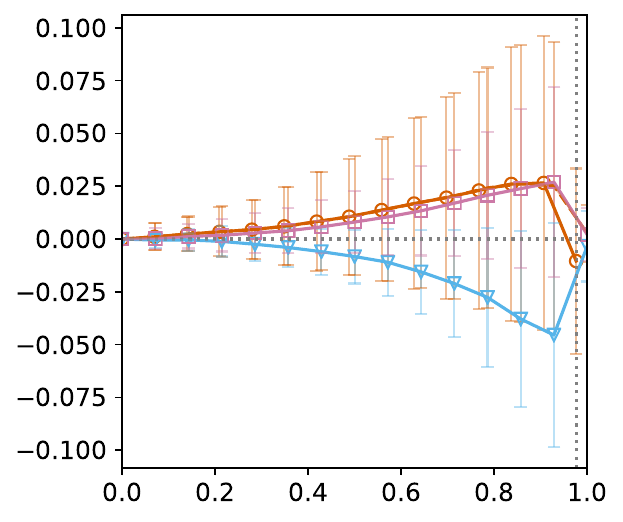}
	&
	\includegraphics[scale=0.4]{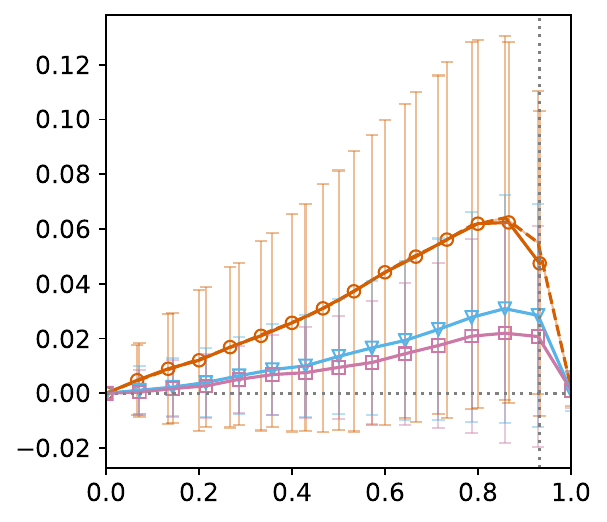}
	&
	\includegraphics[scale=0.4]{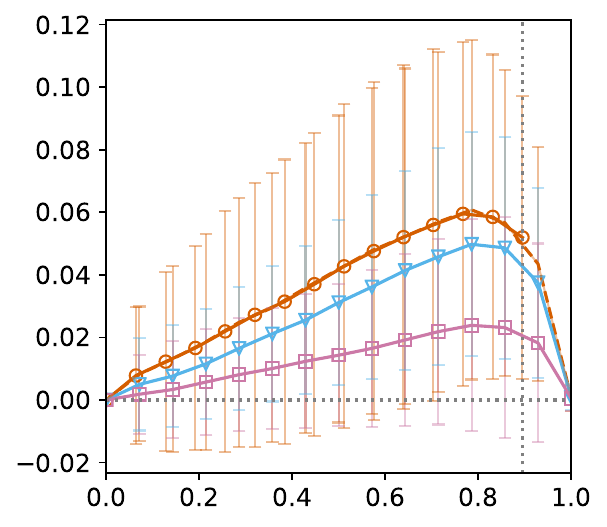} \\
	& \tiny{ATPP} & \tiny{ATPP} & \tiny{ATPP} \\
\end{tabular}
\caption{Top row: $\afdp$-$\atpp$ curve under the setting $\delta=0.6$, $\epsilon=0.4, \sigma\in\{0.25, 0.75, 2\}$. Second row: Y-axis is the difference of $\afdp$ between the other bridge estimators and ridge. One standard deviation of the difference is added. Third and fourth rows: the same type of plots as in the first two rows, under the setting $\delta=0.9$, $\epsilon=0.4, \sigma\in\{1.2, 1.5, 1.9\}$.}\label{fig:fLN_1}
\end{center}
\end{figure}
}

\item $\delta= 2$, $\epsilon=0.4$:  The results are included in Figure \ref{fig:fLN_4}. Here we pick $\sigma \in \{2, 4, 8\}$. Similar phenomena are observed. However for all choices of $\sigma$, the $\afdp$-$\atpp$ curves of different methods are quite close to each other. 

\item $\delta = 0.6$, $\epsilon=0.4$: Figure \ref{fig:fLN_1} contains the
    results for this part. Here we have $\sigma \in \{0.25, 0.75, 2\}$. An 
    important feature of this simulation is that $\delta< M_1(\epsilon)$, which does
    not satisfy the condition of Theorem \ref{THM:LOWNOISE:MAIN}. 
    It is interesting to observe that in this case, ridge outperforms $\lasso$ even for small values of the noise. 
    We thus see that the superiority of LASSO in small noise characterized by Theorem \ref{THM:LOWNOISE:MAIN} may not 
    hold when the conditions of the theorem are violated. In fact, Theorem \ref{THM:LOWNOISE:MAIN} is restricted to the regime below 
    the phase transition (i.e., when the signal can be fully recovered without noise). However, in the current setting, the optimal $\amse$ for $q=1, 1.2,
    2, 4$ at $\sigma=0$ are $14.9, 12.2, 10.2, 11.6$, respectively.

\item $\delta = 0.9$, $\epsilon=0.4$: The results are shown in Figure \ref{fig:fLN_1}. Here we have $\sigma \in \{1.2, 1.5, 1.9\}$.  This
    group of figures provide us with examples where ridge based TVS
    outperforms the other two-stage methods, and at the same time reaches a quite satisfactory AFDP-ATPP trade-off. For instance, when $\sigma=1.5$ and $\afdp\approx 0.2$, for ridge we have $\atpp\approx 0.8$ while that for $\lasso$ is around 0.7. Note that here $M_1(\epsilon) < \delta<1$.
\end{enumerate}

\subsection{Large sample regime}\label{ssec:largenoisesim}

We will validate the results in Theorem \ref{THM:SAMPLE:MAIN},
which are obtained under the limiting case $\delta \rightarrow \infty$. We
fix the probability measure $p_G=\delta_M$ with $M=1$ and consider the following settings for $q\in\{1, 1.5, 2, 4\}$ and debiased LASSO:

{
\setlength{\tabcolsep}{0pt}
\begin{figure}[htbp!]
\begin{center}
\begin{tabular}{rccc}
    & \scriptsize{$\delta=2$} & \scriptsize{$\delta=3$} & \scriptsize{$\delta=4$} \\
	\rotatebox{90}{\qquad\qquad\quad\tiny{AFDP}} & \includegraphics[scale=0.4]{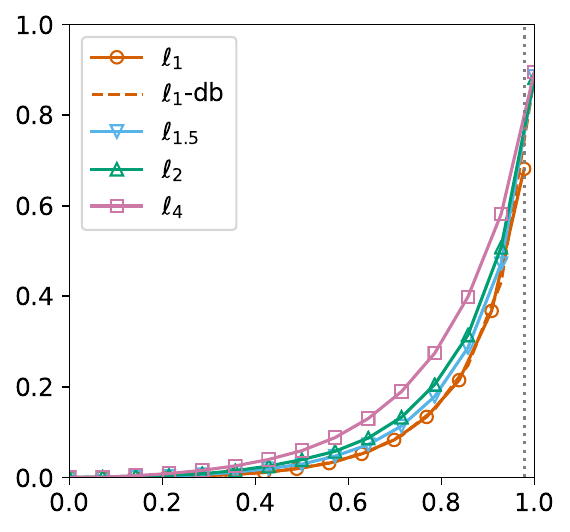}
	&
	\includegraphics[scale=0.4]{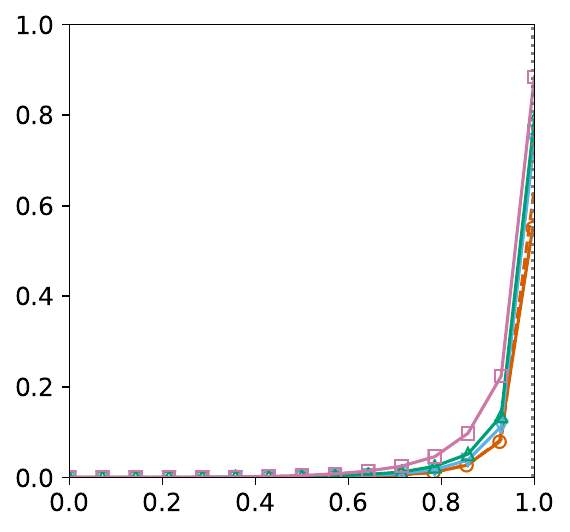}
	&
	\includegraphics[scale=0.4]{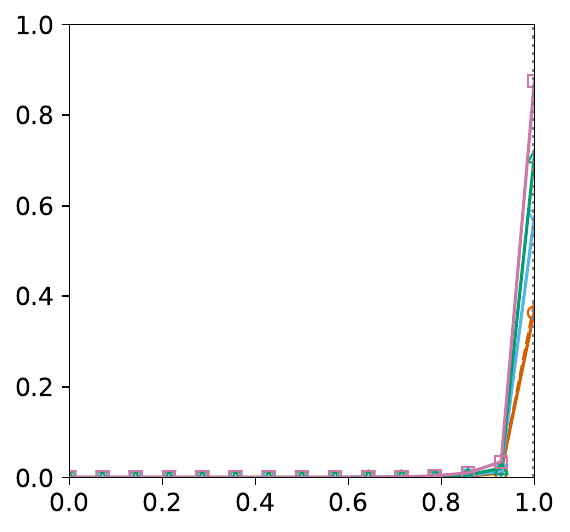} \\
	& \tiny{ATPP} & \tiny{ATPP} & \tiny{ATPP} \\
	&  &  &  \\
	\rotatebox{90}{\qquad\quad\tiny{$\mathrm{AFDP}_q - \mathrm{AFDP}_1$}} & \includegraphics[scale=0.4]{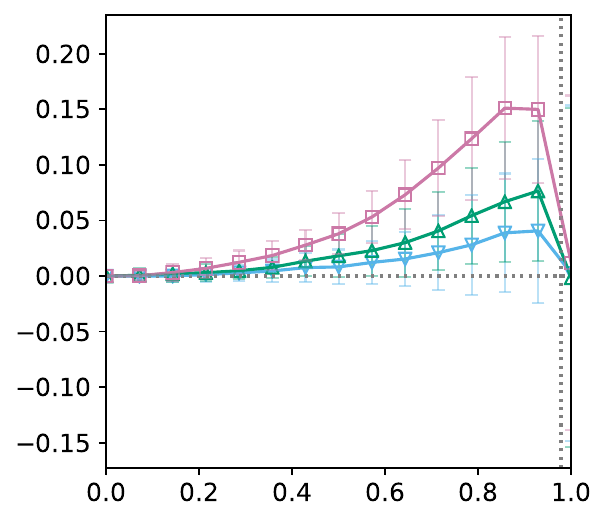}
	&
	\includegraphics[scale=0.4]{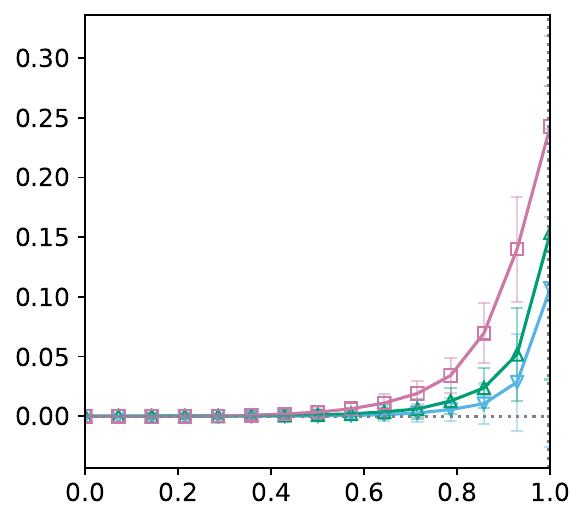}
	&
	\includegraphics[scale=0.4]{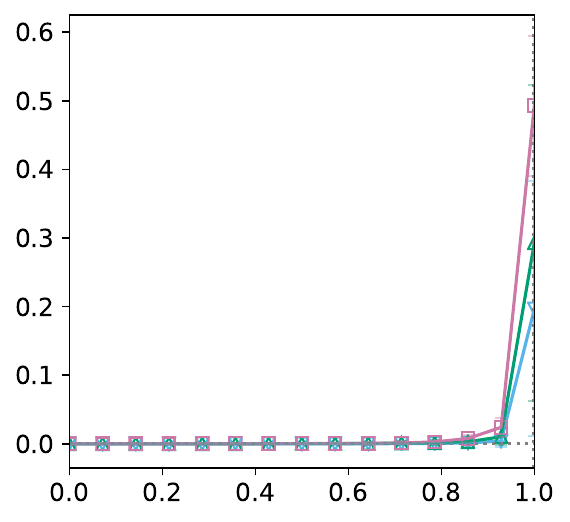} \\
	& \tiny{ATPP} & \tiny{ATPP} & \tiny{ATPP} \\
	&  &  &  \\
	\hline
	&  &  &  \\

    & \scriptsize{$\delta=2$} & \scriptsize{$\delta=3$} & \scriptsize{$\delta=4$} \\
	\rotatebox{90}{\qquad\qquad\quad\tiny{AFDP}} & \includegraphics[scale=0.4]{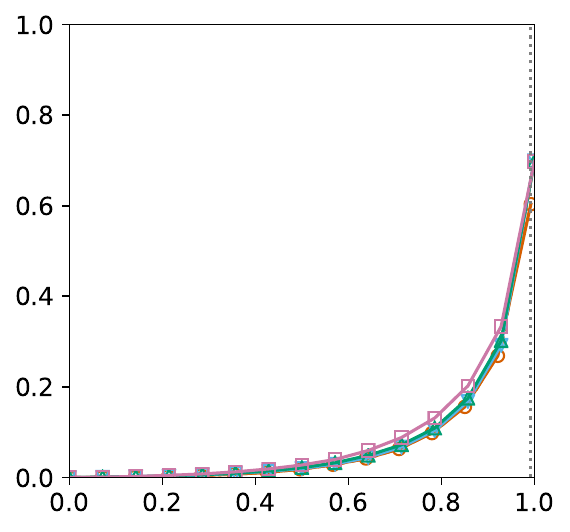}
	&
	\includegraphics[scale=0.4]{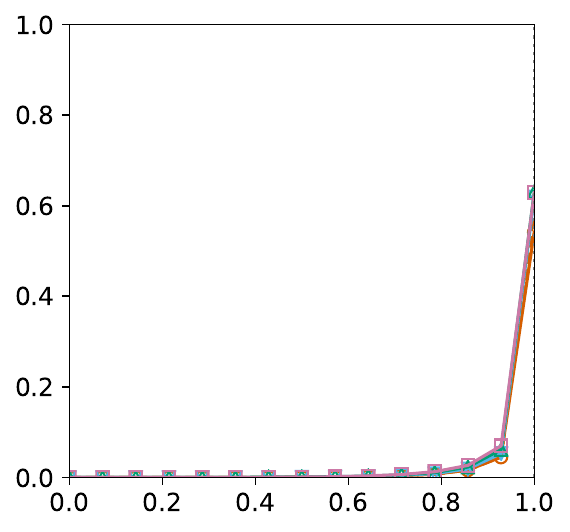}
	&
	\includegraphics[scale=0.4]{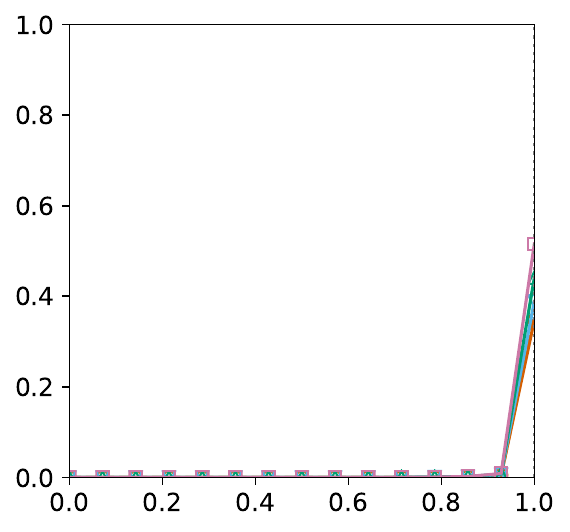} \\
	& \tiny{ATPP} & \tiny{ATPP} & \tiny{ATPP} \\
	&  &  &  \\
	\rotatebox{90}{\qquad\quad\tiny{$\mathrm{AFDP}_q - \mathrm{AFDP}_1$}} & \includegraphics[scale=0.4]{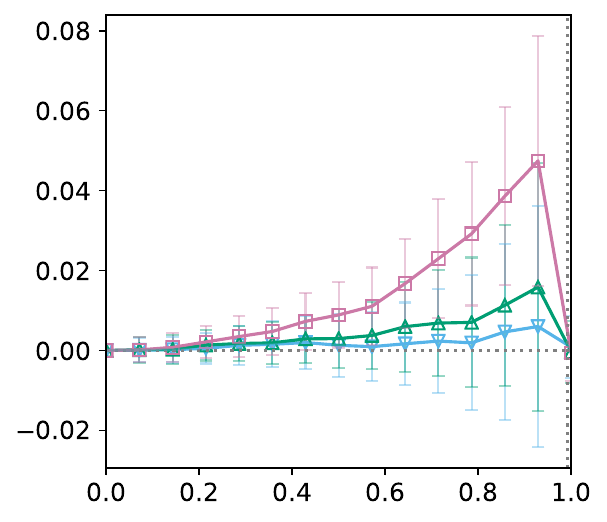}
	&
	\includegraphics[scale=0.4]{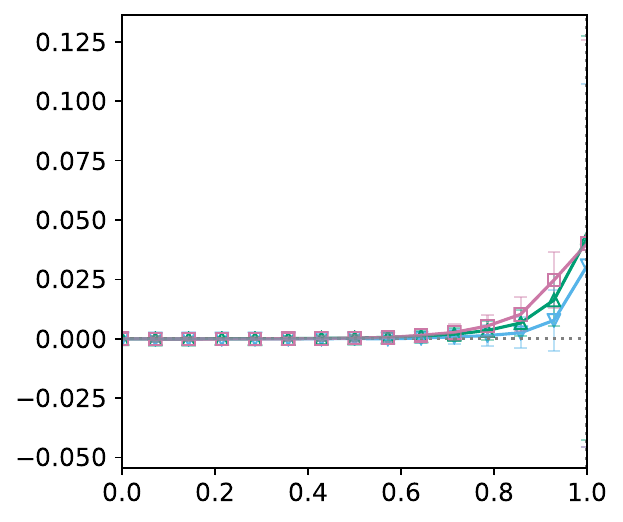}
	&
	\includegraphics[scale=0.4]{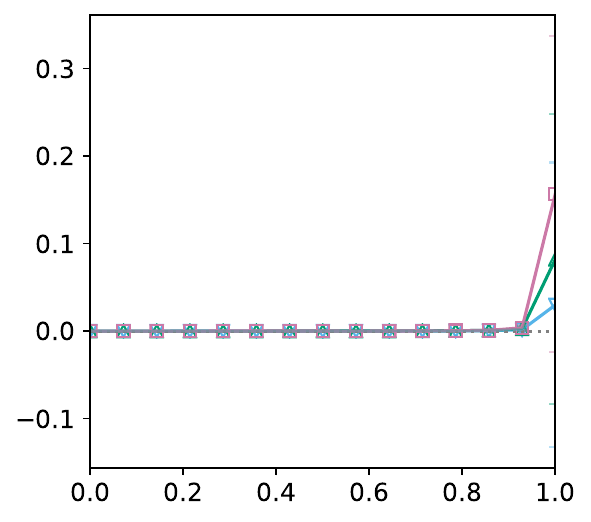} \\
	& \tiny{ATPP} & \tiny{ATPP} & \tiny{ATPP} \\
\end{tabular}
\caption{Top row: $\afdp$-$\atpp$ curve under the setting $\epsilon=0.1, \sigma=0.4, \delta \in\{2, 3, 4\}$. Second row: Y-axis is the difference of $\afdp$ between the other bridge estimators and LASSO. One standard deviation of the difference is added. Third and fourth rows: the same type of plots as in the first two rows, under the setting $\epsilon=0.3, \sigma=0.4, \delta \in\{2, 3, 4\}$.}\label{fig:fLS_1}
\end{center}
\end{figure}
}
 
\begin{enumerate}
    \item $\epsilon = 0.1$, $\sigma = 0.4$: The results for this setting are shown in Figure \ref{fig:fLS_1}. We vary $\delta \in \{2, 3, 4\}$. As is clear, $\lasso$ starts to outperform the others even when $\delta=2$. As $\delta$ increases, $\lasso$ remains the best, but all the methods are becoming better and the $\afdp$-$\atpp$ curves get closer to each other.
    
    \item $\epsilon = 0.3$, $\sigma = 0.4$: The results can be found 
        in Figure \ref{fig:fLS_1}. Again $\delta \in \{2, 3, 4\}$.
        Similar phenomena are observed. Compared to the previous setting, a
        larger $\epsilon$ leads to a higher SNR and all the methods have improved performances.
    \item $\epsilon = 0.4$, $\sigma = 0.22$: The results are shown in Figure
        \ref{fig:fLS_3}. We set $\delta\in\{0.7, 0.8, 1.2\}$. When $\delta$ is 0.7 or 0.8, ridge significantly outperforms the others. As $\delta$ is increased to 1.2, LASSO starts to lead the performances.
\end{enumerate}

{
\setlength{\tabcolsep}{0pt}
\begin{figure}[htbp!]
\begin{center}
\begin{tabular}{rccc}
    & \scriptsize{$\delta=0.7$} & \scriptsize{$\delta=0.8$} & \scriptsize{$\delta=1.2$} \\
	\rotatebox{90}{\qquad\qquad\quad\tiny{AFDP}} & \includegraphics[scale=0.4]{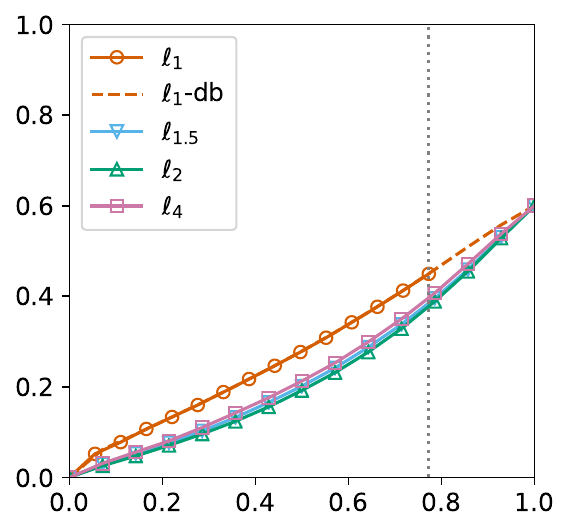}
	&
	\includegraphics[scale=0.4]{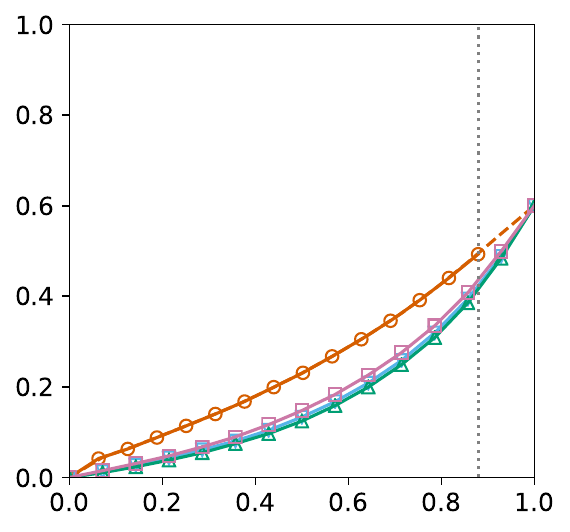}
	&
	\includegraphics[scale=0.4]{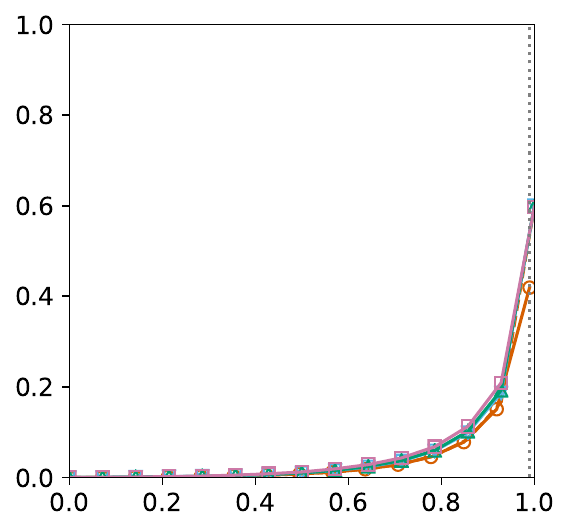} \\
	& \tiny{ATPP} & \tiny{ATPP} & \tiny{ATPP} \\
	&  &  &  \\
	\rotatebox{90}{\qquad\quad\tiny{$\mathrm{AFDP}_q - \mathrm{AFDP}_1$}} & \includegraphics[scale=0.4]{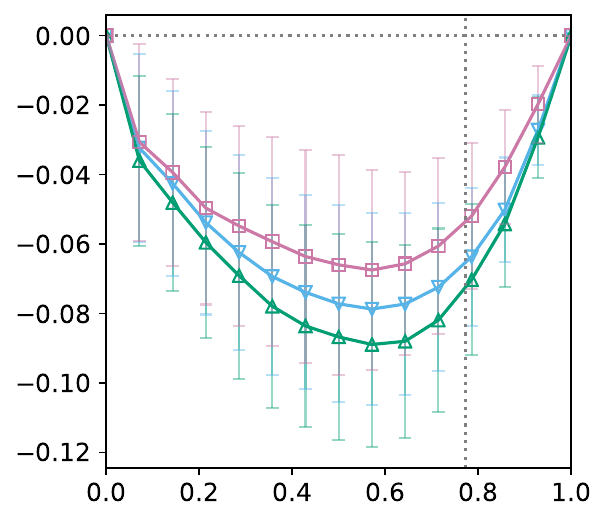}
	&
	\includegraphics[scale=0.4]{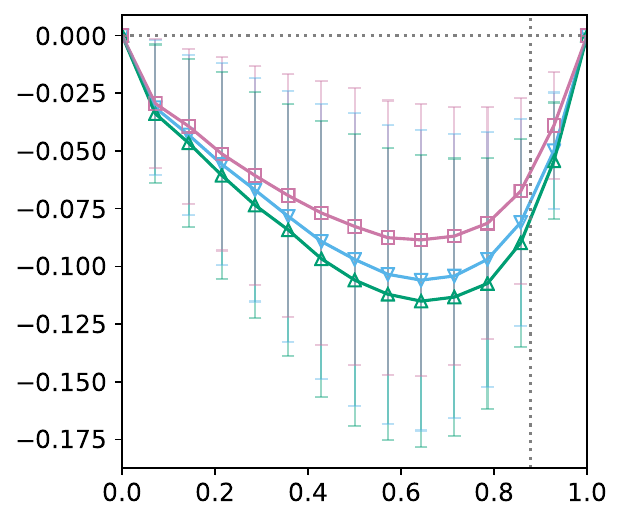}
	&
	\includegraphics[scale=0.4]{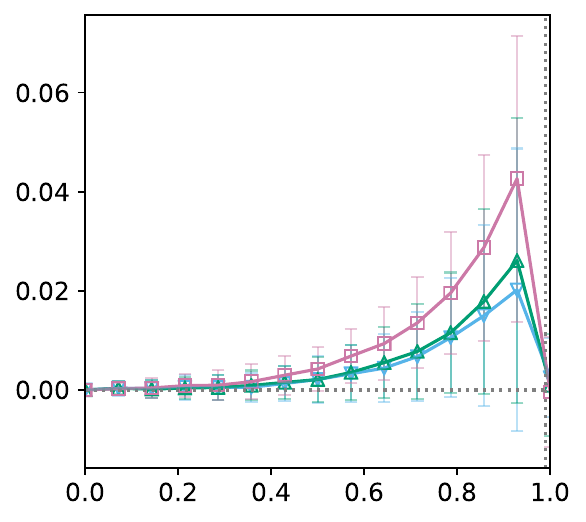} \\
	& \tiny{ATPP} & \tiny{ATPP} & \tiny{ATPP} \\
\end{tabular}
\caption{Top row: $\afdp$-$\atpp$ curve under the setting $\epsilon=0.4, \sigma=0.22, \delta \in\{0.7, 0.8, 1.2\}$. Second row: Y-axis is the difference of $\afdp$ between the other bridge estimators and LASSO. One standard deviation of the difference is added.}\label{fig:fLS_3}
\end{center}
\end{figure}
}

\subsection{Nearly black object}\label{ssec:simu:near-black}
In this section, we verify our theoretical results which are presented in Section \ref{sssec:nbo} for the nearly black object setting. Recall $b_{\epsilon}=\sqrt{\mathbb{E}G^2}$ and $\tilde{G}=G/b_{\epsilon}$. 
We consider the following setting:
$\delta=0.8$, $\sigma\in\{3, 5\}$, $b_\epsilon=4 / \sqrt{\epsilon}, \tilde{G}=1, \epsilon \in \{0.25, 0.0625, 0.04\}$. 
The simulation results are displayed in Figure \ref{fg:near-black-object}. We observe that under both noise levels $\sigma=3, 5$, LASSO is 
suboptimal at sparsity level $\epsilon=0.25$. As $\epsilon$ decreases, LASSO becomes better. When $\epsilon$ is reduced to $0.04$, LASSO outperforms the other bridge estimators by a large margin. Note that in this simulation, the signal strength $b_{\epsilon}$ scales with $\epsilon$ at the rate $\epsilon^{-1/2}$. This is the regime where LASSO is proved to be optimal in Section \ref{sssec:nbo}.

{
\setlength{\tabcolsep}{0pt}
\begin{figure}[htb!]
\begin{center}
\begin{tabular}{rccc}
    & \scriptsize{$\sigma=3$, $\epsilon=0.25$} & \scriptsize{$\sigma=3$, $\epsilon=0.0625$} & \scriptsize{$\sigma=3$, $\epsilon=0.04$} \\
	\rotatebox{90}{\qquad\qquad\quad\tiny{AFDP}} & \includegraphics[scale=0.4]{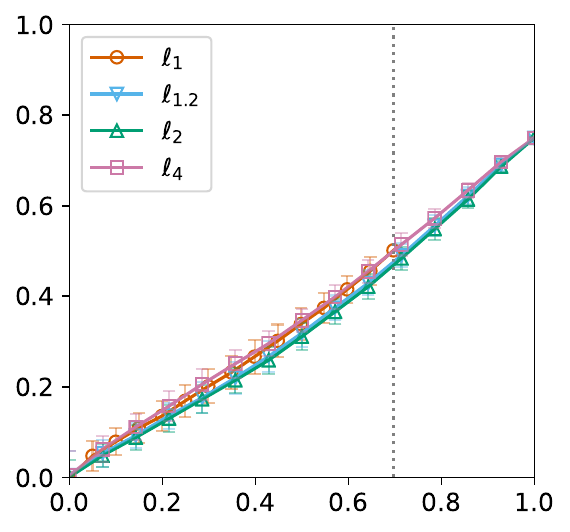}
	&
	\includegraphics[scale=0.4]{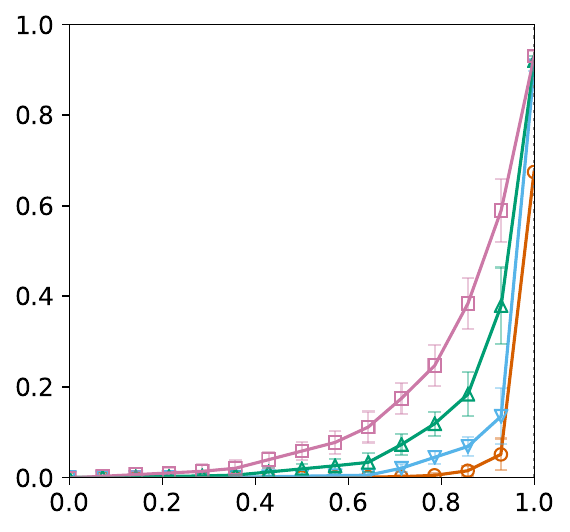}
	&
	\includegraphics[scale=0.4]{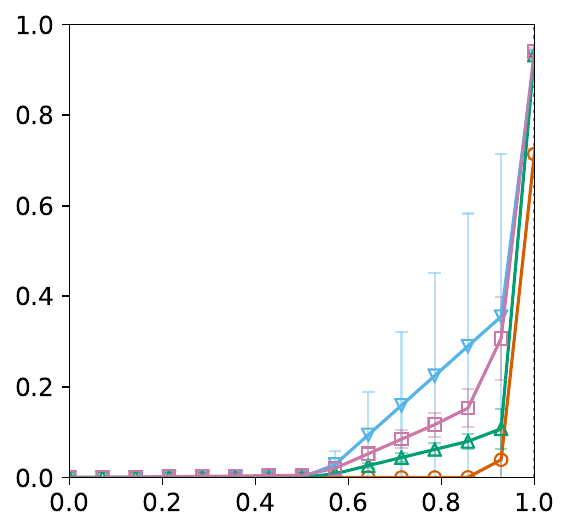} \\
	& \tiny{ATPP} & \tiny{ATPP} & \tiny{ATPP} \\
	&  &  &  \\ 
    & \scriptsize{$\sigma=5$, $\epsilon=0.25$} & \scriptsize{$\sigma=5$, $\epsilon=0.0625$} & \scriptsize{$\sigma=5$, $\epsilon=0.04$} \\
	\rotatebox{90}{\qquad\qquad\quad\tiny{$\mathrm{AFDP}$}} & \includegraphics[scale=0.4]{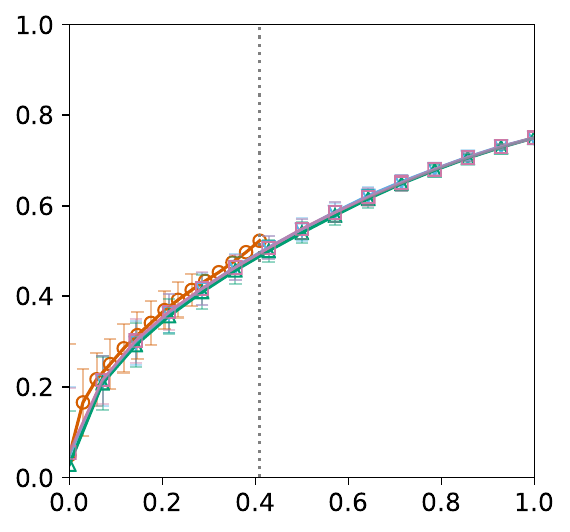}
	&
	\includegraphics[scale=0.4]{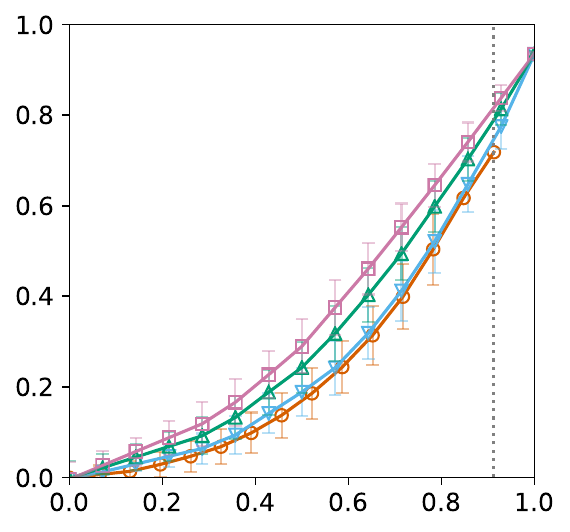}
	&
	\includegraphics[scale=0.4]{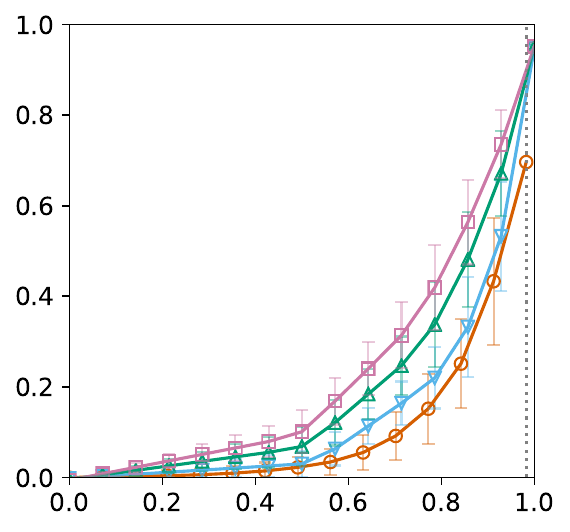} \\
	& \tiny{ATPP} & \tiny{ATPP} & \tiny{ATPP} \\
\end{tabular}
\caption{Top row: $\afdp$-$\atpp$ curve under the setting $b_{\epsilon}=4/\sqrt{\epsilon}, \sigma=3, \delta=0.8, \epsilon \in\{0.25, 0.0625, 0.04\}$. Second row: 
        $\afdp$-$\atpp$ curve under the setting $b_{\epsilon}=4/\sqrt{\epsilon}, \sigma=5, \delta=0.8, \epsilon \in\{0.25, 0.0625, 0.04\}$. One standard deviation is added.}\label{fg:near-black-object}
\end{center}
\end{figure}
}

\subsection{LASSO vs. two-stage LASSO} 
\label{ssec:lasso-vs-two-stage}
In Theorem \ref{THM:LASSOFDPCOMPARE} we proved that two-stage LASSO
with its first stage optimally tuned outperforms LASSO on variable
selection. We now provide a brief simulation to verify this
result. We choose $p_G=\delta_M$ with $M=8$ and set $\delta=0.8, \epsilon=0.2, \sigma\in \{1,3,5\}$. As shown in Figure \ref{fg:lasso-vs-2stage}, two-stage LASSO improves over LASSO. When the noise is small ($\sigma=1$), the improvement is the most significant. As the noise level increases, the difference between the two approaches becomes smaller. When the noise is large
($\sigma=5$), both have large errors. 
{
\setlength{\tabcolsep}{0pt}
\begin{figure}[htb!]
\begin{center}
\begin{tabular}{rccc}
    & \scriptsize{$\sigma=1$} & \scriptsize{$\sigma=3$} & \scriptsize{$\sigma=5$} \\
	\rotatebox{90}{\qquad\qquad\quad\tiny{$\mathrm{AFDP}$}} & \includegraphics[scale=0.4]{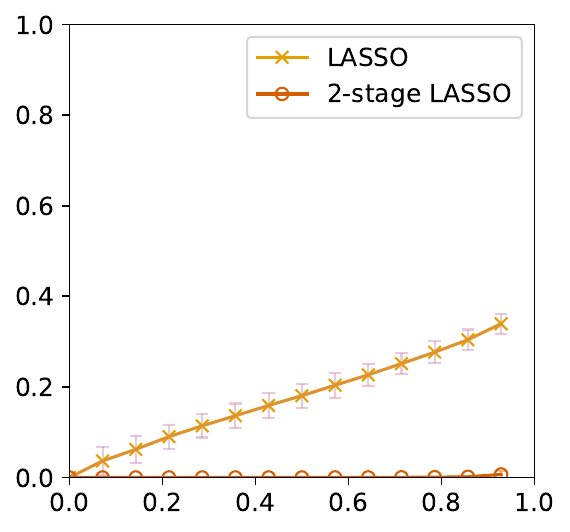}
	&
	\includegraphics[scale=0.4]{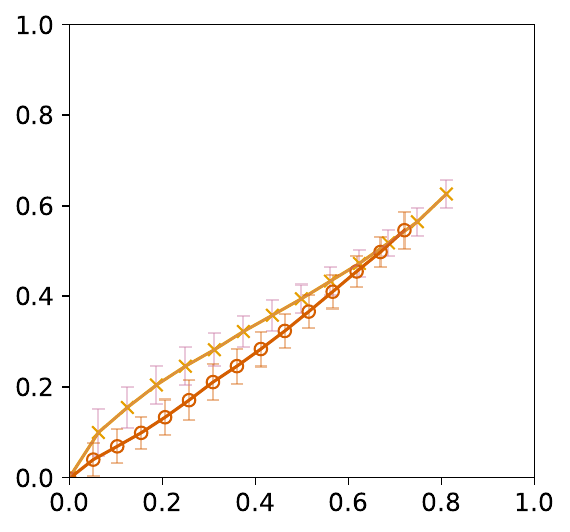}
	&
	\includegraphics[scale=0.4]{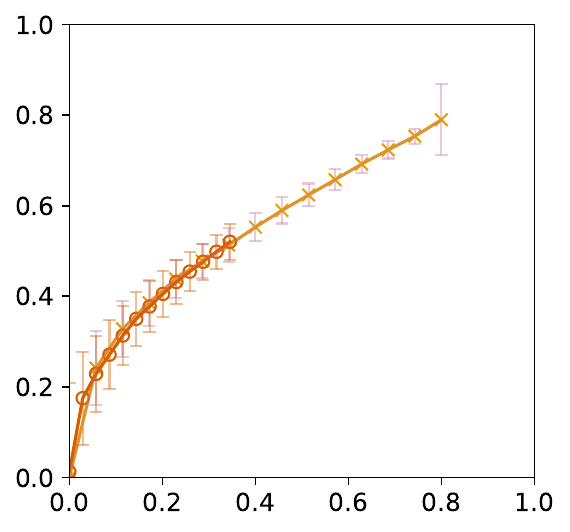} \\
	& \tiny{ATPP} & \tiny{ATPP} & \tiny{ATPP}
\end{tabular}
\caption{LASSO vs. two-stage LASSO. Here $\delta=0.8$, $\epsilon=0.2$,
            $M=8, \sigma\in \{1, 3, 5\}$. The outperformance of two-stage LASSO is
            the most significant when the noise level is low. When noise gets
            higher, the gap becomes smaller and smaller.}\label{fg:lasso-vs-2stage}
\end{center}
\end{figure}
}

\subsection{General design}\label{ssec:general-design}

In this section, we extend our simulations to general design matrices. Given that our  theoretical
results in Section \ref{sec:contribution} are derived under the i.i.d. Gaussian assumption on $X$, 
the aim of this section is to numerically study the validity scope of our main conclusions when such an assumption does not hold. In particular, we consider the following correlated
designs and i.i.d. non-Gaussian designs:
\begin{itemize}
    \item Correlated design: We consider the model $y = X\Sigma^{\frac{1}{2}}\beta + w$, where $X_{ij}\follow \mathcal{N}(0, \frac{1}{n})$ and $\Sigma$ is a Toeplitz matrix with $\Sigma_{ij} = \rho^{|i - j|}$. Here $\rho \in (0, 1)$ controls the
        correlation strength.
    \item i.i.d. non-Gaussian design: 
        We generate $X$ with i.i.d. components $X_{ij} \sim
        \sqrt{\frac{\nu - 2}{n\nu}}t_{\nu}$ where $t_{\nu}$ is the t-distribution with degrees of freedom $\nu$. The scaling $\sqrt{\frac{\nu - 2}{n\nu}}$ ensures $\mbox{var}(X_{ij})=\frac{1}{n}$ as in the i.i.d Gaussian case. 
\end{itemize}

Throughout this section, we choose $p=2500, p_G=\delta_M, n=\delta p, w_i\follow N(0,\sigma^2)$.

\paragraph{Large/small noise}
We set $M=8, \delta=0.9$, $\epsilon=0.4$. For correlated design, we vary $\rho\in\{0.1,
0.5, 0.9\}$ to allow for different levels of correlations among the predictors.
Figure \ref{fg:cor-design} shows the simulation results. There are a few important observations:
\begin{itemize}
\item[(i)] For a given $\rho \in \{0.1,0.5, 0.9\}$, the comparison of bridge estimators under different noise levels is similar to what we observe for 
i.i.d. Gaussian designs: LASSO performs best in low noise case, and ridge becomes optimal when the noise is large.

\item[(ii)] Given the noise level $\sigma=0.8$, as the design correlation $\rho$ varies in $\{0.1,0.5,0.9\}$, it is interesting to observe that, LASSO outperforms the other estimators when the correlation is not high ($\rho=0.1,0.5$), while ridge becomes the optimal one when the correlation is increased to $0.9$. Similar phenomenon happens at the noise level $\sigma=1$. It seems that in terms of variable selection performance comparison of TVS, adding dependency among the predictors is like increasing the noise level in the system. We leave a theoretical analysis of the impact of correlation on our results as an interesting future research.

\end{itemize}

{
\setlength{\tabcolsep}{0pt}
\begin{figure}[htb!]
\begin{center}
\begin{tabular}{rccc}
    & \scriptsize{$\rho=0.1$, $\sigma=0.8$} & \scriptsize{$\rho=0.1$, $\sigma=1$} & \scriptsize{$\rho=0.1$, $\sigma=2$} \\
	\rotatebox{90}{\qquad\qquad\quad\tiny{AFDP}} & \includegraphics[scale=0.4]{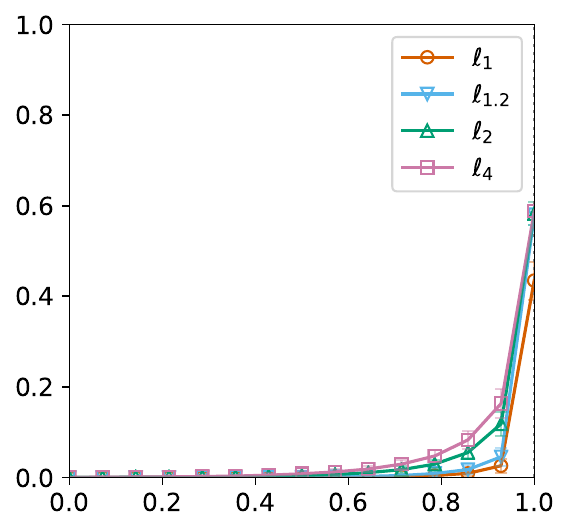}
	&
	\includegraphics[scale=0.4]{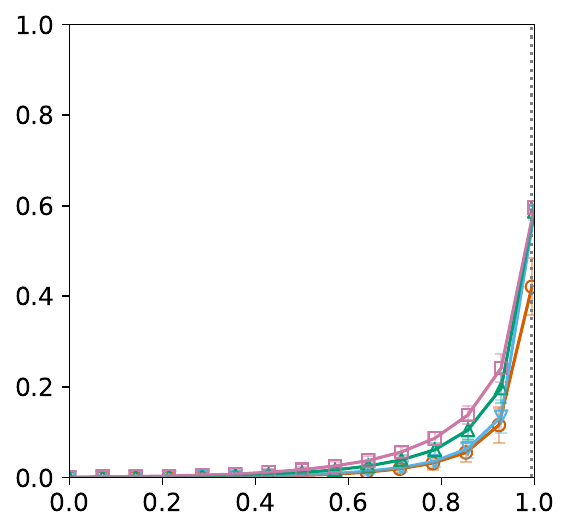}
	&
	\includegraphics[scale=0.4]{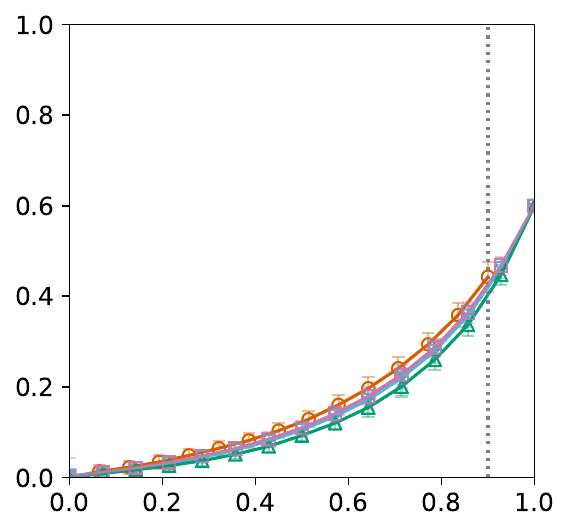} \\
	& \tiny{ATPP} & \tiny{ATPP} & \tiny{ATPP} \\
	&  &  &  \\ 
    & \scriptsize{$\rho=0.5$, $\sigma=0.8$} & \scriptsize{$\rho=0.5$, $\sigma=1$} & \scriptsize{$\rho=0.5$, $\sigma=2$} \\
	\rotatebox{90}{\qquad\qquad\quad\tiny{$\mathrm{AFDP}$}} & \includegraphics[scale=0.4]{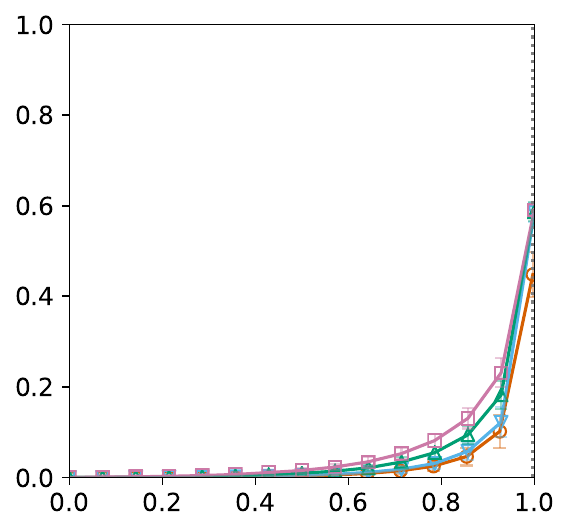}
	&
	\includegraphics[scale=0.4]{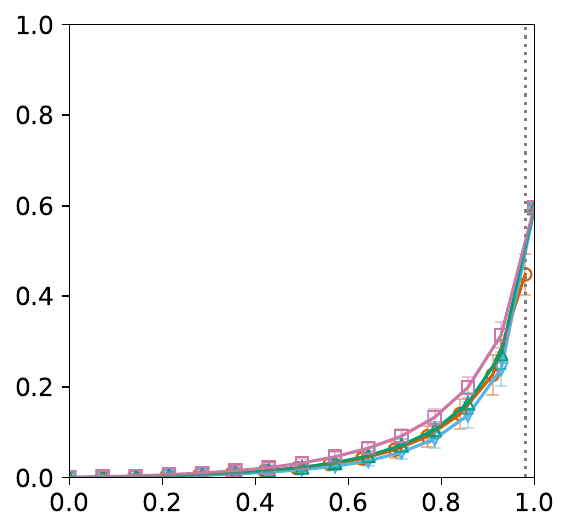}
	&
	\includegraphics[scale=0.4]{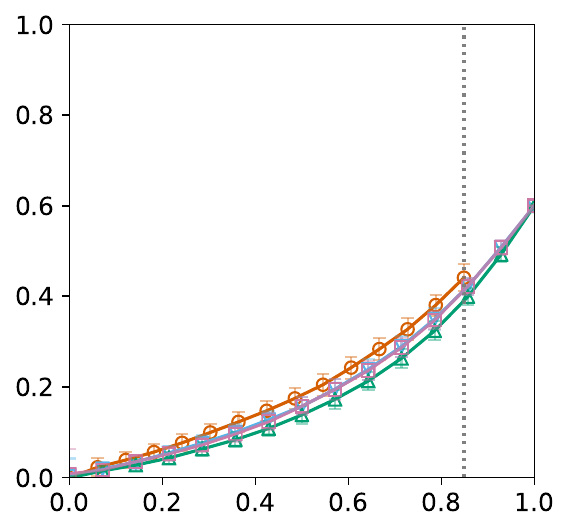} \\
	& \tiny{ATPP} & \tiny{ATPP} & \tiny{ATPP} \\
	&  &  &  \\ 
    & \scriptsize{$\rho=0.9$, $\sigma=0.3$} & \scriptsize{$\rho=0.9$, $\sigma=0.8$} & \scriptsize{$\rho=0.9$, $\sigma=1$} \\
	\rotatebox{90}{\qquad\qquad\quad\tiny{$\mathrm{AFDP}$}} & \includegraphics[scale=0.4]{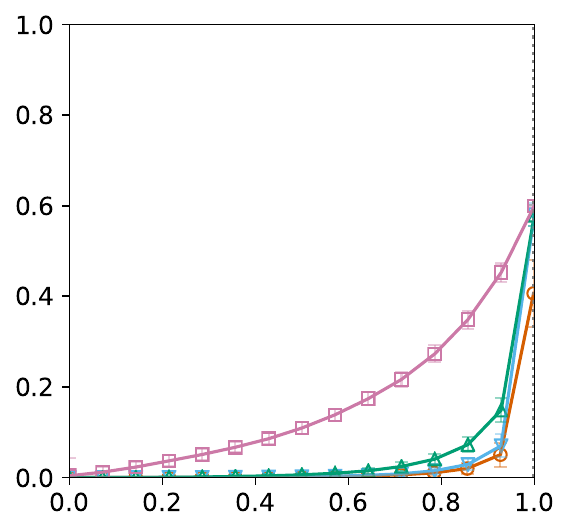}
	&
	\includegraphics[scale=0.4]{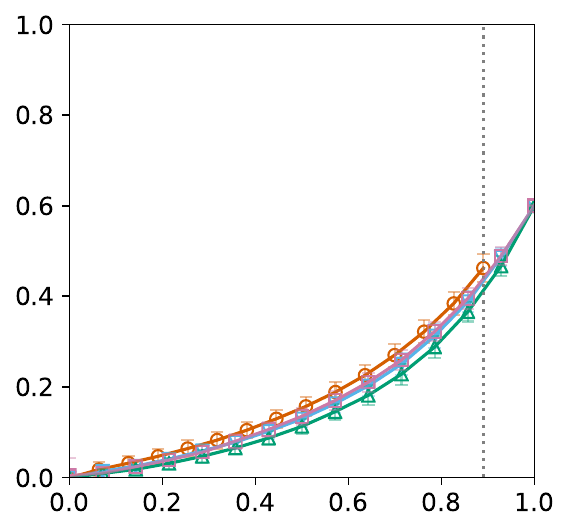}
	&
	\includegraphics[scale=0.4]{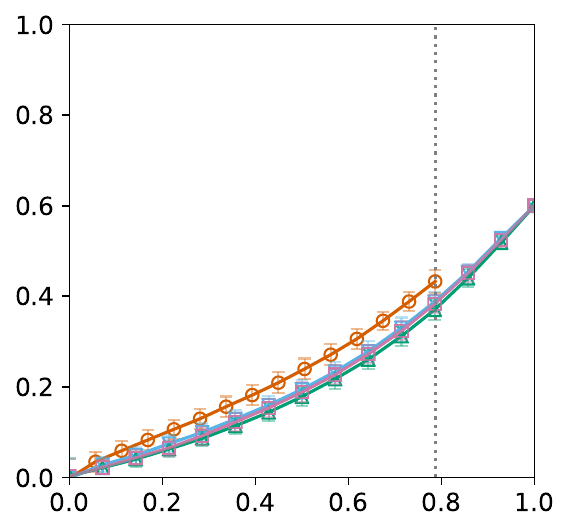} \\
	& \tiny{ATPP} & \tiny{ATPP} & \tiny{ATPP}
\end{tabular}
\caption{Large/small noise scenario under correlated design.}\label{fg:cor-design}
\end{center}
\end{figure}
}

{
\setlength{\tabcolsep}{0pt}
\begin{figure}[htb!]
\begin{center}
\begin{tabular}{rccc}
    & \scriptsize{$\sigma=0.8$} & \scriptsize{$\sigma=1$} & \scriptsize{$\sigma=2$} \\
	\rotatebox{90}{\qquad\qquad\quad\tiny{AFDP}} & \includegraphics[scale=0.4]{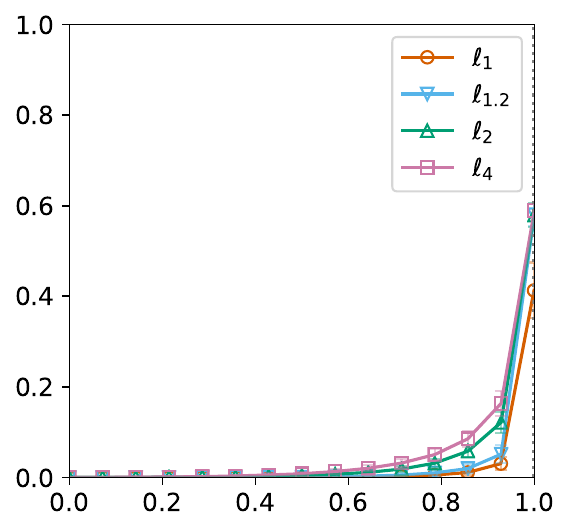}
	&
	\includegraphics[scale=0.4]{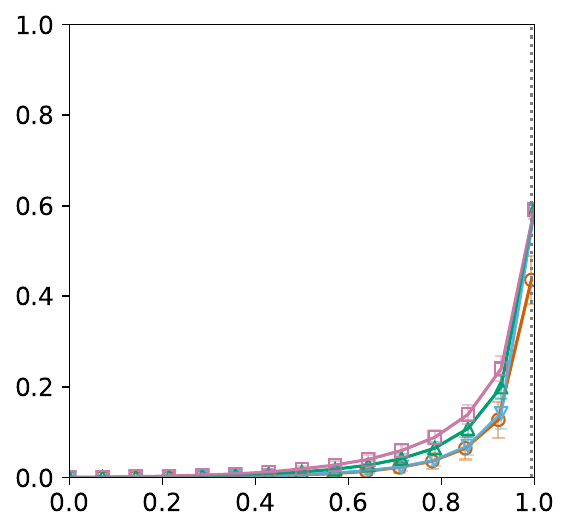}
	&
	\includegraphics[scale=0.4]{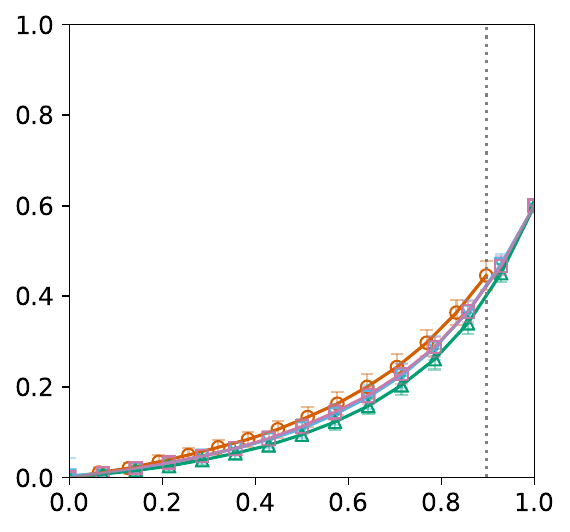} \\
	& \tiny{ATPP} & \tiny{ATPP} & \tiny{ATPP}
\end{tabular}
\caption{Large/small noise scenario under i.i.d. non-Gaussian design. We set 
        $\delta=0.9, \epsilon=0.4, M=8, \sigma\in \{0.8, 1, 2\}$. The degrees of
        freedom of the t-distribution is $\nu=3$.}\label{fg:nonGauss-design-large-noise}
\end{center}
\end{figure}
}

Regarding i.i.d. non-Gaussian design, we choose the t-distribution $t_{\nu}$ with $\nu=3$. Note that among all the t-distributions $\{t_{\nu}, \nu \in \mathbb{N}\}$ with finite variance$, t_3$ has the heaviest tail. The results are shown in Figure
\ref{fg:nonGauss-design-large-noise}. We again observe the comparison predicted by our theory:  LASSO outperforms the other bridge
estimators when the noise level is low ($\sigma=0.8$), and ridge performs best as the noise level increases to $\sigma=2$.

\paragraph{Nearly black object}
For nearly black objects, we consider $\delta=0.8, \sigma=3, b_\epsilon=\frac{4}{\sqrt{\epsilon}}, \tilde{G}=1, \epsilon \in \{0.25, 0.0625, 0.04\}$. We construct the design matrix in the following ways:
\begin{itemize}
\item [(i)] Set a correlated Gaussian design with correlation levels $\rho=0.5, 0.9$.
\item [(ii)] Set an i.i.d. non-Gaussian design with $t_3$.
\end{itemize}

Figures \ref{fg:non-Gauss-design} and \ref{fg:near-black-non-Gauss-design}
contain the results for the correlated design and i.i.d. non-Gaussian design, respectively. We can see that as the model becomes sparser,  LASSO starts to outperform other choices of bridge estimator and eventually becomes optimal. This is consistent with the main conclusion we have proved for the i.i.d. Gaussian designs.
{
\setlength{\tabcolsep}{0pt}
\begin{figure}[htb!]
\begin{center}
\begin{tabular}{rccc}
    & \scriptsize{$\rho=0.5$, $\epsilon=0.25$} & \scriptsize{$\rho=0.5$, $\epsilon=0.0625$} & \scriptsize{$\rho=0.5$, $\epsilon=0.04$} \\
	\rotatebox{90}{\qquad\qquad\quad\tiny{AFDP}} & \includegraphics[scale=0.4]{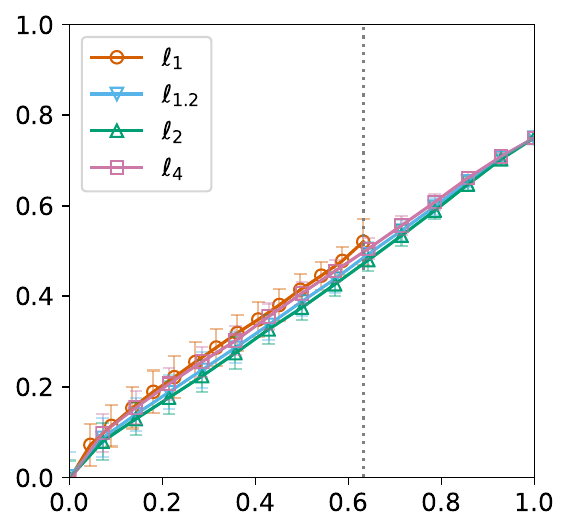}
	&
	\includegraphics[scale=0.4]{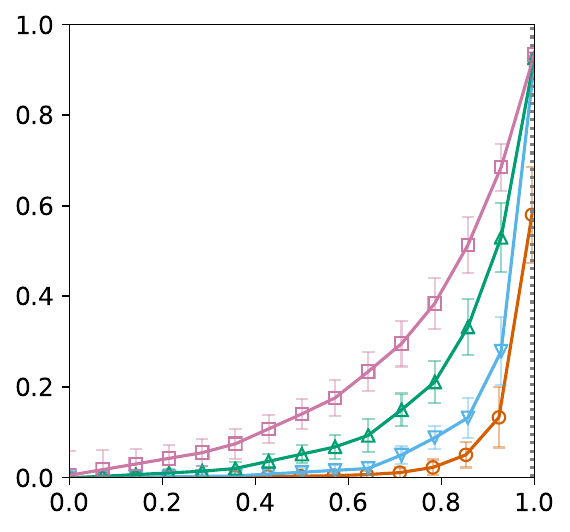}
	&
	\includegraphics[scale=0.4]{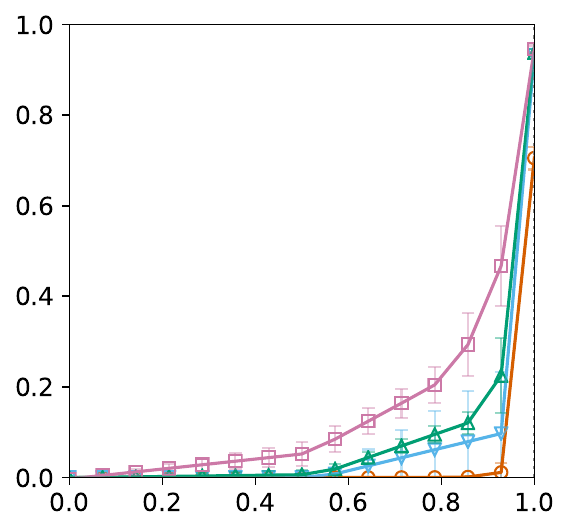} \\
	& \tiny{ATPP} & \tiny{ATPP} & \tiny{ATPP} \\
	&  &  &  \\ 
    & \scriptsize{$\rho=0.9$, $\epsilon=0.25$} & \scriptsize{$\rho=0.9$, $\epsilon=0.0625$} & \scriptsize{$\rho=0.9$, $\epsilon=0.04$} \\
	\rotatebox{90}{\qquad\qquad\quad\tiny{$\mathrm{AFDP}$}} & \includegraphics[scale=0.4]{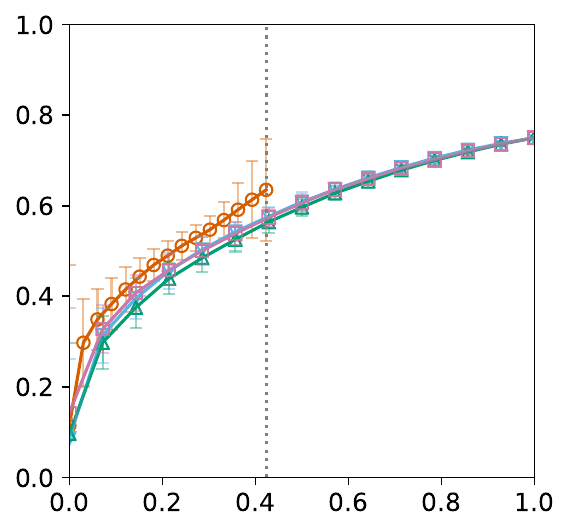}
	&
	\includegraphics[scale=0.4]{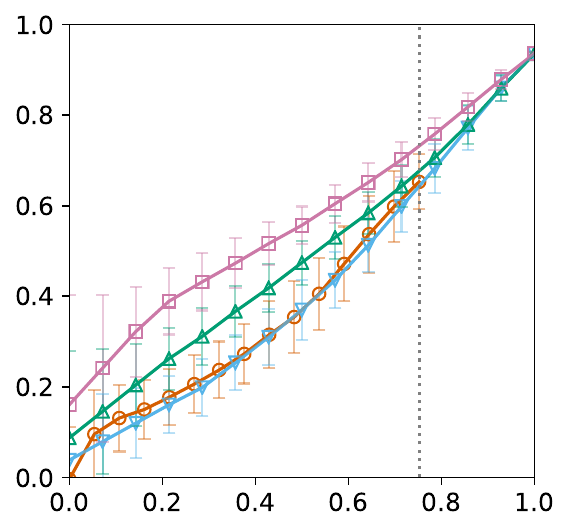}
	&
	\includegraphics[scale=0.4]{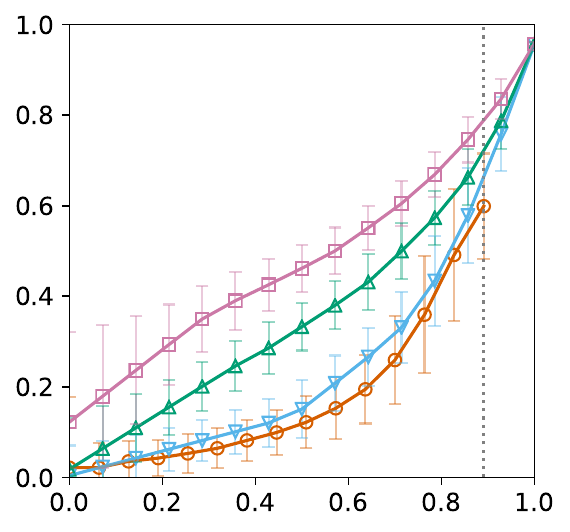} \\
	& \tiny{ATPP} & \tiny{ATPP} & \tiny{ATPP}
\end{tabular}
\caption{Nearly black object with correlated design. We fix
            $\delta=0.8$, $\sigma=3$ and $b_\epsilon=4 / \sqrt{\epsilon}, \epsilon \in
        \{0.25, 0.0625, 0.04\}$. The
        correlation $\rho$ is set to $0.5$ and $0.9$ in the two rows.}\label{fg:non-Gauss-design}
\end{center}
\end{figure}
}

{
\setlength{\tabcolsep}{0pt}
\begin{figure}[htb!]
\begin{center}
\begin{tabular}{rccc}
    & \scriptsize{$\epsilon=0.25$} & \scriptsize{$\epsilon=0.0625$} & \scriptsize{$\epsilon=0.04$} \\
	\rotatebox{90}{\qquad\qquad\quad\tiny{AFDP}} & \includegraphics[scale=0.4]{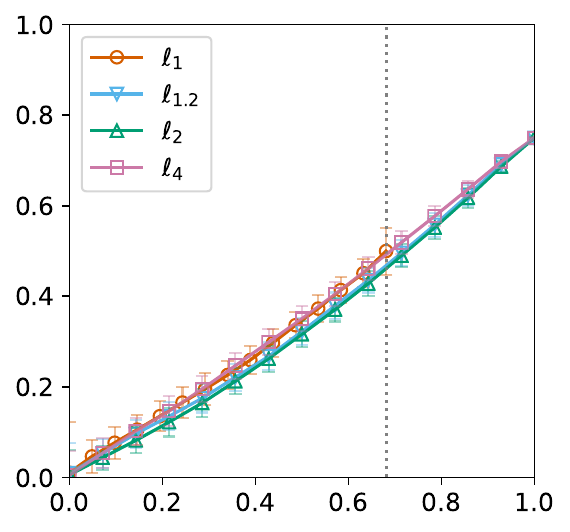}
	&
	\includegraphics[scale=0.4]{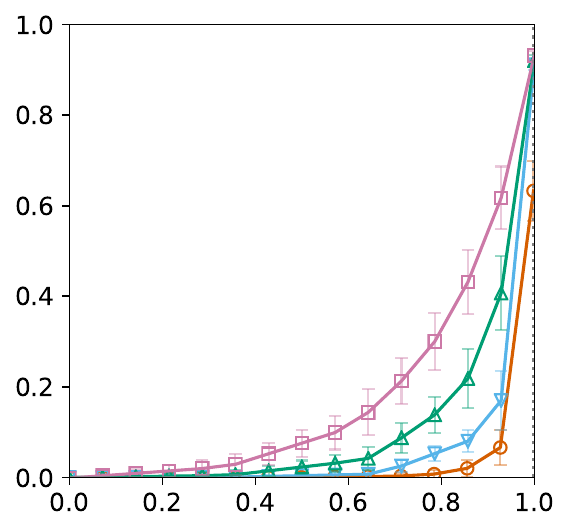}
	&
	\includegraphics[scale=0.4]{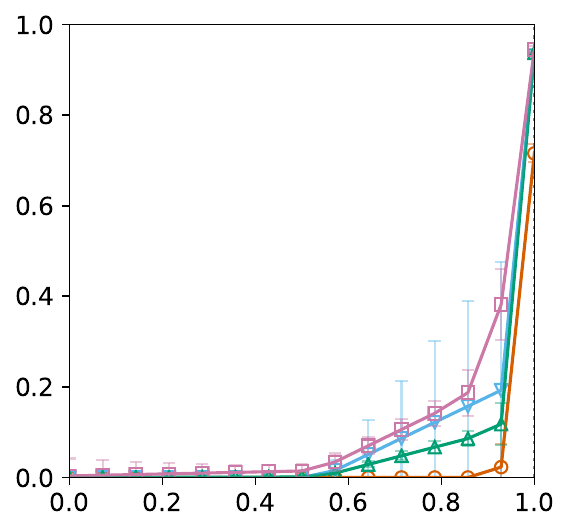} \\
	& \tiny{ATPP} & \tiny{ATPP} & \tiny{ATPP}
\end{tabular}
\caption{Nearly black object with i.i.d. non-Gaussian design. We fix
            $\delta=0.8$, $\sigma=3$ and $b_\epsilon=4 / \sqrt{\epsilon}, \epsilon \in
        \{0.25, 0.0625, 0.04\}$. The degrees
        of freedom for the t-distribution design is $\nu=3$.}\label{fg:near-black-non-Gauss-design}
\end{center}
\end{figure}
}

\paragraph{LASSO vs two-stage LASSO}
We compare LASSO and two-stage LASSO under more general designs. As in Section
\ref{ssec:lasso-vs-two-stage} for i.i.d. Gaussian design, we set $\delta=0.8$, $\epsilon=0.2$, $M=8$ and
$\sigma=1, 3, 5$. For correlated designs, we pick $\rho=0.5, 0.9$. For i.i.d. non-Guassian design, we choose $\nu=3$. As is seen in Figure \ref{fg:lasso-vs-2stage-corr}, the same phenomenon observed in i.i.d. Gaussian design also occurs under general designs:  two-stage LASSO outperforms LASSO by a large margin when the noise is small, and the outperformance becomes marginal in large noise. 

{
\setlength{\tabcolsep}{0pt}
\begin{figure}[htb!]
\begin{center}
\begin{tabular}{rccc}
    & \scriptsize{$\rho=0.5$, $\sigma=1$} & \scriptsize{$\rho=0.5$, $\sigma=3$} & \scriptsize{$\rho=0.5$, $\sigma=5$} \\
	\rotatebox{90}{\qquad\qquad\quad\tiny{AFDP}} & \includegraphics[scale=0.4]{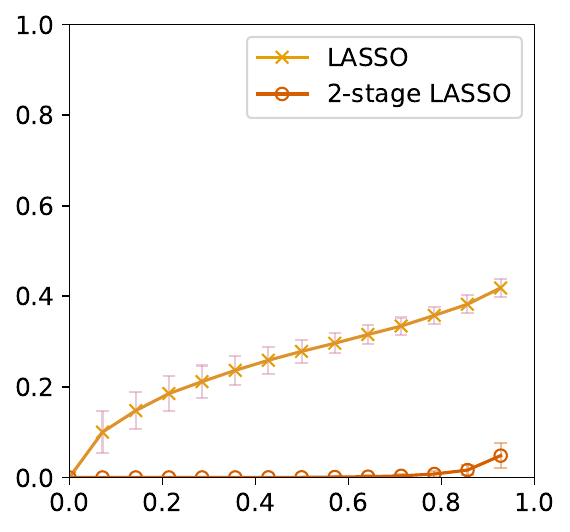}
	&
	\includegraphics[scale=0.4]{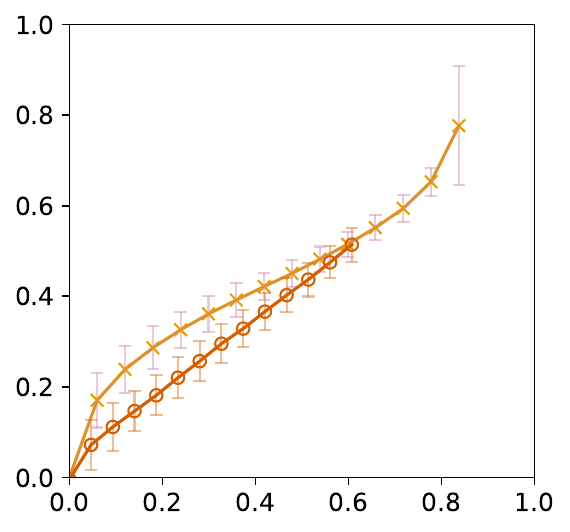}
	&
	\includegraphics[scale=0.4]{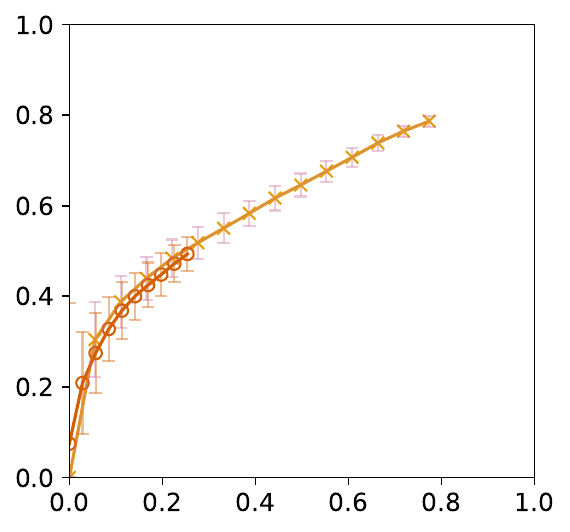} \\
	& \tiny{ATPP} & \tiny{ATPP} & \tiny{ATPP} \\
	&  &  &  \\ 
    & \scriptsize{$\rho=0.9$, $\sigma=1$} & \scriptsize{$\rho=0.9$, $\sigma=3$} & \scriptsize{$\rho=0.9$, $\sigma=5$} \\
	\rotatebox{90}{\qquad\qquad\quad\tiny{$\mathrm{AFDP}$}} & \includegraphics[scale=0.4]{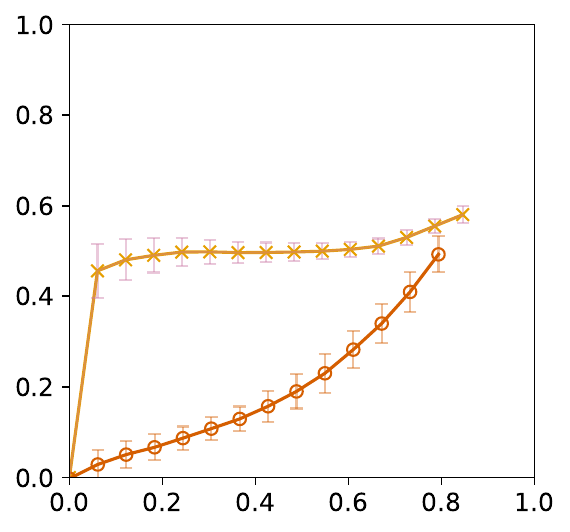}
	&
	\includegraphics[scale=0.4]{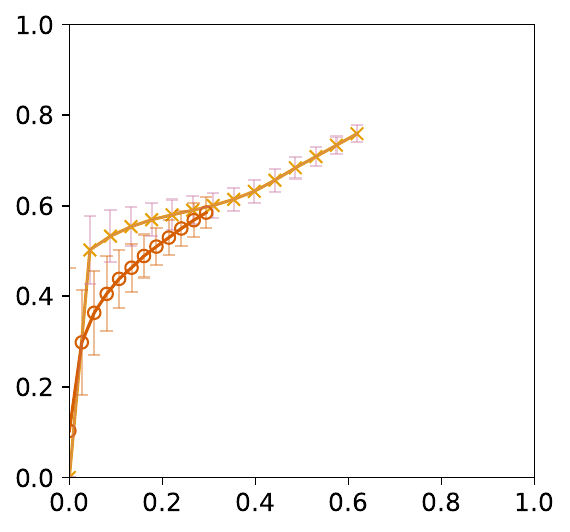}
	&
	\includegraphics[scale=0.4]{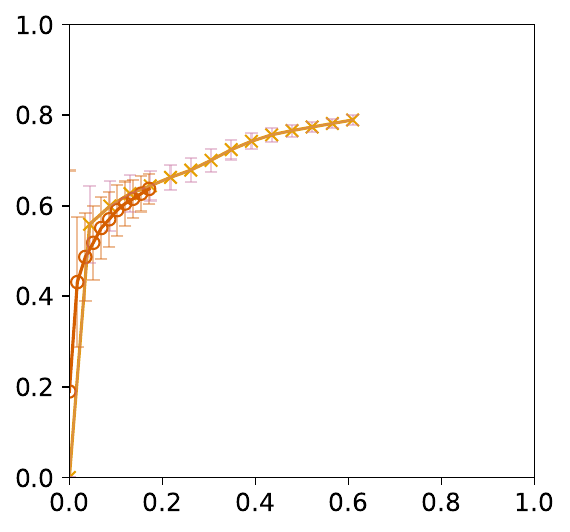} \\
	& \tiny{ATPP} & \tiny{ATPP} & \tiny{ATPP} \\
	&  &  &  \\ 
    & \scriptsize{$\nu=3$, $\sigma=1$} & \scriptsize{$\nu=3$, $\sigma=3$} & \scriptsize{$\nu=3$, $\sigma=5$} \\
	\rotatebox{90}{\qquad\qquad\quad\tiny{$\mathrm{AFDP}$}} & \includegraphics[scale=0.4]{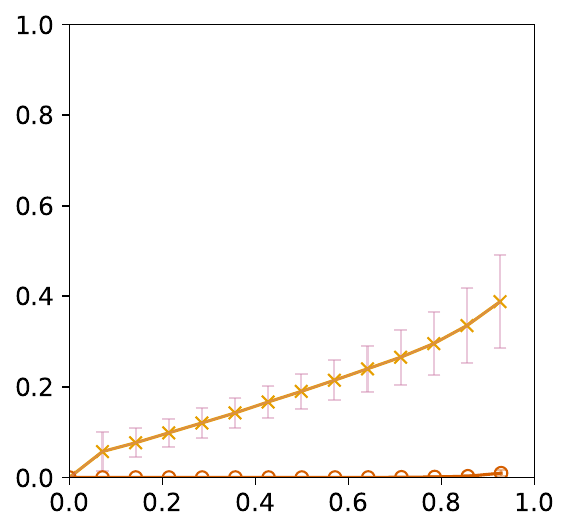}
	&
	\includegraphics[scale=0.4]{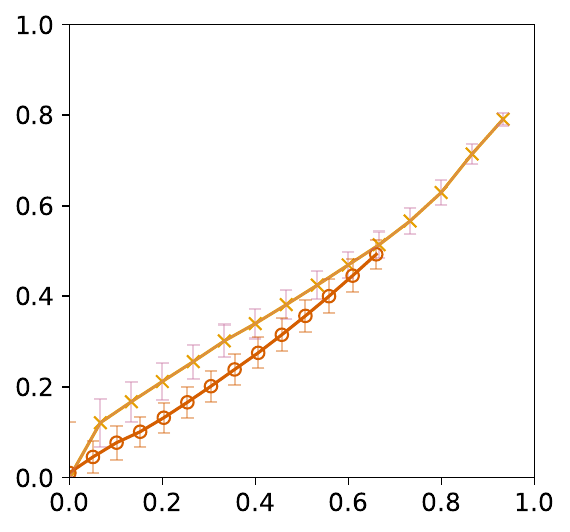}
	&
	\includegraphics[scale=0.4]{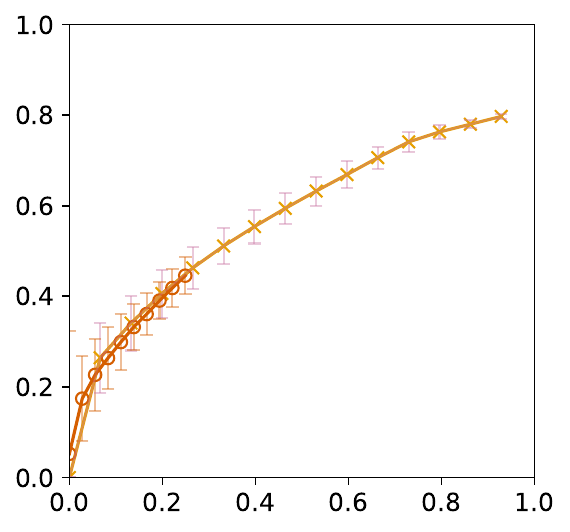} \\
	& \tiny{ATPP} & \tiny{ATPP} & \tiny{ATPP}
\end{tabular}
\caption{LASSO vs. two-stage LASSO under general designs. Here $\delta=0.8, \epsilon=0.2, M=8, \sigma\in \{1, 3, 5\}$. The first two rows are for $\rho=0.5, 0.9$
        in correlated design. The last row is for $\nu=3$ in i.i.d. non-Gaussian design.}\label{fg:lasso-vs-2stage-corr}
\end{center}
\end{figure}
}

\section{Discussion}\label{sec:discussion}

\subsection{Nonconvex bridge estimators} 

In this paper, our discussion has been focused on the bridge estimators with $q\in [1,\infty)$. When $q$ falls in $[0,1)$, the corresponding bridge regression becomes a nonconvex problem. Given that certain nonconvex regularizations have been shown to achieve variable selection consistency under weaker conditions than LASSO \cite{loh2017support}, it is of great interest to analyze the variable selection performance of nonconvex bridge estimators. An early work \cite{huang2008asymptotic} has showed that bridge estimators for $q\in (0,1)$ enjoy an oracle property in the sense of \cite{fan2001variable} under appropriate conditions. However, the asymptotic regime considered in \cite{huang2008asymptotic} is fundamentally different from the linear asymptotic in the current paper. A more relevant work is \cite{zheng2017does} which studied the estimation property of bridge regression when $q$ belongs to $[0,1]$ under a similar asymptotic framework to ours. Nevertheless, the main focus of \cite{zheng2017does} is on the estimators returned by an iterative local algorithm. The analysis of the global minimizer in \cite{zheng2017does} relies on the replica method \cite{rangan2009asymptotic} from statistical physics, which has not been fully rigorous yet. To the best of our knowledge, under the linear asymptotic setting, no existing works have provided a fully rigorous analysis of the global solution from nonconvex regularization in linear regression models. We leave this important and challenging problem as a future research.

\subsection{Tuning parameter selection for a two-stage variable selection scheme}
Two-stage variable selection techniques discussed in this paper have two tuning parameters: the regularization parameter $\lambda$ in the first stage and the threshold $s$ from the second stage. Furthermore, given that TVS using different bridge estimators offer the best performance in different regimes, we may see $q$ as another tuning parameter. How can these parameters be optimally tuned in practice? As proved in Section \ref{sec:contribution}, the TVS with an estimator of smaller AMSE in the first stage provides a better variable selection. Hence, the parameter $\lambda$ can be set by minimizing the estimated risk of the bridge estimator. Similarly, one can estimate the risk for different values of $q$ and choose the one that offers the smallest estimated risk. Section \ref{sec:optimalLambda} has showed how this can be done. 

It remains to determine the parameter $s$. As presented in our results, the threshold $s$ controls the
trade-off between AFDP and ATPP. By increasing $s$ we decrease the number of false discoveries, but at the same time, we decrease the number of correct discoveries. Therefore, the choice of $s$ depends on the accepted level of false discoveries (or similar quantities). For instance, one can control the false discovery rate by combining the two-stage approach with the knockoff framework
\cite{barber2015controlling}. Specifically, if we would like to control FDP at
a rate of $\rho \in (0, 1)$, we can go through the following procedure.

\begin{enumerate}
    \item
        Construct the knockoff features $\tilde{X}\in\mathbb{R}^{n\times p}$ as stated in \cite{barber2015controlling};
    \item
        Run bridge regression on the joint design $[X, \tilde{X}]$ and obtain the corresponding estimator $\begin{bmatrix}\hat{\beta} \\ \tilde{\beta} \end{bmatrix}$. Let
    $W_j=\max(|\hat{\beta}_j|, |\tilde{\beta}_j|) \mathrm{sign}(|\hat{\beta}_j|
    -|\tilde{\beta}_j|), j=1,2,\ldots, p$. Define the threshold
    $s$ as $s = \min\Big\{t>0: \frac{1+\#\{j: W_j \leq -t\}}{\#\{j:W_j \geq t\}
    \vee 1} \leq \rho \Big\}$.
    \item
     Select all the predictors with $\{j: W_j \geq s\}$.
\end{enumerate}
The above procedure only works for $n\geq p$. We may adapt the new knockoff approach in \cite{candes2018panning} when $n<p$.

\section{Conclusion}

We studied two-stage variable selection schemes for linear models under the high-dimensional asymptotic setting, where the number of observations $n$ grows at the same rate as the number of predictors $p$. Our TVS has a bridge estimator in the first stage and a simple threshold function in the second stage. For such schemes, we proved that for a fixed $\atpp$, in order to obtain the smallest $\afdp$ one should pick an estimator that minimizes the asymptotic mean square error in the first stage of TVS. This connection between parameter estimation and variable selection further led us to a thorough investigation of the AMSE under different regimes including rare and weak signals, small/large noise, and large sample. Our analyses revealed several interesting phenomena and provided new insights into variable selection. For instance, the variable selection of LASSO can be improved by debiasing and thresholding; a TVS with ridge in its first stage outperforms TVS with other bridge estimators for large values of noise; the optimality of two-stage LASSO among two-stage bridge estimators holds for very sparse signals until the signal strength is below some threshold. We conducted extensive numerical experiments to support our theoretical findings and validate the scope of our main conclusions for general design matrices.

\bibliographystyle{alpha}
\bibliography{reference}

\newpage


\appendix
\begin{center}
\textbf{\large{Supplementary material}}
\end{center}

\section{Organization}

This supplement contains the proofs of all the main results. Below we mention the organization of this supplement to help the readers.  
\vspace{0.2cm}
\begin{enumerate}
\item Appendix \ref{preliminary:proof35} includes some preliminaries that will
    be extensively used in the latter proofs. We also
        outline the proofs of Theorems
    \ref{THM:STRONGRAREELL_Q} - \ref{THM:SAMPLE:MAIN} in Appendix
\ref{sec:proofThemstrongrareq} - \ref{apxc} since they share certain similarities.
\vspace{0.2cm}
\item Appendix \ref{apxnewb} proves Lemma \ref{LEMMA:FDP:LQ}. 
\vspace{0.2cm}
\item Appendix \ref{apxb} contains the proof of Theorems \ref{THM:MSEMINFDRMIN}, \ref{THM:LASSOFDPCOMPARE} and Corollary \ref{THM:FDPCOMPARE:LQ}.  
\vspace{0.2cm}
\item Appendix \ref{sec:proofThemstrongrareq} proves Theorem
    \ref{THM:STRONGRAREELL_Q}. 
\vspace{0.2cm}
\item Appendix \ref{thmsmallepsilonell_1} proves Theorem
    \ref{EQ:THM:STRONGRAREELL_1}. 
\vspace{0.2cm}
\item Appendix \ref{sec:proofFinitePowerExtremeSPARSE} proves Theorem \ref{THM:SPARSE:MAIN}. 
\vspace{0.2cm}
\item Appendix \ref{thmlargenoise:main} proves Theorems \ref{THM:NOISY:MAIN}. 
\vspace{0.2cm}
\item Appendix \ref{apxc} proves Theorem   \ref{THM:SAMPLE:MAIN}.     
\vspace{0.2cm}
\item Appendix \ref{apxd} includes the proof of Theorems \ref{THEOREM:DEBIASING:VALID}, \ref{THEOREM:DEBIASING:ATTP}, \ref{DEBIASING:Q>1} and Lemma \ref{COMP:SIS}. 
\end{enumerate}

\section{Preliminaries} \label{preliminary:proof35}

\subsection{Some notations}\label{appendix:notations}
We will use the following notations throughout this supplementary file:
\begin{enumerate}[(i)]
    \item
    We will use $\partial_i f$ to denote the partial derivative of $f(x, y, \ldots)$ with respect to its $i^{\rm th}$ argument. Also for the ease of organizing the proof, we may use $\partial_y f$ to be the partial derivative of $f$ with respect to its argument $y$, which is equivalent to $\partial_2 f$.
    \item
    We will use DCT as a short name for Dominated Convergence Theorem.
    \item
    Recall we have $p_B = (1-\epsilon)\delta_0 + \epsilon p_G$. By symmetry, it
    can be easily verified that $B$ and $G$ appearing in the subsequent proofs
    can be equivalently replaced by $|B|$ and $|G|$. Hence without loss of generality, we assume $B$ and $G$ are nonnegative random variables.
    \item
    Let $\Phi$ and $\phi$ denote the cumulative distribution function and
    probability density function of a standard normal random variable
    respectively. Integration by parts gives us the standard result on the
    Gaussian tails expansion: for $k \in \mathbb{N}^+, s>0$
\begin{equation}
\Phi(-s)
=\phi(s)\Bigg[\sum_{i=0}^{k-1}\frac{(-1)^i(2i - 1)!!}{s^{2i+1}} + (-1)^k
(2k-1)!!\int_s^\infty\frac{\phi(t)}{t^{2k}}dt\Bigg],  \label{gaussiantail:exp}
\end{equation}
where $(2i-1)!! \triangleq 1 \times 3 \times 5 \times \ldots \times (2i-1)$.

\item As $a \rightarrow 0$ (or $a \rightarrow \infty$), $g(a) = O(f(a))$, means that there exists a constant $C$ such that for small enough (or large enough) values of $a$, $g(a)\leq Cf(a)$. Furthermore, $g(a) = o(f(a))$ if and only if $\lim_ {a \rightarrow 0} \frac{g(a)}{f(a)} =0$ (or in case of $a \rightarrow \infty$, $\lim_ {a \rightarrow \infty} \frac{g(a)}{f(a)} =0$). 

\item As $a \rightarrow 0$ (or $a \rightarrow \infty$), $g(a) = \Omega(f(a))$,
    if and only if $f(a) = O(g(a))$. Similarly, $g(a) = \omega(f(a))$ if and only if $f(a) = o(g(a))$. Finally, $f(a) = \Theta(g(a))$, if and only if $f(a) = O(g(a))$ and $g(a) = O(f(a))$. 
\end{enumerate}

\subsection{State evolution and properties of the proximal operator}

\begin{definition}[pseudo-Lipschitz function]
\label{definition:pesudolip}
A function $\psi:\mathbb{R}^2\rightarrow\mathbb{R}$ is said to be pseudo-Lipschitz, if $\exists L>0$ s.t., $\forall x,y\in\mathbb{R}^2$, $|\psi(x)-\psi(y)|\leq L(1+\|x\|_2+\|y\|_2)\|x-y\|_2$.
\end{definition}

The following theorem proved by \cite{bayati2011dynamics} and
\cite{weng2018overcoming} will be used in our proof.

\begin{theorem}(\cite{bayati2011dynamics}, \cite{weng2018overcoming})\label{theorem:amp:bridge2}
For a given $q \in [1,\infty)$, let $\hat{\beta}(q,\lambda)$ be the bridge estimator defined in \eqref{eq:bridge regression problem}. Consider a converging sequence $\{\beta(p),X(p),w(p)\}$. Then, for any pseudo-Lipschitz function $\psi:\mathbb{R}^2\rightarrow\mathbb{R}$, almost surely
\begin{equation*}
\lim_{p\rightarrow\infty}\frac{1}{p}\sum_{i=1}^p\psi\left(\hat{\beta}_{i}(q, \lambda),\beta_i(p)\right)
=
\mathbb{E}\psi\left(\eta_q(B+\tau Z;\alpha\tau^{2-q}),B\right),
\end{equation*}
where $B \sim p_B$ and $Z \sim N(0,1)$ are two independent random variables; $\alpha$ and $\tau$ are two positive numbers satisfying \eqref{eq:fixedpointfirstappearance} and \eqref{state_evolution2}.
\end{theorem}

For each tuning parameter $\lambda>0$, \cite{weng2018overcoming} has proved that the solution pair $(\alpha, \tau)$ to the nonlinear equations \eqref{state_evolution1} and \eqref{state_evolution2} is unique. We will denote this unique solution pair for the optimal tuning value $\lambda=\lambda_q^*$ by $(\alpha_*,\tau_*)$. Note that we omit the dependency of these two quantities on $q$, since when they appear in this paper, $q$ is clear from the context. 
 
\begin{lemma}\label{lem:opttunesimpler}
If $(\alpha_*, \tau_*)$ are the solutions of \eqref{state_evolution1} and \eqref{state_evolution2} for $\lambda=\lambda_q^*$, then $\tau_*$ satisfies the following equation:
\begin{align} \label{eq:se11}
    \tau_*^2 =& \sigma^2 + \frac{1}{\delta} \min_{\alpha >0} \mathbb{E} (\eta_q
    (B+ \tau_* Z; \alpha \tau_*^{2-q}) - B)^2, \nonumber \\
    \alpha_* =& \arg\min_{\alpha >0} \mathbb{E} (\eta_q (B+ \tau_* Z; \alpha \tau_*^{2-q}) - B)^2
\end{align}
and 
\begin{equation*}
{\rm AMSE} (q, \lambda^*_q) = \mathbb{E} (\eta_q (B+ \tau_* Z; \alpha_* \tau_*^{2-q}) - B)^2. 
\end{equation*}
 \end{lemma}

This is a simple extension of Lemma 15 in Appendix E of \cite{weng2018overcoming}. Hence we skip the proof. Define
\begin{eqnarray}
R_q(\alpha,\tau) &\triangleq& \mathbb{E}(\eta_q(B/\tau+Z;\alpha)-B/\tau)^2,    \label{eq:def:Rq} \\
\alpha_q(\tau)  &\triangleq& \arg\min_{\alpha\geq 0} R_q(\alpha,\tau).  \label{eq:def:alphaq}
\end{eqnarray}
For the definition \eqref{eq:def:alphaq}, if the minimizer is not unique, we
choose the smallest one.  

Recall the proximal operator:
\begin{equation*}
\eta_q (u; \chi) = \arg\min_z \frac{1}{2} (u - z)^2 + \chi |z|^q.
\end{equation*}
Note that $\eta_q(u;\chi)$ does not have an explicit form except for $q=0,1,2$.
In the following lemma, we summarize some properties of $\eta_q(u; \chi)$. They
will be used to prove our theorems.
\begin{lemma} \label{prox:property}
For any $q \in (1,\infty)$, we have 
\begin{itemize}
\item[(i)] $\eta_q(u;\chi)=-\eta_q(-u;\chi)$.
\item[(ii)] $u= \eta_q(u;\chi)+\chi q|\eta_q(u;\chi)|^{q-1} {\rm sgn} (u)$, where ${\rm sgn}$ denotes the sign of a variable. 
\item[(iii)] $\eta_q(\alpha u;  \alpha^{2-q} \chi)=\alpha \eta_q(u;\chi), \mbox{~~for~}\alpha>0.$
\item[(iv)] $\partial_1 \eta_q(u;\chi)=\frac{1}{1+\chi q(q-1)|\eta_q(u;\chi)|^{q-2}}$. 
\item[(v)] $\partial_2 \eta_q(u;\chi)=\frac{-q|\eta_q(u;\chi)|^{q-1}{\rm sgn}(u)}{1+\chi q(q-1)|\eta_q(u;\chi)|^{q-2}}$.
\item[(vi)] $0\leq \partial_1 \eta_q(u;\chi) \leq 1$.
\item[(vii)] If $1<q<2$, then $\lim_{u \rightarrow 0} \frac{|u|}{|\eta_q(u; \chi)|^{q-1}} = \chi q$.
\item [(viii)] If $1<q<2$, then $\lim_{u \rightarrow \infty} \frac{|u|}{|\eta_q(u; \chi)|} = 1$,
\end{itemize}
\end{lemma}
\begin{proof}
Please refer to Lemmas 7 and 10 in \cite{weng2018overcoming} for the proof of $q\in (1,2]$. The proof for $q>2$ is the same. Hence we do not repeat it. 
\end{proof}

\subsection{Proof sketch for Theorem \ref{THM:STRONGRAREELL_Q} - \ref{THM:SAMPLE:MAIN}}
\label{ssec:proof-sketch}
In Appendix \ref{sec:proofThemstrongrareq} - \ref{apxc} we prove Theorem
\ref{THM:STRONGRAREELL_Q} - \ref{THM:SAMPLE:MAIN}. Since the proofs share some
similarities, we sketch the proof idea in this section.

The results in Theorem \ref{THM:STRONGRAREELL_Q} - \ref{THM:SAMPLE:MAIN}
characterize the asymptotic expansion of the optimal $\mathrm{AMSE}(q,
\lambda_q^*)$ under different
scenarios we considered. In Lemma \ref{lem:opttunesimpler}, we connect
$\mathrm{AMSE}(q, \lambda_q^*)$ with $(\alpha_*, \tau_*)$ through the state
evolution equations. Hence in order to prove our theorems, we will characterize
the behavior of the solution $(\alpha_*, \tau_*)$ of the fixed point equations
\eqref{state_evolution1} and \eqref{state_evolution2} with $\lambda =
\lambda_q^*$ under different scenarios. This can be achieved by making use of \eqref{eq:se11} and its first order condition (notice
$\alpha_*$ minimize the AMSE).

Depending on different scenarios, \eqref{eq:se11} may be presented in
slightly different ways. Specifically for nearly black object, we replace $B$
by $b_\epsilon \tilde{B}$ with $p_{\tilde{B}}=(1-\epsilon)\delta_0 + \epsilon
p_{\tilde{G}}$; For large sample scenario, we replace $\sigma^2$ by $\frac{\sigma^2}{\delta}$.

For $R_q(\alpha, \tau)$, the following decomposition holds:
\begin{equation*}
    R_q (\alpha, \tau)
    =
    (1-\epsilon) \mathbb{E} \eta_q^2 (Z; \alpha)
    + \epsilon \mathbb{E} [\eta_q (G / \tau + Z; \alpha) - G / \tau ]^2. 
\end{equation*}
Since both terms are positve, either
can be used as a lower bound for $R_q(\alpha, \tau)$.

For LASSO, the $\ell_1$ norm enables a simple form for $\eta_1$ and hence for
$\eqref{eq:se11}$ and its first order derivative. We present some useful
formula below.
\begin{align}
    R_1(\alpha, \tau)
    =&
    \underbrace{(1-\epsilon) \tau^2 \mathbb{E} \eta_1^2 (Z;
    \alpha)}_{\triangleq F_1}
    +
    \underbrace{\epsilon \mathbb{E} [ \eta_1(b_\epsilon\tilde{G} + \tau Z; \alpha \tau) -
    b_\epsilon \tilde{G} - \tau Z ]^2}_{\triangleq F_2} -
    \underbrace{\epsilon \tau^2}_{\triangleq F_3} \nonumber \\
    & + \underbrace{2\epsilon \tau^2 \mathbb{E}
    \partial_1\eta_1(b_\epsilon \tilde{G} + \tau Z; \alpha \tau
    )}_{\triangleq F_4} \label{eq:L1-risk-expand1}\\
    =&
    2(1-\epsilon) [(1+\alpha^2)\Phi(-\alpha) - \alpha\phi(\alpha)] + \epsilon\mathbb{E}_G\bigg[ (1+\alpha^2-\frac{G^2}{\tau^2})\Phi(\frac{G}{\tau}-\alpha) + \nonumber \\
& (1+\alpha^2-\frac{G^2}{\tau^2})\Phi(-\frac{G}{\tau}-\alpha)- (\alpha+\frac{G}{\tau})\phi(\alpha-\frac{G}{\tau}) - (\alpha-\frac{G}{\tau})\phi(\alpha+\frac{G}{\tau}) + \frac{G^2}{\tau^2} \bigg] \label{eq:L1-risk-expand2}
\end{align}
Each of the two expansions \eqref{eq:L1-risk-expand1} and
\eqref{eq:L1-risk-expand2} will be handy in certain case. Note that
\begin{equation} \label{eq:F1CalcSoftThresholdNoise}
    F_1 = 2(1-\epsilon) \tau^2 \int_{\alpha}^\infty (z- \alpha)^2 \phi(z) dz
    = 2(1-\epsilon) [(1+\alpha^2)\Phi(-\alpha) - \alpha\phi(\alpha)].
\end{equation}

We also provide the following expansion for the first order derivative
$\partial_{\alpha}R_1(\alpha, \tau)$.
\begin{align}
    \frac{\partial R_1(\alpha, \tau)}{\partial \alpha}
    =&
    2(1-\epsilon)[-\phi(\alpha) + \alpha\Phi(-\alpha)]+\epsilon \mathbb{E}
    \Big[\alpha\Phi(\frac{|G|}{\tau}-\alpha)-\phi(\alpha-\frac{|G|}{\tau})
    \Big] \nonumber \\
    &+
    \epsilon\mathbb{E}
    \Big[\alpha\Phi(-\frac{|G|}{\tau}-\alpha) - \phi(\alpha+\frac{|G|}{\tau})\Big]
    \label{eq:L1-risk-deri-expand}
\end{align}
\section{Proof of Lemma \ref{LEMMA:FDP:LQ}} \label{apxnewb}

Define ${\rm FP}=\sum_{i=1}^p \mathbb{I}(\bar{\beta}_i(q,\lambda,s)\neq 0, \beta_i = 0), ~~{\rm TP}=\sum_{i=1}^p \mathbb{I}(\bar{\beta}_i(q,\lambda, s)\neq 0, \beta_i \neq 0)$.
First note that according to Theorem \ref{theorem:amp:bridge2}, almost surely the empirical distribution of $(\hat{\beta}(q,\lambda),\beta)$ converges weakly to the distribution of $(\eta_q(B+\tau Z;\alpha\tau^{2-q}),B)$. We now choose a sequence $t_m\rightarrow 0$ as $m\rightarrow 0$ such that $G$ does not have any point mass on that sequence. Then by portmanteau lemma we have almost surely
\begin{align*}
&\lim_{p\rightarrow \infty}\frac{1}{p}\sum_{i=1}^{p}\mathbb{I}(\bar{\beta}_i(q,\lambda,s)\neq 0, |\beta_i| \leq t_m)  =  \lim_{p\rightarrow \infty}\frac{1}{p}\sum_{i=1}^{p}\mathbb{I}(|\hat{\beta}_i(q,\lambda)|> s, |\beta_i|\leq t_m) \nonumber \\
    =& \mathbb{P}(|\eta_q(B+\tau Z;\alpha \tau^{2-q})|>s, |B|\leq t_m) \nonumber \\
    =& (1-\epsilon)\mathbb{P}(|\eta_q(\tau Z;\alpha \tau^{2-q})|>s) + \epsilon \mathbb{P}(|\eta_q(G+\tau Z;\alpha \tau^{2-q})|>s, |G|\leq t_m), \nonumber 
\end{align*}
which leads to 
\begin{eqnarray}
&&\lim_{m \rightarrow \infty} \lim_{p\rightarrow \infty}\frac{1}{p}\sum_{i=1}^{p}\mathbb{I}(\bar{\beta}_i(q,\lambda,s)\neq 0, |\beta_i| \leq t_m)  \label{quickproof:one}  =(1-\epsilon)\mathbb{P}(|\eta_q(\tau Z;\alpha \tau^{2-q})|>s). \nonumber
\end{eqnarray}

Moreover, it is clear that 
\begin{eqnarray*}
&&\hspace{-0.8cm}\frac{1}{p} \big | \sum_{i=1}^{p}\mathbb{I}(\bar{\beta}_i(q,\lambda,s)\neq 0, |\beta_i| \leq t_m)-{\rm FP} \big |\leq 
\frac{1}{p} \sum_{i=1}^p \mathbb{I}(|\hat{\beta}_i(q,\lambda)|> s)\cdot \mathbb{I}(0<|\beta_i|\leq t_m) \\
&&\leq \sqrt{\frac{1}{p}\sum_{i=1}^p \mathbb{I}(|\hat{\beta}_i(q,\lambda)|> s)}\cdot \sqrt{\frac{1}{p}\sum_{i=1}^p\mathbb{I}(0<|\beta_i|\leq t_m)} \\
&& \overset{a.s.}{\rightarrow} [\mathbb{P}(|\eta_q(B+\tau Z;\alpha \tau^{2-q})|>s)]^{1/2}\cdot \epsilon^{1/2}[\mathbb{P}(0<|G|\leq t_m)]^{1/2} \mbox{~as~}p \rightarrow \infty.
\end{eqnarray*}

Hence we obtain almost surely
\begin{eqnarray*}
\lim_{m\rightarrow \infty}\lim_{p\rightarrow \infty} \frac{1}{p} \bigg | \sum_{i=1}^{p}\mathbb{I}(\bar{\beta}_i(q,\lambda,s)\neq 0, |\beta_i| \leq t_m)-{\rm FP} \bigg |=0.
\end{eqnarray*}

This combined with \eqref{quickproof:one} implies that as $p\rightarrow \infty$
\begin{equation*}
\frac{{\rm FP}}{p}\overset{a.s.}{\rightarrow} (1-\epsilon)\mathbb{P}(|\eta_q(\tau Z;\alpha \tau^{2-q})|>s).
\end{equation*}

We can now conclude that
\begin{equation*}
{\rm AFDP}(q,\lambda, s)
=\frac{\lim_{p\rightarrow \infty}{\rm FP}/p}{\lim_{p\rightarrow
\infty}\sum_{i=1}^p\mathbb{I}(\hat{\beta}_i(q,\lambda) >s)/p} \nonumber \\
=\frac{(1-\epsilon)\mathbb{P}(|\eta_q(\tau Z;\alpha
\tau^{2-q})|>s)}{\mathbb{P}(|\eta_q(B + \tau Z;\alpha \tau^{2-q})|>s)},\quad a.s.
\end{equation*}

The formula of ${\rm AFDP}(q,\lambda, s)$ in Lemma \ref{LEMMA:FDP:LQ} can then be obtained by Lemma \ref{prox:property} part (iii). Regarding ${\rm ATPP}(q,\lambda, s)$ we have
\begin{align*}
    {\rm ATPP}(q,\lambda, s)
    =&\frac{\lim_{p\rightarrow \infty}\sum_{i=1}^p\mathbb{I}(\hat{\beta}_i(q,\lambda) >s)/p-\lim_{p\rightarrow \infty}{\rm FP}/p}{\lim_{p\rightarrow \infty}\sum_{i=1}^p\mathbb{I}(\beta_i \neq 0)/p} \\
    =& \mathbb{P}(|\eta_q(G+\tau Z;\alpha \tau^{2-q})|>s),
    \quad \text{a.s.}
\end{align*}

\section{Proof of Theorems \ref{THM:MSEMINFDRMIN}, \ref{THM:LASSOFDPCOMPARE} and Corollary \ref{THM:FDPCOMPARE:LQ}}  \label{apxb}

We present the proofs of Theorems \ref{THM:MSEMINFDRMIN}, \ref{THM:LASSOFDPCOMPARE} and Corollary \ref{THM:FDPCOMPARE:LQ} in Sections \ref{proof:THM:MSEMINFDRMIN}, \ref{proof:THM:LASSOFDPCOMPARE} and \ref{proof:THM:FDPCOMPARE:LQ},  respectively.

\subsection{Proof of Theorem \ref{THM:MSEMINFDRMIN}} \label{proof:THM:MSEMINFDRMIN}

\begin{proof}
According to Lemma \ref{LEMMA:FDP:LQ}, we know
\begin{equation*}
{\rm ATPP}(q, \lambda, s) = \mathbb{P}(|\eta_q(G+\tau Z; \alpha\tau^{2-q})|>s)
\end{equation*}
where $(\alpha,\tau)$ is the unique solution to \eqref{state_evolution1} and \eqref{state_evolution2}. From Lemma \ref{prox:property} part (iv), the proximal function $\eta_q(u;\chi)=0$ if and only if $u=0$ for $q>1$. Since $G+\tau Z\neq 0~a.s.$, we have ${\rm ATPP}(q, \lambda, 0)=1$. Moreover, it is clear that ${\rm ATPP}(q, \lambda, +\infty)=0$, and ${\rm ATPP}(q, \lambda, s)$ is a continuous and strictly decreasing function of $s$ over $[0,\infty]$. Hence there exists a unique $s$ for which ${\rm ATPP}(q, \lambda, s) = \zeta\in[0, 1]$.

Now consider all possible pairs $(\lambda, s)$ such that ${\rm ATPP}(q,\lambda, s)=\zeta$. Let $(\alpha_*, \tau_*, s_*)$ be the triplet corresponding to the optimal tuning $\lambda_q^*$ (it  minimizes ${\rm AMSE}(q,\lambda)$), and $(\alpha, \tau, s)$ be the one that corresponds to any other $\lambda$. According to Theorem \ref{theorem:amp:bridge2}, we know ${\rm AMSE}(q,\lambda)=\delta(\tau^2-\sigma^2)$. So $\tau_*<\tau$. By the strict monotonicity and symmetry of $\eta_q$ with respect to its first argument (see Lemma \ref{prox:property} parts (i)(iv)), ${\rm ATPP}(q, \lambda_q^*, s_*) = {\rm ATPP}(q, \lambda, s)$ implies that 
\begin{equation}
\mathbb{P}(|G/\tau_* + Z|>\eta_q^{-1}(s_*/\tau_*; \alpha_*))
=
\mathbb{P}(|G/\tau + Z|>\eta_q^{-1}(s/\tau; \alpha)),  \label{contrad:evd}
\end{equation}
where $\eta_q^{-1}$ is the inverse function of $\eta_q$. Now we claim ${\rm AFDP}(q, \lambda_q^*, s_*) < {\rm AFDP}(q, \lambda, s)$. Otherwise, from the formula of ${\rm AFDP}$ in \eqref{afdp:atpp}, we will have
\[
 \mathbb{P}(\eta_q(|Z|;\alpha_*) > s_*/\tau_*) \geq \mathbb{P}(\eta_q(|Z|;\alpha) > s/\tau),
 \]
 which is equivalent to $ \mathbb{P}(|Z| > \eta_q^{-1}(s_*/\tau_*;\alpha_*)) \geq \mathbb{P}(|Z| > \eta_q^{-1}(s/\tau;\alpha)).$
 This implies $\eta_q^{-1}(s_*/\tau_*;\alpha_*) \leq \eta_q^{-1}(s/\tau;\alpha)$. However, combining this result  with $\tau_* < \tau$ and the fact that $\mathbb{P}(|\mu+Z|>t)$ is an strictly increasing function of $\mu$ over $[0,\infty)$, we must have 
 \begin{equation*}
     \small \mathbb{P}\Big(\Big|\frac{G}{\tau_*} + Z\Big| >
     \eta_q^{-1}\Big(\frac{s_*}{\tau_*}; \alpha_*\Big)\Big)
     \geq  \mathbb{P}\Big(\Big|\frac{G}{\tau_*} +
     Z\Big|>\eta_q^{-1}\Big(\frac{s}{\tau}; \alpha\Big)\Big)
     > \mathbb{P}\Big(\Big|G/\tau + Z\Big|>\eta_q^{-1}\Big(\frac{s}{\tau};
     \alpha\Big)\Big).
\end{equation*}

This is in contradiction with \eqref{contrad:evd}. The conclusion follows.
\end{proof}

\subsection{Proof of Theorem \ref{THM:LASSOFDPCOMPARE}} \label{proof:THM:LASSOFDPCOMPARE}
According to Lemma \ref{CITE:WEIJIE}, 
\[
{\rm ATPP}(1,\lambda)=\mathbb{P}(|G + \tau Z| > \alpha\tau)=\mathbb{E}[\Phi(G/\tau-\alpha) + \Phi(-G/\tau - \alpha)].
\]
It has been shown in \cite{bayati2012lasso} that, $\alpha$ is an increasing and continuous function of $\lambda$, and $\alpha\rightarrow\infty$ as $\lambda\rightarrow\infty$. Hence, $\atpp(1, \lambda)$ is continuous in $\lambda$ and $\lim_{\lambda\rightarrow\infty}\atpp(1, \lambda) = \lim_{\alpha\rightarrow\infty}\mathbb{P}(|G + \tau Z| > \alpha\tau) = 0$. Now let $(\alpha_*, \tau_*)$ be the solution to \eqref{state_evolution1} and \eqref{state_evolution2} when $\lambda=\lambda_1^*$. As we decrease $\lambda$ from $\infty$ to $\lambda_1^*$, $\atpp(1, \lambda)$ continuously changes from $0$ to $\atpp(1, \lambda_1^*)$. Therefore, for any $\atpp$ level $\zeta \in [0,{\rm ATPP}(1,\lambda_1^*)]$, there always exists at least a value of $\lambda \in [\lambda_1^*,\infty]$ such that ${\rm ATPP}(1,\lambda)=\zeta$. Regarding the thresholded $\lasso$ $\bar{\beta}(1,\lambda_1^*,s)$, Lemma \ref{LEMMA:FDP:LQ} shows that
\[
{\rm ATPP}(1,\lambda_1^*,s)=\mathbb{P}(|\eta_1(G+\tau_*Z;\alpha_*\tau_*)|>s).
\] 
Note that when $s=0$ the thresholded $\lasso$ is equal to $\lasso$ and thus ${\rm ATPP}(1,\lambda^*_1,0)={\rm ATPP}(1,\lambda_1^*)$. It is also clear that ${\rm ATPP}(1,\lambda_1^*,s)$ is a continuous and strictly decreasing function of $s$ on $[0,\infty]$. As a result, a unique threshold $s_{\zeta}$ exists s.t. ${\rm ATPP}(1,\lambda_1^*,s_{\zeta})$ reaches a given level $\zeta \in [0,{\rm ATPP}(1,\lambda_1^*)]$. We now compare the AFDP of different estimators that have the same ATPP.  Suppose $\hat{\beta}(1, \lambda)$ and $\bar{\beta}(1, \lambda_1^*, s)$ reach the same level of $\atpp$. We have 
\[
\mathbb{P}(|\eta_1(G + \tau Z; \alpha\tau)| > 0)= \mathbb{P}(|\eta_1(G + \tau_*Z; \alpha_*\tau_*)| > s),
\]
which is equivalent to
\begin{equation}
\mathbb{P}(|G/\tau + Z| > \alpha)=
\mathbb{P}(|G/\tau_* + Z| > \alpha_* + s/\tau_*). \label{proof:l1atppafdpcpr}
\end{equation}
Similar to the argument in the proof of Theorem \ref{THM:MSEMINFDRMIN}, we have $\alpha < \alpha_* + s/\tau_*$, since otherwise the left hand side in \eqref{proof:l1atppafdpcpr} will be smaller than the right hand side. Hence, we obtain
\[
\mathbb{P}(|Z|>\alpha)> \mathbb{P}(|Z|>\alpha_*+s/\tau_*)=\mathbb{P}(|\eta_1(Z;\alpha_*)|>s/\tau_*).
\]
This implies ${\rm AFDP}(1,\lambda)>{\rm AFDP}(1,\lambda_1^*,s)$ based on Lemmas \ref{CITE:WEIJIE} and \ref{LEMMA:FDP:LQ}. By the same argument, we can show that $\bar{\beta}(1, \lambda_1^*, s)$ also has smaller AFDP than $\bar{\beta}(1, \lambda, s)$ if $\lambda\neq\lambda_1^*$.

\subsection{Proof of Corollary \ref{THM:FDPCOMPARE:LQ}}\label{proof:THM:FDPCOMPARE:LQ}

This theorem compares the two-stage estimators $\bar{\beta}(q,\lambda_q^*,s)$ for $q\in [1,\infty)$. Consider $q_1,q_2\geq 1$, and ${\rm AMSE}(q_1,\lambda^*_{q_1})<{\rm AMSE}(q_2,\lambda^*_{q_2})$. Let $(\alpha_{q_i*},\tau_{q_i*})$ be the solution to \eqref{state_evolution1} and \eqref{state_evolution2} when $\lambda=\lambda_{q_i}^*$, for $i=1,2$. Then, according to Theorem \ref{theorem:amp:bridge2}, $\tau_{q_1*}<\tau_{q_2*}$. Suppose ${\rm ATPP}(q_1,\lambda_{q_1}, s_1)={\rm ATPP}(q_2,\lambda_{q_2}, s_2)$, i.e.,
\[
\mathbb{P}(\eta_{q_1}(G+\tau_{q_1*}Z;\alpha_{q_1*}\tau_{q_1*}^{2-q_1})>s_1)=\mathbb{P}(\eta_{q_2}(G+\tau_{q_2*}Z;\alpha_{q_2*}\tau_{q_2*}^{2-q_2})>s_2).
\]
When the ATPP level is 0 or 1, we see $s_1$ and $s_2$ are either both $\infty$ or 0. The corresponding AFDP will be the same. We now consider the level of ATPP belong to $(0,1)$. Using arguments similar to the ones presented in the proof of Theorem \ref{THM:MSEMINFDRMIN}, we can conclude $\eta_{q_1}^{-1}(s_1/\tau_{q_1*};\alpha_{q_1*}) > \eta_{q_2}^{-1}(s_2/\tau_{q_2*};\alpha_{q_2*})$\footnote{Note that $\eta_1^{-1}(u;\chi)$ is not well defined for $u=0$ and we define it as $\eta_1^{-1}(0;\chi)=\chi$.}. This gives us
\begin{eqnarray*}
&&\mathbb{P}(|\eta_{q_1}(Z;\alpha_{q_1*})|>s_1/\tau_{q_1*})=
\mathbb{P}(|Z|>\eta_{q_1}^{-1}(s_1/\tau_{q_1*};\alpha_{q_1*})) \\
&&< \mathbb{P}(|Z|>\eta_{q_2}^{-1}(s_2/\tau_{q_2*};\alpha_{q_2*}))=\mathbb{P}(|\eta_{q_2}(Z;\alpha_{q_2*})|>s_2/\tau_{q_2*}),
\end{eqnarray*}
implying ${\rm AFDP}(q_1,\lambda_{q_1}, s_1)<{\rm AFDP}(q_2,\lambda_{q_2}, s_2)$.

\section{Proof of Theorem \ref{THM:STRONGRAREELL_Q}} \label{sec:proofThemstrongrareq}

\subsection{Roadmap of the proof}\label{ssec:roadmapstrraresig}
As we have mentioned in Section \ref{ssec:proof-sketch}, we will characterize
the behavior of $(\alpha_*, \tau_*)$ defined through equation 
\eqref{eq:se11}. Since we are dealing with the nearly black object model, we replace
$B$ by $b_\epsilon \tilde{B}$ with $p_{\tilde{B}}=(1-\epsilon)\delta_0 +
\epsilon p_{\tilde{B}}$. We first handle $q<2$ in Section \ref{ssec:prooftaustarzero} - \ref{ssec:phasetransell_qorder}. Then in Section \ref{ssec:nbo-qlarger2} we deal with $q \geq 2$. We will prove in Section \ref{ssec:prooftaustarzero} that as $\epsilon \rightarrow 0$, $\tau_* \rightarrow \sigma$. Furthermore, it is straightforward to see that $\alpha_* \rightarrow \infty$ as $\epsilon \rightarrow 0$. Otherwise, if $\alpha_* \rightarrow C$, then 
\begin{align*}
\varliminf_{\epsilon \rightarrow 0}\mathbb{E} (\eta_q (\aeps \tilde{B}+ \tau_* Z; \alpha_* \tau_*^{2-q}) -\aeps \tilde{B})^2
\geq \varliminf_{\epsilon \rightarrow 0} (1-\epsilon) \mathbb{E} \eta_q^2 (\tau_* Z; \alpha_* \tau_*^{2-q})
=  \mathbb{E} \eta_q^2 (\sigma Z; C \sigma^{2-q}) > 0.
\end{align*}

However, in Section \ref{ssec:prooftaustarzero} we will prove that
$\lim_{\epsilon \rightarrow 0}\mathbb{E} (\eta_q (\aeps \tilde{B}+ \tau_* Z;
\alpha_* \tau_*^{2-q}) -\aeps \tilde{B})^2 \rightarrow 0$. In order to show the
optimal $\amse$ vanishes as $\epsilon \rightarrow 0$, we need to characterize the rate at which
$\alpha_* \rightarrow \infty$. This requires an accurate analysis of
$\arg\min_{\alpha} \tilde{R}_q (\alpha, \epsilon, \tau_*)$, where
\begin{equation}\label{eq:definitiontilde{R}}
\tilde{R}_q (\alpha, \epsilon, \tau) \triangleq  \mathbb{E} (\eta_q (\aeps \tilde{B} + \tau Z; \alpha \tau^{2-q}) - \aeps \tilde{B}).
\end{equation}

We note the slight differences between $\tilde{R}_q$ and $R_q$ in
\eqref{eq:def:Rq} and $\amse(q, \lambda_q^*) = \tilde{R}(\alpha_*, \epsilon, \tau_*)$. The behavior of $\alpha_*$ depends on the relation between $\aeps$ and $\epsilon$ in the following way:
\begin{itemize}
\item
    Case I - If $\aeps = o (\epsilon^{\frac{1-q}{2}})$, then $\lim_{\epsilon
    \rightarrow 0} \epsilon^{-1} \aeps^{-2} \tilde{R}_q (\alpha_*, \epsilon,
    \tau_*) = \mathbb{E} |\tilde{G}|^2.$ This claim is proved in Section
    \ref{sec:notsostrongsingal}. Note that $\lim_{\alpha \rightarrow \infty}
    \tilde{R}_q (\alpha, \epsilon, \tau)= \epsilon  \aeps^{2}  \mathbb{E}
    |\tilde{G}|^2$ too.

\item Case II - If $\aeps = \omega (\epsilon^{\frac{1-q}{2}})$, then $\epsilon^{\frac{q-1}{2q}} \aeps^{\frac{(q-1)^2}{q}} \alpha_* = \Theta(1)$. This claim is proved in Section \ref{sec:sostrongsingal}. Furthermore, we will show that
\begin{equation*}
    \epsilon^{-\frac{1}{q}} \aeps^{-\frac{2(q-1)}{q}}
    \tilde{R}_q (\alpha_*,\epsilon, \tau_*)
    \rightarrow
    q (q-1)^{\frac{1}{q} - 1} \sigma^\frac{2}{q}
    \big[\mathbb{E} |Z|^{\frac{2}{q-1}} \big]^{\frac{q-1}{q}}
    \big[ \mathbb{E} |\tilde{G}|^{2q - 2} \big]^{\frac{1}{q}}
\end{equation*}

\item Case III - If $\aeps = \Theta (\epsilon^{\frac{1-q}{2}})$, then still the optimal choice of $\alpha$ satisfies  $\epsilon^{\frac{q-1}{2q}} \aeps^{\frac{(q-1)^2}{q}} \alpha_* = \Theta(1)$. This will be proved in Section \ref{ssec:phasetransell_qorder}. 
After obtaining this result, we will show
\[
\lim_{\epsilon \rightarrow 0} \epsilon^{-\frac{1}{q}} \aeps^{-\frac{2(q-1)}{q}} \tilde{R}_q (\alpha_*,\epsilon, \tau_*) = \min_C h(C),
\]
where $h(C)\triangleq (Cq)^{-\frac{2}{q-1}}  \sigma^2  \mathbb{E}
|Z|^{\frac{2}{q-1}} + \mathbb{E} \big(\eta_q (c_r G; C
\sigma^{2-q}) -  c_r G \big)^2$, and $c_r
\triangleq \lim_{\epsilon \rightarrow 0} \aeps \epsilon^{\frac{q-1}{2}}$.

\end{itemize}
 
\subsection{Proof of $\tau_* \rightarrow \sigma$ as $\epsilon \rightarrow 0$}\label{ssec:prooftaustarzero}
We first prove a simple lemma which helps with bounding the optimal $\tau_*^2$.

\begin{lemma}\label{lem:upperbound}
For any value of $\epsilon>0$ we have  
\[
    \sigma^2 \leq \tau_*^2 \leq \sigma^2 + \frac{\epsilon b_\epsilon^2}{\delta} \mathbb{E} \tilde{G}^2.
\]
\end{lemma}
\begin{proof}
$\tau_*> \sigma$ is clear from $\tau_*^2 = \sigma^2 + \frac{1}{\delta} \min_{\alpha >0} \tilde{R}_q (\alpha,\epsilon, \tau_*)$. Furthermore,
\begin{equation*}
    \tau_*^2 - \sigma^2 = \frac{1}{\delta} \min_{\alpha >0} \tilde{R}_q (\alpha,\epsilon, \tau_*)
    \leq \frac{1}{\delta} \lim_{\alpha \rightarrow \infty} \tilde{R}_q (\alpha,\epsilon, \tau_*)
    = \frac{\epsilon \aeps^2}{\delta} \mathbb{E} \tilde{G}^2.
\end{equation*}
\end{proof}

If $\sqrt{\epsilon}\aeps \rightarrow 0$ as $\epsilon \rightarrow 0$, by
Lemma \ref{lem:upperbound}, we have $\tau_* \rightarrow \sigma$. So next we
focus on the case when $\sqrt{\epsilon}\aeps \rightarrow c$, where $c
\in (0, \infty)$. In order to prove $\tau_* \rightarrow \sigma$ under this
case, we prove $\tilde{R}_q (\alpha, \epsilon, \tau_*) \rightarrow 0$ for a
specific choice of $\alpha$.
 
\begin{lemma}\label{lem:correctordergoeszero} 
If $\sqrt{\epsilon} \aeps \rightarrow c$, and $\tilde{\sigma}^2 \triangleq
\sigma^2 + \frac{\epsilon \aeps^2}{\delta} \mathbb{E} \tilde{G}^2$, then as
$\epsilon \rightarrow 0$
\begin{equation*} 
    \sup_{\sigma \leq \tau \leq \tilde{\sigma}} \tilde{R}_q \Big(
    \epsilon^{\frac{(2-q)(q-1)}{-2q}}, \epsilon, \tau \Big)  \rightarrow 0.
\end{equation*}
\end{lemma}
\begin{proof}
Define $\alpha_0 \triangleq \epsilon^{\frac{(2-q)(q-1)}{-2q}}$. We have
\begin{equation}\label{eq:Adefs}
\tilde{R}_q (\alpha_0, \epsilon, \tau)
=
\underbrace{(1- \epsilon) \mathbb{E} \eta_q^2 (\tau Z ; \alpha_0
\tau^{2-q})}_{\triangleq A_1(\epsilon)} + \underbrace{\epsilon
\mathbb{E} (\eta_q (\aeps \tilde{G}+ \tau Z; \alpha_0 \tau^{2-q})-\aeps
\tilde{G})^2}_{\triangleq A_2(\epsilon)}
\end{equation}

We first prove that $\varlimsup_{\epsilon \rightarrow 0} \sup_{\sigma  \leq \tau \leq \tilde{\sigma}}A_1(\epsilon)=0$. Note that
\begin{align}\label{eq:A1calcfinal}
    \alpha_0^{\frac{2}{q-1}} A_1(\epsilon)
    \overset{(a)}{=}
    (1-\epsilon) \tau^2 q^{-\frac{2}{q-1}}
    \mathbb{E} (|Z| - |\eta_q(Z; \alpha_0)| )^{\frac{2}{q-1}}
\end{align}
Equality (a) is due to Lemma \ref{prox:property} (ii) (iii). Hence,
\[
\sup_{\sigma \leq \tau \leq \tilde{\sigma}}  \alpha_0^{\frac{2}{q-1}} A_1(\epsilon)
\leq
(1-\epsilon) \tilde{\sigma}^2 q^{-\frac{2}{q-1}} \mathbb{E} |Z|^{\frac{2}{q-1}}.
\]
Since $\alpha_0 \rightarrow \infty$ as $\epsilon \rightarrow 0$, this
immediately implies that $\lim_{\epsilon \rightarrow 0} \sup_{\sigma \leq \tau
\leq \tilde{\sigma}}A_1(\epsilon)=0$. Now we discuss $A_2 (\epsilon)$. We have
\begin{align}\label{eq:A2terms1}
    A_2(\epsilon)
    =&  \epsilon \mathbb{E} (\eta_q (b_\epsilon \tilde{G}+ \tau Z; \alpha_0 \tau^{2-q})-\aeps \tilde{G} -\tau Z)^2 + \epsilon \tau^2 \nonumber \\
     &+2 \epsilon \tau \mathbb{E} (Z (\eta_q (\aeps \tilde{G}+ \tau Z; \alpha_0 \tau^{2-q})-\aeps \tilde{G} -\tau Z))
    \triangleq \epsilon B_1(\epsilon) + \epsilon \tau^2 + 2 \epsilon B_2(\epsilon).
\end{align}
We study $B_1(\epsilon)$ and $B_2(\epsilon)$ separately. 
\begin{equation*}
    B_1(\epsilon)
    \overset{(a)}{=}
    q^2 \alpha_0^2 \tau^{4 - 2q} \mathbb{E} |\eta_q (b_\epsilon\tilde{G} + \tau Z; \alpha_0 \tau^{2-q})|^{2q-2},
\end{equation*}
where Equality (a) is due to Lemma \ref{prox:property}(ii). We note that the
choice of $\alpha_0$ implies $\epsilon \alpha_0^2 b_\epsilon^{2q-2} \rightarrow 0$.  Hence, as
$\epsilon \rightarrow 0$
\begin{equation*}
    \epsilon B_1(\epsilon)
    \leq
    \epsilon \alpha_0^2 q^2 \tau^{4 - 2q} \mathbb{E} |b_\epsilon\tilde{G} + \tau Z|^{2q-2}
    \rightarrow
    0
\end{equation*}

It is straightforward to see that $\lim_{\epsilon \rightarrow 0} \sup_{\tau \leq \tilde{\sigma}} \epsilon B_1(\epsilon)=0$.  
Now let us discuss $B_2(\epsilon)$.  By using Stein's lemma we have
\begin{align*}
B_2(\epsilon)
=
\tau^2 \mathbb{E} (\partial_1\eta_q (\aeps \tilde{G}+ \tau Z; \alpha_0 \tau^{2-q}) - 1)
\overset{(a)}{=}
\tau^2 \mathbb{E} \bigg[ \frac{- \alpha_0 \tau^{2-q} q (q-1) |\eta_q (b_\epsilon
        \tilde{G} + \tau Z;
\alpha_0 \tau)|^{q-2} } { 1 + \alpha_0 \tau^{2-q} q (q-1) |\eta_q ( b_\epsilon \tilde{G} +
\tau Z; \alpha_0 \tau^{2-q})|^{q-2}} \bigg]. 
\end{align*}
Equality (a) is due to Lemma \ref{prox:property}(v). Hence, $|B_2(\epsilon)| < \tau^2$ and
\begin{equation}\label{eq:B2calc}
 \sup_{\tau \leq \tilde{\sigma}} |\epsilon B_2(\epsilon)| \rightarrow 0.
\end{equation}
Combining \eqref{eq:Adefs}, \eqref{eq:A1calcfinal}, \eqref{eq:A2terms1}, and \eqref{eq:B2calc} completes the proof.
\end{proof}

Now Lemma \ref{lem:upperbound} implies that $\tau_* \in [\sigma,
\tilde{\sigma}]$. By combining this observation with Lemma
\ref{lem:correctordergoeszero}, it is straightforward to conclude that
\begin{equation*}
    \delta(\tau_*^2 - \sigma^2)
    =
    \min_{\alpha >0} \tilde{R}_q(\alpha, \epsilon, \tau_*) \nonumber \\
    \leq \tilde{R}_q \Big( \epsilon^{\frac{(2-q)(q-1)}{-2q}},
    \epsilon, \tau_* \Big)
    \leq \sup_{\sigma <\tau < \tilde{\sigma}}
    \tilde{R}_q \Big( \epsilon^{\frac{(2-q)(q-1)}{-2q}}, \epsilon, \tau \Big)
    \rightarrow 0.
\end{equation*}
This finishes our proof of $\tau_* \rightarrow \sigma$.
 
\subsection{Case I - $\aeps = o (\epsilon^{\frac{1-q}{2}})$}\label{sec:notsostrongsingal}
Since Case I is the simplest case, we start with this one. As
discussed in Section \ref{ssec:roadmapstrraresig}, $\alpha_* \rightarrow
\infty$ as $\epsilon \rightarrow 0$. In Lemma \ref{lem:risklbgen1} we use this fact to derive a lower
bound for $\tilde{R}_q (\alpha, \epsilon, \tau)$, then use it to obtain a finer
information about $\alpha_*$. 

\begin{lemma}\label{lem:risklbgen1}
If $\alpha \rightarrow \infty$ and $\tau \rightarrow \sigma>0$ as $\epsilon \rightarrow 0$, then 
\begin{equation*}
    \varliminf_{\epsilon \rightarrow 0} \alpha^{\frac{2}{q-1}}\tilde{R}_q (\alpha, \epsilon, \tau)
    \geq
    \sigma^2 q^{-\frac{2}{q-1}} \mathbb{E} |Z|^{\frac{2}{q-1}}
\end{equation*}
\end{lemma}
\begin{proof}
First note that
\begin{equation*}
    \alpha^{\frac{2}{q-1}}\tilde{R}_q (\alpha, \epsilon, \tau)
    \overset{(a)}{\geq}
    (1-\epsilon) \alpha^{\frac{2}{q-1}} \tau^2 \mathbb{E} \eta_q^2 (Z; \alpha )
    \overset{(b)}{=}
    (1-\epsilon) \tau^2 q^{-\frac{2}{q-1}} \mathbb{E} \big| |Z| -
    |\eta_q(Z; \alpha)| \big|^{\frac{2}{q-1}}.
\end{equation*}
where inequality (a) is due to Lemma \ref{prox:property}(iii) and inequality (b)
is due to Lemma \ref{prox:property}(ii). We note that an application of DCT proves
that the last term of expectation converges to $\mathbb{E}|Z|^{\frac{2}{q-1}}$
as $\epsilon \rightarrow 0$. We should mention that it is straightforward to
prove that for every $u$, $\eta_q(u; \alpha) \rightarrow 0$ as $\alpha
\rightarrow \infty$.
\end{proof}

The rest of the proof goes as follows: we first use Lemma
\ref{lem:risklbgen1} to prove $\aeps \alpha_*^{-\frac{1}{2-q}} \rightarrow 0$.
This will further help us to characterize the accurate behavior of $\tilde{R}_q
(\alpha_*, \epsilon, \tau_*)$.

\begin{lemma}\label{lem:chiinfty}
    We have $\lim_{\alpha \rightarrow \infty}
    \tilde{R}_q (\alpha, \epsilon, \tau_*) = \epsilon \aeps^2 \mathbb{E}
    |\tilde{G}|^2$.
\end{lemma}
The proof of this lemma is straightforward and is hence skipped.

\begin{lemma}\label{thm:behavioralpha_infty}
If $\aeps =o (\epsilon^{\frac{1-q}{2}})$, then $\aeps \alpha_*^{-\frac{1}{2-q}} \rightarrow 0$ as $\epsilon \rightarrow 0$. 
\end{lemma}
\begin{proof}
We prove by contradiction. Assume the assertion of the lemma is incorrect,
i.e. $\frac{\alpha_*}{\aeps^{2-q}} = O(1)$. Then,
\begin{equation*}
    \epsilon^{-1} b_\epsilon^{-2} \tilde{R}_q(\alpha_*, \epsilon, \tau_*)
    =
    \big(\alpha_*^{-1} \aeps^{2-q}\big)^{\frac{2}{q-1}}
    \big(\aeps^{\frac{2}{(q-1)}}\epsilon \big)^{-1}
    \alpha_*^{\frac{2}{q-1}} \tilde{R}_q (\alpha_*, \epsilon, \tau_*).
\end{equation*}

According to Lemma \ref{lem:risklbgen1}, since $\alpha_* \rightarrow \infty$
and $\tau_* \rightarrow \sigma$, we have $\alpha_*^{\frac{2}{q-1}}\tilde{R}_q
(\alpha_*, \epsilon, \tau_*) =\Omega(1)$. Furthermore our assumption indicates
that $\alpha_*^{-1}b_\epsilon^{2-q} = \Omega(1)$. Finally, due to the condition
of the lemma, we have $\aeps^{\frac{2}{(q-1)}} \epsilon \rightarrow 0$. Hence,
$ \epsilon^{-1} b_\epsilon^{-2} \tilde{R}_q(\alpha_*, \epsilon, \tau_*) \rightarrow \infty$.
Based on Lemma \ref{lem:chiinfty}, $\lim_{\epsilon \rightarrow 0} \tilde{R}_q
(\alpha, \epsilon, \tau_*)$ is proportional to  $\epsilon \aeps^2$.
This forms a contradiction with the optimality of $\alpha_*$ and completes the
proof.
\end{proof}

In the next theorem we use Lemma \ref{thm:behavioralpha_infty} to characterize $R_q(\alpha_*, \epsilon, \tau_*)$. 

\begin{theorem}\label{thm:finalchininfty}
    If $\aeps = o (\epsilon^{\frac{1-q}{2}})$, then $\varliminf_{\epsilon
    \rightarrow 0} \frac{\tilde{R}_q (\alpha_*, \epsilon, \tau_*)}{\epsilon
    \aeps^2\mathbb{E}|\tilde{G}|^2} \geq 1$.
\end{theorem}
\begin{proof}
It is not hard to see that $\frac{\tilde{R} (\alpha_*, \epsilon, \tau_*)}{\epsilon \aeps^2}
    \geq
    \mathbb{E} \Big[\eta_q \Big( \tilde{G} + \frac{\tau_*}{\aeps} Z;
    \tau_*^{2-q}\frac{\alpha_*}{\aeps^{2-q}} \Big) -  \tilde{G} \Big]^2$.
Since $\aeps \rightarrow \infty$ and according to Lemma
\ref{thm:behavioralpha_infty}, $\frac{\alpha_*}{\aeps^{2-q}} \rightarrow
\infty$, it is straightforward to apply DCT and obtain that $\mathbb{E}
\Big[\eta_q \Big( \tilde{G}+ \frac{\tau_*}{\aeps} Z; \tau_*^{2-q}
\frac{\alpha_*}{\aeps^{2-q}} \Big) -  \tilde{G} \Big]^2 \rightarrow
\mathbb{E}|\tilde{G}|^2$. The conclusion then follows.
\end{proof}

A direct corollary of Lemma \ref{lem:chiinfty} and Theorem
\ref{thm:finalchininfty} is if $\aeps =o (\epsilon^{\frac{1-q}{2}})$, then
$\tilde{R}_q(\alpha_*, \epsilon, \tau_*) \sim \epsilon \aeps^2 \mathbb{E}
|\tilde{G}|^2$. This completes the first piece of Theorem
\ref{THM:STRONGRAREELL_Q}.

\subsection {Case II - $\aeps =\omega (\epsilon^{\frac{1-q}{2}})$} \label{sec:sostrongsingal}
We first characterize the risk for a specific choice of $\alpha$. This
offers an upper bound for $\tilde{R}_q(\alpha_*, \epsilon, \tau_*)$, and will
later help us obtain the exact behavior of $\alpha_*$.

\begin{lemma}\label{lem:optimalriskorder}
Suppose that $\aeps =\omega (\epsilon^{\frac{1-q}{2}})$. If $\alpha = C
\epsilon^{\frac{1-q}{2q}} \aeps^{-\frac{(q-1)^2}{q}}$, then,
\begin{equation*}
    \lim_{\epsilon \rightarrow 0}\epsilon^{-\frac{1}{q}} \aeps^{-\frac{2(q-1)}{q}}  \tilde{R}(\alpha,\epsilon, \tau_*)
    =
    C^{\frac{-2}{q-1}} q^{-\frac{2}{q-1}} \sigma^2 \mathbb{E}
    |Z|^{\frac{2}{q-1}} + C^2 q^2 \sigma^{4-2q)} \mathbb{E} |\tilde{G}|^{2(q-1)}.
\end{equation*}
\end{lemma}
\begin{proof}
    We again start our argument with the same decomposition as in
    \eqref{eq:Adefs}.
Note that as $\epsilon \rightarrow 0$, $\alpha \rightarrow \infty$, and as
discussed in Section \ref{ssec:prooftaustarzero}, $\tau_* \rightarrow \sigma$.
We further have
\begin{align}
    &\epsilon^{-\frac{1}{q}} \aeps^{\frac{-2(q-1)}{q}} A_1(\epsilon)
=
(1-\epsilon) \epsilon^{-\frac{1}{q}} \aeps^{\frac{-2(q-1)}{q}} \tau_*^2 \; \mathbb{E}
\eta_q^2 ( Z; \alpha) \label{eq:noisepartonlyeq12} \\
\overset{(a)}{=}&
(1 - \epsilon) q^{-\frac{2}{q-1}} \Big(\alpha \epsilon^{\frac{q-1}{2q}}
\aeps^{\frac{(q-1)^2}{q}} \Big)^{-\frac{2}{q-1}} \tau_*^2 \;
\mathbb{E} \big| |Z| - |\eta_q(Z; \alpha)| \big|^{\frac{2}{q-1}}
\rightarrow
q^{-\frac{2}{q-1}} C^{-\frac{2}{q-1}} \sigma^2
\mathbb{E} |Z|^{\frac{2}{q-1}}. \nonumber
\end{align}
Equality (a) is due to Lemma \ref{prox:property}(ii), and the last step is a
result of DCT. We should also emphasize that since $\alpha \rightarrow \infty$,
$|\eta_q (Z;\alpha)| \rightarrow 0$. Furthermore, similar to the steps in the
proof of Lemma \eqref{lem:correctordergoeszero}, we can obtain
\begin{align}\label{eq:risksecondtermfinite1}
    A_2(\epsilon)
    =\;
    \epsilon \alpha^2 \tau_*^{4 - 2q} q^2 \mathbb{E} \big|\eta_q( b_\epsilon \tilde{G} + \tau_* Z; \alpha
    \tau_*^{2-q}) \big|^{2q-2} + \epsilon \tau_*^2
    + 2 \epsilon \tau_*^2 \mathbb{E} [\partial_1 \eta_q (b_\epsilon \tilde{G} + \tau_* Z; \alpha \tau_*^{2-q}) - 1 ].  
\end{align}

First we have that
\begin{align}
    & \epsilon^{-\frac{1}{q}} \aeps^{\frac{-2(q-1)}{q}} \epsilon \alpha^2 \tau_*^{4 - 2q} q^2
    \mathbb{E} \big|\eta_q( b_\epsilon \tilde{G} + \tau_* Z; \alpha
    \tau_*^{2-q}) \big|^{2q-2} \nonumber \\
    =&
    \alpha^{2} \epsilon^{\frac{q-1}{q}} \aeps^{\frac{2(q-1)^2}{q}} \tau_*^{4 - 2q} q^2
    \mathbb{E} \bigg[ \frac{|\eta_q( b_\epsilon \alpha^{-\frac{1}{2-q}}
    \tilde{G} + \alpha^{-\frac{1}{2-q}}\tau_* Z; \tau_*^{2-q})|}
    { |b_\epsilon \alpha^{-\frac{1}{2-q}} \tilde{G} + \alpha^{-\frac{1}{2-q}}\tau_* Z| }
    |\tilde{G} +  b_\epsilon^{-1} \tau_* Z| \bigg]^{2q-2}
\end{align}

We note our condition on the growth of $\alpha$ and the following relation:
\begin{equation*}
    b_\epsilon \alpha^{-\frac{1}{2-q}}
    =
    C^{-\frac{1}{2-q}} \big[\epsilon^{\frac{q-1}{2}} \aeps \big]^{\frac{1}{q(2-q)}}
    \rightarrow \infty,
    \quad
    \epsilon^{-\frac{1}{q}} \aeps^{\frac{-2(q-1)}{q}} \epsilon
    =
    \big[\epsilon \aeps^{-2}\big]^{\frac{q-1}{q}}
    \rightarrow 0
\end{equation*}
The first relation above implies that
\begin{equation}
    \lim_{\epsilon \rightarrow 0}
    \epsilon^{-\frac{1}{q}} \aeps^{\frac{-2(q-1)}{q}} \epsilon \alpha^2 \tau_*^{4 - 2q} q^2
    \mathbb{E} \big|\eta_q( b_\epsilon \tilde{G} + \tau_* Z; \alpha
    \tau_*^{2-q}) \big|^{2q-2}
    =
    C^2 \sigma^{4 - 2q} q^2 \mathbb{E} |\tilde{G}|^{2q-2}
\end{equation}

Since $\big|\partial_1 \eta_q (b_\epsilon \tilde{G} + \tau_* Z;
\alpha \tau_*^{2-q}) - 1 \big| \leq 1$,
we are able to conclude that
\begin{equation} \label{eq:lasteqnthistheoremF4}
    \epsilon^{-\frac{1}{q}} \aeps^{\frac{-2(q-1)}{q}} A_2(\epsilon)
    \rightarrow
    C^2 \sigma^{4 - 2q} q^2 \mathbb{E} |\tilde{G}|^{2q-2}
\end{equation}
where the last step is a simple application of DCT (According to Lemma
\ref{prox:property}(ii) $\frac{|\eta_q (u; \chi)|}{|u|} \leq 1$ for every $u$
and $\chi$), combined with the fact that ${\aeps} \rightarrow \infty$ and
$\tau_* \rightarrow \sigma$ as $\epsilon \rightarrow 0$. Combining
\eqref{eq:Adefs}, \eqref{eq:noisepartonlyeq12},
\eqref{eq:risksecondtermfinite1}, and \eqref{eq:lasteqnthistheoremF4} finishes
the proof.
\end{proof}

So far, we know that $\alpha_* \rightarrow \infty$. Our next theorem provides more accurate information about $\alpha_*$.
\begin{theorem}\label{thm:caseiialphaorderprel}
If $\aeps =\omega (\epsilon^{\frac{1-q}{2}})$, then $\aeps \alpha_*^{-\frac{1}{2-q}} \rightarrow \infty$. 
\end{theorem}
\begin{proof}
Suppose this is not correct, then $\aeps \alpha_*^{-\frac{1}{2-q}} =
O(1)$. According to \eqref{eq:risksecondtermfinite1} and
\eqref{eq:Adefs} we have
\begin{align}\label{eq:contractprovingnicecase}
\tilde{R}_q (\alpha_*, \epsilon, \tau_*)
\geq&\;
\epsilon \mathbb{E} \big( \eta_q (b_\epsilon \tilde{G} + \tau_* Z;
\alpha_* \tau_*^{2-q}) - b_\epsilon \tilde{G} \big)^2 \nonumber \\
\overset{(a)}{=}&\;
\epsilon  \aeps^2 \mathbb{E} \big( \eta_q (\tilde{G} + b_\epsilon^{-1} \tau_*
Z; b_\epsilon^{q-2} \alpha_*\tau_*^{2-q} ) - \tilde{G} \big)^2,
\end{align}
where the last equality is due to Lemma \ref{prox:property}(iii). Note that 
\begin{equation*}
\frac{\tilde{R}_q (\alpha_*, \epsilon, \tau_*)}{
\tilde{R}_q(C \epsilon^{ \frac{1-q}{2q}} \aeps^{-\frac{(q-1)^2}{q}}, \epsilon,
\tau_*)}
=
\frac{\epsilon^{-1}  \aeps^{-2}
\tilde{R}_q (\alpha_*, \epsilon, \tau_*)}{  \epsilon^{-\frac{1}{q}}
\aeps^{\frac{2(1 - q)}{q}} \tilde{R}_q(C \epsilon^{ \frac{1-q}{2q}}
\aeps^{-\frac{(q-1)^2}{q}}, \epsilon, \tau_*)} \times (\epsilon^{q-1} \aeps^{2})^{\frac{1}{q}}.
\end{equation*}

According to Lemma \ref{lem:optimalriskorder},
$\epsilon^{-\frac{1}{q}} \aeps^{\frac{2(1-q)}{q}}\tilde{R}_q(C
\epsilon^{ \frac{1-q}{2q}} \aeps^{-\frac{(q-1)^2}{q}},\epsilon, \tau_*)
= \Theta(1)$. By using the DCT in
\eqref{eq:contractprovingnicecase} (combined with the assumption that $\aeps
\alpha_*^{-\frac{1}{2-q}} = O(1)$), it is straightforward to confirm that
$\epsilon^{-1}  \aeps^{-2} \tilde{R}_q (\alpha_*, \epsilon, \tau_*) =
\Omega(1)$. Since, $(\epsilon^{q-1}  \aeps^{2})^{\frac{1}{q}} \rightarrow
\infty$, we conclude that
\begin{equation*}
\varliminf_{\epsilon \rightarrow 0}
\frac{\tilde{R}_q (\alpha_*, \epsilon, \tau_*)}{ \tilde{R}_q(C \sqrt{\epsilon}^{ \frac{1-q}{q}} \aeps^{-\frac{(q-1)^2}{q}},\epsilon, \tau_*)}
= \infty. 
\end{equation*}
This is contradicted with the optimality of $\alpha_*$. Hence, $\aeps
\alpha_*^{-\frac{1}{2-q}} \rightarrow \infty$.
\end{proof}

Finally, we are ready to prove the main claim of this section. 

\begin{theorem}\label{thm:refinedinfoalphastar}
Suppose that $\aeps =\omega (\epsilon^{\frac{1-q}{2}})$. Then,
$\epsilon^{\frac{q-1}{2q}} \aeps^{\frac{(q-1)^2}{q}} \alpha_* \rightarrow C_*$,
where $C_* = \bigg[\frac { \sigma^{2q-2} \mathbb{E} |Z|^{\frac{2}{q-1}}} {
(q-1) q^{\frac{2q}{q-1}} \mathbb{E} |\tilde{G}|^{2q-2}}\bigg]^{\frac{q-1}{2q}}$. Furthermore,
\begin{equation} \label{eq:ell_q_omega_case_risk}
    \epsilon^{-\frac{1}{q}} \aeps^{-\frac{2(q-1)}{q}} \tilde{R}_q (\alpha_*,\epsilon, \tau_*)
    \rightarrow
    (C_* q)^{\frac{-2}{q-1}} \sigma^2 \mathbb{E} |Z|^{\frac{2}{q-1}} + (C_*
    q)^2 \sigma^{4-2q} \mathbb{E} |\tilde{G}|^{2(q-1)}.
\end{equation}
\end{theorem}
\begin{proof}
We know that $\frac{\partial \tilde{R}_q(\alpha_*, \epsilon, \tau_*) }{\partial \alpha} =0$. Hence,
\begin{align}\label{eq:zeroder}
    0
    =&
    (1- \epsilon) \tau_*^2
    \mathbb{E} [\eta_q (Z; \alpha_*) \partial_2 \eta_q(Z; \alpha_*)] \nonumber \\
    &+
    \epsilon \tau_*^{2-q} \mathbb{E} \big[(\eta_q (\aeps \tilde{G}+ \tau_*
    Z; \alpha_* \tau_*^{2-q}) - \aeps \tilde{G}) \partial_2 \eta_q (\aeps
    \tilde{G}+ \tau_* Z; \alpha_* \tau_*^{2-q}) \big] \nonumber \\
    \overset{(a)}{=}&
    (1- \epsilon) \tau_*^2 \mathbb{E} \bigg[ \frac{- q |\eta_q (Z;
    \alpha_*)|^q} {1 + \alpha_* q (q-1) |\eta_q (Z; \alpha_*)|^{q-2} } \bigg]
    + \epsilon \tau_*^{3-q} \mathbb{E} \big[ Z \partial_2 \eta_q (\aeps \tilde{G}+
    \tau_* Z;\alpha_* \tau_*^{2-q}) \big] \nonumber \\
    & + \epsilon \alpha_* \tau_*^{4-2q} q^2 \mathbb{E} \bigg[ \frac{ |\eta_q
    (\aeps \tilde{G}+ \tau_* Z;\alpha_* \tau_*^{2-q})|^{2q-2}} {1 +
    \alpha_*\tau_*^{2-q} q(q-1)|\eta_q (\aeps \tilde{G}+ \tau_* Z; \alpha_*
    \tau_*^{2-q})|^{q-2} } \bigg] \nonumber \\
    \triangleq&
    - (1- \epsilon) q \tau_*^2 H_1 + \epsilon \tau_*^{3-q} H_2 + \epsilon \alpha_*
    \tau_*^{4-2q} q^2 H_3. 
\end{align}
To obtain (a) we have used Lemma \ref{prox:property} parts (ii) and (v). 
We now study each term separately. We should mention at this point that the
rest of our analyses relies heavily on Theorem \ref{thm:caseiialphaorderprel}.
By using Lemma \ref{prox:property}(iii) we have
\begin{align}\label{eq:H1Firsteq1}
    H_1
    =& \mathbb{E} \bigg[ \frac{|\eta_q (Z; \alpha_*)|^q}
    {1 + \alpha_* q (q-1) |\eta_q (Z; \alpha_*)|^{q-2} } \bigg] \nonumber \\
    =&
    \alpha_*^{-1} \mathbb{E} \Bigg[ \bigg| \frac{|Z| - |\eta_q(Z; \alpha_*)|}
    {q\alpha_*} \bigg|^\frac{2}{q-1}
    \frac{|\eta_q(\alpha_*^{-\frac{1}{2-q}} Z; 1)|^{q-2}} {1 + q(q-1)
    |\eta_q(\alpha_*^{-\frac{1}{2-q}} Z; 1)|^{q-2}} \Bigg]. 
\end{align}

Since $ |\eta_q(\alpha_*^{-\frac{1}{2-q}} Z; 1)| \rightarrow 0$, we have
\begin{equation}\label{eq:H1calculationeq1step3}
    \frac{|\eta_q(\alpha_*^{-\frac{1}{2-q}} Z; 1)|^{q-2}} {1 + q(q-1)
    |\eta_q(\alpha_*^{-\frac{1}{2-q}} Z; 1)|^{q-2}}
    \rightarrow
    \frac{1}{q(q-1)}. 
\end{equation}
By combining \eqref{eq:H1Firsteq1} and \eqref{eq:H1calculationeq1step3} we have
\begin{equation}\label{eq:H1calcFinal}
    \alpha_*^{\frac{q+1}{q-1}} H_1
    \rightarrow
    (q-1)^{-1} q^{-\frac{q+1}{q-1}} \mathbb{E} | Z|^{\frac{2}{q-1}}.
\end{equation}

Now we focus on $H_3$. Define $D \triangleq 1+  q(q-1) \tau_*^{2-q}|\eta_q
(\aGoChStar+  \ZoChiStarTaus;\tau_*^{2-q})|^{q-2}$. Then,
\begin{align*}
    H_3
    =&
    \mathbb{E} \bigg[ \frac{ |\eta_q (\aeps \tilde{G}+ \tau_* Z; \alpha_* \tau_*^{2-q})|^{2q-2}}
    {1 + \alpha_*\tau_*^{2-q} q(q-1)|\eta_q (\aeps \tilde{G}+ \tau_* Z; \alpha_* \tau_*^{2-q})|^{q-2} } \bigg] \nonumber \\
    =&
    \mathbb{E} \bigg[ \frac{ b_\epsilon^{2q-2} |\eta_q (\tilde{G} +
    b_\epsilon^{-1}\tau_* Z; \alpha_* b_\epsilon^{q-2} \tau_*^{2-q})|^{2q-2}}
    {1 + \alpha_* b_\epsilon^{q-2} \tau_*^{2-q} q(q-1)|\eta_q (\tilde{G} + b_\epsilon^{-1}\tau_* Z; \alpha_* b_\epsilon^{q-2} \tau_*^{2-q})|^{q-2} } \bigg].
\end{align*}
According to Theorem \ref{thm:caseiialphaorderprel}, $\alpha_* b_\epsilon^{q-2}
\rightarrow 0$. This drives the term $\alpha_*\tau_*^{2-q} q(q-1)|\eta_q (\aeps
\tilde{G}+ \tau_* Z; \alpha_* \tau_*^{2-q})|^{q-2} \rightarrow 0$ in the denominator.
It is then straightforward to use DCT and show that 
\begin{equation}\label{eq:H3calcFinal1}
    \lim_{\epsilon \rightarrow 0} \aeps^{2-2q} H_3
    =
    \mathbb{E} |\tilde{G}|^{2q-2}.
\end{equation}

Finally, we discuss $H_2$. By using Stein's lemma and after some algebraic calculations we have
\begin{align*}
    H_2
    =&
    - q \mathbb{E} \bigg[Z \frac{|\eta_q(\aeps \tilde{G}+\tau_* Z ;\alpha_*
    \tau_*^{2-q}) |^{q-1} \mathrm{sign} (\aeps \tilde{G} + \tau_*
    Z;\alpha_* \tau_*^{2-q})}{1+ \alpha_* \tau_*^{2-q}q (q-1) |\eta_q(\aeps
    \tilde{G}+ \tau_* Z ) |^{q-2} } \bigg] \nonumber \\
    =&
    -q (q-1) \tau_* \mathbb{E} \bigg[ \frac{|\eta_q(\aeps \tilde{G}+ \tau_* Z ;\alpha_* \tau_*^{2-q})|^{q-2}}
    {(1+ \alpha_* \tau_*^{2-q} q (q-1) |\eta_q(\aeps \tilde{G}+ \tau_* Z;
    \alpha_* \tau_*^{2-q}) |^{q-2})^3} \bigg] \nonumber \\
    &-
    q^2 (q - 1) \alpha_* \tau_*^{3-q} \mathbb{E} \bigg[ \frac{|\eta_q(\aeps
    \tilde{G}+ \tau_* Z ;\alpha_* \tau_*^{2-q}) |^{2q-4}}{(1+ \alpha_*
    \tau_*^{2-q} q (q-1) |\eta_q(\aeps \tilde{G}+ \tau_* Z ;\alpha_* \tau_*^{2-q})
    |^{q-2})^3} \bigg]  \nonumber \\
    \triangleq&
    - q(q-1) \tau_* H_4 - q^2(q-1) \alpha_* \tau_*^{3-q} H_5.
\end{align*}

Now we bound $H_4$ and $H_5$. Due to exactly the same reason when we analyzing
$H_3$, the denominator of $H_4$ and $H_5$ converges to 1. According to Lemma
\ref{prox:property}(iii), we have
\begin{equation}\label{eq:H4calculations1}
    b_\epsilon^{2-q} H_4
    =
    \mathbb{E} \bigg[ \frac{|\eta_q(\tilde{G}+ b_\epsilon^{-1} \tau_* Z;
    \alpha_*b_\epsilon^{q-2} \tau_*^{2-q})|^{q-2}}
    {(1+ \alpha_* \tau_*^{2-q} q (q-1) |\eta_q(\aeps \tilde{G}+ \tau_* Z; \alpha_* \tau_*^{2-q}) |^{q-2})^3} \bigg]
    \rightarrow
    \mathbb{E}|\tilde{G}|^{q-2}
\end{equation}

Similarly for $H_5$ we have that
\begin{equation}\label{eq:H5calculations}
    b_\epsilon^{4-2q} H_5
    =
    \mathbb{E} \bigg[ \frac{|\eta_q(\tilde{G}+ b_\epsilon^{-1} \tau_* Z;
    \alpha_*b_\epsilon^{q-2} \tau_*^{2-q})|^{2q-4}}
    {(1+ \alpha_* \tau_*^{2-q} q (q-1) |\eta_q(\aeps \tilde{G}+ \tau_* Z; \alpha_* \tau_*^{2-q}) |^{q-2})^3} \bigg]
    \rightarrow
    \mathbb{E}|\tilde{G}|^{2q-4}
\end{equation}

From \eqref{eq:zeroder}  and with some algebra we have
\begin{align*}
    1
    =&
    \lim_{\epsilon \rightarrow 0}
    \frac{\epsilon \tau_*^{3-q} H_2 + \epsilon \alpha_*
    \tau_*^{4-2q} q^2 H_3}{ (1- \epsilon) q \tau_*^2 H_1 }
    =
    \lim_{\epsilon \rightarrow 0}
    \frac{\tau_*^{3-q} \alpha_*^{-1} b_\epsilon^{2-2q} H_2 +
    \tau_*^{4-2q} q^2 b_\epsilon^{2-2q} H_3}
    {(1- \epsilon) q \tau_*^2 \alpha_*^{\frac{q+1}{q-1}} H_1}
    \epsilon \alpha_*^{\frac{2q}{q-1}}b_\epsilon^{2q-2} \nonumber \\
    =&
    \frac{ (q-1) q^{\frac{2q}{q-1}} \mathbb{E} |\tilde{G}|^{2q-2}}
    { \sigma^{2q - 2} \mathbb{E} |Z|^{\frac{2}{q-1}}}
    \lim_{\epsilon \rightarrow 0}
    \epsilon \alpha_*^{\frac{2q}{q-1}}b_\epsilon^{2q-2}
\end{align*}
where in the last step, we use the fact that $\alpha^{-1} b_\epsilon^{2-2q} H_2
\rightarrow 0$ which is an implication of \eqref{eq:H4calculations1} and
\eqref{eq:H5calculations}. We also used \eqref{eq:H1calcFinal} and \eqref{eq:H3calcFinal1} to
simplify the part involving $H_1$ and $H_3$. Overall, these give us that
\begin{equation}\label{eq:finalformula_sigma}
    \lim_{\epsilon \rightarrow 0}
    \epsilon \alpha_*^{\frac{2q}{q-1}}b_\epsilon^{2q-2}
    =
    \frac { \sigma^{2q-2} \mathbb{E} |Z|^{\frac{2}{q-1}}}
    { (q-1) q^{\frac{2q}{q-1}} \mathbb{E} |\tilde{G}|^{2q-2}}
    =
    C_*^{\frac{2q}{q-1}}
\end{equation}

This proves the first claim of our theorem. The behavior of
$\tilde{R} (\alpha_*, \epsilon, \tau_*)$ now follows once we combine
\eqref{eq:finalformula_sigma} with Lemma \ref{lem:optimalriskorder}. In order
to obtain the final form presented in Theorem \ref{THM:STRONGRAREELL_Q}, we
need to substitute $C_*$ into \eqref{eq:ell_q_omega_case_risk} and simplify the expression.
\end{proof}

\subsection{Case III - $\frac{\aeps}{\sqrt{\epsilon}^{1-q}} \rightarrow c_{r}$
for $c_{r} \in (0,\infty)$}
\label{ssec:phasetransell_qorder}

Since the proof is very similar to the one we presented in the last section, we
only present the sketch of the proof, and do not discuss the details. We only
emphasize on the major differences. First note that similar to what we had
before
\begin{equation*}
\tilde{R}_q (\alpha, \epsilon, \tau_*)
=
(1-\epsilon) \tau_*^2 \mathbb{E} \eta_q^2 (Z; \alpha)
+ \epsilon \mathbb{E} \big(\eta_q (b_\epsilon \tilde{G} + \tau_* Z; \alpha\tau_*^{2-q}) - b_\epsilon \tilde{G} \big)^2. 
\end{equation*}
We can prove the following claims:
\begin{enumerate}

\item It is straightforward to prove that 
\begin{equation}\label{eq:riskatinfinityformula1}
\lim_{\alpha \rightarrow \infty } \tilde{R}_q (\alpha, \epsilon, \tau_*) = \epsilon \aeps^2 \mathbb{E} (\tilde{G})^2. 
\end{equation}

\item We claim that $\alpha_*^{-\frac{1}{2-q}} \aeps = O(1)$. We then have 
\begin{equation}\label{eq:lowerboundusualorder1}
    \varliminf_{\epsilon \rightarrow 0} \alpha_*^{\frac{2}{q-1}} \tilde{R}_q (\alpha_*, \epsilon, \tau_*)
    \geq
    \sigma^2 \varliminf_{\epsilon \rightarrow 0} \alpha_*^{\frac{2}{q-1}}
    \mathbb{E} \eta_q^2 ( Z; \alpha_*)
    \overset{(a)}{=}  \Theta(1).
\end{equation}
The reasoning for (a) is similar to what we did in \eqref{eq:noisepartonlyeq12}.
We connect \eqref{eq:riskatinfinityformula1} and
\eqref{eq:lowerboundusualorder1} through the optimality of $\alpha_*$ to
conclude that
\begin{equation*}
    \epsilon b_\epsilon^2 \alpha_*^{\frac{2}{q-1}} \geq \Theta(1)
\end{equation*}
Our claim then follows by substituting the relation $\epsilon \sim b_\epsilon^{- \frac{2}{q-1}}$ into the
above equation.

\item Given the previous case two scenarios can happen, each of which is discussed below:
\begin{itemize}
\item Case I:
$\alpha_*^{-\frac{1}{2-q}} \aeps \rightarrow 0$. 
Note that in this case, we have
\begin{equation*}
\varliminf_{\epsilon \rightarrow 0} \frac{\tilde{R}_q (\alpha_*, \epsilon, \tau_*)}{\epsilon \aeps^2 \mathbb{E} |\tilde{G}|^2}
\geq
\varliminf_{\epsilon \rightarrow 0}
\frac{ \mathbb{E} \big(\eta_q ( \tilde{G} + b_\epsilon^{-1} \tau_* Z;
\alpha_*\aeps^{q-2} \tau_*^{2-q}) -
\tilde{G} \big)^2} { \mathbb{E} |\tilde{G}|^2} = 1,
\end{equation*}
where the last step is a result of DCT and the assumption that
$\alpha_*^{-\frac{1}{2-q}} \aeps \rightarrow 0$. Note that the lower bound is
achievable by $\alpha = \infty$.

\item Case II:
$\alpha_*^{-\frac{1}{2-q}} \aeps = \Theta(1)$. Under this
assumption, we have $\alpha_* \epsilon^{\frac{(q-1)(2-q)}{2}} \rightarrow C$,
where $C > 0$ is fixed. We will specify the optimal choice of $C$ later. Furthermore,
\begin{equation*}
    \aeps \alpha_*^{-\frac{1}{2-q}} \rightarrow c_r C^{-\frac{1}{2-q}}.
\end{equation*}
We remind the reader that $c_r \triangleq \lim_{\epsilon\rightarrow 0}\frac{\aeps}{\sqrt{\epsilon}^{1-q}} \in (0, \infty)$. Similar to the proof of Lemma \ref{lem:risklbgen1} we have
\begin{equation}\label{eq:nosietermlastcase1}
    \alpha_*^{\frac{2}{q-1}} \mathbb{E} \eta_q^2 (Z; \alpha_*)
    \rightarrow
    q^{-\frac{2}{q-1}} \mathbb{E} |Z|^{\frac{2}{q-1}}. 
\end{equation}
Also, with DCT we can show that, as $\epsilon \rightarrow 0$
\begin{align}\label{eq:signaltermlastcase1}
    \small{\alpha_*^{-\frac{2}{2-q}} \mathbb{E} \big(\eta_q ( b_\epsilon \tilde{G} +
    \tau_* Z; \alpha_* \tau_*^{2-q} ) - b_\epsilon \tilde{G} \big)^2
    \rightarrow
    \mathbb{E} \big(\eta_q (c_r C^{-\frac{1}{2-q}} G;\sigma^{2-q} ) - c_r C^{-\frac{1}{2-q}} G \big)^2.}
\end{align}

Now we can characterize the risk accurately as
\begin{eqnarray*}
    \lefteqn{\lim_{\epsilon \rightarrow 0}
    \epsilon^{q-2} \tilde{R}_q (\alpha_*, \epsilon, \tau_*) =
    \lim_{\epsilon \rightarrow 0}
    \epsilon^{q-2} \alpha_*^{-\frac{2}{q-1}} \tau_*^2 \alpha_*^{\frac{2}{q-1}}
    \mathbb{E}\eta_q^2 (Z; \alpha_*)} \nonumber \\
    &&+
    \lim_{\epsilon \rightarrow 0}
    \epsilon^{q-1} \alpha_*^{\frac{2}{2-q}} \alpha_*^{-\frac{2}{2-q}}
    \mathbb{E} \big(\eta_q (b_\epsilon \tilde{G} + \tau_* Z;
    \alpha\tau_*^{2-q}) - b_\epsilon \tilde{G} \big)^2 \nonumber \\
    &\overset{(a)}{=}&
    C^{-\frac{2}{q-1}} \sigma^2 q^{-\frac{2}{q-1}}
    \mathbb{E} |Z|^{\frac{2}{q-1}}
    +
    \mathbb{E} \big(\eta_q (c_r G; C \sigma^{2-q} ) - c_r G \big)^2 =: h(C)
\end{eqnarray*}
 
To obtain Equality (a), we have combined \eqref{eq:nosietermlastcase1},
\eqref{eq:signaltermlastcase1} and the fact that $\epsilon^{q-2}
\alpha_*^{-\frac{2}{q-1}} \rightarrow C^{-\frac{2}{q-1}}$, $\epsilon^{q-1}
\alpha_*^{\frac{2}{2-q}} \rightarrow C^{\frac{2}{2-q}}$.

To get the optimal choice of $C$, we take the derivative with respect to $C$
and obtain
\begin{align*}
    h'(C)
    =&
    -\frac{2}{q-1} C^{-\frac{q+1}{q-1}} \sigma^2 q^{-\frac{2}{q-1}} \mathbb{E}
    |Z|^{\frac{2}{q-1}} \nonumber \\
    &+
    2 \sigma^{2-q} \mathbb{E} \big[\big(\eta_q (c_r G; C \sigma^{2-q} ) - c_r G
    \big) \partial_2 \eta_q (c_r G; C \sigma^{2-q} ) \big] \nonumber \\
    =&
    -\frac{2}{q-1} C^{-\frac{q+1}{q-1}} \sigma^2 q^{-\frac{2}{q-1}} \mathbb{E}
    |Z|^{\frac{2}{q-1}} \nonumber \\
    &+
    2 C \sigma^{4-2q} q^2 \mathbb{E} \bigg[ \frac{|\eta_q (c_r G; C \sigma^{2-q} )|^{2q-2}}
    {1 + C\sigma^{2-q} q(q-1) |\eta_q (c_r G; C \sigma^{2-q} )|^{q-2}} \bigg]
\end{align*}
To obtain the last equality we have used Lemma \ref{prox:property} parts (i),
(ii), and (v). We would like to show that the optimal choice of $C$
is finite. Toward this goal, we characterize the limiting
behavior of the ratio of the positive and negative terms in $\frac{dh(C)}{dC}$. First it is straighjtforward to see that $\lim_{C \rightarrow 0} h'(C) = -\infty$. Further we have
\begin{align*}
    &\lim_{C \rightarrow \infty} C^{\frac{q}{q-1}} h'(C)     =
    - \lim_{C \rightarrow \infty} \frac{2}{q-1} C^{-\frac{1}{q-1}} \sigma^2 q^{-\frac{2}{q-1}} \mathbb{E}
    |Z|^{\frac{2}{q-1}} \nonumber \\
    & + 2 \sigma^{4-2q} q^2 \lim_{C \rightarrow \infty} C^{\frac{q}{q-1}} C \mathbb{E} \bigg[ \frac{|\eta_q (c_r G; C \sigma^{2-q} )|^{2q-2}}
    {1 + C\sigma^{2-q} q(q-1) |\eta_q (c_r G; C \sigma^{2-q} )|^{q-2}} \bigg]
    \nonumber \\
    =&
    2 \sigma^{4-2q} q^2 \lim_{C \rightarrow \infty} \mathbb{E} \bigg[
    \bigg|\frac{|c_rG| - |\eta_q (c_r G; C \sigma^{2-q} )|}{q
    \sigma^{2-q}}\bigg|^{\frac{q}{q-1}} \nonumber \\
    &\quad\quad\quad\quad\quad\quad\quad\quad\quad
    \cdot \frac{|\eta_q (c_r C^{-\frac{1}{2-q}} G; \sigma^{2-q} )|^{q-2}}
    {1 + \sigma^{2-q} q(q-1) |\eta_q (c_r C^{-\frac{1}{2-q}} G; \sigma^{2-q})|^{q-2}} \bigg] \nonumber \\
    =&
    2 \sigma^{-\frac{2-q}{q-1}} (q-1)^{-1} q^{-\frac{1}{q-1}}
    \mathbb{E}|c_rG|^{\frac{q}{q-1}} > 0
\end{align*}
where in the last step we used the fact that $|\eta_q (c_r C^{-\frac{1}{2-q}} G; \sigma^{2-q})|
\rightarrow 0$ as $C \rightarrow \infty$. We should finally emphasize that,
since $\lim_{C \rightarrow \infty} h(C)$ equals the risk of
$\lim_{\epsilon \rightarrow 0} \epsilon^{q-2} \tilde{R}_q
(\alpha, \epsilon, \sigma)$ when $\alpha_*^{-\frac{1}{2-q}}\aeps
\rightarrow 0$, we conclude that $\alpha_*\epsilon^{\frac{(q-1)(2-q)}{2}} \rightarrow C_*$, where $C_*$ is
the minimizer of $h(C)$.
\end{itemize}
\end{enumerate}

        \subsection{$q \geq 2$} \label{ssec:nbo-qlarger2}
In this part, we prove the rate for $q \geq 2$ in the nearly black object
model. The proof when $\sqrt{\epsilon}b_\epsilon = o(1)$ can be simply
obtained according to the previous proof for $q < 2$. When $\sqrt{\epsilon}
b_\epsilon = \Theta(1)$, a slightly longer argument is involved.

\paragraph{$\sqrt{\epsilon}b_\epsilon = o(1)$}
In this case, we have $\tau_* \rightarrow \sigma$ according to Lemma
\ref{lem:upperbound}. Using the same argument as the start of Section
\ref{ssec:roadmapstrraresig}, we know $\alpha_* \rightarrow \infty$. Blessed by
the condition $q \geq 2$, we know $b_\epsilon = o(\epsilon^{\frac{1-q}{2}})$
and $\alpha_*b_\epsilon^{q-2} \rightarrow \infty$. The conclusion of
Lemma \ref{lem:chiinfty} and Theorem \ref{thm:finalchininfty} simply follows
and we have $\tilde{R}(\alpha_*, \epsilon, \tau_*) \sim \epsilon
b_\epsilon^2$.

\paragraph{$\sqrt{\epsilon}b_\epsilon = \Theta(1)$}
Assume $\lim_{\epsilon \rightarrow 0}\sqrt{\epsilon}b_\epsilon = c > 0$. Let  
$\lim_{\epsilon \rightarrow 0} \alpha_* = \alpha_0 \in [0, \infty]$ (we may
focus on one of the convergent subsequences). The limit of the
optimal $\tau_*^2$ is bounded in the sense that $\sigma^2 \leq
\varliminf_{\epsilon \rightarrow 0} \tau_*^2 \leq \varlimsup_{\epsilon
\rightarrow 0} \tau_*^2 \leq \sigma^2 + \frac{1}{\delta}c^2$. Let us consider a
convergent subsequence of $\tau_*$ (since it is bounded). By using the
state-evolution equation we have
\begin{align*}
    \lim_{\epsilon \rightarrow 0} \tau_*^2
    =&
    \sigma^2 + \frac{1}{\delta} \lim_{\epsilon \rightarrow 0}
    \mathbb{E}[\eta_q(B + \tau_*Z; \alpha_*\tau_*^{2-q}) - B]^2 \nonumber \\
    =&
    \sigma^2 + \frac{1}{\delta} \lim_{\epsilon \rightarrow 0}
    \Big[ (1 - \epsilon)\tau_*^2 \mathbb{E}\eta_q^2(Z; \alpha_*) + \epsilon b_\epsilon^2
    \mathbb{E}[\eta_q(\tilde{G} + \tau_*b_\epsilon^{-1}Z;
    \alpha_*b_\epsilon^{q-2} \tau_*^{2-q}) - \tilde{G}]^2 \Big] \nonumber \\
    =&
    \sigma^2 + \frac{1}{\delta} \lim_{\epsilon \rightarrow 0}
    \tau_*^2 \mathbb{E}\eta_q^2(Z; \alpha_0) + \frac{c^2}{\delta}
    \mathbb{E}[\eta_q(\tilde{G}; \lim_{\epsilon \rightarrow 0} (\alpha_*b_\epsilon^{q-2}) \lim_{\epsilon
    \rightarrow 0} \tau_*^{2-q}) - \tilde{G}]^2
\end{align*}

Under the assumption $\delta < 1$, the right hand side is always
larger than the left hand side when $\alpha_0 = 0$. This implies
that $\alpha_0 > 0$.

When $q > 2$, we have $\lim_{\epsilon \rightarrow 0} \alpha_*b_\epsilon^{q-2}
\rightarrow \infty$. This leads to the following result for $\tau_*^2$:
\begin{equation*}
    \lim_{\epsilon\rightarrow 0} \tau_*^2
    = \frac{\sigma^2 + \frac{c^2}\delta{}}{1 -
    \frac{1}{\delta}\mathbb{E}\eta_q^2(Z; \alpha_0)}
\end{equation*}

The larger $\alpha_0$ is, the smaller $\tau_*^2$ is. Hence we have
$\alpha_* \rightarrow \infty$ and $\tau_*^2 \rightarrow \sigma^2 +
\frac{c^2}{\delta}$. This gives us
\begin{equation*}
    \tilde{R}_q(\alpha_*, \epsilon, \tau_*) \rightarrow c^2,
    \quad \text{when} \quad
    q > 2.
\end{equation*}

When $q = 2$, the above argument becomes invalid. However in this case
$\eta_q(u; \chi) = \frac{u}{1 + 2\chi}$, leading to an explicit form of the
optimal $\alpha_*$ and $\tau_*$. A careful calculation exhibits that
\begin{equation*}
    \alpha_*
    =
    \frac{1}{4}\bigg( \frac{\sigma^2}{\mathbb{E}B^2} +
    \frac{1}{\delta} - 1 + \sqrt{\Big(\frac{\sigma^2}{\mathbb{E}B^2} +
    \frac{1}{\delta} - 1\Big)^2 + \frac{4\sigma^2}{\mathbb{E}B^2}} \bigg)
    \rightarrow
    \frac{1}{4}\bigg( \frac{\sigma^2}{c^2} + \frac{1}{\delta} - 1 +
    \sqrt{\Big(\frac{\sigma^2}{c^2} + \frac{1}{\delta} - 1\Big)^2 +
    \frac{4\sigma^2}{c^2}} \bigg)
\end{equation*}

The corresponding limit of $\mathrm{MSE}$ can then be explicitly represented as
\begin{equation} \label{eq:near-black-L2-const}
    \tilde{R}_2(\alpha_*, \epsilon, \tau_*)
    =
    \frac{\delta\sigma^2 + 4\delta \alpha_*^2 c^2}{(1 + 2\alpha_*)^2\delta - 1}
\end{equation}
This completes the proof.


\section{Proof of Theorem \ref{EQ:THM:STRONGRAREELL_1}}
\label{thmsmallepsilonell_1}

\subsection{Roadmap of the proof}
The roadmap of the proof is similar to the one presented in Section
\ref{ssec:roadmapstrraresig}. As we discussed there, the main goal is to
characterize the behavior of $(\alpha_*, \tau_*)$ in \eqref{eq:se11} for $q=1$
with $B$ replaced by $b_\epsilon \tilde{B}$, where
$\tilde{B}=(1-\epsilon)\delta_0 + \epsilon p_{\tilde{G}}$.

Similar to the proof in Section \ref{ssec:roadmapstrraresig}, we can
again prove that as $\epsilon \rightarrow 0$, (i) $\tau_* \rightarrow \sigma,$
and (ii) $\alpha_* \rightarrow \infty$. For the sake of brevity we skip the
proof of this claim. The rest of this proof is to obtain a more accurate
statement about the behvaior of $\alpha_*$ and ${\rm AMSE} (1, \lambda^*_1) $.
The optimal choice of $\alpha$ depends on the relation between $\aeps$ and
$\epsilon$ in the following way:
\begin{itemize}
\item Case I - $\aeps = \omega(\sqrt{- \log\epsilon})$. Under this rate, we will prove that 
$\lim_{\epsilon \rightarrow 0} \frac{\alpha_*}{ \sqrt{- 2 \log \epsilon}} =1$.
We then use this result to show
$\lim_{\epsilon \rightarrow 0} \frac{{\rm AMSE} (1, \lambda^*_1)}{-\epsilon \log\epsilon} = 2 \sigma^2$. 
The proofs are presented in \ref{ssec:ell_1omegacase}. 

\item Case II - $\aeps = o(\sqrt{-\log\epsilon})$. If $\aeps =
    \omega(1)$, then $\lim_{\epsilon \rightarrow 0} \frac{{\rm AMSE} (1,
    \lambda^*_1)}{\epsilon b_\epsilon^2 } =\mathbb{E}\tilde{G}^2$. We prove
    this result in Section \ref{ssec:ell_1ocaseproof}. 

\item Case III - $\aeps = \Theta(\sqrt{-\log\epsilon})$. If $\frac{\aeps}{\sqrt{- 2\log \epsilon}} \rightarrow c$, then  
$\lim_{\epsilon \rightarrow 0}\frac{{\rm AMSE} (1, \lambda^*_1)}{-2\epsilon
\log \epsilon} =  \mathbb{E} (\eta_1(c\tilde{G}; \sigma) -c\tilde{G})^2$.
This claim is proved in Section \ref{ssec:ell_1equalProof}. 
\end{itemize}

\subsection{Case I - $\aeps = \omega(\sqrt{-\log\epsilon})$}\label{ssec:ell_1omegacase}
Before we start, we would like to remind our reader of the definition of
$\tilde{R}$ in \eqref{eq:definitiontilde{R}}. We will study the behavior of
$\tilde{R}$ as $\epsilon \rightarrow 0$ to obtain the rate of $\amse(1, \lambda_1^*)$.
Similar to the procedures in Section \ref{ssec:roadmapstrraresig}, we
characterize the rate of $\alpha_*$ in several steps: First we
describe the behavior of the AMSE for a specific choice of $\alpha$. The
suboptimality of this special choice then narrow down the scope of the
optimal $\alpha_*$. Finally, this information about $\alpha_*$ enables us to
accurately analyze the derivative of the risk with respect to $\alpha$
and the increasing rate of $\alpha_*$.

\begin{lemma}\label{lem:firstoneEll1}
Suppose that $\aeps = \omega(\sqrt{-\log\epsilon})$. If $\alpha =  \sqrt{-2 \log\epsilon}$, then 
\begin{equation}\label{eq:RiskEll1SpecificChoice}
    \lim_{\epsilon \rightarrow 0}
    \frac{\tilde{R}_1(\alpha, \epsilon, \tau_*)}
    {-\epsilon \log\epsilon} = 2 \sigma^2. 
\end{equation}
\end{lemma}

\begin{proof}
Recall the expansions of $R_1(\alpha, \tau)$ in \eqref{eq:L1-risk-expand1}, we
have the following expansion for $\tilde{R}_1(\alpha, \epsilon, \tau_*)$.
\begin{align}\label{eq:firstexpansionell_1}
    \tilde{R}_1(\alpha, \epsilon, \tau_*)
    =&
    (1-\epsilon) \tau_*^2 \mathbb{E} \eta_1^2 (Z; \alpha)
    +
    \epsilon \mathbb{E} [ \eta_1(b_\epsilon\tilde{G} + \tau_* Z; \alpha \tau_*) -
    b_\epsilon \tilde{G} - \tau_* Z ]^2 \nonumber \\
    &-
    \epsilon \tau_*^2 + 2\epsilon \tau_*^2 \mathbb{E}
    [ \partial_1\eta_1(b_\epsilon \tilde{G} + \tau_* Z; \alpha \tau_* ) ]
    \triangleq \tau_*^2(F_1 +  F_2- F_3 +  F_4).
\end{align}

As what we pointed out in \eqref{eq:F1CalcSoftThresholdNoise}, $F_1 =
2(1-\epsilon) [(1+\alpha_*^2)\Phi(-\alpha_*) - \alpha_*\phi(\alpha_*)]$.
Since $\alpha =  \sqrt{-2 \log \epsilon} \rightarrow \infty$, \eqref{gaussiantail:exp} implies that
\begin{equation}\label{eq:F1finalCalc}
    \lim_{\epsilon \rightarrow 0} \frac{F_1}{4 \phi(\alpha) / \alpha^3} = 1. 
\end{equation}

To calculate $F_2$ we note that
$\big|\eta_1\big(\frac{b_\epsilon\tilde{G}}{\alpha\tau_*} + \frac{Z}{\alpha};
1\big) - \frac{b_\epsilon\tilde{G}}{\alpha\tau_*} - \frac{Z}{\alpha} \big| \leq 1$ and $\tau_* \rightarrow
\sigma$. By using DCT and the fact that $\frac{\aeps}{\alpha} \rightarrow \infty$ we have
\begin{equation}\label{eq:F2finalCalc}
    \lim_{\epsilon \rightarrow 0} \frac{F_2}{\epsilon \alpha^2} = 1.
\end{equation}

It is straightforward to check that $|\partial_1\eta_1(b_\epsilon \tilde{G} + \tau_* Z; \alpha \tau_*)| < 1$, these give us that
\begin{equation}\label{eq:F3finalCalc}
    F_3=O(\epsilon),
    \quad
    F_4=O(\epsilon). 
\end{equation}

By combining \eqref{eq:firstexpansionell_1}, \eqref{eq:F1finalCalc}, \eqref{eq:F2finalCalc}, and \eqref{eq:F3finalCalc} we obtain \eqref{eq:RiskEll1SpecificChoice}. 
\end{proof}

Our next lemma provides a more refined information about $\alpha_*$.

\begin{lemma}\label{lem:orderell1firstcase1}
    For $\aeps = \omega(\sqrt{-\log \epsilon})$ there exists a $c \in [0,1]$, such that
\begin{equation*}
    \lim_{\epsilon \rightarrow 0} \frac{\sqrt{- 2 \log \epsilon}}{\alpha_*} = c.
\end{equation*}
\end{lemma}

\begin{proof}
From \eqref{eq:firstexpansionell_1} we have
\begin{align}
    \lim_{\epsilon \rightarrow 0}
    \frac{\tilde{R}_1(\alpha_*, \epsilon, \tau_*)} {\tilde{R}_1(\sqrt{- 2 \log \epsilon}, \epsilon, \tau_*)}
    \geq&
    \lim_{\epsilon \rightarrow 0} \frac{(1 - \epsilon) \tau_*^2 \mathbb{E} \eta_1^2(Z;\alpha_*) }
    {\tilde{R}_1(\sqrt{- 2 \log \epsilon}, \epsilon, \tau_*)}
    =
    \lim_{\epsilon \rightarrow 0} \frac{2 (1-\epsilon) \phi(\alpha_*)}
    { - \alpha_*^3 \epsilon \log\epsilon},
\end{align}
where the second inequality is due to \eqref{eq:RiskEll1SpecificChoice} and
\eqref{eq:F1finalCalc}. Furthermore, if $\lim_{\epsilon \rightarrow 0}
\frac{\sqrt{- 2 \log \epsilon}}{\alpha_*} >1$, then
\begin{equation}
    \lim_{\epsilon \rightarrow 0}
    \frac{2 (1-\epsilon) \phi(\alpha_*) }{ - \alpha_*^3 \epsilon \log \epsilon}
    >
    \lim_{\epsilon \rightarrow 0} \frac{ (1 - \epsilon) \phi(\alpha_*) }
    { \sqrt{2} \epsilon (- \log\epsilon)^{\frac{5}{2}}} \rightarrow \infty. 
\end{equation}
This is in contradiction with the optimality of $\alpha_*$. 
\end{proof}

\begin{lemma}\label{lem:aepsalphainfty}
For $\aeps = \omega(\sqrt{-\log \epsilon})$, we have
\[
\lim_{\epsilon \rightarrow 0} \frac{\aeps}{\alpha_*} = \infty. 
\]
\end{lemma}
\begin{proof}
We would like to prove this with contradiction. First, suppose that $\lim_{\epsilon \rightarrow 0} \frac{\aeps}{\alpha_*}  = 0$. Under this assumption, we have
\begin{eqnarray}\label{eq:orderaepsalpha1first}
 \lim_{\epsilon \rightarrow 0}  \frac{ \mathbb{E}
  \big[ \eta_1\big(b_\epsilon\tilde{G} + \tau_* Z; \alpha_* \tau_* \big) -
    b_\epsilon \tilde{G} \big]^2}{\mathbb{E}(\aeps^2 \tilde{G}^2)} = \lim_{\epsilon \rightarrow 0} \frac{ \mathbb{E}
  \big[ \eta_1\big(\tilde{G} + \frac{\tau_*}{\aeps} Z; \frac{\alpha_*}{\aeps} \tau_* \big) -
     \tilde{G} \big]^2}{\mathbb{E}(\tilde{G}^2)}=1,
\end{eqnarray}
where to obtain the last equality we have used DCT. Now we have
\begin{align}
    \lim_{\epsilon \rightarrow 0}
    \frac{\tilde{R}_1(\alpha_*, \epsilon, \tau_*)} {\tilde{R}_1(\sqrt{- 2 \log \epsilon}, \epsilon, \tau_*)}
    \geq
    \lim_{\epsilon \rightarrow 0} \frac{\epsilon \mathbb{E}
  \big[ \eta_1\big(b_\epsilon\tilde{G} + \tau_* Z; \alpha_* \tau_* \big) -
    b_\epsilon \tilde{G} \big]^2}
    {\tilde{R}_1(\sqrt{- 2 \log \epsilon}, \epsilon, \tau_*)}
    \overset{(a)}{=}
    \lim_{\epsilon \rightarrow 0} \frac{\epsilon \aeps^2 \mathbb{E} (\tilde{G}^2)}
    { - \epsilon \log\epsilon} = \infty.
\end{align}
Equality (a) is due to \eqref{eq:orderaepsalpha1first}, and the last equality is in contradiction with the optimality of $\alpha_*$. Similarly, we can show that if $\lim_{\epsilon \rightarrow 0} \frac{\aeps}{\alpha_*} = c$ ($c< \infty$), then $ \lim_{\epsilon \rightarrow 0} \frac{\tilde{R}_1(\alpha_*, \epsilon, \tau_*)} {\tilde{R}_1(\sqrt{- 2 \log \epsilon}, \epsilon, \tau_*)}= \infty$, which is again in contradiction with the optimality of $\alpha_*$. For brevity we skip this proof.
\end{proof}

\begin{theorem}\label{thm:LASSOcase1finalTheorem}
If $\aeps = \omega(\sqrt{ - \log \epsilon})$, then
\begin{equation}
    \lim_{\epsilon \rightarrow 0} \frac{\alpha_*}{ \sqrt{ - 2 \log \epsilon}} = 1. 
\end{equation}
\end{theorem}

\begin{proof}
We analyze the derivative of the risk with respect to $\alpha$. Recall the form
of $\frac{\partial R_1(\alpha, \tau)}{\partial\alpha}$ in
\eqref{eq:L1-risk-deri-expand}, we have
\begin{align}
    \frac{1}{\tau_*^2}\frac{\partial \tilde{R}_1(\alpha, \epsilon,
    \tau_*)}{\partial \alpha} \Big|_{\alpha=\alpha_*}
    =&
    2(1-\epsilon)\underbrace{[-\phi(\alpha_*) + \alpha_*\Phi(-\alpha_*)]}_{:=D_1}+\epsilon \underbrace{\mathbb{E}
        \Big[\alpha_*\Phi(\frac{|b_\epsilon\tilde{G}|}{\tau_*}-\alpha_*)-\phi(\alpha_*-\frac{|b_\epsilon\tilde{G}|}{\tau_*})
    \Big]}_{:=D_2} \nonumber \\
    &+
    \epsilon\underbrace{\mathbb{E}\Big[\alpha_*\Phi(-\frac{|b_\epsilon\tilde{G}|}{\tau_*}-\alpha_*) - \phi(\alpha_*+\frac{|b_\epsilon\tilde{G}|}{\tau_*})\Big]}_{:=D_3}
    \label{eq:ell_1firstderivative}
\end{align}

Since $\alpha_* \rightarrow \infty$, similar calculations as the one presented
\eqref{eq:F1finalCalc} lead to the rate for $D_1$; On the other hand, according
to Lemma \ref{lem:aepsalphainfty}, $\aeps/\alpha_* \rightarrow \infty$, This
gives us the rate for $D_2 + D_3$. Overall we have
\begin{equation}\label{eq:calc2D1Firstcase}
    \lim_{\epsilon \rightarrow 0} \frac{D_1}{\phi(\alpha_*)/ \alpha_*^2}
    \rightarrow -1,
    \qquad
    \lim_{\epsilon \rightarrow 0} \frac{D_2+D_3}{\alpha_*} \rightarrow 1. 
\end{equation}

Hence, by combining \eqref{eq:ell_1firstderivative} and \eqref{eq:calc2D1Firstcase}, we have
\[
\lim_{\epsilon \rightarrow 0} \frac{2\phi(\alpha_*)/\alpha_*^2}{\epsilon \alpha_*} = \lim_{\epsilon \rightarrow 0} \frac{- \frac{\epsilon D_2 +\epsilon D_3}{\epsilon \alpha_*}}{ D_1 \frac{\alpha_*^2}{ \phi(\alpha_*)} } = 1.
\]
By taking logarithm, $\lim_{\epsilon \rightarrow 0} - \frac{\alpha^2}{2 } - 3 \log \alpha - \log \epsilon =0.$ Since $\alpha \rightarrow \infty$, $\lim_{\epsilon \rightarrow 0} - \frac{1}{2} - \frac{\log \epsilon}{\alpha^2}= 0$. \end{proof}

Combining Lemma \ref{lem:firstoneEll1} and Theorem \ref{thm:LASSOcase1finalTheorem} proves  
$\lim_{\epsilon \rightarrow 0} \frac{{\rm AMSE} (1, \lambda^*_1) }{{-2 \sigma^2 \epsilon}\log\epsilon } =1$.

\subsection{Case II- $\aeps = o(\sqrt{- \log\epsilon})$}\label{ssec:ell_1ocaseproof}

\begin{lemma}\label{lem:lassolog2}
If $\aeps = \omega(1)$ and $\aeps = o(\sqrt{-\log\epsilon})$, then there exists $c \in [0,1]$, such that 
\begin{equation}\label{eq:cgoodbound}
\lim_{\epsilon \rightarrow 0} \frac{ \sqrt{-2 \log \epsilon}}{\alpha_*} = c,
\end{equation}
\end{lemma}
\begin{proof}
Since $\lim_{\alpha \rightarrow \infty} \tilde{R}_1(\alpha, \epsilon, \tau_*) = \epsilon b_\epsilon^2 \mathbb{E} \tilde{G}^2$, we have
\begin{eqnarray}
\lim_{\epsilon \rightarrow 0} \frac{\tilde{R}_1(\alpha_*, \epsilon, \tau_*)}{\epsilon b_\epsilon^2 } \leq \mathbb{E} \tilde{G}^2. 
\end{eqnarray}
Note that $\tilde{R}_1(\alpha_*, \epsilon, \tau_*) \geq (1-\epsilon) \tau_*^2
\mathbb{E} \eta_1^2 (Z;\alpha_* ).$ Hence, 
\[
    \lim_{\epsilon \rightarrow 0} \frac{(1-\epsilon) \tau_*^2 \mathbb{E} \eta_1^2 (Z;\alpha_*)}{\epsilon b_\epsilon^2 } \leq \mathbb{E} \tilde{G}^2.
\]
In \eqref{eq:F1CalcSoftThresholdNoise} and \eqref{eq:F1finalCalc} we prove that
$\lim_{\epsilon \rightarrow 0} \frac{(1-\epsilon) \mathbb{E} \eta_1^2(Z; \alpha_*)}
{\phi(\alpha_*) \alpha_*^{-3} }   = 1.$ Hence, 
\begin{equation}\label{eq:whatconfirmedell_1}
\lim_{\epsilon \rightarrow 0}  \frac{\tau_*^2 \phi(\alpha_*) \alpha_*^{-3} }{\epsilon b_\epsilon^2 } =  \lim_{\epsilon \rightarrow 0}  \frac{\sigma^2 \phi(\alpha_*) \alpha_*^{-3} }{\epsilon b_\epsilon^2 }\leq \mathbb{E} (\tilde{G})^2.
\end{equation}
It is straightforward to see that if \eqref{eq:cgoodbound} does not hold, then \eqref{eq:whatconfirmedell_1} will not be correct either. Hence, our claim is proved. \end{proof}

\begin{theorem}
If $\aeps = \omega(1)$ and $\aeps = o(\sqrt{-\log\epsilon})$, then 
\begin{equation*}
    \lim_{\epsilon \rightarrow 0} \frac{\tilde{R}_1(\alpha_*, \epsilon, \tau_*)}
    {\epsilon b_\epsilon^2 } =\mathbb{E}\tilde{G}^2. 
\end{equation*}
In other words, its dominant term is the same as that of $\alpha = \infty$. 
\end{theorem}
\begin{proof}
Note that
\begin{equation*}
    \tilde{R}_1(\alpha_*, \epsilon, \tau_*)
    =
    (1-\epsilon) \tau_*^2 \mathbb{E} \eta_1^2 ( Z;\alpha_*)
    +
    \epsilon \aeps^2 \mathbb{E} \big[ \eta_1\big( \tilde{G} + b_\epsilon^{-1} \tau_*
    Z; b_\epsilon^{-1}\alpha_* \tau_* \big) - \tilde{G} \big]^2
\end{equation*}
According to Lemmas \ref{lem:lassolog2} we know that $\alpha_*/ \aeps
\rightarrow \infty$ and $\alpha_* \rightarrow \infty$. Hence, by using DCT we can prove that 
\begin{equation*}
    \lim_{\epsilon \rightarrow 0}
    \mathbb{E} \big[ \eta_1\big(\tilde{G} + b_\epsilon^{-1} \tau_* Z;
    b_\epsilon^{-1} \alpha_* \tau_* \big) - \tilde{G} \big]^2
    = \mathbb{E} \tilde{G}^2. 
\end{equation*}
Hence,
\begin{equation*}
\lim_{\epsilon \rightarrow 0}
\frac{\tilde{R}_1(\alpha_*, \epsilon, \tau_*)}{\epsilon b_\epsilon^2}
\geq
\lim_{\epsilon \rightarrow 0} \mathbb{E} \big[ \eta_1\big(\tilde{G} +
b_\epsilon^{-1} \tau_* Z; b_\epsilon^{-1}\alpha_* \tau_* \big) - \tilde{G}
\big]^2
= \mathbb{E} \tilde{G}^2.
\end{equation*}
On the other hand, this lower bound is achieved for $\alpha = \infty$. Hence, the proof is complete. 
\end{proof}

\subsection{Case III- $\aeps = \Theta(\sqrt{- \log \epsilon})$}\label{ssec:ell_1equalProof}

\begin{theorem}
If $\frac{\aeps}{\sqrt{- 2\log \epsilon}} \rightarrow c$, then  
\begin{equation*}
\lim_{\epsilon \rightarrow 0} \frac{\tilde{R} (\alpha_*, \epsilon , \tau_*)}{- 2 \epsilon \log \epsilon}
=
\mathbb{E} (\eta_1(c\tilde{G}; \sigma) - c \tilde{G})^2.
\end{equation*}
\end{theorem}
\begin{proof}
Since the proof is very similar to the proof of Theorem \ref{thm:LASSOcase1finalTheorem}, we only present a proof sketch here. It is straightforward to check the following steps:
\begin{enumerate}
\item
    If $\alpha = \sqrt{-2 \log \epsilon}$, then $\lim_{\epsilon \rightarrow 0}
    \frac{\tilde{R}_1(\alpha, \epsilon, \tau_*)}{-2\epsilon \log\epsilon} = \mathbb{E} (\eta_1(c\tilde{G}; \sigma) -c\tilde{G})^2$.  The proof is similar to the proof of Lemma \ref{lem:firstoneEll1}. 
\item
    $\lim_{\epsilon \rightarrow 0} \frac{\sqrt{-2 \log \epsilon}}{\alpha_*}
    = \frac{1}{\tilde{c}}$, where $\tilde{c} \in [1,\infty)$.  The proof
        is exactly the same as the proof of Lemma
        \ref{lem:orderell1firstcase1}. $\tilde{c}$ can reach $\infty$?
\item
    For notational simplicity, suppose $\alpha = \tilde{c} \sqrt{- 2 \log
    \epsilon}$, where $\tilde{c}\geq 1$ (it is straightforward to show that
    $\lim_{\epsilon \rightarrow 0} \frac{\alpha_*}{\sqrt{-2 \log \epsilon}} $
    is not infinite. This will be clear from the rest of the proof too.).
    Recall the expansion of $\tilde{R}_1(\alpha, \epsilon, \tau_*)=\tau_*^2(F_1
    + F_2 - F_3 + F_4)$ in \eqref{eq:firstexpansionell_1}. It is first
    straightforward to confirm the following claims.
\begin{equation*}
    \frac{F_1}{- 2 \epsilon \log \epsilon}
    \rightarrow 0,
    \qquad
    \frac{F_3}{- 2 \epsilon \log \epsilon }
    \rightarrow 0,
    \qquad
    \frac{F_4}{- 2 \epsilon \log \epsilon }
    \rightarrow 0.
\end{equation*}

Furthermore, it is straightforward to show that
\begin{equation*}
    \frac{\tau_*^2 F_2}{- 2\epsilon \log \epsilon }
    \rightarrow
    \mathbb{E}(\eta_1 (c \tilde{G} ; \tilde{c} \sigma)- c \tilde{G} )^2. \nonumber \\
\end{equation*}
Note that 
\[
(\eta_1 (c \tilde{G} ; \tilde{c} \sigma)- c \tilde{G} )^2 = \min (\tilde{c}^2
\sigma^2, c^2 \tilde{G}^2)
\geq \min (\sigma^2, c^2 \tilde{G}^2),
\]
where the last inequality is due to $\tilde{c} \geq 1$. 
Hence, for any $\alpha = \tilde{c} \sqrt{- 2 \log \epsilon}$, we have the following lower bound:
\[
\varliminf_{\epsilon \rightarrow 0}
\frac{\tilde{R}_1(\alpha, \epsilon, \tau_*) }{- 2 \epsilon \log \epsilon} \geq
\mathbb{E} \min (\sigma^2, c^2 \tilde{G}^2). 
\]
Note that this lower bound is achieved by $\alpha = \sqrt{ -2 \log \epsilon}$. This completes the proof. 
\end{enumerate}
\end{proof}

\section{Proof of Theorem \ref{THM:SPARSE:MAIN}}
\label{sec:proofFinitePowerExtremeSPARSE}

Before we discuss the proof of our main theorem, we mention a preliminary lemma that will later be used in our proof. 

\subsection{Preliminaries}

\begin{lemma}[Laplace Approximation]
\label{lemma:sparse:l1:laplace}
Suppose $G$ is nonnegative and $\esssup(G)=M$. Then for arbitrary nonnegative continuous function $f$ we have $\frac{\mathbb{E}\left(f(G)e^{\alpha G}\right)}{f(M)\mathbb{E}e^{\alpha G}}\rightarrow1$ as $\alpha\rightarrow\infty$.
\end{lemma}

\begin{proof}[Proof of Lemma \ref{lemma:sparse:l1:laplace}]
Let $\mathcal{G}$ denote the distribution of $G$. For any small $\delta>0$ and large $A$, let $0<\delta_1<\delta$, we have
\begin{equation*}
\frac{\int_{G>M-\delta} e^{\alpha G}d\mathcal{G}}{\int_{G\leq M-\delta} e^{\alpha G}d\mathcal{G}}
\geq
\frac{\int_{G>M-\delta_1} e^{\alpha G}d\mathcal{G}}{\int_{G\leq M-\delta} e^{\alpha G}d\mathcal{G}}
\geq
\frac{e^{\alpha(M-\delta_1)}\mathbb{P}(G>M-\delta_1)}{e^{\alpha(M-\delta)}\mathbb{P}(G\leq M-\delta)}
= \frac{\mathbb{P}(G>M-\delta_1)}{\mathbb{P}(G\leq M-\delta)}e^{\alpha(\delta-\delta_1)}>A
\end{equation*}
for large enough $\alpha$. This implies that $\frac{\int_{G>M-\delta} e^{\alpha G}d\mathcal{G}}{\mathbb{E}e^{\alpha G}}>\frac{A}{A+1}$ for large $\alpha$. Notice the continuity of $f$, we have $|f(G)-f(M)|<\epsilon$ when $G$ is close enough to $M$ and $|f|\leq C$ on $[0,M]$. Thus we have
\begin{align*}
\frac{f(M)-\epsilon}{f(M)}\frac{A}{A+1}
\leq &
\frac{\int_{G>M-\delta}f(G)e^{\alpha G}d\mathcal{G}}{f(M)\mathbb{E}e^{\alpha G}}
\leq
\frac{\mathbb{E}\left(f(G)e^{\alpha G}\right)}{f(M)\mathbb{E}e^{\alpha G}}\\
\leq&
\frac{\int_{G>M-\delta}f(G)e^{\alpha G}d\mathcal{G}+C\int_{G\leq M-\delta}e^{\alpha G}d\mathcal{G}}{f(M)\int_{G>M-\delta}e^{\alpha G}d\mathcal{G}}
\leq
\frac{f(M)+\epsilon}{f(M)}+\frac{C}{Af(M)}
\end{align*}

This holds for arbitrary $\epsilon,\delta$ and $A$. Our conclusion follows.
\end{proof}

\subsection{$q=1$: LASSO case}

Recall the definition of $(\alpha_*, \tau_*)$ in \eqref{eq:se11}. As we discussed in Section \ref{ssec:proof-sketch}, the main objective is to characterize the behavior of $\alpha_*$ and $\tau_*$ for large values of $\epsilon$. First, we prove that $\alpha_* \rightarrow \infty$ and $\tau_* \rightarrow \sigma$ as $\epsilon \rightarrow 0$.
\begin{lemma}\label{lemma:l1:alpha-tau-trend}
    As $\epsilon \rightarrow 0$, we have $\alpha_* \rightarrow \infty$ and
    $\tau_* \rightarrow \sigma$.
\end{lemma}
\begin{proof}
    First as $\epsilon \rightarrow 0$, we can pick the sequence of $\alpha
    \rightarrow \infty$, noticing that $\tau^2 = \frac{\delta\sigma^2}{\delta -
    \mathbb{E}[\eta_1(\beta / \tau + Z; \alpha) - \beta / \tau]^2}$, the
    corresponding fixed point solution $\tau^2 \rightarrow \sigma^2$.
    Now suppose $\varlimsup_{\epsilon \rightarrow
    0}\alpha_* < \infty$, then we consider a convergent subsequence $\alpha_*
    \rightarrow \bar{\alpha}$. If $\tau_* \rightarrow \infty$, then
    $\lim_{\epsilon \rightarrow \infty} \tau_*^2 = \frac{\delta\sigma^2}{\delta
    - \mathbb{E}[\eta_1(Z; \bar{\alpha})]^2} < \infty$ which forms a
    contradiction. Assume $\tau_* \rightarrow \bar{\tau} < \infty$ (say we
    pick a sub-subsequence), then it is not hard to see $\bar{\tau}^2 = \sigma^2 +
    \frac{1}{\delta} \mathbb{E}[\eta_1(\beta + \bar{\tau} Z; \bar{\alpha}
    \bar{\tau}) - \beta]^2 > \sigma^2$. This forms a contradiction with the
    optimality of $\alpha_*$.
\end{proof}

The next step is to obtain more accurate information about $\alpha_*$ and $\sigma_*$. Lemma  \ref{appendix:noisy:l1:tail} paves the way toward this goal.

\begin{lemma}\label{appendix:noisy:l1:tail}
For any given $h(G)$ being a positive function over $[0,+\infty)$ with $\mathbb{E}h(|G|)<\infty$, there exists a constant $\xi>0$ such that the following results hold for all sufficiently small $\epsilon$,
\begin{equation*}
    \mathbb{E} \big[ h(|G|) \phi(\alpha_*-|G|/\tau_*) \cdot \mathbbm{1}(\xi \leq |G|\leq t \sigma \alpha_*) \big]
    >
    \mathbb{E} \big[ h(|G|) \phi(\alpha_*-|G|/\tau_*) \cdot \mathbbm{1}( |G|\leq \xi ) \big].
\end{equation*}
\end{lemma}

\begin{proof}
We present the proof for $\alpha_*$. Let $p(x)$ be the probability function of $|G|$ and $g_{\xi,\tau}(\alpha)=\frac{\int_{\xi}^{t\sigma \alpha}h(x)p(x)e^{\frac{\alpha x}{ \tau}-\frac{x^2}{2\tau^2}}dx}{\int_0^{\xi} h(x)p(x) e^{\frac{\alpha x}{ \tau}-\frac{x^2}{2\tau^2}} dx}$. For fixed $\xi, \tau>0$, we calculate the derivative of $g$ with respect to $\alpha$:
\begin{eqnarray*}
g'_{\xi, \tau}(\alpha)&=&\frac{t\sigma h(t\sigma \alpha)p(t\sigma \alpha)e^{(\frac{t\sigma}{2\tau}-\frac{t^2\sigma^2}{2\tau^2})\alpha^2}}{\int_0^{\xi} h(x)p(x) e^{\frac{\alpha x}{ \tau}-\frac{x^2}{2\tau^2}} dx} + \frac{\int_{\xi}^{t\sigma \alpha}\frac{x}{\tau}h(x)p(x)e^{\frac{\alpha x}{ \tau}-\frac{x^2}{2\tau^2}}dx}{\int_0^{\xi} h(x)p(x) e^{\frac{\alpha x}{ \tau}-\frac{x^2}{2\tau^2}} dx} \nonumber \\
&&-\frac{\int_{\xi}^{t\sigma \alpha}h(x)p(x)e^{\frac{\alpha x}{ \tau}-\frac{x^2}{2\tau^2}}dx \cdot \int_0^{\xi}\frac{x}{\tau}h(x)p(x)e^{\frac{\alpha x}{ \tau}-\frac{x^2}{2\tau^2}}dx }{(\int_0^{\xi} h(x)p(x) e^{\frac{\alpha x}{ \tau}-\frac{x^2}{2\tau^2}} dx)^2} \\
&\geq&  \frac{\int_{\xi}^{t\sigma \alpha} \frac{x}{\tau}h(x)p(x)e^{\frac{\alpha x}{ \tau}-\frac{x^2}{2\tau^2}}dx \cdot \int_0^{\xi}h(x)p(x)e^{\frac{\alpha x}{ \tau}-\frac{x^2}{2\tau^2}}dx}{(\int_0^{\xi} h(x)p(x) e^{\frac{\alpha x}{ \tau}-\frac{x^2}{2\tau^2}} dx)^2} \nonumber \\
&&- \frac{\int_{\xi}^{t\sigma \alpha}h(x)p(x)e^{\frac{\alpha x}{ \tau}-\frac{x^2}{2\tau^2}}dx \cdot \int_0^{\xi}\frac{x}{\tau}h(x)p(x)e^{\frac{\alpha x}{ \tau}-\frac{x^2}{2\tau^2}}dx }{(\int_0^{\xi} h(x)p(x) e^{\frac{\alpha x}{ \tau}-\frac{x^2}{2\tau^2}} dx)^2}
 \geq 0.
\end{eqnarray*}
Hence $g_{\xi, \tau}(\alpha)$ is an increasing function of $\alpha$, for any fixed $\xi, \tau>0$. Now we consider a small neighbor around $\sigma$ : $\mathcal{I}_{\Delta}=[\sigma-\Delta, \sigma+\Delta]$, where $\Delta>0$ is small enough so that $\sigma -\Delta>0$.  We would like to show there exists a positive constant $\xi_0$ s.t. the following holds,
\begin{eqnarray}\label{trick:three}
g_{\xi_0,\tau}(1)>1, \quad \forall \tau \in \mathcal{I}_{\Delta}.
\end{eqnarray}
To show the above, we first notice that $\forall \tau \in \mathcal{I}_{\Delta}$
\begin{eqnarray}
&& \int_{\xi}^{t\sigma}h(x)p(x)e^{\frac{x}{\tau}-\frac{x^2}{2\tau^2}}dx \geq \int_{\xi}^{t\sigma}h(x)p(x)e^{\frac{x}{\sigma +\Delta}-\frac{x^2}{2(\sigma-\Delta)^2}}dx \label{trick:one} \\
&& \int_0^{\xi} h(x)p(x)e^{\frac{x}{\tau}-\frac{x^2}{2\tau^2}}dx \leq  \int_0^{\xi} h(x)p(x)e^{\frac{x}{\sigma-\Delta}-\frac{x^2}{2(\sigma+\Delta)^2}}dx. \label{trick:two}
\end{eqnarray}
Moreover, we can easily pick a small constant $\xi_0>0$ to satisfy $\int_{\xi_0}^{t\sigma}h(x)p(x)$ $e^{\frac{x}{\sigma +\Delta}-\frac{x^2}{2(\sigma-\Delta)^2}}dx>  \int_0^{\xi_0} h(x)p(x)e^{\frac{x}{\sigma-\Delta}-\frac{x^2}{2(\sigma+\Delta)^2}}dx$, which together with \eqref{trick:one}\eqref{trick:two} proves \eqref{trick:three}. Since $g_{\xi_0,\tau}(\alpha)$ is monotonically increasing, we have
\begin{eqnarray*}
g_{\xi_0,\tau}(\alpha_*) >1 , \forall \tau \in \mathcal{I}_{\Delta}.
\end{eqnarray*}
Since $\tau^* \rightarrow \sigma$, we know when $\epsilon$ is small enough, $\tau^* \in \mathcal{I}_{\Delta}$. This implies that $g_{\xi_0, \tau^*}(\alpha_*)>1$. It is straightforward to use this inequality to derive the result for $\alpha_*$.
\end{proof}

The next Lemma obtains a simple equation between $\alpha_*$ and $\tau_*$. This equation will be later used to obtain an accurate characterization of $\alpha_*$. 

\begin{lemma}\label{lemma:sparse:l1:alpha1}
    Assume there exists a constant $c > 0$ such that the tail of $G$ satisfies
    $\frac{\mathbb{P}(|G| \geq a)} {\mathbb{P}(|G| \geq b)} \leq e^{-c(a^2 -
    b^2)}$ for $a > b > 0$. Then there exists a $0<t_*<1$ such that for any
    given $1>t>t_*$, the following holds
    \begin{equation}\label{general:one}
        \lim_{\epsilon \rightarrow 0}\epsilon \mathbb{E}\Bigg[
        \frac{(\alpha_*)^2|G|}{\alpha_*\tau_*-|G|} {\rm exp}\Big(\frac{\alpha_*
        |G|}{\tau_*}-\frac{G^2}{2\tau_*^2}  \Big) \cdot \mathbbm{1}(|G| \leq  t\alpha_*
        \tau_*) \Bigg] = 2.
    \end{equation}
\end{lemma}

\begin{proof}[Proof of Lemma \ref{lemma:sparse:l1:alpha1}]
Since, $\alpha_*$ minimizes $\mathbb{E} (\eta_1 ({B}+ \tau_* Z; \alpha \tau_*)
-{B})^2$, we have $\frac{\partial R_1(\alpha, \tau_*)}{\partial
\alpha}\Big|_{\alpha=\alpha_*}=0$. By setting
\eqref{eq:L1-risk-deri-expand} as 0, we obtain that
\begin{align} \label{gen:optima}
    &2(1-\epsilon)\underbrace{(\phi(\alpha_*)-\alpha_*\Phi(-\alpha_*))}_{:=G_1}+\epsilon
    \underbrace{\mathbb{E}
    \Big[\phi(\alpha_*+\frac{|G|}{\tau_*})-\alpha_*\Phi(-\frac{|G|}{\tau_*}-\alpha_*)\Big
]}_{:=G_3} \nonumber \\
=&
\epsilon \underbrace{\mathbb{E} \Big[\alpha_*\Phi(\frac{|G|}{\tau_*}-\alpha_*)-\phi(\alpha_*-\frac{|G|}{\tau_*}) \Big ]}_{:=G_2}
\end{align}

Since $\alpha_* \rightarrow \infty$, according to \eqref{gaussiantail:exp}, $G_1 \sim \frac{\phi(\alpha_*)}{(\alpha_*)^2}$ and 
\begin{equation*}
\mathbb{E}\Big[\frac{|G|}{\alpha_* \tau_* +|G|}\phi(\alpha_*+|G|/\tau_*)\Big]\leq G_3 \leq \mathbb{E}\Big[\frac{|G|\phi(\alpha_*+|G|/\tau_*)}{\alpha_* \tau_* +|G|} + \frac{\alpha_*\phi(\alpha_*+|G|/\tau_*)}{(\alpha_*+|G|/\tau_*)^3} \Big].
\end{equation*}

Hence,
\begin{equation} \label{gen:eq:one}
\Big| \frac{G_3}{G_1} \Big| \lesssim \mathbb{E}\bigg[\frac{\alpha_*^2 |G|}{\alpha_* \tau_* +|G|}{\rm exp}\Big(-{\frac{\alpha_*|G|}{\tau_*}} -\frac{|G|^2}{2\tau_*^2}\Big) \bigg]
+ \mathbb{E} \bigg[ \frac{\alpha_*^3}{(\alpha_*+|G|/\tau_*)^3} \exp \Big(-{\frac{\alpha_*|G|}{\tau_*}} -\frac{|G|^2}{2\tau_*^2}\Big) \bigg].
\end{equation}

Moreover, it is straightforward to see that
\begin{equation*}
\frac{(\alpha_*)^2 |G|}{\alpha_* \tau_* +|G|}{\rm exp}\Big(-{\frac{\alpha_*|G|}{\tau_*}} -\frac{|G|^2}{2\tau_*^2}\Big) \leq \frac{\alpha_* |G|}{\tau_*}{\rm exp}\Big(-{\frac{\alpha_*|G|}{\tau_*}} \Big)\leq e^{-1}
\end{equation*}

We can then apply DCT to conclude the first term on the right hand side of \eqref{gen:eq:one} goes to zero. Similar arguments work for the second term. Hence $\frac{G_3}{G_1} \rightarrow 0$, as $\epsilon \rightarrow 0$. Hence, according to  \eqref{gen:optima} 
\begin{eqnarray}\label{final:one}
\lim_{\epsilon \rightarrow 0} \frac{\epsilon \alpha_*^2 G_2}{\phi(\alpha_*)}=2.
\end{eqnarray}

Next we would like to simplify $G_2$. The idea is to approximate $\Phi(|G|/\tau_*-\alpha_*)$ by $\frac{1}{\alpha_*-|G|/\tau_*}\phi(\alpha_*-|G|/\tau_*)$, but since $|G|$ is not necessarily bounded, the approximation may not be accurate.  Therefore, we first consider an approximation to a truncated version of $G_3$. More specifically, given a constant $0<t<1$, we focus on
\begin{equation*}
T \triangleq 
\mathbb{E} \Big\{\big[\alpha_*\Phi(|G| / \tau_* - \alpha_*) - \phi(\alpha_* -
|G| / \tau_*) \big ] \cdot \mathbbm{1}(|G|\leq t \tau_* \alpha_*) \Big\}.
\end{equation*}

It is straightforward to confirm that 
\begin{eqnarray} \label{gen:eq:two}
\lefteqn{- \mathbb{E} \Big [\frac{\alpha_*}{(\alpha_*-|G|/\tau_*)^3} \phi(\alpha_*-|G|/\tau_*) \cdot \mathbbm{1}(|G|\leq t \tau_* \alpha_*)  \Big ] } \nonumber \\
&\leq& T - \mathbb{E} \Big [\frac{|G| }{\alpha_* \tau_* -|G|}\phi(\alpha_*-|G|/\tau_*) \cdot \mathbbm{1}(|G|\leq t \tau_* \alpha_*) \Big ] \leq 0.
\end{eqnarray}

To show $T$ has the same order as the gaussian tail approximation, we analyze the following ratio:
\begin{align}\label{gen:eq:three}
    \frac{\mathbb{E} \big[ \frac{\alpha_*}{(\alpha_*-|G|/\tau_*)^3} \phi(\alpha_*-|G|/\tau_*) \cdot \mathbbm{1}(|G|\leq t \tau_* \alpha_*) \big]}{\mathbb{E} \big[\frac{|G| }{\alpha_* \tau_* -|G|}\phi(\alpha_*-|G|/\tau_*) \cdot \mathbbm{1}(|G|\leq t \tau_* \alpha_*) \big] }
    \leq&
    \frac{\frac{\alpha_*}{(\alpha_*-t \alpha_*)^3} \mathbb{E} \big[ \phi(\alpha_*-|G|/\tau_*) \cdot \mathbbm{1}(|G|\leq t \tau_* \alpha_*) \big]}{\frac{1}{\alpha_* \tau_*} \mathbb{E} \big[ |G|\phi(\alpha_*-|G|/\tau_*) \cdot \mathbbm{1}(|G|\leq t \tau_* \alpha_*) \big] } \nonumber \\
    \leq&
    \frac{ \tau_*}{\alpha_* (1-t)^3} \frac{ \mathbb{E} \big[ \phi(\alpha_*-|G|/\tau_*) \cdot \mathbbm{1}(|G|\leq t \tau_* \alpha_*)  \big] }{ \mathbb{E} \big[ |G|\phi(\alpha_*-|G|/\tau_*) \cdot \mathbbm{1}(|G|\leq t \tau_* \alpha_*) \big] } .
\end{align}

According to Lemma \ref{appendix:noisy:l1:tail} and the fact $\tau_* >\sigma$, there exists $\xi>0$ such that
\begin{align*}
    \mathbb{E} \big[ \phi(\alpha_*-|G|/\tau_*) \cdot \mathbbm{1}(|G|\leq t \tau_* \alpha_*) \big]
    >&
    \mathbb{E} \big[ \phi(\alpha_*-|G|/\tau_*) \cdot \mathbbm{1}(\xi \leq
    |G|\leq t \sigma \alpha_*) \big] \nonumber \\
    >&
    \frac{1}{2} \mathbb{E} \big[ \phi(\alpha_*-|G|/\tau_*) \cdot \mathbbm{1}(|G| \leq t \tau_* \alpha_*) \big] .
\end{align*}

Hence, 
\begin{align*}
    &\mathbb{E} \big[ |G|\phi(\alpha_*-|G|/\tau_*) \cdot \mathbbm{1}(|G|\leq t \tau_* \alpha_*) \big]
    \geq \mathbb{E} \big[ |G|\phi(\alpha_*-|G|/\tau_*) \cdot \mathbbm{1}(\xi
    \leq |G|\leq t \tau_* \alpha_*) \big] \nonumber \\
    \geq&
    \xi \mathbb{E} \big[\phi(\alpha_*-|G|/\tau_*) \cdot \mathbbm{1}(\xi \leq |G|\leq t \tau_* \alpha_*) \big]
    >
    \frac{\xi}{2}\mathbb{E} \big[ \phi(\alpha_*-|G|/\tau_*) \cdot \mathbbm{1}(|G| \leq t \tau_* \alpha_*) \big] .
\end{align*}

The above inequality together with \eqref{gen:eq:three} yields,
\begin{equation*}
    \frac{\mathbb{E} \Big[\frac{\alpha_*}{(\alpha_*-|G|/\tau_*)^3} \phi(\alpha_*-|G|/\tau_*) \cdot \mathbbm{1}(|G|\leq t \tau_* \alpha_*) \Big]}
    {\mathbb{E} \Big [\frac{|G| }{\alpha_* \tau_* -|G|}\phi(\alpha_*-|G|/\tau_*) \cdot \mathbbm{1}(|G|\leq t \tau_* \alpha_*) \Big ] }
    \rightarrow 0,
    \quad
    \text{as } \epsilon \rightarrow 0.
\end{equation*}

This combined with \eqref{gen:eq:two} gives us,
\begin{equation*}
    T \sim  \mathbb{E} \Big[\frac{|G| }{\alpha_* \tau_* -|G|}\phi(\alpha_*-|G|/\tau_*) \cdot \mathbbm{1}(|G|\leq t \tau_* \alpha_*) \Big].
\end{equation*}

Now we turn to analyzing the term 
\begin{equation*}
    G_2 - T
    =
    \mathbb{E} \Big\{\big[\alpha_*\Phi(|G|/\tau_* - \alpha_*) - \phi(\alpha_*
    - |G| / \tau_*) \big] \cdot \mathbbm{1}(|G|> t \tau_* \alpha_*) \Big\}.
\end{equation*}

We aim to show $G_2 - T$ has smaller order than $T$. Equivalently, we would like to prove
\begin{equation}\label{gen:eq:four}
    \frac{\mathbb{E} \Big\{\big[\alpha_*\Phi( |G| / \tau_* - \alpha_*) -
    \phi(\alpha_* - |G| / \tau_*) \big] \cdot \mathbbm{1}(|G|> t \tau_* \alpha_*) \Big\}}
    { \mathbb{E} \Big [\frac{|G| }{\alpha_* \tau_* -|G|}\phi(\alpha_*-|G|/\tau_*) \cdot \mathbbm{1}(|G|\leq t \tau_* \alpha_*) \Big ]}=o(1).
\end{equation}

First note that
\begin{equation*}
    \mathbb{E} \Big\{\big[ \alpha_*\Phi( |G| / \tau_* - \alpha_*) -
    \phi(\alpha_* - |G| / \tau_*) \big] \cdot \mathbbm{1}(|G| > t \tau_* \alpha_*) \Big\}
    =
    O(\alpha_* P(|G| > t \tau_* \alpha_*))
\end{equation*}

Furthermore, for any $0<\tilde{t}<t$, we have
\begin{align*}
    & \mathbb{E} \Big [\frac{|G| }{\alpha_* \tau_* -|G|}\phi(\alpha_*-|G|/\tau_*) \cdot \mathbbm{1}(|G|\leq t \tau_* \alpha_*) \Big ] \nonumber \\
    &>
    \mathbb{E} \Big[ \frac{|G| }{\alpha_* \tau_* -|G|}\phi(\alpha_*-|G|/\tau_*)
    \cdot \mathbbm{1}(\tilde{t} \tau_* \alpha_* \leq |G|\leq t \tau_* \alpha_*)
    \Big] \nonumber \\
    \geq &
    \mathbb{E} \Big[ \frac{\tilde{t}\tau_* \alpha_*}{\tau_* \alpha_*-\tilde{t}\tau_* \alpha_*}\phi(\alpha_*-\tilde{t}\tau_* \alpha_*/\tau_*)\cdot \mathbbm{1}(\tilde{t} \tau_* \alpha_* \leq |G|\leq t \tau_* \alpha_*) \Big] \nonumber \\
    =&
    \frac{\tilde{t}}{1-\tilde{t}}\phi((1-\tilde{t})\alpha_*) P(\tilde{t} \tau_* \alpha_* \leq |G|\leq t \tau_* \alpha_*).
\end{align*}

So \eqref{gen:eq:four} would hold if we can show 
\begin{equation}\label{eq:intermedeqforalpha}
\frac{\alpha_*P(|G|>t\tau_* \alpha_*)}{\phi((1-\tilde{t})\alpha_*) P(\tilde{t} \tau_* \alpha_* \leq |G|\leq t \tau_* \alpha_*)}=o(1)
\end{equation}

Based on the condition we impose on the tail probability of $G$ in the statement of Lemma \ref{lemma:sparse:l1:alpha1}, \eqref{eq:intermedeqforalpha} is equivalent to
\begin{equation*}
\frac{\alpha_*P(|G|>t\tau_* \alpha_*)}{\phi((1-\tilde{t})\alpha_*) P( |G|> \tilde{t} \tau_* \alpha_*)}=o(1).
\end{equation*}

Using the tail probability condition again, we obtain
\begin{equation}\label{gen:eq:five}
\frac{\alpha_*P(|G|>t\tau_* \alpha_*)}{\phi((1-\tilde{t})\alpha_*) P( |G|> \tilde{t} \tau_* \alpha_*)}
=
O(\alpha_* e^{\frac{(1-\tilde{t})^2(\alpha_*)^2}{2}} \cdot e^{-c(t^2-\tilde{t}^2)(\tau_*\alpha_*)^2}).
\end{equation}

Also note that 
\begin{equation*}
    c(t^2-\tilde{t}^2)\sigma^2-\frac{(1-\tilde{t})^2}{2}=c(t^2-1)\sigma^2+(1-\tilde{t})
    \Big[  c(1+\tilde{t})\sigma^2-\frac{1}{2}(1-\tilde{t}) \Big].
\end{equation*}

Hence we can choose $t$ and $\tilde{t}$ close to $1$ so that $c(t^2-\tilde{t}^2)\sigma^2 - \frac{(1-\tilde{t})^2}{2}>0$ (Set $t=1$ and $\tilde{t}$ close enough to 1, the expression is negative. Conclusion follows by the continuity in $t$). Since $\tau_* \rightarrow \sigma$, when $\alpha_*$ is large enough, $c(t^2-\tilde{t}^2)(\tau_*)^2 - \frac{(1-\tilde{t})^2}{2}$ is bounded below away from zero. This implies the term on the right hand side of \eqref{gen:eq:five} is $o(1)$. Putting the preceding results we derived so far, we have showed that $G_2 \sim \mathbb{E} \Big [\frac{|G| }{\alpha_* \tau_* -|G|}\phi(\alpha_*-|G|/\tau_*) \cdot \mathbbm{1}(|G|\leq t \tau_* \alpha_*) \Big ]$. This result together with \eqref{final:one} completes the proof.
\end{proof}

Equation \eqref{general:one} can potentially enable us to obtain accurate information about $\alpha_*$. The only remaining difficulty is the existence of $\tau_*$ in this equation that can depend on $\epsilon$. Our next lemma proves that \eqref{general:one} still
holds, even if we replace $\tau_*$ with $\sigma$. Hence, we obtain a simple equation for $\alpha_*$. 

\begin{lemma}\label{lemma:sparse:l1:alpha2}
Under the conditions of Lemma \ref{lemma:sparse:l1:alpha1}, we have
\begin{eqnarray}\label{general:two}
\lim_{\epsilon \rightarrow 0}\epsilon \mathbb{E}\bigg[ \frac{\alpha_*^2|G|}{\alpha_*\sigma -|G|} {\rm exp}\Big(\frac{\alpha_* |G|}{\sigma}-\frac{G^2}{2\sigma^2}  \Big) \cdot \mathbbm{1}(|G| \leq  t\alpha_* \sigma) \bigg] = 2.
\end{eqnarray}
\end{lemma}

\begin{proof}
    Firstly, it is not hard to confirm that the proof of Lemma
    \ref{lemma:sparse:l1:alpha1} works through if we consider the truncation
    $\mathbbm{1}(|G|\leq t \alpha_* \sigma)$, which leads to the following
    result
    \begin{equation}\label{sigma:one}
        \lim_{\epsilon \rightarrow 0}\epsilon \mathbb{E}\bigg[ \frac{(\alpha_*)^2|G|}{\alpha_*\tau_* -|G|} {\rm exp}\Big(\frac{\alpha_* |G|}{\tau_*}-\frac{G^2}{2\tau_*^2}  \Big) \cdot \mathbbm{1}(|G| \leq  t\alpha_* \sigma) \bigg] = 2.
    \end{equation}

Denoting $h(z) = \mathbb{E}\Big[ \frac{\alpha_*^2|G|}{\alpha_*z -|G|} {\rm exp}\Big(\frac{\alpha_* |G|}{z}-\frac{G^2}{2z^2}  \Big) \cdot \mathbbm{1}(|G| \leq  t\alpha_* \sigma) \Big]$, we have
\begin{equation}\label{sigma:zero}
    \epsilon h(\tau_*)-\epsilon h(\sigma)
    =
    \epsilon h'(\tilde{\tau})(\tau_*-\sigma),
\end{equation}
where $\tilde{\tau}$ is between $\tau_*$ and $\sigma$. We then calculate the derivative 
\begin{equation*}
    h'(z)=\mathbb{E}\bigg[\Big( -\frac{\alpha_*^3|G|}{(\alpha_*z - |G|)^2} -
    \frac{\alpha_*^2|G|^2}{z^3} \Big) {\rm exp}\Big(\frac{\alpha_* |G|}{z} -
    \frac{G^2}{2z^2}  \Big) \cdot \mathbbm{1}(|G| \leq t\alpha_* \sigma) \bigg]
\end{equation*}

Hence we have
\begin{align}\label{sigma:four}
|h'(\tilde{\tau})|
\leq&
\bigg( \frac{1}{\sigma^2(1-t)^2} +\frac{t\alpha_*^2}{\sigma^2} \bigg) \mathbb{E}\Big[ \alpha_*|G| {\rm exp}\Big(\frac{\alpha_* |G|}{\tilde{\tau}}-\frac{G^2}{2\tilde{\tau}^2}  \Big) \cdot \mathbbm{1}(|G| \leq  t\alpha_* \sigma) \Big] \nonumber \\
\leq &
\bigg( \frac{1}{\sigma^2(1-t)^2} +\frac{t\alpha_*^2}{\sigma^2} \bigg) \mathbb{E}\Big[ \alpha_*|G| {\rm exp}\Big(\frac{\alpha_* |G|}{\sigma}-\frac{G^2}{2\tau_*^2}  \Big) \cdot \mathbbm{1}(|G| \leq  t\alpha_* \sigma) \Big]
\end{align}

We have used $\tau_*\geq\tilde{\tau} \geq \sigma$ in the above derivation.
Moreover, according to \eqref{sigma:one}, it is easily seen that
\begin{align} \label{sigma:two}
    \Theta(1)
    =&
    \epsilon \mathbb{E}\Big[ \alpha_* |G| {\rm exp}\Big(\frac{\alpha_* |G|}{\tau_*}-\frac{G^2}{2\tau_*^2}  \Big) \cdot \mathbbm{1}(|G| \leq  t\alpha_* \sigma) \Big] \nonumber \\
    \geq&
    \epsilon  \mathbb{E}\Big[ \alpha_* |G| {\rm exp}\Big(\frac{\alpha_* |G|}{\tau_*}-\frac{G^2}{2\tau_*^2}  \Big) \cdot \mathbbm{1}(\xi \leq |G| \leq  t\alpha_* \sigma) \Big] \nonumber \\
    \geq&
    \epsilon \alpha_* \xi e^{\frac{\alpha_*\xi}{\tau_*}-\frac{\xi^2}{2\tau_*^2}}P(\xi \leq |G| \leq t\alpha_* \sigma)=\Theta(\epsilon \alpha_* e^{\frac{\alpha_* \xi}{\tau_*}}),
\end{align}
where $\xi$ is a small constant such that $\mathbb{P}(\xi \leq |G|) > 0$. This
tells us that $\epsilon\alpha_*^2\rightarrow 0$. Thus we have
\begin{align*}
0\leq&
\mathbb{E}\Big[ \alpha_*|G| {\rm exp}\Big(\frac{\alpha_* |G|}{\sigma}-\frac{G^2}{2\tau_*^2}  \Big) \cdot \mathbbm{1}(|G| \leq  t\alpha_* \sigma) \Big]
-
\mathbb{E}\Big[ \alpha_*|G| {\rm exp}\Big(\frac{\alpha_* |G|}{\tau_*}-\frac{G^2}{2\tau_*^2}  \Big) \cdot \mathbbm{1}(|G| \leq  t\alpha_* \sigma) \Big]\\
=&
\epsilon \mathbb{E}\Big[ \alpha_* |G| {\rm exp}\Big(\frac{\alpha_* |G|}{\tau_*}-\frac{G^2}{2\tau_*^2}  \Big)\Big({\rm exp}\Big(\frac{\alpha_* |G|}{\sigma\tau_*}(\tau_*-\sigma)\Big) -1 \Big) \cdot \mathbbm{1}(|G| \leq  t\alpha_* \sigma) \Big]\\
\leq&
\epsilon \mathbb{E}\Big[ \alpha_* |G| {\rm exp}\Big(\frac{\alpha_* |G|}{\tau_*}-\frac{G^2}{2\tau_*^2}  \Big) \cdot \mathbbm{1}(|G| \leq  t\alpha_* \sigma) \Big]\Big({\rm exp}\Big(\frac{t\alpha_*^2}{\tau_*}(\tau_*-\sigma)\Big) -1 \Big)\\
\approx&
2\Big({\rm exp}\Big(\frac{t\alpha_*^2}{\tau_*}(\tau_*-\sigma)\Big) -1 \Big)
\end{align*}
Note that according to \eqref{eq:se11} we have
\begin{equation}\label{eq:orderalphataustar1}
\frac{\delta(\tau_*^2-\sigma^2)}{\tau_*^2}  = \frac{1}{\tau_*^2} \mathbb{E} (\eta_1 (B+ \tau_* Z; \alpha_* \tau_*) -B)^2.  
\end{equation}
By canceling terms in \eqref{eq:L1-risk-expand2} using
\eqref{eq:L1-risk-deri-expand}=0, we also have
\begin{align}
\frac{1}{\tau_*^2} \mathbb{E} (\eta_1 (B&+ \tau_* Z; \alpha_* \tau_*) -B)^2
=
2(1-\epsilon)\Phi(-\alpha_*)
+ \epsilon\mathbb{E}_G\Bigg[ (1-\frac{G^2}{\tau_*^2})\Phi(\frac{G}{\tau_*}-\alpha_*) + \nonumber \\ 
& (1-\frac{G^2}{\tau_*^2})\Phi(-\frac{G}{\tau_*}-\alpha_*) - \frac{G}{\tau_*}\phi(\alpha_*-\frac{G}{\tau_*}) + \frac{G}{\tau_*}\phi(\alpha_*+\frac{G}{\tau_*}) + \frac{G^2}{\tau_*^2} \Bigg]  \leq \frac{2\phi(\alpha_*)}{\alpha_*} + \epsilon O(1). \label{eq:orderalphataustar2}
\end{align}

By combining \eqref{eq:orderalphataustar1} and \eqref{eq:orderalphataustar2}, we have
\begin{equation*}
2\alpha_*^2(\tau_*-\sigma)
\leq
2\phi(\alpha_*)\alpha_*+\epsilon\alpha_*^2O(1)
\rightarrow 0
\end{equation*}

This shows
\begin{equation}\label{sigma:three}
    \epsilon \mathbb{E}\Big[ \alpha_*|G| {\rm exp}\Big(\frac{\alpha_*
    |G|}{\sigma}-\frac{G^2}{2\tau_*^2}  \Big) \cdot \mathbbm{1}(|G| \leq
    t\alpha_* \sigma) \Big]=\Theta(1).
\end{equation}

By combining \eqref{sigma:zero},\eqref{sigma:four}, \eqref{sigma:two},
\eqref{sigma:three} and \eqref{sigma:one} we have $\epsilon h(\tau_*)-\epsilon h(\sigma) =o(1)$, which completes the proof.
\end{proof}

Based on Lemma \ref{lemma:sparse:l1:alpha2}, we can build the following
explicit convergence of $\alpha$ w.r.t. $\epsilon$.
\begin{lemma}\label{lemma:sparse:l1:alpha3}
    Assume $\esssup(G)=M<\infty$, then we have
    $\frac{\alpha_*}{\log\frac{1}{\epsilon}}\rightarrow\frac{\sigma}{M}$.
\end{lemma}
\begin{proof}
    Obviously the condition of Lemma \ref{lemma:sparse:l1:alpha2} is satisfied
    when $G$ is bounded. It is then easy to see that \eqref{general:two} becomes
\begin{equation*}
\lim_{\epsilon \rightarrow 0}\frac{\epsilon\alpha_*}{\sigma} \mathbb{E}\bigg[|G|{\rm exp}\Big(\frac{\alpha_* |G|}{\sigma}-\frac{G^2}{2\sigma^2}  \Big) \bigg]
=
2
\end{equation*}

By Lemma \ref{lemma:sparse:l1:laplace}, we have
\begin{equation*}
\lim_{\epsilon\rightarrow0}\frac{\epsilon\alpha_*}{\sigma}Me^{-\frac{M^2}{2\sigma^2}}\mathbb{E}e^{\frac{\alpha_* G}{\sigma}}=2
\end{equation*}
Some simple algebra proves $\frac{\log\epsilon}{\alpha_*}+\frac{M}{\sigma}\rightarrow 0$.
\end{proof}

Now we are ready to establish the first part of Theorem \ref{THM:SPARSE:MAIN}. For $G$
bounded as in Lemma \ref{lemma:sparse:l1:alpha3}, by the first order condition, we have
\begin{align*}
    &\lim_{\epsilon\rightarrow0}\frac{\alpha_*^3}{\phi(\alpha_*)}\big(R -
    \mathbb{E}\beta^2 / \tau_*^2 \big) =
    \lim_{\epsilon\rightarrow0} 4(1-\epsilon) \big(1 + O(\alpha_*^{-2}) \big) \nonumber \\
    &+ 2 \epsilon\alpha_*\mathbb{E} \bigg[ - \alpha_*^2e^{\frac{\alpha_* G}{\tau_*}-\frac{G^2}{2\tau_*^2}}
    \frac{G / \tau_* + O(\alpha_*^{-1})} {(\alpha_* - G / \tau_*)^2}+ \alpha_*^2e^{-\frac{\alpha_* G}{\tau_*} - \frac{G^2}{2 \tau_*^2}}
    \frac{G / \tau_* + O(\alpha_*^{-1})}{(\alpha_* + G / \tau_*)^2} \bigg]
    \nonumber \\
    =&
    \lim_{\epsilon\rightarrow0}
    4 - 2 \epsilon \alpha_* \mathbb{E}\Big(\frac{G}{\tau_*}e^{\frac{\alpha_*
    G}{\tau_*} - \frac{G^2}{2\tau_*^2}}\Big)
    =
    0
\end{align*}

Using the divergent speed of $\alpha_*$ derived in Lemma \ref{lemma:sparse:l1:alpha3},
we can obtain that
\begin{align*}
\amse_1^*
=
\mathbb{E}\beta^2 + \alpha_*^{-3} \phi(\alpha_*) o(1)
=
\mathbb{E}\beta^2 +
o\big(\epsilon^{\frac{\sigma^2}{2M^2}\log\frac{1}{\epsilon}(1+o(1))}\big)
=
\mathbb{E}\beta^2+o(\epsilon^k),
\quad
\forall k \in \mathbb{N}.
\end{align*}

\subsection{$q > 1$}
Again recall \eqref{eq:se11} which defines $(\alpha_*, \tau_*)$. Similar to the
proof of Lemma \ref{lemma:l1:alpha-tau-trend} we can show
$\alpha_*\rightarrow\infty$ and $\tau_*\rightarrow\sigma$. The next step is to
get more accurate info about $\alpha_*$.

\begin{lemma}\label{lemma:sparse:lq:alpha2}
If all the moments of $B$ are bounded, then
\begin{equation*}
\lim_{\epsilon\rightarrow0}
\frac{\alpha_*}{\epsilon^{1-q}}
=
\frac{1}{q}\Bigg( \frac{\sigma\mathbb{E}|Z|^{\frac{2}{q-1}}}
{\mathbb{E}\big[ | G / \sigma + Z|^{\frac{1}{q-1}} G
{\rm sign}(G / \sigma + Z)\big]}\Bigg)^{q-1}.
\end{equation*}
\end{lemma}

\begin{proof}
    From $u-\eta_q(u,\alpha)=\alpha q \sgn(u)|\eta_q(u;\alpha)|^{q-1}$ and
    $\eta_q\rightarrow0$ as $\alpha\rightarrow\infty$, we have
    $\lim_{\alpha\rightarrow\infty}\alpha|\eta_q(u;\alpha)|^{q-1}=\frac{|u|}{q}$.
    Since $\alpha_*$ is optimal, the first order optimality condition yields
\small
\begin{align*}
&\epsilon\mathbb{E}\bigg[\frac{q|\eta_q(\frac{G}{\tau_*}+Z;\alpha_*)|^q}{1+\alpha_* q(q-1)|\eta_q(\frac{G}{\tau_*}+Z;\alpha_*)|^{q-2}}\bigg]
+
(1-\epsilon)\mathbb{E}\bigg[\frac{q|\eta_q(Z;\alpha_*)|^q}{1+\alpha_* q(q-1)|\eta_q(Z;\alpha_*)|^{q-2}}\bigg] \nonumber \\
&=
\epsilon\mathbb{E}\bigg[\frac{G}{\tau_*}\frac{q|\eta_q(\frac{G}{\tau_*}+Z;\alpha_*)|^{q-1}\sgn(\frac{G}{\tau_*}+Z)}{1+\alpha_* q(q-1)|\eta_q(\frac{G}{\tau_*}+Z;\alpha_*)|^{q-2}}\bigg].
\end{align*}
\normalsize

Denote the three parts inside $\mathbb{E}(\cdot)$ by $T_1,T_2,T_3$ respectively. Multiply each side by $\alpha_*^{\frac{q+1}{q-1}}$, and note that
\begin{align*}
0\leq&
\alpha_*^{\frac{q+1}{q-1}}T_1
=
\frac{q|\alpha_*^{\frac{1}{q-1}}\eta_q(G / \tau_* + Z;
\alpha_*)|^q}{\alpha_*^{-\frac{1}{q-1}} + q(q-1)
|\alpha_*^{\frac{1}{q-1}}\eta_q(G / \tau_* + Z; \alpha_*)|^{q-2}} \nonumber \\
\leq &
\frac{1}{q-1}|\alpha_*^{\frac{1}{q-1}}\eta_q(G / \tau_* + Z;\alpha_*)|^2
=
\frac{1}{q-1}\bigg|\frac{G / \tau_* + Z-\eta_q(G / \tau_* + Z;\alpha_*)}{q}\bigg|^{\frac{2}{q-1}} \nonumber \\
\leq &
\frac{1}{q-1}\bigg|\frac{G / \tau_* + Z}{q}\bigg|^{\frac{2}{q-1}}
\leq
\frac{1}{q-1}\bigg[\frac{|G| / \sigma + |Z|}{q}\bigg]^{\frac{2}{q-1}}
<\infty.
\end{align*}

The last step holds if we assume finite moments of all orders for $G$. Similar inequalities hold for $T_2,T_3$. Thus by DCT, we have
\begin{equation*}
\small\lim_{\epsilon\rightarrow0}\epsilon\alpha_*^{\frac{q+1}{q-1}}\mathbb{E}T_1=0,
\quad
\lim_{\epsilon\rightarrow0}\alpha_*^{\frac{q+1}{q-1}}\mathbb{E}T_2
=\frac{\mathbb{E}|Z|^{\frac{2}{q-1}}}{(q-1)q^{\frac{2}{q-1}}},
\quad
\lim_{\epsilon\rightarrow0}\alpha_*^{\frac{q}{q-1}}\mathbb{E}T_3
=\frac{\mathbb{E}\big[G \big|\frac{G}{\sigma} + Z
        \big|^{\frac{1}{q-1}}\sgn\big(\frac{G}{\tau_*} + Z\big)\big]}{\sigma (q-1)q^{\frac{1}{q-1}}},
\end{equation*}
and
\begin{equation*}
\lim_{\epsilon\rightarrow0}\epsilon\alpha_*^{\frac{1}{q-1}}
=
\frac{1}{q^{\frac{1}{q-1}}}\frac{\mathbb{E}|Z|^{\frac{2}{q-1}}}{\mathbb{E}\Big[\frac{G}{\sigma}\left|\frac{G}{\sigma}+Z\right|^{\frac{1}{q-1}}\sgn(\frac{G}{\tau_*}+Z)\Big]}.
\end{equation*}
\end{proof}

Now we prove the second part of Theorem \ref{THM:SPARSE:MAIN}. From
\eqref{eq:se11} we have
\begin{equation*}
    \mathbb{E}[\eta_q(B / \tau_* + Z;\alpha_*) - B/ \tau_* ]^2
    =
    \mathbb{E}\eta_q^2(B / \tau_* + Z;\alpha_*) - \frac{2}{\tau_*}
    \mathbb{E}[B\eta_q(B / \tau_* + Z;\alpha_*)]
    + \frac{\mathbb{E}B^2}{\tau_*}.
\end{equation*}
Thus we have
\begin{equation*}
  \frac{\amse(q, \lambda_q^*)}{\tau_*^2}-\epsilon\frac{\mathbb{E}G^2}{\tau_*^2}
= \epsilon\mathbb{E}\eta_q^2(G / \tau_* + Z;\alpha_*) +
(1-\epsilon)\mathbb{E}\eta_q^2(Z;\alpha_*) - \frac{2 \epsilon}{\tau_*}
\mathbb{E} [G \eta_q(G/\tau_* + Z;\alpha_*)].
\end{equation*}

Note that by DCT (similar argument as the ones mentioned in Lemma \ref{lemma:sparse:lq:alpha2}), we have
\begin{equation*}
   \lim_{\epsilon \rightarrow 0} \alpha_*^{\frac{2}{q-1}}
   \mathbb{E}\eta_q^2\Big(\frac{G}{\tau_*} + Z; \alpha_*\Big)
    =
    \lim_{\epsilon \rightarrow 0} q^{ - \frac{2}{q-1}}
    \mathbb{E}\Big[\Big|\frac{G}{\tau_*} + Z \Big| -
    \Big|\eta_q\Big(\frac{G}{\tau_*} + Z;\alpha_*\Big)\Big|\Big]^{\frac{2}{q-1}} \nonumber \\
    =
    q^{ - \frac{2}{q-1}} \mathbb{E}\Big|\frac{G}{\sigma} + Z\Big|^{\frac{2}{q-1}} \nonumber
    \end{equation*}
    Similarly,
    \begin{eqnarray*}
   \lim_{\epsilon \rightarrow 0} \alpha_*^{\frac{2}{q-1}}\mathbb{E}\eta_q^2(Z;\alpha_*)
    =
    \lim_{\epsilon \rightarrow 0}  q^{-\frac{2}{q-1}} \mathbb{E}[|Z| - |\eta_q(Z;\alpha_*)|]^{\frac{2}{q-1}}
   =
    q^{-\frac{2}{q-1}} \mathbb{E}|Z|^{\frac{2}{q-1}}, \nonumber 
    \end{eqnarray*}
    and finally,
    \begin{eqnarray*}
    \lefteqn{\alpha_*^{\frac{1}{q-1}} \mathbb{E}[G \eta_q(G / \tau_* + Z; \alpha_*)]
    =
    \alpha_*^{\frac{1}{q-1}}
    \mathbb{E}[G |\eta_q(G / \tau_* + Z;\alpha_*)| \sgn(G / \tau_* + Z)]} \nonumber \\
   & =& q^{-\frac{1}{q-1}} \mathbb{E}\big[ G \big(|G / \tau_* + Z| - |\eta_q(G / \tau_* + Z; \alpha_*)| \big)^{\frac{1}{q-1}}
    \sgn(G / \tau_* + Z ) \big] \nonumber \\
    &\rightarrow&
    q^{-\frac{1}{q-1}} \mathbb{E} \big[G |G / \sigma + Z|^{\frac{1}{q-1}} \sgn(G / \sigma + Z)].
\end{eqnarray*}

Based on these information, we have
\begin{eqnarray*}
   \lefteqn{ \lim_{\epsilon \rightarrow 0}\frac{\amse(q, \lambda_q^*) - \epsilon\mathbb{E}G^2}{\epsilon^2 \tau_*^2}} \nonumber \\
    &=& \lim_{\epsilon \rightarrow 0}
    \frac{1}{\epsilon} \mathbb{E}\eta_q^2(G / \tau_* + Z; \alpha_*)
    +
    \frac{1-\epsilon}{\epsilon^2} \mathbb{E}\eta_q^2(Z;\alpha_*)
    -
    \frac{2}{\epsilon \tau_*} \mathbb{E}[G \eta_q(G / \tau_* + Z; \alpha_*)] \nonumber \\
    &=&
    \frac{\mathbb{E}^2[ G|G / \sigma + Z |^{\frac{1}{q-1}}\sgn(G / \sigma + Z)]}{\sigma^2 \mathbb{E}|Z|^{\frac{2}{q-1}}}
    -
    2\frac{\mathbb{E}^2 [G |G / \sigma + Z|^{\frac{1}{q-1}}\sgn(G / \sigma + Z)]}{\sigma^2 \mathbb{E}|Z|^{\frac{2}{q-1}}}\\
    &=&
    -\frac{\mathbb{E}^2[G |G / \sigma + Z|^{\frac{1}{q-1}}\sgn(G / \sigma + Z)]}{\sigma^2 \mathbb{E}|Z|^{\frac{2}{q-1}}}.
\end{eqnarray*}

Hence, we have $\frac{\amse(q, \lambda_q^*) -\epsilon\mathbb{E}G^2}{\epsilon^2} =    -\frac{\mathbb{E}^2\big[G |G / \sigma + Z|^{\frac{1}{q-1}}\sgn(G / \sigma + Z) \big]}{\mathbb{E}|Z|^{\frac{2}{q-1}}}.$


\section{Proof of Theorem \ref{THM:NOISY:MAIN}}  \label{thmlargenoise:main}
\subsection{Preliminaries}
Before we start the proof we discuss a useful lemma. 
\begin{lemma}\label{appendix:noisy:l1:alpha3}
Consider a nonnegative random variable $X$ with probability distribution $\mu$ and $\mathbb{P}(X>0)=1$. Let $\xi>\zeta>0$ be the points such that $\mathbb{P}(X\leq\zeta)\leq\frac{1}{4}$ and $\mathbb{P}(\zeta<X\leq\xi)\geq\frac{1}{4}$. Let $a, b, c:\mathbb{R}_+\rightarrow\mathbb{R}_+$ be three deterministic positive functions such that $a(s), c(s)\rightarrow\infty$ as $s\rightarrow\infty$. Then there exists a positive constant $s_0$ depending on $a, c, X$, such that when $s>s_0$,
\begin{equation*}
\int_{0}^{a(s)} e^{b(s)x-\frac{x^2}{c(s)}}d\mu(x) \leq 3 \int_{\zeta}^{a(s)}e^{b(s)x-\frac{x^2}{c(s)}}d\mu(x).
\end{equation*}
\end{lemma}
\begin{proof}
For large enough $s$ such that $a(s)>\xi$,
\begin{eqnarray}
&&\int_{\zeta}^{a(s)}e^{b(s)x-\frac{x^2}{c(s)}}d\mu(x)
\geq
\int_{\zeta}^{\xi}e^{b(s)x-\frac{x^2}{c(s)}}d\mu(x) 
\geq
e^{b(s)\zeta-\frac{\xi^2}{c(s)}}\mathbb{P}(\zeta<X\leq\xi) \nonumber \\
&&\geq
e^{b(s)\zeta-\frac{\xi^2}{c(s)}}\mathbb{P}(X\leq\zeta)
\geq
e^{-\frac{\xi^2}{c(s)}}\int_0^{\zeta}e^{b(s)x-\frac{x^2}{c(s)}}d\mu(x). \nonumber
\end{eqnarray}
As a result we have the following inequality, 
\begin{equation*}
\int_0^{a(s)}e^{b(s)x-\frac{x^2}{c(s)}}d\mu(x)
\leq
(1+e^{\frac{\xi^2}{c(s)}})\int_\zeta^{a(s)}e^{b(s)x-\frac{x^2}{c(s)}} d\mu(x).
\end{equation*}
For sufficiently large $s$ such that $e^{\frac{\xi^2}{c(s)}}<2$, the conclusion follows. 
\end{proof}

\subsection{Roadmap} 
Recall that we have $(\alpha_*, \tau_*)$ in \eqref{eq:se11}. As mentioned in
Section \ref{ssec:proof-sketch}, we need to characterize $(\alpha_*, \tau_*)$
as $\sigma\rightarrow \infty$. Accordingly ${\rm AMSE}(q,\lambda_q^*)=\delta(\tau_*^2-\sigma^2)$.
It is clear from \eqref{eq:se11} that $\tau_* \rightarrow \infty$ as $\sigma
\rightarrow \infty$. However, to derive the second order expansion of ${\rm
AMSE}(q,\lambda_q^*)$ as $\sigma \rightarrow \infty$, we need to obtain the
convergence rate of $\tau_*$. We will achieve this goal by first characterizing
the convergence rate of the term $\min_{\alpha \geq 0}
\mathbb{E}(\eta_q(B+\tau_*Z;\alpha \tau_*^{2-q})-B)^2$ as $\tau_*\rightarrow
\infty$. We then use that result to derive the convergence rate of $\tau_*$
based on \eqref{eq:se11} and finally calculate AMSE$(q,\lambda_q^*)$. Since the proof techniques look different for $q=1, 1<q\leq 2, q>2$, we prove the theorem for these three cases in Sections \ref{proof35:q1}, \ref{proof35:q12} and \ref{proof35:q2} respectively.

\subsection{Proof of Theorem \ref{THM:NOISY:MAIN} for $q=1$} \label{proof35:q1}
As explained in the roadmap of the proof, the key step is to characterize the
convergence rate of $\tau_*$. Towards this goal, we first derive the
convergence rate of $\alpha_q(\tau)$ as $\tau \rightarrow \infty$ in Section
\ref{alphaorder:sec}. We then bound the convergence rate of
$R_q(\alpha_q(\tau),\tau)$ as $\tau \rightarrow \infty$ in Section
\ref{sssec:growth:rate:l1}. This enables us to study the rate of $\tau_*$ when $\sigma \rightarrow \infty$, and derive the expansion of AMSE$(q,\lambda_q^*)$ as $\sigma \rightarrow \infty$ in Section \ref{amse:qeq1}.

\subsubsection{Deriving the convergence rate of $\alpha_q(\tau)$ as $\tau \rightarrow \infty$ for $q=1$} \label{alphaorder:sec}

We first prove $\alpha_q(\tau)\rightarrow \infty$ as $\tau \rightarrow \infty$ in the next lemma.

\begin{lemma}\label{lemma:noisy:alpha:infty}
Recall the definition of $\alpha_q(\tau)$ in \eqref{eq:def:alphaq}. Assume $\mathbb{E}|G|^2<\infty$. Then, $\alpha_q(\tau)\rightarrow \infty$ as $\tau \rightarrow \infty$. 
\end{lemma}

\begin{proof}
Suppose this is not true, then there exists a sequence $\{\tau_n\}$ such that $\alpha_q(\tau_n)\rightarrow \alpha_0 <\infty$ and $\tau_n \rightarrow \infty$, as $n\rightarrow \infty$. Notice that
\[
|\eta_q(B/\tau_n+Z;\alpha_q(\tau_n))|\leq |B|/ \tau_n+Z \leq |B|+Z,
\]
for sufficiently large $n$. 
We can apply DCT to obtain
\begin{eqnarray*}
\lim_{n\rightarrow \infty}R_q(\alpha_q(\tau_n),\tau_n)=\mathbb{E}\eta^2_q(Z;\alpha_0)>0.
\end{eqnarray*}
On the other hand, since $\alpha=\alpha_q(\tau_n)$ minimizes $R_q(\alpha,\tau_n)$
\[
\lim_{n\rightarrow \infty}R_q(\alpha_q(\tau_n),\tau_n) \leq \lim_{n\rightarrow \infty} \lim_{\alpha\rightarrow \infty}R_q(\alpha,\tau_n) =0.
\]
A contradiction arises. 
\end{proof}

Based on Lemma \ref{lemma:noisy:alpha:infty}, we can further derive the convergence rate of $\alpha_q(\tau)$.

\begin{lemma}\label{lemma:noisy:alpha:trend:l1}

If $G$ has a sub-Gaussian tail, then
\[
\lim_{\tau \rightarrow \infty} \frac{\alpha_q(\tau)}{\tau}= C_0,
\]
where $C=C_0$ is the unique solution of the following equation:
\[
\mathbb{E}\left(e^{CG}(CG-1)+e^{-CG}(-CG-1)\right)=\frac{2(1-\epsilon)}{\epsilon}.
\]
\end{lemma}
\begin{proof}
Since $\alpha=\alpha_q(\tau)$ minimizes $R_q(\alpha,\tau)$, we know
$\partial_1R_q(\alpha_q(\tau),\tau)=0$. To simplify the notation, we will
simply write $\alpha$ for $\alpha_q(\tau)$ in the rest of this proof.
Rearranging the terms in \eqref{eq:L1-risk-deri-expand} gives us
\begin{equation*}
\frac{2(1-\epsilon)}{\epsilon}= 
\mathbb{E}\underbrace{\frac{\alpha^2}{\phi(\alpha)}\Big[\alpha\Phi \big(\frac{|G|}{\tau}-\alpha\big)+\alpha\Phi\big(-\frac{|G|}{\tau}-\alpha\big)-\phi\big(\frac{|G|}{\tau}-\alpha\big)-\phi\big(\frac{|G|}{\tau}+\alpha\big)\Big]}_{T(G,\alpha,\tau)}.
\end{equation*}
Fixing $t\in(0, 1)$, we reformulate the above equation in the following way:
\begin{equation} 
\hspace{.1cm} \frac{2(1-\epsilon)}{\epsilon}= \mathbb{E}[T(G,\alpha,\tau)\mathbb{I}(|G|\leq t\tau\alpha)]+ \mathbb{E}[T(G,\alpha,\tau)\mathbb{I}(|G|> t\tau\alpha)].  \label{eq:lemma:noisy:alpha:trend:l1:1}
\end{equation}
We now analyze the two terms on the right hand side of the above equation. Since $G$ has a sub-Gaussian tail, there exists a constant $\gamma>0$ such that $\mathbb{P}(|G|>x)\leq  e^{-\gamma x^2}$ for $x$ large. We can then have the following bound,
\begin{eqnarray*}
&&|\mathbb{E}[T(G,\alpha,\tau)\mathbb{I}(|G|> t\tau\alpha)]|
\leq
\frac{\alpha^2}{\phi(\alpha)}\big(2\alpha+\sqrt{2/\pi}\big)\mathbb{P}(|G|>t\tau\alpha) \\
&&\leq
\alpha^2\big(2\sqrt{2\pi}\alpha+2 \big) e^{-(\gamma t^2\tau^2-\frac{1}{2})\alpha^2}
\rightarrow0, \mbox{~~~as~}\tau \rightarrow \infty,
\end{eqnarray*}
where we have used the fact that $\alpha \rightarrow \infty$ as $\tau \rightarrow \infty$ from Lemma \ref{lemma:noisy:alpha:infty}. This result combined with \eqref{eq:lemma:noisy:alpha:trend:l1:1} implies that as $\tau\rightarrow \infty$
\begin{equation}
\mathbb{E}[T(G,\alpha,\tau)\mathbb{I}(|G|\leq t\tau\alpha)] \rightarrow \frac{2(1-\epsilon)}{\epsilon}. \label{fromhere}
\end{equation}
Moreover, using the tail approximation of normal distribution in \eqref{gaussiantail:exp} with $k=3$, we have for sufficiently large $\tau$,
\begin{align*} 
    \mathbb{E}[T(G,\alpha,\tau)&\mathbb{I}(|G|\leq t\tau\alpha)]
\leq
\mathbb{E}\Big[\underbrace{\frac{\alpha}{\alpha-|G|/\tau}e^{\frac{\alpha |G|}{\tau}-\frac{G^2}{2\tau^2}}\bigg(\frac{\alpha |G|}{\tau}-\frac{\alpha^2}{(\alpha-|G|/\tau)^2}+\frac{3\alpha^2}{(\alpha-|G|/\tau)^4}\bigg)}_{U_1(G,\alpha, \tau)} \nonumber\\
&+\underbrace{\frac{\alpha}{\alpha+|G|/\tau}e^{-\frac{\alpha |G|}{\tau}-\frac{G^2}{2\tau^2}}\bigg(-\frac{\alpha |G|}{\tau}-\frac{\alpha^2}{(\alpha + |G|/\tau)^2}+\frac{3\alpha^2}{(\alpha + |G|/\tau)^4}\bigg)}_{U_2(G,\alpha, \tau)}\Big]\cdot \mathbb{I}(|G|\leq t\tau \alpha).
\end{align*}

Similarly applying  \eqref{gaussiantail:exp} with $k=2$ gives us for large $\tau$
\begin{eqnarray*}
&&\hspace{-0.8cm} \mathbb{E}[T(G,\alpha,\tau)\mathbb{I}(|G|\leq t\tau\alpha)]
\geq
\mathbb{E}\Big[\underbrace{\frac{\alpha}{\alpha-|G|/\tau}e^{\frac{\alpha |G|}{\tau}-\frac{G^2}{2\tau^2}}\bigg(\frac{\alpha |G|}{\tau}-\frac{\alpha^2}{(\alpha - |G|/\tau)^2}\bigg)}_{L_1(G,\alpha, \tau)}+\nonumber\\
&& \hspace{0.cm}
\underbrace{\frac{\alpha}{\alpha + |G|/\tau}e^{-\frac{\alpha |G|}{\tau}-\frac{G^2}{2\tau^2}}\bigg(-\frac{\alpha |G|}{\tau}-\frac{\alpha^2}{(\alpha + |G|/\tau)^2}\bigg)}_{L_2(G,\alpha, \tau)}\Big]\cdot \mathbb{I}(|G|\leq t\tau \alpha).
\end{eqnarray*}

We claim based on the two bounds that $\varlimsup_{\tau
\rightarrow\infty}\frac{\alpha}{\tau}=C_1$ with $0<C_1<\infty$. Otherwise:
\begin{itemize}
	\item
	If $C_1=\infty$, there exists a sequence $\alpha_n/\tau_n \rightarrow \infty$ and $\tau_n\rightarrow \infty$, as $n\rightarrow \infty$. Since $|L_2(G,\alpha_n, \tau_n)|\leq e^{-\frac{\alpha_n |G|}{\tau_n}}(\frac{\alpha_n |G|}{\tau_n} + 1)\leq 2$, we can apply DCT to obtain 
\[
\lim_{n\rightarrow \infty}\mathbb{E}(L_2(G,\alpha_n, \tau_n)\mathbb{I}(|G|\leq t\tau_n\alpha_n))=0.
\]
Furthermore, we choose a positive constant $\zeta>0$ satisfying the condition in Lemma \ref{appendix:noisy:l1:alpha3} for the nonnegative random variable $|G|$. Then
\begin{eqnarray*}
&&\hspace{-0.cm} \mathbb{E}(L_1(G,\alpha_n, \tau_n)\mathbb{I}(|G|\leq t\tau_n\alpha_n)) \\
&&\geq \mathbb{E}\bigg[e^{\frac{\alpha_n |G|}{\tau_n}-\frac{G^2}{2\tau_n^2}}\bigg(\frac{\alpha_n |G|}{\tau_n}-\frac{1}{(1-t)^3}\bigg)\mathbb{I}(|G|\leq t\tau_n\alpha_n)\bigg] \\
&&\geq \int_{\zeta<g\leq t\tau_n\alpha_n}e^{\frac{\alpha_n g}{\tau_n}-\frac{g^2}{2\tau_n^2}}\frac{\alpha_n g}{\tau_n}dF(g)-\int_{g\leq t\tau_n\alpha_n}\frac{1}{(1-t)^3}e^{\frac{\alpha_n g}{\tau_n}-\frac{g^2}{2\tau_n^2}}dF(g) \\
&& \overset{(a)}{\geq} \Big(\frac{\zeta\alpha_n}{\tau_n}-\frac{2}{(1-t)^3}\Big)\int_{\zeta<g\leq t\tau_n\alpha_n}e^{\frac{\alpha_n g}{\tau_n}-\frac{g^2}{2\tau_n^2}} dF(g)\\
&&\geq \Big(\frac{\zeta\alpha_n}{\tau_n}-\frac{2}{(1-t)^3} \Big) e^{\frac{\alpha_n\zeta}{\tau_n}}\int_{\zeta<g \leq t\tau_n\alpha_n}e^{-\frac{g^2}{2\tau_n^2}}dF(g) \rightarrow\infty,
\end{eqnarray*}		
where we have used Lemma \ref{appendix:noisy:l1:alpha3} in $(a)$. This forms a contradiction.
\vspace{0.2cm}
\item
If $C_1=0$, for large enough $\tau$ we have $\frac{\alpha}{\tau}<1$ and then on $|G|\leq t\tau\alpha$,
\[
|U_1(G,\alpha, \tau)+U_2(G,\alpha, \tau)|\leq\frac{2}{1-t}e^{G}\Big[G+\frac{1}{(1-t)^2}+\frac{3}{\alpha^2(1-t)^4}\Big],
\] 
which is integrable since $G$ has sub-Gaussian tail. Hence we apply DCT to obtain as $\tau\rightarrow \infty$
\[
\mathbb{E}\left[(U_1(G,\alpha, \tau)+U_2(G,\alpha, \tau))\mathbb{I}(|G|\leq t\tau\alpha)\right]\rightarrow -2
\]
This forms another contradiction.
\end{itemize}
Similar to the above arguments, we can conclude that $\varliminf_{\tau\rightarrow\infty}\frac{\alpha}{\tau}=C_2\in(0,\infty)$. Now that $\frac{\alpha}{\tau}=O(1)$, we can use DCT to obtain
\[
\lim_{\tau\rightarrow \infty}\mathbb{E}\Big[ \frac{\alpha}{\alpha\pm |G|/\tau}e^{\frac{\alpha |G|}{\tau}-\frac{G^2}{2\tau^2}}\frac{3\alpha^2}{(\alpha \pm |G|/\tau)^4}\mathbb{I}(|G|\leq t\tau\alpha) \Big]=0.
\]
This result combined together with \eqref{fromhere} and the upper and lower bounds on $\mathbb{E}[T(G,\alpha,\tau)\mathbb{I}(|G|\leq t\tau\alpha)] $ enables us to show
\begin{eqnarray*}
\lim_{\tau\rightarrow \infty} \mathbb{E}[(L_1(G,\alpha,\tau)+L_2(G,\alpha,\tau))\mathbb{I}(|G|\leq t\tau\alpha)]=\frac{2(1-\epsilon)}{\epsilon}.
\end{eqnarray*}

Now consider a convergent sequence $\frac{\alpha_n}{\tau_n} \rightarrow C_1 \in (0,\infty)$ and $\tau_n\rightarrow \infty$ as $n\rightarrow \infty$. On $|G|\leq t\tau_n\alpha_n$ we can bound for large $n$
\[
|L_1(G,\alpha_n,\tau_n)+L_2(G,\alpha_n,\tau_n)|\leq  \frac{2}{1-t}e^{2C_1G}\bigg(2C_1G+\frac{1}{(1-t)^2}\bigg),
\]
which is again integrable. Thus DCT gives us
\begin{eqnarray*}
\frac{2(1-\epsilon)}{\epsilon}
&=& \lim_{n\rightarrow \infty} \mathbb{E}[(L_1(G,\alpha_n,\tau_n)+L_2(G,\alpha_n,\tau_n))\mathbb{I}(|G|\leq t\tau_n\alpha_n)]  \\
&=& \mathbb{E}\big[e^{C_1|G|}(C_1|G|-1)+e^{-C_1|G|}(-C_1|G|-1)\big].
\end{eqnarray*}
For $C_2$ the same equation holds. By calculating the derivative we can easily verify $h(c)=e^{c|G|}(c|G|-1)+e^{-c|G|}(-c|G|-1)$, as a function of $c$ over $(0,\infty)$, is strictly increasing. This determines $C_1=C_2$. Above all we have shown
\begin{equation*}
\frac{\alpha_q(\tau)}{\tau}\rightarrow C_0,  \mbox{~~as~} \tau\rightarrow \infty,
\end{equation*}
where $\mathbb{E}\big[e^{C_0G}(C_0G-1)+e^{-C_0G}(-C_0G-1)\big]=\frac{2(1-\epsilon)}{\epsilon}$.
\end{proof}

\subsubsection{Bounding the convergence rate of $R_1(\alpha_1(\tau),\tau)$ as $\tau \rightarrow \infty$}
\label{sssec:growth:rate:l1}                        
We state the main result in the next lemma.
\begin{lemma}\label{theorem:noisy:l1:tau} 
If $G$ has sub-Gaussian tail, then as $\tau\rightarrow \infty$
\[
R_1(\alpha_1(\tau),\tau)=\frac{\epsilon \mathbb{E}|G|^2}{\tau^2}+o\Big(\frac{\phi(\alpha_1(\tau))}{\alpha^3_1(\tau)}\Big).
\]
\end{lemma}

\begin{proof}
For notational simplicity, we will use $\alpha$ to denote $\alpha_1(\tau)$ in
the rest of the proof. Rearranging \eqref{eq:L1-risk-expand2}, we can write $R_1(\alpha,\tau)$ in the following form:
\begin{align*}
    R_1(\alpha,\tau)=&
    2(1-\epsilon)[(1+\alpha^2)\Phi(-\alpha)-\alpha\phi(\alpha)]
+
\epsilon\mathbb{E}\bigg[\underbrace{(1+\alpha^2-G^2/\tau^2)[\Phi(G/\tau-\alpha)+\Phi(-G/\tau-\alpha)]}_{S_1(G,\alpha, \tau)}\nonumber\\
&\underbrace{-(G/\tau+\alpha)\phi(\alpha-G/\tau)+(G/\tau-\alpha)\phi(\alpha+G/\tau)}_{S_2(G,\alpha,\tau)}+G^2/\tau^2\bigg].
\end{align*}
Hence, we have
\begin{align}
&\lim_{\tau\rightarrow \infty}
    \frac{\alpha^3}{\phi(\alpha)}\Big(R_q(\alpha,\tau)-\frac{\epsilon
    \mathbb{E}|G|^2}{\tau^2}\Big) \label{bound:secondorder} \\
    =& 2(1-\epsilon)\lim_{\tau\rightarrow\infty}\frac{\alpha^3}{\phi(\alpha)}[(1+\alpha^2)\Phi(-\alpha)-\alpha\phi(\alpha)] +\epsilon \lim_{\tau\rightarrow \infty}\frac{\alpha^3}{\phi(\alpha)}\mathbb{E}[S_1(G,\alpha, \tau)+S_2(G,\alpha, \tau)] \nonumber \\
    \overset{(a)}{=}& 4(1-\epsilon)+\epsilon \lim_{\tau\rightarrow
\infty}\frac{\alpha^3}{\phi(\alpha)}\mathbb{E}[S_1(G,\alpha,
\tau)+S_2(G,\alpha, \tau)] \nonumber.
\end{align}
We use the tail expansion \eqref{gaussiantail:exp} with $k=3,4$ to obtain $(a)$. Since $|x\phi(x)|\leq\frac{e^{-1/2}}{\sqrt{2\pi}}$, we have
\begin{equation}
|S_1(G,\alpha, \tau)+S_2(G,\alpha, \tau)|
\leq
\frac{2e^{-1/2}}{\sqrt{2\pi}}+\frac{4\alpha}{\sqrt{2\pi}}+2\bigg(1+\alpha^2+\frac{G^2}{\tau^2}\bigg).\nonumber
\end{equation}

Moreover, it is not hard to use the sub-Gaussian condition $\mathbb{P}(|G|>x)\leq e^{-\gamma x^2}$ to obtain
\begin{eqnarray*}
&&\mathbb{E}(G^2\mathbb{I}(|G|>t\tau \alpha))=\int_0^{t\tau\alpha}2x\mathbb{P}(G>t\tau\alpha)dx+\int_{t\tau\alpha}^\infty2x\mathbb{P}(G>x)dx  \\
&&\hspace{1.4cm} \leq (t\tau\alpha)^2e^{-\gamma t^2\tau^2\alpha^2}+\frac{1}{\gamma}e^{-\gamma t^2\tau^2\alpha^2},
\end{eqnarray*}
where $t\in (0,1)$ is a constant. Combining the last two bounds we can derive
\begin{eqnarray*}
&&\hspace{0.5cm} \frac{\alpha^3}{\phi(\alpha)}\mathbb{E}[(S_1(G,\alpha, \tau)+S_2(G,\alpha, \tau))\mathbb{I}(|G|>t\tau\alpha)] \\
&& \leq
\alpha^3\big(2e^{-1/2}+4\alpha+2\sqrt{2\pi}(1+\alpha^2)\big) e^{-(\gamma t^2\tau^2-\frac{1}{2})\alpha^2}
+ \\
&&\hspace{0.5cm} \frac{2\sqrt{2\pi}\alpha^3}{\tau^2} (t^2\tau^2\alpha^2+1/\gamma)e^{-(\gamma t^2\tau^2-\frac{1}{2})\alpha^2}\rightarrow0, \mbox{~~as~}\tau \rightarrow \infty.
\end{eqnarray*}
On the other hand, we can build an upper bound and lower bound for $|S_1(G,\alpha, \tau)+S_2(G,\alpha, \tau)|$ on $\{|G|\leq t\tau\alpha\}$ with the tail expansion \eqref{gaussiantail:exp} as we did in the proof of Lemma \ref{lemma:noisy:alpha:trend:l1}, For both bounds we can argue they converge to the same limit as $\tau\rightarrow\infty$ by using DCT and Lemma \ref{lemma:noisy:alpha:trend:l1}. Here we give the details of using DCT for the upper bound. Using \eqref{gaussiantail:exp} with $k=3$ we can obtain the upper bound,
\begin{eqnarray*}
&&\frac{\alpha^3}{\phi(\alpha)}(S_1(G,\alpha, \tau)+S_2(G,\alpha, \tau)) \\
&&\hspace{-0.8cm} \leq \frac{\alpha^3\phi(\alpha-G/\tau)}{\phi(\alpha)}\left[\frac{2G^2/\tau^2-2\alpha G/\tau-1}{(\alpha-G/\tau)^3}+\frac{3(1+\alpha^2-G^2/\tau^2)}{(\alpha-G/\tau)^5}\right]+ \\
&& \frac{\alpha^3\phi(\alpha+G/\tau)}{\phi(\alpha)}\left[\frac{2G^2/\tau^2+2\alpha G/\tau-1}{(\alpha+G/\tau)^3}+\frac{3(1+\alpha^2-G^2/\tau^2)}{(\alpha+G/\tau)^5}\right].
\end{eqnarray*}
It is straightforward to see that on $\{|G|\leq t\tau \alpha\}$ for
sufficiently large $\alpha$, there exist three positive constants $C_1,C_2,
C_3$ such that the upper bound can be further bounded by
$\Big[\frac{C_1|G|+1}{(1-t)^3}+\frac{C_2}{(1-t)^5}+\frac{C_1|G|+1}{(1+t)^3}+\frac{C_2}{(1+t)^5}\Big]e^{C_3|G|}$, which is integrable by the condition that $G$ has sub-Gaussian tail. Hence we can apply DCT to derive the limit of the upper bound. Similar arguments enable us to calculate the limit of the lower bound. By calculating the limits of the upper and lower bounds we can obtain the following result:
\begin{eqnarray*}
&&\frac{\alpha^3}{\phi(\alpha)}\mathbb{E}[(S_1(G,\alpha, \tau)+S_2(G,\alpha, \tau))\mathbb{I}(|G|\leq t\tau\alpha)]\\&&\rightarrow
-2\mathbb{E}\left(e^{C_0G}(C_0G-1)+e^{-C_0G}(-C_0G-1)\right)
=-\frac{4(1-\epsilon)}{\epsilon}.
\end{eqnarray*}
This completes the proof.
\end{proof}

\subsubsection{Deriving the expansion of ${\rm AMSE}(q,\lambda^*_q)$ for $q=1$} \label{amse:qeq1}

We are now in the  position to derive the result \eqref{largenoise:lasso} in Theorem \ref{THM:NOISY:MAIN}. As we explained in the roadmap, we know
\begin{eqnarray}
{\rm AMSE}(q,\lambda_q^*)=\tau^2_* R_q(\alpha_q(\tau_*),\tau_*)=\delta(\tau_*^2-\sigma^2). \label{final:eq}
\end{eqnarray}
First note that $\tau_* \rightarrow \infty$ as $\sigma \rightarrow \infty$ since $\tau_*\geq \sigma$. According to Lemma \ref{theorem:noisy:l1:tau} and \eqref{final:eq}, we have
\begin{equation}
\lim_{\sigma \rightarrow \infty} \frac{\sigma^2}{\tau_*^2}=\lim_{\tau_* \rightarrow \infty} \frac{\sigma^2}{\tau_*^2}=\lim_{\tau_* \rightarrow \infty} \big(1-\frac{R_q(\alpha_q(\tau_*),\tau_*)}{\delta} \big)=1. \label{final:eq2}
\end{equation}
Furthermore, Lemma \ref{lemma:noisy:alpha:trend:l1} shows that
\begin{equation}
\lim_{\sigma \rightarrow \infty} \frac{\alpha_q(\tau_*)}{\tau_*}=\lim_{\tau_* \rightarrow \infty}  \label{final:eq3}\frac{\alpha_q(\tau_*)}{\tau_*}=C_0.
\end{equation}
Combining Lemma \ref{theorem:noisy:l1:tau} with \eqref{final:eq}, \eqref{final:eq2}, and \eqref{final:eq3} we obtain as $\sigma \rightarrow \infty$,
\begin{eqnarray*}
&&e^{\frac{C^2\sigma^2}{2}}({\rm AMSE}(q,\lambda_q^*)-\epsilon \mathbb{E}|G|^2)=e^{\frac{C^2\sigma^2}{2}}\tau_*^2(R_q(\alpha_q(\tau_*),\tau_*)-\epsilon \mathbb{E}|G|^2/\tau_*^2) \\
&=&e^{\frac{C^2\sigma^2}{2}}\tau_*^2e^{-\frac{\alpha^2_q(\tau_*)}{2}}(\alpha_q(\tau_*))^{-3}o(1)=e^{-\frac{\sigma^2}{2}(\frac{\alpha^2_q(\tau_*)}{\tau_*^2}\cdot \frac{\tau_*^2}{\sigma^2}-C^2)}\tau_*^2(\alpha_q(\tau_*))^{-3}o(1)=o(1).
\end{eqnarray*}
We have used the fact $0<C<C_0$ to get the last equality. 

\subsection{Proof of Theorem \ref{THM:NOISY:MAIN} for $q\in (1,2]$} \label{proof35:q12}

The basic idea of the proof for $q\in (1,2]$ is the same as that for $q=1$. We characterize the convergence rate of $R_q(\alpha_q(\tau),\tau)$ in Section \ref{riskrate:q12}. We can derive the expansion of ${\rm AMSE}(q,\lambda_q^*)$ in Section \ref{amseexp:q12}.

\subsubsection{Characterizing the convergence rate of $R_q(\alpha_q(\tau),\tau)$ as $\tau\rightarrow \infty$ for $q\in (1,2]$}  \label{riskrate:q12}

We first derive the convergence rate of $\alpha_q(\tau)$ as $\tau \rightarrow \infty$.
\begin{lemma}\label{lemma:noisy:alpha:tau:trend:q}
For $q\in(1,2]$, assume $G$ has finite moments of all order. We have,
\begin{equation*}
\frac{\alpha_q(\tau)}{\tau^{2(q-1)}}
\rightarrow
\left(\frac{q-1}{q^{\frac{1}{q-1}}}\frac{\mathbb{E}|Z|^{\frac{2}{q-1}}}{\mathbb{E}B^2\mathbb{E}|Z|^{\frac{2-q}{q-1}}}\right)^{q-1},
\quad \text{as }\tau\rightarrow \infty
\end{equation*}
\end{lemma}
\begin{proof}
First note that Lemma \ref{lemma:noisy:alpha:infty} holds for $q\in (1,2]$ as well. Hence $\alpha_q(\tau)\rightarrow \infty$ as $\tau \rightarrow \infty$. We aim to characterize its convergence rate. Since $\eta_2(u;\chi)=\frac{u}{1+2\chi}$, the result can be easily verified for $q=2$. We will focus on the case $q \in (1,2)$. For notational simplicity, we will use $\alpha$ to represent $\alpha_q(\tau)$ in the rest of the proof.  By the first order condition of the optimality, we have $\partial_1 R_q(\alpha,\tau)=0$, which can be further written out:
\begin{align}\label{starting:p}
0
=& \mathbb{E}[(\eta_q(B/\tau+Z;\alpha)-B/\tau)\partial_2\eta_q(B/\tau+Z;\alpha)] \nonumber \\
=& \underbrace{\mathbb{E}\Bigg [\frac{-q|\eta_q(B/\tau+Z;\alpha)|^q}{1+\alpha q(q-1)|\eta_q(B/\tau+Z;\alpha)|^{q-2}} \Bigg]}_{H_1}
+ \underbrace{\mathbb{E}\Bigg [\frac{B
q|\eta_q(B/\tau+Z;\alpha)|^{q-1}\mbox{sgn}(B/\tau+Z)}{\tau(1+\alpha
q(q-1)|\eta_q(B/\tau+Z;\alpha)|^{q-2})} \Bigg]}_{H_2}
\end{align}
where we have used Lemma \ref{prox:property} part (v). We now analyze the two terms $H_1$ and $H_2$ respectively. Regarding $H_1$ from Lemma \ref{prox:property} part (ii) we have
\begin{eqnarray*}
&&\frac{\alpha^{\frac{q+1}{q-1}}q\big|\eta_q\big(B/\tau+Z;\alpha\big)\big|^q}{1+\alpha q(q-1)\big|\eta_q\big(B/\tau+Z;\alpha\big)\big|^{q-2}}
\leq
\frac{\big|\alpha^{\frac{1}{q-1}}\eta_q\big(B/\tau+Z;\alpha\big)\big|^2}{q-1}\nonumber\\
&&=
\frac{\Big| |B/\tau+Z|-\big|\eta_q (B/\tau+Z;\alpha)\big|\Big|^\frac{2}{q-1}}{q^{\frac{2}{q-1}}(q-1)}
\leq
\frac{(|B|+|Z|)^\frac{2}{q-1}}{q^{\frac{2}{q-1}}(q-1)}, \mbox{~~for~}\tau \geq 1.
\end{eqnarray*}
Since $G$ has finite moments of all orders, the upper bound above is integrable. Hence DCT enables us to conclude
\begin{equation}
\lim_{\tau \rightarrow \infty} \alpha^{\frac{q+1}{q-1}}H_1=
\frac{\mathbb{E}|Z|^{\frac{2}{q-1}}}{q^{\frac{2}{q-1}}(1-q)}. \label{eq:noisy:alpha:tau:trend:2}
\end{equation}
For the term $H_2$, according to Lemma \ref{prox:property} parts (ii)(iv) we can obtain
\begin{eqnarray*}
&&H_2=
\frac{1}{\tau\alpha}\mathbb{E}\left[\frac{B\big(B/\tau+Z-\eta_q(B/\tau+Z;\alpha)\big)}{1+\alpha q(q-1)\big|\eta_q\big(B/\tau+Z;\alpha\big)\big|^{q-2}}\right]\nonumber\\
&&\hspace{-0.5cm}=
\underbrace{\frac{1}{\tau^2\alpha}\mathbb{E}\Bigg[\frac{B^2}{1+\alpha q(q-1)\big|\eta_q(B/\tau+Z;\alpha)\big|^{q-2}}\Bigg]}_{I_1}
+
\underbrace{\mathbb{E}\bigg[\frac{B Z\partial_1\eta_q(B/\tau+Z;\alpha)}{\tau\alpha}\bigg]}_{I_2}\nonumber\\
&&\hspace{1.2cm} -
\underbrace{\frac{1}{\tau\alpha} \mathbb{E}\Bigg[\frac{B\eta_q(B/\tau+Z;\alpha)}{1+\alpha q(q-1)\big|\eta_q(B/\tau+Z;\alpha)\big|^{q-2}}\Bigg]}_{I_3}.  \label{termh2:analy}
\end{eqnarray*}
By a similar argument and using DCT, it is not hard to see that,
\begin{equation} \label{extra:need1}
\lim_{\tau\rightarrow \infty} \tau^2\alpha^{\frac{q}{q-1}}I_1 =
\frac{\mathbb{E}B^2\mathbb{E}|Z|^{\frac{2-q}{q-1}}}{q^{\frac{1}{q-1}}(q-1)},
\qquad
\lim_{\tau\rightarrow \infty} \tau\alpha^{\frac{q+1}{q-1}}I_3
=\frac{\mathbb{E}B\mathbb{E}\big(|Z|^{\frac{3-q}{q-1}}\sgn(Z)\big)}{q^{\frac{2}{q-1}}(q-1)}=0.
\end{equation}
Regarding the term $I_2$, by using Stein's lemma and Taylor expansion, we can obtain a sequel of equalities:
\begin{align*}
    I_2=&
\frac{\mathbb{E}\big[B(Z^2-1)\eta_q\big(B/\tau+Z;\alpha\big)\big]}{\alpha\tau}
=
\frac{\mathbb{E}\big[B(Z^2-1)\big(\eta_q(Z;\alpha)+\partial_1\eta_q(\gamma B/\tau+Z;\alpha)B/\tau\big)\big]}{\alpha\tau}\nonumber\\
=&
\frac{\mathbb{E}\big[B^2(Z^2-1)\partial_1\eta_q\big(\gamma B/\tau+Z;\alpha\big)\big]}{\alpha\tau^2}
=
\frac{1}{\alpha\tau^2}\mathbb{E}\bigg[\frac{B^2(Z^2-1)}{1+\alpha q(q-1)\big|\eta_q\big(\gamma B/\tau+Z;\alpha\big)\big|^{q-2}}\bigg],
\end{align*}
where the second step is simply due to Lemma \ref{prox:property} part (i); $\gamma\in(0,1)$ is a random variable depending on $B$ and $Z$. With a similar argument to verify the conditions of DCT we obtain
\begin{equation}
\lim_{\tau \rightarrow \infty} \alpha^{\frac{q}{q-1}}\tau^2I_2=
\frac{(2-q)\mathbb{E}B^2\mathbb{E}|Z|^{\frac{2-q}{q-1}}}{q^{\frac{1}{q-1}}(q-1)^2}
\label{eq:noisy:alpha:tau:trend:5}.
\end{equation}
Finally, \eqref{starting:p}, \eqref{eq:noisy:alpha:tau:trend:2},
\eqref{extra:need1} and \eqref{eq:noisy:alpha:tau:trend:5} together enable us to have as $\tau \rightarrow \infty$,
\begin{equation*}
\frac{\alpha}{\tau^{2(q-1)}}=\left[\frac{\alpha^{\frac{q+1}{q-1}}(I_3-H_1)}{\tau^2\alpha^{\frac{q}{q-1}}(I_1+I_2)}\right]^{q-1}\rightarrow  \left(\frac{q-1}{q^{\frac{1}{q-1}}}\frac{\mathbb{E}|Z|^{\frac{2}{q-1}}}{\mathbb{E}B^2\mathbb{E}|Z|^{\frac{2-q}{q-1}}}\right)^{q-1}.
\end{equation*}
\end{proof}

We now characterize the convergence rate of $R_q(\alpha_q(\tau),\tau)$.
\begin{lemma} \label{boundrisk:rateq12}
Suppose $1<q \leq 2$ and $G$ has finite moments of all orders, then as $\tau \rightarrow \infty$,
\begin{eqnarray*}
R_q(\alpha_q(\tau),\tau)=\frac{\epsilon \mathbb{E}|G|^2}{\tau^2}-\frac{\epsilon^2 \big(\mathbb{E}|G|^2\mathbb{E}|Z|^{\frac{2-q}{q-1}}\big)^2}{(q-1)^2\mathbb{E}|Z|^{\frac{2}{q-1}}}\frac{1}{\tau^4}+o(1/\tau^4).
\end{eqnarray*}
\end{lemma}

\begin{proof}
It is straightforward to prove the result for $q=2$. Now we only consider $1<q<2$. We write $\alpha$ for $\alpha_q(\tau)$ in the rest of the proof to simplify the notation. First we have
\begin{eqnarray}
&&R_q(\alpha,\tau)-\frac{\epsilon \mathbb{E}|G|^2}{\tau^2}
= \mathbb{E}\eta_q^2(B/\tau+Z;\alpha)- 2\mathbb{E}[\eta_q(B/\tau+Z;\alpha)B/\tau] \nonumber \\
&&=
\mathbb{E}\eta_q^2(B/\tau+Z;\alpha)
-
2\mathbb{E}[(\eta_q(Z;\alpha)+\partial_1\eta_q(\gamma B/\tau+Z;\alpha)B/\tau)B/\tau]\nonumber\\
&&=
\mathbb{E}\eta_q^2(B/\tau+Z;\alpha)
-
2\mathbb{E}[\partial_1\eta_q(\gamma B/\tau+Z;\alpha)B^2/\tau^2], \label{starting:eq}
\end{eqnarray}
where we have used Taylor expansion in the second step and $\gamma \in (0,1)$
is a random variable depending on $B, Z$. According to Lemma
\ref{prox:property} part (ii), for $\tau \geq 1$,
\begin{align*}
\alpha^{\frac{2}{q-1}} \eta_q^2(B/\tau+Z;\alpha)
=q^{\frac{2}{1-q}}(|B/\tau+Z|-|\eta_q(B/\tau+Z;\alpha)|)^{\frac{2}{q-1}}
 \leq& q^{\frac{2}{1-q}} (|B|+|Z|)^{\frac{2}{q-1}}.
\end{align*}
The upper bound is integrable since $G$ has finite moments of all orders. Hence we can apply DCT to obtain            
\begin{equation}
\lim_{\tau\rightarrow \infty} \alpha^{\frac{2}{q-1}} \mathbb{E}\eta_q^2(B/\tau+Z;\alpha)=q^{\frac{2}{1-q}} \mathbb{E}|Z|^{\frac{2}{q-1}}. \label{eq:noisy:mse:sigma:trend:q1}
\end{equation}
We can follow a similar argument to use DCT to have
\begin{align}
    &\lim_{\tau\rightarrow \infty} \alpha^{\frac{2}{q-1}}\mathbb{E}[\partial_1\eta_q(\gamma B/\tau+Z;\alpha)B^2/\tau^2]  \nonumber \\
    \overset{(a)}{=}&\lim_{\tau\rightarrow \infty}
\frac{\alpha^{\frac{1}{q-1}}}{\tau^2}\cdot \lim_{\tau\rightarrow \infty} \mathbb{E}\left[\frac{B^2}{\alpha^{-\frac{1}{q-1}}+q(q-1)|\alpha^{\frac{1}{q-1}}\eta_q(\gamma B/\tau+Z;\alpha)|^{q-2}}\right]\nonumber\\
\overset{(b)}{=}&
\frac{q-1}{q^{\frac{1}{q-1}}}\frac{\mathbb{E}|Z|^{\frac{2}{q-1}}}{\mathbb{E}B^2\mathbb{E}|Z|^{\frac{2-q}{q-1}}}\cdot \frac{\mathbb{E}B^2\mathbb{E}|Z|^{\frac{2-q}{q-1}}}{q^{\frac{1}{q-1}}(q-1)}
= q^{\frac{2}{1-q}} \mathbb{E}|Z|^{\frac{2}{q-1}} ,
\label{eq:noisy:mse:sigma:trend:q2}
\end{align}
where $(a)$ holds due to Lemma \ref{prox:property} part (iv); we have used Lemma \ref{lemma:noisy:alpha:tau:trend:q} and DCT to obtain $(b)$. Finally, we put the results \eqref{starting:eq}, \eqref{eq:noisy:mse:sigma:trend:q1}, \eqref{eq:noisy:mse:sigma:trend:q2} and Lemma \ref{lemma:noisy:alpha:tau:trend:q} together to derive 
\begin{align*}
&\lim_{\tau\rightarrow \infty} \tau^4(R_q(\alpha,\tau)-\epsilon \mathbb{E}|G|^2/\tau^2) \\
=&\lim_{\tau\rightarrow \infty} \frac{\tau^4}{\alpha^{\frac{2}{q-1}}} \cdot \big[\lim_{\tau\rightarrow \infty} \alpha^{\frac{2}{q-1}}\mathbb{E}\eta_q^2(B/\tau+Z;\alpha) - 2\lim_{\tau\rightarrow \infty} \alpha^{\frac{2}{q-1}}\mathbb{E}(\partial_1\eta_q(\gamma B/\tau+Z;\alpha)B^2/\tau^2)  \big] \\
=& \left(\frac{q-1}{q^{\frac{1}{q-1}}}\frac{\mathbb{E}|Z|^{\frac{2}{q-1}}}{\mathbb{E}B^2\mathbb{E}|Z|^{\frac{2-q}{q-1}}}\right)^{-2}\cdot (q^{\frac{2}{1-q}} \mathbb{E}|Z|^{\frac{2}{q-1}}-2q^{\frac{2}{1-q}} \mathbb{E}|Z|^{\frac{2}{q-1}})
=-\frac{\epsilon^2 \big(\mathbb{E}|G|^2\mathbb{E}|Z|^{\frac{2-q}{q-1}}\big)^2}{(q-1)^2\mathbb{E}|Z|^{\frac{2}{q-1}}}.
\end{align*}
This finishes the proof.
\end{proof}

\subsubsection{Deriving the expansion of ${\rm AMSE}(q,\lambda^*_q)$ for $q\in (1,2]$} \label{amseexp:q12}

The way we derive the result \eqref{largenoiseresult:q12} of Theorem \ref{THM:NOISY:MAIN} is similar to that in Section \ref{amse:qeq1}.  We hence do not repeat all the details. The key step is applying Lemma \ref{boundrisk:rateq12} to obtain
\begin{align*}
    &\lim_{\sigma \rightarrow \infty} \sigma^2({\rm AMSE}(q,\lambda_q^*)-\epsilon \mathbb{E}|G|^2)=\lim_{\tau_* \rightarrow \infty} \sigma^2({\rm AMSE}(q,\lambda_q^*)-\epsilon \mathbb{E}|G|^2 \\
=&\lim_{\tau_* \rightarrow \infty}\frac{\sigma^2}{\tau_*^2} \cdot \lim_{\tau_* \rightarrow \infty} \tau^4_*(R_q(\alpha_q(\tau_*),\tau_*)-\epsilon \mathbb{E}|G|^2/\tau_*^2)
= -\epsilon^2 (\mathbb{E}|G|^2)^2c_q.
\end{align*}

\subsection{Proof of Theorem \ref{THM:NOISY:MAIN} for $q>2$}  \label{proof35:q2}

We aim to prove the same results as presented in Lemmas \ref{lemma:noisy:alpha:tau:trend:q} and \ref{boundrisk:rateq12}. However, many of the limits we took when proving for the case $1<q\leq 2$ become invalid for $q>2$ because DCT may not be applicable. Therefore, here we assume a slightly stronger condition that $G$ has a sub-Gaussian tail and use a different reasoning to validate the results in Lemmas \ref{lemma:noisy:alpha:tau:trend:q} and \ref{boundrisk:rateq12}. Throughout this section, we use $\alpha$ to denote $\alpha_q(\tau)$ for simplicity. First note that Lemma \ref{lemma:noisy:alpha:infty} holds for $q>2$ as well. Hence we already know $\alpha \rightarrow \infty$ as $\tau \rightarrow \infty$. The following key lemma paves our way for the proof.

\begin{lemma}\label{lemma:noisy:lq:general1}
Suppose function $h:\mathbb{R}^2\rightarrow\mathbb{R}$ satisfies $|h(x,y)|\leq C(|x|^{m_1}+|y|^{m_2})$ for some $C>0$ and $0\leq m_1, m_2<\infty$. $B$ has sub-Gaussian tail. Then the following result holds for any constants $v\geq 0, \gamma\in[0,1]$ and $q>2$,
\begin{equation} \label{eq:noisy:mse:sigma:trend:q5}
\lim_{\tau \rightarrow \infty} \alpha^{\frac{v+1}{q-1}}
\mathbb{E}\bigg[\frac{h(B,Z)|\eta_q(B/\tau+Z;\alpha)|^v}{1+\alpha q(q-1)|\eta_q(\gamma B/\tau+Z;\alpha)|^{q-2}}\bigg]
=
\frac{q^{\frac{-v-1}{q-1}}}{q-1}
\mathbb{E}[h(B,Z) |Z|^{\frac{v+2-q}{q-1}}],
\quad \text{as }\tau \rightarrow \infty.
\end{equation}
Moreover, there is a finite constant $K$ such that for sufficiently large $\tau$,
\begin{equation}
\max_{0\leq \gamma \leq 1}\alpha^{\frac{v+1}{q-1}}
\mathbb{E}\bigg[\frac{|h(B,Z)||\eta_q(B/\tau+Z;\alpha)|^v}{1+\alpha q(q-1)|\eta_q(\gamma B/\tau+Z;\alpha)|^{q-2}}\bigg]
\leq
K.
\label{eq:noisy:mse:sigma:trend:q6}
\end{equation}
\end{lemma}

\begin{proof}
Define $A=\{|\eta_q(\gamma B/\tau+Z;\alpha)|\leq\frac{1}{2}|\gamma
B/\tau+Z|\}$.

We evaluate the expectation on the set $A$ and its complement $A^c$ respectively. Recall we use $p_B$ to denote the distribution of $B$. By a change of variable we then have
\begin{eqnarray*}
&&\mathbb{E}\left[\frac{h(B,Z)|\eta_q(B/\tau+Z;\alpha)|^v}{1+\alpha q(q-1)|\eta_q(\gamma B/\tau+Z;\alpha)|^{q-2}}\right]   \\
&=&\int\frac{h(x,y-\gamma x/\tau)|\eta_q(y+(1-\gamma)x/\tau;\alpha)|^v}{1+\alpha q(q-1)|\eta_q(y;\alpha)|^{q-2}}\phi(y-\gamma x/\tau)dydp_B(x).
\end{eqnarray*}
We have on $\{|\eta_q(y;\alpha)|\leq\frac{1}{2}|y|\}$ when $\tau$ is large enough,
\begin{eqnarray*}
&&\hspace{.cm}\frac{\alpha^{ \frac{v+1}{q-1}}|h(x,y-\gamma x/\tau)|\cdot |\eta_q(y+(1-\gamma)x/\tau;\alpha)|^v}{1+\alpha q(q-1)|\eta_q(y;\alpha)|^{q-2}}\phi(y-\gamma x/\tau)\label{upperbound:qlarger2}   \\
&\overset{(a)}{\leq}&
\frac{|h(x,y-\gamma x/\tau)|\cdot |\alpha^{\frac{1}{q-1}}\eta_q(y+(1-\gamma)x/\tau;\alpha)|^v}{q(q-1)|\alpha^{\frac{1}{q-1}}\eta_q(y;\alpha)|^{q-2}}\phi(y-\gamma x/\tau)  \nonumber\\
&\overset{(b)}{=}&
\frac{q^{\frac{q-2-v}{q-1}}|h(x,y-\gamma x/\tau)| \cdot |y+(1-\gamma)x/\tau|^\frac{v}{q-1}}{q(q-1)(|y|-|\eta_q(y;\alpha)|)^{\frac{q-2}{q-1}}}\phi(y/\sqrt{2})e^{-\frac{1}{4}(y-\frac{2\gamma x}{\tau})^2+\frac{\gamma^2x^2}{2\tau^2}}  \nonumber\\
&\overset{(c)}{\leq}&
\frac{2^{\frac{q-2}{q-1}}q^{\frac{q-2-v}{q-1}}|h(x,y-\gamma x/\tau)|\cdot |y+(1-\gamma)x/\tau |^\frac{v}{q-1}}{q(q-1)\left|y\right|^{\frac{w}{q-1}}}\phi(y/\sqrt{2})e^{\frac{\gamma^2x^2}{2\tau^2}}  \nonumber\\
&\overset{(d)}{\leq}&
\frac{2^{\frac{q-2}{q-1}}q^{\frac{q-2-v}{q-1}}\left(|x|^{m_1}+(|y|+|x|)^{m_2}\right)\cdot (|y|+|x|)^\frac{v}{q-1}}{q(q-1) \left|y\right|^{\frac{q-2}{q-1}}}\phi(y/\sqrt{2})e^{c_0x^2}.  \nonumber 
\end{eqnarray*}
We have used Lemma \ref{prox:property} part (ii) to obtain $(a)(b)$; $(c)$ is
due to the condition $|\eta_q(y;\alpha)|\leq\frac{1}{2}|y|$; and $(d)$ holds
because of the condition on the function $h(x,y)$. Notice that the numerator of
the upper bound is essentially a polynomial in $|x|$ and $|y|$. Since $B$ has
sub-Gaussian tail, if we choose $c_0$ small enough (when $\tau$ is sufficiently
large), the integrability with respect to $x$ is guaranteed. The integrability w.r.t. $y$ is clear since $(2-q)/(q-1)>-1$. Thus we can apply DCT to obtain
\begin{eqnarray*}
&&\lim_{\tau\rightarrow \infty}\alpha^{ \frac{v+1}{q-1}}\mathbb{E}\left[\frac{h(B,Z)|\eta_q(B/\tau+Z;\alpha)|^v}{1+\alpha q(q-1)|\eta_q(\gamma B/\tau+Z;\alpha)|^{q-2}} \mathbb{I}_{A}\right]  \\
&&\hspace{-.9cm} =\int \lim_{\tau \rightarrow \infty} \frac{h(x,y-\gamma x/\tau)|\alpha^{\frac{1}{q-1}}\eta_q(y+(1-\gamma)x/\tau;\alpha)|^v}{\alpha^{\frac{1}{1-q}}+ q(q-1)|\alpha^{\frac{1}{q-1}}\eta_q(y;\alpha)|^{q-2}}\phi(y-\gamma x/\tau)\mathbb{I}_{\{|\eta_q(y;\alpha)|\leq\frac{1}{2}|y|\}}dydp_B(x)        \nonumber\\
&&\hspace{-.9cm} =
\int\frac{q^{\frac{-1-v}{q-1}}h(x,y)}{(q-1)|y|^{\frac{q-2-v}{q-1}}}\phi(y)dydp_B(x)
=
\frac{q^{\frac{-1-v}{q-1}}}{q-1}\mathbb{E}[h(B,Z)|Z|^{\frac{v+2-q}{q-1}}].
\end{eqnarray*}
We now evaluate the expectation on the event $A^c$. Note that $A^c$ implies
\begin{align*}
|\gamma B/\tau+Z|
=&
\alpha q|\eta_q(\gamma B/\tau+Z;\alpha)|^{q-1}+|\eta_q(\gamma
B/\tau+Z;\alpha)| \nonumber \\
>&
\frac{\alpha q}{2^{q-1}}|\gamma B/\tau+Z |^{q-1}+\frac{1}{2} |\gamma B/\tau+Z|.
\end{align*}

This implies $|\gamma B/\tau+Z|<2(\alpha q)^{\frac{1}{2-q}}$ on $A^c$. Hence we have the following bounds,
\begin{eqnarray*}
&&\hspace{-.2cm} \alpha^{ \frac{v+1}{q-1}}
\mathbb{E}\left[\frac{|h(B,Z)|\cdot |\eta_q(B/\tau+Z;\alpha)|^v}{1+\alpha q(q-1)|\eta_q(\gamma B/\tau+Z;\alpha)|^{q-2}}\mathbb{I}_{A^c}\right] \\
&&  \hspace{-.8cm} \leq
\alpha^{\frac{v+1}{q-1}}
\mathbb{E}(|h(B,Z)|\cdot |\eta_q(B/\tau+Z;\alpha)|^v\mathbb{I}_{A^c})  \nonumber\\
&&\hspace{-.8cm} \leq
\alpha^{\frac{1}{q-1}}\int_{|y|<2(\alpha q)^{\frac{1}{2-q}}} |h\big(x,y-\frac{\gamma x}{\tau}\big)|\cdot |\alpha^{\frac{1}{q-1}}\eta_q\big(y+\frac{1-\gamma}{\tau}x;\alpha\big)|^v\phi(y)e^{\frac{\gamma yx}{\tau}}dydp_B(x)\nonumber\\
&&\hspace{-.8cm} \overset{(e)}{\leq}
q^{\frac{v}{1-q}}\alpha^{\frac{1}{q-1}}\int_{|y|<2(\alpha q)^{\frac{1}{2-q}}}\left(|x|^{m_1}+(|y|+|x|)^{m_2}\right)(|y|+|x|)^{\frac{v}{q-1}}\phi(y)e^{\frac{2(\alpha q)^{\frac{1}{2-q}}}{\tau}x} dydp_B(x)\nonumber\\
&&\hspace{-.8cm} \leq
q^{\frac{v}{1-q}}\alpha^{\frac{1}{q-1}}\int_{|y|<2(\alpha q)^{\frac{1}{2-q}}} P(|x|, |y|)\phi(y)e^{\frac{2(\alpha q)^{\frac{1}{2-q}}}{\tau}x}dydp_B(x)\nonumber\\
&&\hspace{-.8cm}  \overset{(f)}{\leq}
c_1 \alpha^{\frac{1}{q-1}}\alpha^{\frac{1}{2-q}}\int \tilde{P}(|x|)e^{x}dp_B(x) \leq c_2 \alpha^{\frac{-1}{(q-1)(q-2)}} \rightarrow \infty \mbox{~as~}\tau \rightarrow \infty,   \nonumber
\end{eqnarray*} 
where $(e)$ is due to Lemma \ref{prox:property} part (ii) and condition on $h(x,y)$; $P(\cdot,\cdot), \tilde{P}(\cdot)$ are two polynomials; the extra term $\alpha^{\frac{1}{2-q}}$ in step $(f)$ is derived from the condition $|y|<2(\alpha q)^{\frac{1}{2-q}}$. We thus have finished the proof of \eqref{eq:noisy:mse:sigma:trend:q5}. Finally, note that the two upper bounds we derived do not depend on $\gamma$, hence \eqref{eq:noisy:mse:sigma:trend:q6} follows directly. 
\end{proof}

We are now ready to prove Theorem \ref{THM:NOISY:MAIN} for $q>2$. We will prove the results of Lemmas \ref{lemma:noisy:alpha:tau:trend:q} and \ref{boundrisk:rateq12} for $q>2$. After that the exactly same arguments presented in Section \ref{amseexp:q12} will close the proof. Since the basic idea of proving Lemmas \ref{lemma:noisy:alpha:tau:trend:q} and \ref{boundrisk:rateq12} for $q>2$ is the same as for the case $q\in (1,2]$, we do not detail out the entire proof and instead highlight the differences. The major difference is that we apply Lemma \ref{lemma:noisy:lq:general1} to make some of the limiting arguments valid in the case $q>2$. Adopting the same notations in Section \ref{riskrate:q12}, we list the settings in the use of Lemma \ref{lemma:noisy:lq:general1} below
\vspace{0.0cm}
\begin{itemize}
	\item
	Lemma \ref{lemma:noisy:alpha:tau:trend:q} $I_1$: set $h(x,y)=x^2, v=0, \gamma=1$.
	\vspace{0.2cm}
	\item
	Lemma \ref{lemma:noisy:alpha:tau:trend:q} $I_3$: set $h(x,y)=x\sgn(\frac{x}{\tau}+y), v=1, \gamma=1$. Note that the dependence of $h(x,y)$ on $\tau$ does not affect the result.
	\vspace{0.2cm}
	\item
	Lemma \ref{lemma:noisy:alpha:tau:trend:q} $I_2$: Notice we have
	\begin{align}
	\alpha^{\frac{q}{q-1}}\tau^2 I_2
	=&
	\alpha^{\frac{1}{q-1}}\tau\mathbb{E}\left[B(Z^2-1)\Big(\eta_q(Z;\alpha)+\frac{B}{\tau}\int_0^1\partial_1\eta_q(sB/\tau+Z;\alpha)ds\Big)\right]\nonumber\\
	=&\alpha^{\frac{1}{q-1}} \int_0^1\mathbb{E}\Big[B^2(Z^2-1)\partial_1\eta_q(sB/\tau+Z;\alpha)\Big]ds\nonumber\\
	=&\int_0^1\alpha^{\frac{1}{q-1}}\mathbb{E}\Bigg[\frac{B^2(Z^2-1)}{1+\alpha q(q-1)\big|\eta_q(sB/\tau+Z;\alpha)\big|^{q-2}}\Bigg]ds.   \nonumber
	\end{align}
	We have switched the integral and expectation in the second step above due to the integrability.  Set $h(x,y)=x^2(y^2-1), v=0, \gamma=s$; then by the bound \eqref{eq:noisy:mse:sigma:trend:q6} in Lemma \ref{lemma:noisy:lq:general1}, we can bring the limit $\tau \rightarrow \infty$ inside the above integral to obtain the result of $\lim_{\tau \rightarrow \infty} \alpha^{\frac{q}{q-1}}\tau^2 I_2$.
	\vspace{0.2cm}
	\item
	In Lemma \ref{boundrisk:rateq12}, we need to rebound the term $\mathbb{E}[\eta_q(B/\tau+Z;\alpha)B/\tau]$ in \eqref{starting:eq}. 
	\begin{align}
	\alpha^{\frac{1}{q-1}} \tau^2\mathbb{E}[\eta_q(B/\tau+Z;\alpha)B/\tau]
	=&
	\alpha^{\frac{1}{q-1}} \tau^2\mathbb{E}\bigg[\frac{B}{\tau}\Big(\eta_q(Z;\alpha)+\frac{B}{\tau}\int_0^1\partial_1\eta_q(sB/\tau+Z;\alpha)ds\Big)\bigg]\nonumber\\
	=&
	\int_0^1\alpha^{\frac{1}{q-1}} \mathbb{E}\bigg[\frac{B^2}{1+\alpha q(q-1)\big|\eta_q(sB/\tau+Z;\alpha)\big|}\bigg]ds\nonumber
	\end{align}
	We set $h(x,y)=x^2, v=0, \gamma=s$. The rest arguments are similar to the previous one.
\end{itemize}

\section{Proof of Theorems \ref{THM:SAMPLE:MAIN}} \label{apxc}

Since the roadmap of the proof is similar to that of Theorem \ref{THM:NOISY:MAIN}, we will not repeat it. We suggest the reader study Appendix \ref{thmlargenoise:main} before reading this appendix.

We remind the reader that in the large sample regime, we have scaled the noise term. Hence $\tau_*$ will satisfy
\begin{eqnarray}
\tau_*^2=\frac{\sigma^2}{\delta}+\frac{\tau_*^2R_q(\alpha_q(\tau_*),\tau_*)}{\delta}.  \label{scalestate}
\end{eqnarray}

We first derive the convergence rate of $\tau_*$ as $\delta \rightarrow \infty$. 

\begin{lemma}\label{lemma:large:delta:tau:1}

For a given $q\in [1,\infty)$, as $\delta \rightarrow \infty$,
\[
\tau_*^2=\frac{\sigma^2}{\delta}+o\left(\frac{1}{\delta}\right).
\]
\end{lemma}

\begin{proof}
Since $\alpha=\alpha_q(\tau_*)$ minimizes $R_q(\alpha, \tau_*)$, from \eqref{scalestate} we obtain
\begin{equation}
\delta(\tau_*^2-\sigma^2/\delta)\leq R_q(0, \tau_*)=\tau_*^2, \label{single:eqenough}
\end{equation}
which yields $\tau_*^2\leq \frac{\sigma^2}{\delta-1} \rightarrow 0$ as $\delta \rightarrow \infty$. This completes the proof.
\end{proof}
Lemma \ref{lemma:large:delta:tau:1} shows that $\tau_*\rightarrow 0$ as $\delta \rightarrow \infty$. Hence we need to characterize the convergence rate of $R_q(\alpha_q(\tau),\tau)$ as $\tau \rightarrow 0$. The results have been derived in the small noise regime analysis. We collect the results together in the next lemma.

\begin{lemma} \label{riskconvergence:all}
As $\tau \rightarrow 0$ we have
\begin{itemize}
\item[(1)] For $q=1$, assume $\mathbb{P}(|G|\geq \mu)=1$ with $\mu$ a positive constant and $\mathbb{E}|G|^2<\infty$, then
\begin{eqnarray*}
R_q(\alpha_q(\tau),\tau)-f(\chi_0)=O(\phi(\mu/\tau-\chi_0)),
\end{eqnarray*} 
where $\chi= \chi_0$ is the minimizer of $f(\chi)=(1-\epsilon)\mathbb{E}\eta^2_q(Z;\chi)+\epsilon(1+\chi^2)$.
\item[(2)] For  $1<q<2$, assume $\mathbb{P}(|G| \leq x)=O(x)$ (as $x \rightarrow 0$) and $\mathbb{E}|G|^2<\infty$, 
\begin{eqnarray*}
R_q(\alpha_q(\tau),\tau)=1-\frac{(1-\epsilon)^2(\mathbb{E}|Z|^q)^2}{\epsilon \mathbb{E}|G|^{2q-2}}\tau^{2q-2}+o(\tau^{2q-2}).
\end{eqnarray*}
\item[(3)] For $q>2$, assume $\mathbb{E}|G|^{2q-2}<\infty$, then
\begin{eqnarray*}
R_q(\alpha_q(\tau),\tau)=1-\tau^2\frac{\epsilon (q-1)^2(\mathbb{E}|G|^{q-2})^2}{\mathbb{E}|G|^{2q-2}}+o(\tau^2).
\end{eqnarray*}
\end{itemize}
\end{lemma}

\begin{proof}
Result (1) is Lemma 5 in \cite{weng2018overcoming}; Result (2) is Lemma 20 in \cite{weng2018overcoming}; Result (3) is Lemma I.2 in \cite{wang2017bridge}. 
\end{proof}

We now use the results in Lemmas \ref{lemma:large:delta:tau:1} and \ref{riskconvergence:all} to prove Theorem \ref{THM:SAMPLE:MAIN}. We only present the proof for $q\in (1,2)$. Similar arguments work for other values of $q$. By Theorem \ref{theorem:amp:bridge2} and \eqref{scalestate},
\begin{align*}
    &\delta^q({\rm AMSE}(q,\lambda_q^*)-\sigma^2/\delta)=\delta^q(\tau_*^2R_q(\alpha_q(\tau_*),\tau_*)-\tau_*^2+\tau^2_*R_q(\alpha_q(\tau_*),\tau_*)/\delta) \\
    =&\delta^q\tau_*^2(R_q(\alpha_q(\tau_*),\tau_*)-1)+\delta^{q-1} \tau_*^2R_q(\alpha_q(\tau_*),\tau_*)
\overset{(a)}{\rightarrow} -\sigma^{2q}\frac{(1-\epsilon)^2(\mathbb{E}|Z|^q)^2}{\epsilon \mathbb{E}|G|^{2q-2}},
\quad \text{as }\delta \rightarrow \infty .
\end{align*}
Step $(a)$ is due to Lemmas \ref{lemma:large:delta:tau:1} and \ref{riskconvergence:all} part (2). This finishes the proof.

\section{Proof of Theorems \ref{THEOREM:DEBIASING:VALID},
\ref{THEOREM:DEBIASING:ATTP}, \ref{DEBIASING:Q>1} and Lemma \ref{COMP:SIS}}
\label{apxd}

The proof of Theorems \ref{THEOREM:DEBIASING:VALID}, \ref{THEOREM:DEBIASING:ATTP} (\ref{DEBIASING:Q>1}) and Lemma \ref{COMP:SIS} can be found in Sections \ref{proof:debiasing}, \ref{debiased:comp} and \ref{location:lemma4.1} respectively.

\subsection{Proof of Theorem \ref{THEOREM:DEBIASING:VALID}}\label{proof:debiasing}

Since some technical details for $q=1$ and $q>1$ are different, we prove the two cases separately in Sections \ref{proofdebiasing:q1more} and \ref{proofdebiasing:q1} respectively.

\subsubsection{Proof of Theorem \ref{THEOREM:DEBIASING:VALID} for $q=1$}\label{proofdebiasing:q1}

In this section, we apply the approximate message passing (AMP) framework to prove the result for $\lasso$. We first briefly review the approximate message passing algorithm and state some relevant results that will be later used  in the proof. We then describe the main proof steps. 

\vspace{0.2cm}

\textbf{I. Approximate message passing algorithms}. \cite{bayati2012lasso} has utilized AMP theory to characterize the sharp asymptotic risk of $\lasso$. The authors considered a sequence of estimates $\beta^t\in\mathbb{R}^p$ generated from an approximate message passing algorithm with the following iterations (initialized at $\beta^0=0,z^0=y$):
\begin{eqnarray}
\beta^{t+1}&=&\eta_q(X^Tz^t+\beta^t;  \alpha \tau_t^{2-q}), \nonumber \\
z^t&=& y- X\beta^t +  \frac{1}{\delta}z^{t-1}\langle \partial_1 \eta_q(X^Tz^{t-1}+\beta^{t-1}; \alpha \tau_{t-1}^{2-q}) \rangle, \label{iter:beta}
\end{eqnarray}
where $\langle v \rangle = \frac{1}{p}\sum_{i=1}^p v_i$ denotes the average of a vector's components; $\alpha$ is the solution to Equations \eqref{state_evolution1} and \eqref{state_evolution2}; and $\tau_t$ satisfies ($\tau^2_0=\sigma^2+\mathbb{E}|B|^2/\delta$):
\begin{eqnarray}\label{eq:stateevoAMPellq}
\tau^2_{t+1}=\sigma_w^2+\frac{1}{\delta}\mathbb{E}[\eta_q(B+\tau_t Z;\alpha \tau_t^{2-q})-B]^2, ~~t \geq 0. \label{iter:tau}
\end{eqnarray}
The asymptotics of many quantities in AMP can be sharply characterized. We summarize some results of \cite{bayati2012lasso} that we will use in our proof.

\begin{theorem}[\cite{bayati2012lasso}]\label{tobeused:later}
Let $\{\beta(p),X(p),w(p)\}$ be a converging sequence, and $\psi:\mathbb{R}^2\rightarrow\mathbb{R}$ be a pseudo-Lipschitz function. For $q=1$, almost surely
\begin{eqnarray*}
&(i)& \lim_{t\rightarrow \infty} \lim_{p\rightarrow\infty} \frac{1}{p}\| \hat{\beta}(1,\lambda)-\beta^t\|_2^2=0, \\
&(ii)&\lim_{n\rightarrow \infty} \frac{1}{n}\|z^t\|^2_2=\tau_t^2, \quad \lim_{t\rightarrow }\lim_{n\rightarrow \infty}\frac{1}{n}\|z^t-z^{t-1}\|_2^2=0, \\
&(iii)&\lim_{n\rightarrow\infty}\frac{1}{p}\|\beta^t\|_0=\mathbb{P}(|B+\tau_tZ|>\alpha \tau_t), \\
&(iv)& \lim_{p\rightarrow \infty}\frac{1}{p}\sum_{i=1}^p \psi(\beta^t_i+(X^Tz^t)_i, \beta_i)=\mathbb{E}\psi(B+\tau_t Z, B),\\
&(v)&  \lim_{p\rightarrow \infty}\frac{1}{p}\sum_{i=1}^p \psi(\beta^t_i, \beta_i)=\mathbb{E}\psi(\eta_1(B+\tau_t Z), B),
\end{eqnarray*}
where $\hat{\beta}(1,\lambda)$ is the $\lasso$ solution and $\tau_t$ is defined in \eqref{eq:stateevoAMPellq}.
\end{theorem}

\textbf{II. Main proof steps}. We first have the following bounds:
\begin{align*}
    \small &\frac{1}{p}\big\|\hat{\beta}^\dagger(1, \lambda) - \beta^t - X^T\frac{y -
X\beta^t }{1 - \|\beta^t\|_0 / n}\big\|_2^2 \nonumber \\
    \leq& \underbrace{\frac{2}{p}\big\|\hat{\beta}(1, \lambda) -\beta^t \big\|_2^2}_{Q_1}
    + \underbrace{\frac{8}{p(1 - \mathbb{P}(|B+\tau Z|>\alpha\tau) / \delta)^2} \big\|X^TX(\hat{\beta}(1, \lambda) - \beta^t)\big\|_2^2 }_{Q_2} \nonumber\\
    &+
    \underbrace{\frac{8}{p}\big\|X^T(y - X\hat{\beta}(1, \lambda))\big\|_2^2\Big(\frac{1}{1 - \|\hat{\beta}(1, \lambda)\|_0 / n} - \frac{1}{1 - \mathbb{P}(|B+\tau Z|>\alpha\tau)/\delta}\Big)^2}_{Q_3}\nonumber\\
    &+
    \underbrace{\frac{8}{p}\big\|X^T(y - X \beta^t)\big\|_2^2\Big(\frac{1}{1 - \|\beta^t \|_0 / n} - \frac{1}{1 - \mathbb{P}(|B+\tau Z|>\alpha\tau)/\delta}\Big)^2}_{Q_4},
\end{align*}
where $(\alpha, \tau)$ is the solution to \eqref{state_evolution1} and \eqref{state_evolution2}. From Theorem \ref{tobeused:later} part (i), we know $\lim_{t\rightarrow \infty}\lim_{p\rightarrow \infty}Q_1=0, a.s.$. Since the largest singular value of $X$ is bounded almost surely \cite{bai1993limit}, we can also obtain $\lim_{t\rightarrow \infty}\lim_{p\rightarrow \infty}Q_2=0, a.s.$. Moreover, from Theorem \ref{theorem:amp:bridge2} we can easily see the term $\|X^T(y - X\hat{\beta}(1, \lambda))\|_2^2/p\leq 2\|X^TX(\hat{\beta}(1,\lambda)-\beta)\|_2^2/p+2\|X^Tw\|_2^2/p$ is almost surely bounded. Also we know from \cite{bogdan2013supplementary} that $\frac{1}{p}\|\hat{\beta}(1,\lambda)\|_0=\mathbb{P}(|B+\tau Z|>\alpha \tau), a.s$. Therefore, we obtain $\lim_{p\rightarrow \infty}Q_3=0, a.s$. Regarding $Q_4$, it is not hard to see from \eqref{eq:stateevoAMPellq} that $\tau_t\rightarrow \tau$ as $t \rightarrow \infty$.  Then a similar argument as for $Q_3$ combined with Theorem \ref{tobeused:later} parts (i)(iii) gives us $\lim_{t\rightarrow \infty}\lim_{p\rightarrow \infty}Q_4=0, a.s$. Above all we are able to derive almost surely
\begin{eqnarray}\label{keykey:one}
\lim_{t\rightarrow \infty} \lim_{p\rightarrow \infty} \frac{1}{p}\big\|\hat{\beta}^\dagger(1, \lambda) - \beta^t - X^T\frac{y - X\beta^t }{1 - \|\beta^t\|_0 / n}\big\|_2^2=0.
\end{eqnarray}
Next from Equation \eqref{iter:beta} we have the following,
\begin{eqnarray*}
X^Tz^t-X^T\frac{y - X\beta^t }{1 - \|\beta^t\|_0 / n}=X^T\frac{\|\beta^t\|_0/n(-z^t+z^{t-1}n/(p\delta))}{1-\|\beta^t\|_0/n}.
\end{eqnarray*}
Using the result of Theorem \ref{tobeused:later} part (ii) and $n/p\rightarrow \delta$, we can obtain
\begin{eqnarray}  \label{keykey:two}
\lim_{t\rightarrow \infty}\lim_{p\rightarrow \infty}\frac{1}{p}\|X^Tz^t-X^T\frac{y - X\beta^t }{1 - \|\beta^t\|_0 / n}\|_2^2=0, ~~a.s.
\end{eqnarray}
The results \eqref{keykey:one} and \eqref{keykey:two} together imply that
\begin{eqnarray*}
\lim_{t\rightarrow \infty}\lim_{p\rightarrow \infty}\frac{1}{p}\|\hat{\beta}^\dagger(1, \lambda)-\beta^t-X^Tz^t\|_2^2=0, ~~a.s.
\end{eqnarray*}
According to Theorem \ref{tobeused:later} part (iv),  for any bounded Lipschitz function $L(x):\mathbb{R}\rightarrow \mathbb{R}$,
$\lim_{p\rightarrow \infty} \frac{1}{p}\sum_{i=1}^p L(\beta^t_i+(X^Tz^t)_i-\beta_i)=\mathbb{E}L(\tau_tZ).$ Putting the last two results together, it is not hard to confirm 
\begin{eqnarray*}
\lim_{p\rightarrow \infty} \frac{1}{p}\sum_{i=1}^p L(\hat{\beta}_i^{\dagger}(1,\lambda)-\beta_i)=\mathbb{E}L(\tau Z).
\end{eqnarray*}
Hence, the empirical distribution of $\hat{\beta}^{\dagger}(1,\lambda)-\beta$ converges to the distribution of $\tau Z$.

\subsubsection{Proof of Theorem \ref{THEOREM:DEBIASING:VALID} for $q>1$}\label{proofdebiasing:q1more}

The proof idea for $q>1$ is the same as for $q=1$. However since the debiased estimator for $q>1$ takes a different form, we need take care of some subtle details. Recall the definition of $f(v,w)$, $\hat{\gamma}_{\lambda}$ in \eqref{eq:debiasing:lq}. We first obtain the bound:
\begin{align*}
    \small &\frac{1}{p}\big\|\hat{\beta}^\dagger(q, \lambda) - \beta^t - X^T\frac{y - X\beta^t }{1 - f(\beta^t,\alpha \tau_t^{2-q}) / n}\big\|_2^2 \nonumber \\
    \leq& \underbrace{\frac{2}{p}\big\|\hat{\beta}(q, \lambda) -\beta^t \big\|_2^2}_{Q_1}
    + \underbrace{\frac{8}{p(1 - f(\eta_q(B+\tau Z;\alpha \tau^{2-q}),\alpha \tau^{2-q}) / \delta)^2} \big\|X^TX(\hat{\beta}(q, \lambda) - \beta^t)\big\|_2^2 }_{Q_2} \nonumber\\
    &+ \underbrace{\frac{8}{p}\big\|X^T(y - X\hat{\beta}(q, \lambda))\big\|_2^2\Big(\frac{1}{1 - f(\hat{\beta}(q,\lambda),\hat{\gamma}_{\lambda}) / n} - \frac{1}{1 - f(\eta_q(B+\tau Z;\alpha \tau^{2-q}),\alpha \tau^{2-q}) / \delta}\Big)^2}_{Q_3}\nonumber\\
    &+ \underbrace{\frac{8}{p}\big\|X^T(y - X \beta^t)\big\|_2^2\Big(\frac{1}{1 - f(\beta^t,\alpha \tau_t^{2-q}) / n} - \frac{1}{1 -  f(\eta_q(B+\tau Z;\alpha \tau^{2-q}),\alpha \tau^{2-q})/\delta}\Big)^2}_{Q_4}, 
\end{align*}

As in the proof of $q=1$, we show that $Q_i (i=1,2,3,4)$ vanishes asymptotically. For that purpose we first note that Theorem \ref{tobeused:later} (except part (iii)) holds for $q>1$ as well. Hence the same argument for $q=1$ gives us $Q_1, Q_2 \overset{a.s.}{\rightarrow}0$. Regarding $Q_3$, by the facts that the empirical distribution of $\hat{\beta}(q,\lambda)$ converges weakly to the distribution of $\eta_q(B+\tau Z;\alpha \tau^{2-q})$ and $\frac{1}{1+\alpha \tau^{2-q}q(q-1)|x|^{q-2}}$ is a bounded continuous function of $x$, we have
\[
\lim_{p\rightarrow \infty} \frac{1}{p}f(\hat{\beta}(q,\lambda),\alpha \tau^{2-q})=f(\eta_q(B+\tau Z;\alpha \tau^{2-q}),\alpha \tau^{2-q}),~~a.s.
\]
Moreover, according to Lemma \ref{lemma:debiasing:tuning:converge} we obtain as $p\rightarrow \infty$,
\begin{eqnarray*}
\frac{1}{p}|f(\hat{\beta}(q,\lambda),\alpha \tau^{2-q})-f(\hat{\beta}(q,\lambda),\hat{\gamma}_{\lambda})|\leq \alpha^{-1}\tau^{q-2} |\hat{\gamma}_{\lambda}-\alpha \tau^{2-q}| \overset{a.s.}{\rightarrow} 0. 
\end{eqnarray*}
The last two results together lead to $Q_3\overset{a.s.}{\rightarrow} 0$. For $Q_4$, it is not hard to apply Theorem \ref{tobeused:later} part (v) and the fact $\tau_t \rightarrow \tau$ to show 
\[
\lim_{t\rightarrow \infty}\lim_{p\rightarrow \infty} \frac{1}{p} f(\beta^t,\alpha \tau_t^{2-q})=f(\eta_q(B+\tau Z;\alpha \tau^{2-q}),\alpha \tau^{2-q}), ~~~a.s.
\]
which implies $Q_4 \overset{a.s.}{\rightarrow} 0$. The rest of the proof is almost the same as the one for $q=1$. We hence do not repeat the arguments. 

\begin{lemma}\label{lemma:debiasing:tuning:converge}
For $\hat{\gamma_\lambda}$ defined in \eqref{eq:solvegamma}, as $p\rightarrow \infty$
\begin{equation*}
\hat{\gamma}_\lambda
\rightarrowas
 \alpha\tau^{2-q}.
\end{equation*}
\end{lemma}

\begin{proof}
Denote $a(\gamma) = \delta(1 - \frac{\lambda}{\gamma})$, $\hat{b}(\gamma)=\ave\Big[\frac{1}{1 + \gamma q(q - 1)|\hat{\beta}(q; \lambda)|^{q-2}}\Big], b(\gamma) = \mathbb{E}\Big[\frac{1}{1 + \gamma q(q - 1)|\eta_q(B+\tau Z; \alpha\tau^{2-q})|}\Big]$, and $\gamma_{\lambda}=\alpha \tau^{2-q}$. Clearly from \eqref{eq:solvegamma} and \eqref{state_evolution2}, $\hat{\gamma}_{\lambda}$ is the unique solution of $a(\gamma)=\hat{b}(\gamma)$, and $\gamma_{\lambda}$ is the unique solution of $a(\gamma)=b(\gamma)$

As a simple corollary of Theorem \ref{theorem:amp:bridge2}, almost surely the empirical distribution of $\hat{\beta}(q,\lambda)$ converges weakly to the distribution of $\eta_q(B+\tau Z;\alpha \tau^{2-q})$. As a result, for $h(x) = \frac{1}{1 + \gamma q(q - 1)|x|^{q-2}}$ which is bounded and continuous on $\mathbb{R}$,  we have almost surely
\begin{equation*}
\hat{b}(\gamma)\rightarrow b(\gamma), \mbox{~~as~} p \rightarrow \infty.
\end{equation*}
The above convergence is pointwise in $\gamma$. In fact we can obtain a stronger result. That is, there is a $\Omega_0\subset\Omega$ with $\mathbb{P}(\Omega_0) = 1$, such that for any $\omega\in\Omega_0$, $\hat{b}(\gamma,\omega)\rightarrow b(\gamma)$ for all $\gamma \geq 0$. (we can first construct $\Omega_0$ for $\gamma\in\mathbb{Q}$, then extend to $\mathbb{R}_+$ by continuity and monotonicity of $a(\gamma),\hat{b}(\gamma)$ and $b(\gamma)$.)

Now for any $\omega\in\Omega_0$, for any $\epsilon > 0$, consider the neighborhood $[\gamma_\lambda - \epsilon, \gamma_\lambda + \epsilon]$. Let $\eta_\epsilon = \min\{b(\gamma_\lambda - \epsilon) - a(\gamma_\lambda - \epsilon), a(\gamma_\lambda + \epsilon) - b(\gamma_\lambda + \epsilon)\}$. Monotonicity of $a(\gamma), b(\gamma)$ and uniqueness of the solution $\gamma_\lambda$ guarantee $\eta_\epsilon > 0$. At $\gamma_\lambda - \epsilon$ and $\gamma_\lambda + \epsilon$, we know as $p\rightarrow \infty$,
\begin{equation*}
\hat{b}(\gamma_\lambda - \epsilon, \omega) \rightarrow b(\gamma_\lambda - \epsilon),
\quad
\hat{b}(\gamma_\lambda + \epsilon, \omega) \rightarrow b(\gamma_\lambda + \epsilon).
\end{equation*}
Thus there exists $N_\epsilon(\omega)$, for any $p > N_\epsilon(\omega)$, 
\begin{equation*}
|\hat{b}(\gamma_\lambda - \epsilon, \omega) - b(\gamma_\lambda - \epsilon)| < \frac{\eta_\epsilon}{2},
\quad
|\hat{b}(\gamma_\lambda + \epsilon, \omega) - b(\gamma_\lambda + \epsilon)| < \frac{\eta_\epsilon}{2}.
\end{equation*}
By noticing the distance between $a(\gamma)$ and $b(\gamma)$ on the two end-points, we have $\hat{b}(\gamma_\lambda - \epsilon, \omega) - a(\gamma_\lambda - \epsilon) > \frac{\eta_\epsilon}{2}$ and $a(\gamma_\lambda + \epsilon) - \hat{b}(\gamma_\lambda + \epsilon, \omega) > \frac{\eta_{\epsilon}}{2}$. The monotonicity of the function $\hat{b}(\gamma,\omega)$ determines that $\hat{\gamma}_\lambda(\omega)\in(\gamma_\lambda - \epsilon, \gamma_\lambda + \epsilon)$, i.e., $|\hat{\gamma}_\lambda(\omega) - \gamma_\lambda| < \epsilon$. As a conclusion, we have $\hat{\gamma}_\lambda \rightarrowas \gamma_\lambda$.
\end{proof}

\subsection{Proof of Theorems \ref{THEOREM:DEBIASING:ATTP} and \ref{DEBIASING:Q>1}}\label{debiased:comp}

We only prove the case $q=1$. The proof for $q>1$ is similar. In the proof of Theorem \ref{THEOREM:DEBIASING:VALID}, we have showed 
\begin{eqnarray*}
\lim_{t\rightarrow \infty}\lim_{p \rightarrow \infty}\|\hat{\beta}^{\dagger}(1,\lambda)-(\beta^t+X^T z^t)\|_2=0, ~~a.s.
\end{eqnarray*}
Combining this result with Theorem \ref{tobeused:later} part (iv), we know for any bounded Lipschitz function $\psi: \mathbb{R}^2\rightarrow \mathbb{R}$
\[
\lim_{p\rightarrow \infty}\frac{1}{p}\sum_{i=1}^p\psi(\hat{\beta}^{\dagger}_i(1,\lambda),\beta_i)=\mathbb{E}\psi(B+\tau Z,B).
\]
Hence the empirical distribution of $(\hat{\beta}^\dagger(1, \lambda),\beta)$
converges weakly to the distribution of  $(B + \tau Z, B)$. We then follow the
same calculations as the proof of Lemma \ref{LEMMA:FDP:LQ} and obtain 
\begin{eqnarray}\label{eq:debias:atpp:afdp}
&&{\rm AFDP}^{\dagger}(1,\lambda, s)=\frac{(1-\epsilon)\mathbb{P}(\tau |Z|>s)}{(1-\epsilon)\mathbb{P}(\tau |Z|>s)+\epsilon \mathbb{P}(|G+\tau Z|>s)} \nonumber \\
&&{\rm ATPP}^\dagger(1, \lambda, s)=\mathbb{P}(|G+\tau Z| > s)
\end{eqnarray}
Notice that $\tau>0$ and $G + \tau Z$ is continuous. Thus as we vary $s$, ${\rm ATPP}^\dagger(1, \lambda, s)$ can reach all values in $[0, 1]$. Furthermore, by comparing the formula above with that in Lemma \ref{LEMMA:FDP:LQ} (when $q=1$), it is clear that ${\rm ATPP}^{\dagger}(1,\lambda, s)={\rm ATPP}(1,\lambda,\tilde{s})$ will imply ${\rm AFDP}^{\dagger}(1,\lambda, s)={\rm AFDP}(1,\lambda,\tilde{s})$. The proof for the second part is similar to that of Theorem \ref{THM:MSEMINFDRMIN}, and is skipped.

\subsection{Proof of Lemma \ref{COMP:SIS}} \label{location:lemma4.1}
SIS thresholds $X^Ty$, which is also the initialization of AMP in \eqref{iter:beta}. By setting $t=0$ in Theorem \ref{tobeused:later} (iv), we obtain
\begin{equation}
\lim_{p\rightarrow\infty}\sum_{i=1}^p\psi((X^T y)_i, \beta_i)
=
\mathbb{E}\psi(B + \tau_0 Z, B)
\end{equation}
where $\tau_0^2=\sigma^2 + \frac{\mathbb{E}\beta^2}{\delta} > 0$. This implies that almost surely the empirical distribution of $\{((X^Ty)_i, \beta_i)\}_{i=1}^p$ converges weakly to $(B+\tau_0 Z, B)$. Following the same argument as in the proof of Lemma \ref{LEMMA:FDP:LQ}, we have
\begin{eqnarray*}
&&{\rm AFDP}_{\rm sis}(s)=\frac{(1-\epsilon)\mathbb{P}(\tau_0 |Z|>s)}{(1-\epsilon)\mathbb{P}(\tau_0 |Z|>s)+\epsilon \mathbb{P}(|G+\tau_0 Z|>s)}   \\
&&{\rm ATPP}_{\rm sis}(s)=\mathbb{P}(|G+\tau_0 Z| > s)
\end{eqnarray*}

On the other hand, based on Equation \eqref{eq:debias:atpp:afdp}, we obtain
\begin{eqnarray*}
&&{\rm AFDP}^{\dagger}(q,\lambda_q^*, s)=\frac{(1-\epsilon)\mathbb{P}(\tau_* |Z|>s)}{(1-\epsilon)\mathbb{P}(\tau_* |Z|>s)+\epsilon \mathbb{P}(|G+\tau_* Z|>s)} \\
&&{\rm ATPP}^\dagger(q, \lambda_q^*, s)=\mathbb{P}(|G+\tau_* Z| > s)
\end{eqnarray*}

Note that
\begin{eqnarray*}
\tau_*^2
 \leq
\sigma^2 + \frac{1}{\delta}\lim_{\alpha\rightarrow\infty}\mathbb{E}\big(\eta_q(B + \tau_* Z; \alpha\tau_*^{2-q}) - B\big)^2
=
\tau_0^2.
\end{eqnarray*}

With the same argument as the proof of Theorem \ref{THM:MSEMINFDRMIN}, we have the conclusion follow.


\end{document}